\documentclass[11pt]{article}
\usepackage[margin=1in]{geometry}
\usepackage{enumitem}
\usepackage{placeins}
\usepackage{graphicx}
\usepackage{stmaryrd}
\usepackage{epstopdf}
\usepackage[utf8]{inputenc}
\usepackage[english]{babel}
\usepackage[square,numbers,sort&compress]{natbib} 
\usepackage{scalerel}[2016/12/29]
\usepackage{bbm}
\usepackage{amsmath,amsfonts,amsthm,amssymb}
\usepackage{fancyhdr}
\usepackage[bookmarksopen, bookmarksnumbered]{hyperref}
\usepackage[nottoc,numbib]{tocbibind}
\hypersetup{
	colorlinks=true,
	linkcolor=blue,
	citecolor=blue,
	filecolor=magenta,      
	urlcolor=cyan,
	linktoc=page
}

\theoremstyle{plain}
\newtheorem{theorem}{Theorem}[section]
\newtheorem{corollary}[theorem]{Corollary}
\newtheorem{conj}[theorem]{Conjecture} 
\newtheorem{prop}[theorem]{Proposition}

\newtheorem{lemma}[theorem]{Lemma}
\newtheorem{remark}[theorem]{Remark}

\theoremstyle{definition}
\newtheorem{definition}[theorem]{Definition}

\theoremstyle{plain}

\newcommand{\cL}{\mathcal{L}}
\newcommand{\cN}{\mathcal{N}}

\newcommand{\N}{\mathbb{N}}

\newcommand{\prob}{\mathbb{P}}
\newcommand{\p}{\mathbb{P}}
\newcommand{\E}{\mathbb{E}}

\newcommand{\floor}[1]{{\left\lfloor #1 \right\rfloor}}

\newcommand{\II}[1]{\left \llbracket #1 \right \rrbracket}

\newcommand{\sset}{\subset}

\newcommand{\al}{\alpha}
\newcommand{\Om}{\Omega}
\newcommand{\mathforall}{\text{ for all }}

\newcommand{\mathand}{\;\text{and}\;}
\newcommand{\mathfor}{\;\text{for}\;}

\newcommand{\ga}{\gamma}
\newcommand{\Ga}{\Gamma}
\newcommand{\ep}{\epsilon}
\newcommand{\om}{\omega}

\newcommand{\de}{\delta}
\newcommand{\ze}{\zeta}

\newcommand{\sig}{\sigma}

\newcommand{\Pass}{\mathsf{Pass}}
\newcommand{\Fav}{\mathsf{Fav}}
\newcommand{\Favl}{\mathsf{Fav!}}

\newcommand{\scrA}{\mathcal{A}}
\newcommand{\scrB}{\mathcal{B}}

\newcommand{\scrE}{\mathcal{E}}
\newcommand{\scrG}{\mathcal{G}}
\newcommand{\scrP}{\mathcal{P}}
\newcommand{\scrK}{\mathcal{K}}

\newcommand{\scrM}{\mathcal{M}}
\newcommand{\scrC}{\mathcal{C}}
\newcommand{\scrJ}{\mathcal{J}}

\newcommand{\scrR}{\mathcal{R}}
\newcommand{\scrL}{\mathcal{L}}
\newcommand{\scrH}{\mathcal{H}}
\newcommand{\scrS}{\mathcal{S}}

\newcommand{\scrF}{\mathcal{F}}

\newcommand{\fX}{\mathfrak{X}}

\newcommand{\Z}{\mathbb{Z}}

\newcommand{\R}{\mathbb{R}}
\newcommand{\Q}{\mathbb{Q}}

\usepackage{MnSymbol}

\DeclareMathOperator*{\argmax}{arg\,max}

\usepackage{yfonts}
\usepackage{ wasysym }

\newcommand{\eqd}{\stackrel{d}{=}}

\newcommand{\cvgd}{\stackrel{d}{\to}}

\newcommand{\X}{\times}

\newcommand{\cvgdown}{\downarrow}

\newcommand{\as}{\text{almost surely}}

\newenvironment{claim}[1]{\par\noindent\underline{Claim:}\space#1}{}

\newcommand{\smin}{\setminus}
\newcommand{\lf}{\left}
\newcommand{\rg}{\right}

\newcommand{\bx}{\mathbf{x}}
\newcommand{\bu}{\mathbf{u}}
\newcommand{\bv}{\mathbf{v}}
\newcommand{\bp}{\mathbf{p}}
\newcommand{\bq}{\mathbf{q}}
\newcommand{\ba}{\mathbf{a}}
\newcommand{\bb}{\mathbf{b}}

\newcommand{\by}{\mathbf{y}}
\newcommand{\bz}{\mathbf{z}}

\newcommand{\oo}{\mathfrak{o}}
\newcommand{\fg}{\mathfrak{g}}

\newcommand{\tbx}{\tilde{\bx}}
\newcommand{\tby}{\tilde{\by}}
\newcommand{\tbz}{\tilde{\bz}}

\newcommand{\tz}{\tilde{z}}
\newcommand{\tw}{\tilde{w}}
\newcommand{\tB}{\tilde{B}}

\newcommand{\Rd}{\mathbb{R}^4_\uparrow}

\usepackage{parskip}  
\usepackage[usenames]{color}

\usepackage{tikz}
\usetikzlibrary{arrows}
\usetikzlibrary{arrows.meta}
\definecolor{wwhhii}{rgb}{1.,1.,1.}
\definecolor{rreedd}{rgb}{1.,0.,0.}
\definecolor{uuuuuu}{rgb}{0.26666666666666666,0.26666666666666666,0.26666666666666666}

\title{Disjoint optimizers and the directed landscape}
\author{Duncan Dauvergne 
\thanks{Department of Mathematics, University of Toronto, e-mail: duncan.dauvergne@utoronto.ca}
\and Lingfu Zhang
\thanks{Department of Mathematics, Princeton University, e-mail: lingfuz@math.princeton.edu}
}

\begin{document}

\maketitle

\begin{abstract}
We study maximal length collections of disjoint paths, or `disjoint optimizers', in the directed landscape. We show that disjoint optimizers always exist, and that their lengths can be used to construct an extended directed landscape. The extended directed landscape can be built from an independent collection of extended Airy sheets, which we define from the parabolic Airy line ensemble. We show that the extended directed landscape and disjoint optimizers are scaling limits of the corresponding objects in Brownian last passage percolation (LPP). As two consequences of this work, we show that one direction of the Robinson-Schensted-Knuth bijection passes to the KPZ limit, and we find a criterion for geodesic disjointness in the directed landscape that uses only a single parabolic Airy line ensemble.

The proofs rely on a new notion of multi-point LPP across the parabolic Airy line ensemble, combinatorial properties of multi-point LPP, and probabilistic resampling ideas.
\end{abstract}

\tableofcontents

\FloatBarrier

\begin{figure}
    \centering
    \includegraphics{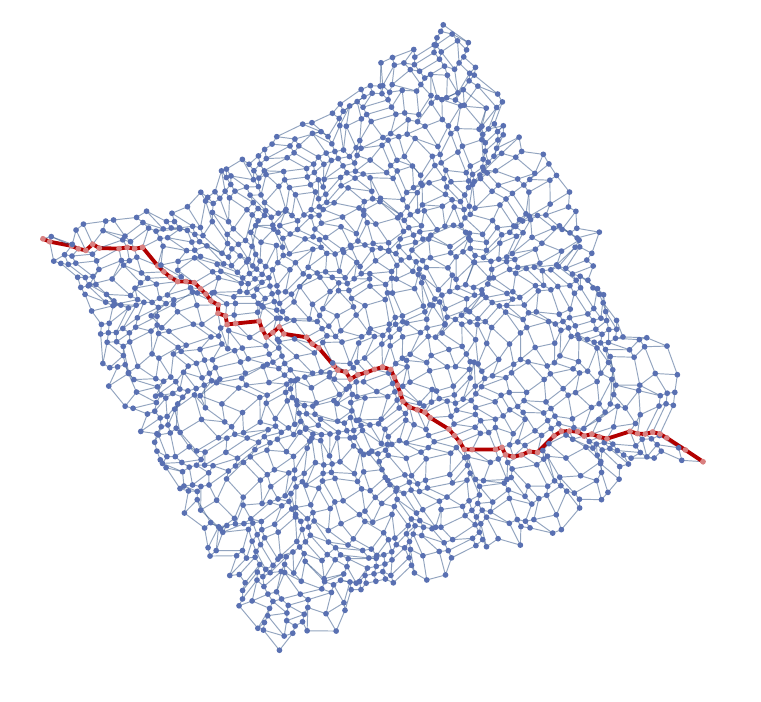}
    \caption{First passage percolation on $\Z^2$ with i.i.d.\ edge weights equal to $1$ or $2$ with equal probability. This model is expected to lie in the KPZ universality class. The red path is a geodesic and here the weighted graph is drawn using a spring electrical embedding (Mathematica). Figure generated by B\'alint Vir\'ag.}
    \label{fig:FPP}
\end{figure}

\section{Introduction}
\subsection{The KPZ universality class}

We start by considering a simple model of a random metric on the plane, see Figure \ref{fig:FPP} for an example. Take the lattice $\Z^2$, and assign i.i.d.\ positive edge weights to all edges. For vertices $p, q \in \Z^2$, let $d(p, q)$ denote graph distance in this randomly weighted graph. This model is an example of first passage percolation on $\Z^2$.

What does this metric look like as the points $p$ and $q$ become further and further away from each other? As long as the weight distribution is sufficiently nice, we expect the following picture. 
For concreteness, we take $p = (0,0)$ and $q = (0, n)$. At leading order, $d(p, q)$ grows proportionally to $n$ and the fluctuations of $d(p, q)$ are of order $O(n^{1/3})$. Moreover, geodesics from $p$ to $q$ are expected to only use edges in an $O(n^{2/3})$-window around the vertical axis. In particular, $d(p, q)$ will only have nontrivial correlations with another distance $d((x_1, y_1), (x_2, y_2))$ when $x_1, x_2 = O(n^{2/3})$.

This $n:n^{1/3}:n^{2/3}$ scaling (usually referred to as a $1:2:3$ scaling) is also observed in other models with an underlying random planar geometry including random interface growth (e.g. the $1+1$-dimensional Kardar-Parisi-Zhang (KPZ) equation, TASEP), random polymers, and last passage percolation. Collectively, these models are said to lie in the \textit{KPZ universality class}. See Section \ref{ss:relatedwork} for background on this area.

\subsection{The directed landscape}

Any model satisfying the $1:2:3$ scaling described above is expected to converge to a scaling limit. Indeed, for first passage percolation $d(\cdot, \cdot)$ on $\Z^2$ defined using a sufficiently nice edge weight distribution we expect that there is a four-parameter continuum object $\cL$ such that
\begin{equation}
\label{E:landscape-heuristic}
d((x n^{2/3}, sn), (yn^{2/3}, tn)) = c_1(t-s) n - n^{1/3} \cL(x, s; y, t) + o(n^{1/3})
\end{equation}
for some constant $c_1$.\footnote{The negative sign in front of the limit $\cL$ here is a convention in the field.} The continuum limit $\cL$ should be universal: up to possibly a linear rescaling of parameters, we should see the same limit regardless of our initial choice of weight distribution. 
Moreover, it should be the limit of all models in the KPZ universality class.

While the possibility of verifying \eqref{E:landscape-heuristic} for general first passage percolation is pure conjecture at this point, over the past twenty-five years a handful of integrable models of \textit{last passage percolation} have been discovered where establishing the existence of this scaling limit is possible. Indeed, recently Dauvergne, Ortmann, and Vir\'ag \cite{DOV} constructed the limit object $\cL$ and verified the analogue of \eqref{E:landscape-heuristic} for one particular integrable model: Brownian last passage percolation. They called the identified limit the \textbf{directed landscape}. This result was extended to other integrable models in \cite{dauvergne2021scaling}.
The papers \cite{DOV, dauvergne2021scaling} build on many previous results in the field. We refer the reader to Section \ref{ss:relatedwork} for more about of this history, and continue describing the directed landscape $\cL$.

The directed landscape $\mathcal L$ is a random continuous function from the parameter space 
$$
\R^4_\uparrow = \{u = (p; q) = (x, s; y, t) \in \R^4 : s < t\}
$$
to $\R$. 

As with first passage percolation, the value $\scrL(p; q) = \scrL(x, s; y, t)$ is best thought of as a distance between two points $p$ and $q$. Here we think of $x, y$ as spatial coordinates and $s, t$ as time coordinates. The two types of coordinates play very different roles, as is evident from the scaling in \eqref{E:landscape-heuristic}. The domain of $\scrL$ is not equal to all of $\R^4$ since the tilt of $c_1 (t-s) n$ in \eqref{E:landscape-heuristic} forces rescaled distances to tend to $-\infty$ if their time coordinates $s, t$ are in the wrong order. In other words, in the limit we cannot move backwards or instantaneously in time. 

Unlike an ordinary metric, $\scrL$ is not symmetric and may take negative values. It also satisfies the triangle inequality backwards (i.e. because of the sign convention in \eqref{E:landscape-heuristic}):
\begin{equation}
\label{E:triangle}
\scrL(p; r) \ge \scrL(p; q) + \scrL(q; r) \qquad \mbox{ for all }(p; r), (p; q), (q; r) \in \Rd.
\end{equation}
Just as in true metric spaces, we can define path lengths in $\scrL$, see \cite[Section 12]{DOV}. In the limiting setup, a \textbf{path} from $(x, s)$ to $(y, t)$ is a continuous function $\pi:[s, t] \to \R$ with $\pi(s) = x$ and $\pi(t) = y$. We can define the \textbf{length} of a path by
\begin{equation}
\label{E:length-L}
\|\pi\|_\scrL=\inf_{k\in \N}\inf_{s=t_0<t_1<\ldots<t_k=t}\sum_{i=1}^k\cL(\pi(t_{i-1}),t_{i-1};\pi(t_i),t_i)\,.
\end{equation}
This is analogous to defining the length of a curve in Euclidean space by piecewise linear approximation.
A path $\pi$ is a \textbf{directed geodesic}, or geodesic for brevity, if $\|\pi\|_\scrL$ is maximal among all paths with the same start and endpoints. Geodesics maximize, rather than minimize, path length because the triangle inequality \eqref{E:triangle} is backwards. Equivalently, a geodesic is any path $\pi$ with $\|\pi\|_\scrL = \scrL(\pi(s),s;\pi(t),t)$. Almost surely, directed geodesics exist between every pair of points $(x, s), (y, t)$ with $s < t$. Moreover, there is almost surely a unique geodesic between any fixed pair $(x, s), (y, t)$. In models which converge to the directed landscape, directed geodesics are limits of geodesics in the prelimiting metrics.

Before moving on to the main goal of this paper -- understanding disjoint optimizers in the directed landscape -- we mention a few relationships between the directed landscape and other well-known limit objects in order to better orient the reader. The first fluctuation limit theorem in the KPZ universality class is the Baik-Deift-Johansson theorem on the length of the longest increasing subsequence in a uniform permutation \cite{baik1999distribution}, see also \cite{johansson2000shape} for the same result in a different model. From our present perspective, these theorems show that for any fixed point $u = (x, s; y, t) \in \Rd$, we have
$$
\scrL(u) \eqd (t-s)^{1/3} T - \frac{(x-y)^2}{t-s},
$$
where $T$ is a Tracy-Widom GUE random variable.
In other words, these results identify the one-point distributions of $\scrL$. 

Shortly afterwards, Pr\"ahofer and Spohn \cite{prahofer2002scale} found a richer one-parameter scaling limit while studying the polynuclear growth model. Again, from our perspective this theorem amounts to identifying the process
$$
y \mapsto \scrL(0,0; y, t) + y^2
$$
as a stationary Airy$_2$ process. The work of Pr\"ahofer and Spohn also introduces a system of functions $\scrA =\{\scrA_i: \R \to \R, i \in \N\}$ with $\scrA_1$ being the Airy$_2$ process. This system was shown to consist of locally Brownian, ordered curves $\scrA_1 > \scrA_2 > \dots$ by Corwin and Hammond \cite{CH}. The collection $\scrA$ is called the \textbf{Airy line ensemble}, and it is the crucial integrable input needed in the construction of the directed landscape. In particular, the marginal $\scrL(\cdot, s; \cdot; t)$ is expressed in terms of an Airy line ensemble, see Section \ref{S:Airysheet} for more details.

Finally, there are exact formulas for marginals of the form $\scrL(0,0; y_i, t_i), i = 1, \dots, k$ from  Johansson and Rahman \citep{johansson2019multi} and Liu \citep{liu2019multi}, as well as formulas for the Markov process 
$$
\mathfrak h_t(y) = \max_{x \in \R} f(x) + \scrL(x, 0; y, t)
$$
for any upper semicontinuous function $f:\R \to \R \cup \{-\infty\}$ satisfying a certain growth condition. This Markov process is the KPZ fixed point, constructed by Matetski, Quastel, and Remenik \citep{matetski2016kpz}.

\subsection{Two perspectives on last passage percolation}

The goal of this paper is to study and understand the structure of maximal length collections of disjoint paths in the directed landscape. This may initially seem like a rather peripheral object of study in a random metric. However, it turns out that optimal collections of disjoint paths plays a central role in random planar geometry.
To properly motivate the study of these objects, we must first return to the prelimit and describe \textbf{Brownian last passage percolation} (henceforth Brownian LPP).

Let $f=\{f_i : i \in \Z\}$ be a sequence of continuous functions. For a nonincreasing cadlag function $\pi$ from $[x, y]$ to the integer interval $\II{m, n}$ with $\pi(y) = m$, henceforth a \textbf{path} from $(x, n)$ to $(y, m)$, define the length of $\pi$ with respect to $f$ by
$$
\|\pi\|_f = \sum_{i = m}^n f_i(z_i) - f_i(z_{i+1}).
$$ 
Here $z_m = y$ and for $i > m$, $z_i$ is the first time when $\pi$ is less than $i$. 
We remark that we use the indexing convention where the path $\pi$ is nonincreasing (rather than nondecreasing) to be consistent with \cite{DOV}; this indexing convention is also convenient to work with in the limit transition to the Airy line ensemble, see Figure \ref{fig:parabolic-path}.

For $x \le y$ and integers $m \le n$, define the \textbf{last passage value}
\begin{equation}
\label{E:BLPP-multi-intro}
f[(x, n) \to (y, m)] = \sup_{\pi} \|\pi\|_f,
\end{equation} 
where the supremum is over all paths from $(x, n)$ to $(y, m)$.
A function $\pi$ that achieves this supremum is called a \textbf{geodesic}. When the function $f$ is a collection of independent two-sided standard Brownian motions $B = \{B_i : i \in \Z\}$, this model is Brownian last passage percolation LPP
$$
(x, n; y, m) \mapsto B[(x, n) \to (y, m)].
$$
Going back to the work of Logan and Shepp \cite{logan1977variational} and Vershik and Kerov \cite{vershik1977asymptotics} on longest increasing subsequences, much of the progress on understanding integrable LPP models has come by understanding the Robinson-Schensted-Knuth (RSK) bijection. One direction of the classical RSK bijection maps an array of numbers to a pair of semistandard Young tableaux of the same shape. This pair of Young tableaux is built out of differences of certain \textit{multi-point} last passage values.
 In the context of last passage percolation across a sequence of functions $f = (f_1, \dots, f_n)$ with domain $[0, t]$, these multi-point last passage values are precisely the data
\begin{equation}
\label{E:f0k-to-yk}
f[(0^k, n) \to (y^k, m)] := \sup_\pi \sum_{i=1}^k \|\pi_i\|_f, \qquad  (y, m) \in [0, t] \X\{1\} \cup \{t\} \X \II{1, n}, k \in \II{1, n - m +1}.
\end{equation}
The supremum is over all $k$-tuples of disjoint paths $\pi = (\pi_1, \dots, \pi_k)$ from $(0, n)$ to $(y, m)$.
Here and throughout the paper we write $x^k = (x, \dots, x) \in \R^k$ for $x\in\R$.
In other words, one direction of this bijection records all multi-point last passage values from $(0,n)$, the bottom corner of the box $[0, t] \X \II{1, n}$, to points on the two far sides. It turns out that the whole function $f$ can be reconstructed from this data. Given the importance of the RSK bijection, it is natural to ask what becomes of it in the directed landscape limit, and how it relates to the finite RSK bijection.
	
On the nonintegrable side, going back at least to the work of Licea and Newman \cite{licea1996geodesics} on first passage percolation, the joint structure of geodesics in random metric models has been an object of fruitful study. Questions about geodesic coalescence and disjointness are closely linked with questions about limit shapes, fluctuation exponents, and the structure of shocks in related growth models. More recently, geodesic coalescence and disjointness have been studied in the more tractable context of integrable last passage percolation by using probabilistic and geometric techniques, e.g. see Hammond \citep{hammond2017exponents}; Pimentel \cite{pimentel2016duality}; Basu, Sarkar, Sly and Zhang \cite{basu2019coal, zhang2020optimal}; Bal{\'a}zs, Busani, Georgiou, Rassoul-Agha, Sepp\"al\"ainen, Shen~\citep{georgiou2017stationary, seppalainen2020coalescence, balazs2020local}. Questions of geodesic coalescence and disjointness still make sense in the directed landscape, and studying these reveals interesting probabilistic structures, e.g. see Bates, Ganguly, and Hammond \cite{bates2019hausdorff}.

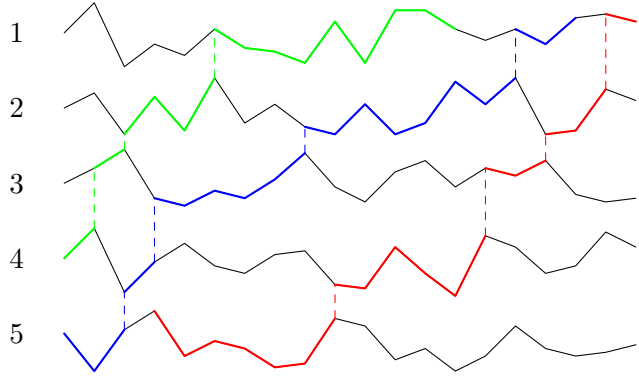
\begin{figure}
	\centering
	\begin{tikzpicture}[line cap=round,line join=round,>=triangle 45,x=4cm,y=5cm]
	\clip(-0.15,-0.15) rectangle (2.15,1.15);
	
	\draw (0.,0.1) node[anchor=east]{$5$};
	\draw (0.,0.3) node[anchor=east]{$4$};
	\draw (0.,0.5) node[anchor=east]{$3$};
	\draw (0.,0.7) node[anchor=east]{$2$};
	\draw (0.,0.9) node[anchor=east]{$1$};

	\draw plot coordinates {(0.1,0.1) (0.2,0.) (0.3,0.11) (0.4,0.16) (0.5,0.04) (0.6,0.08) (0.7,0.06) (0.8,0.01) (0.9,0.02) (1.,0.14) (1.1,0.12) (1.2,0.03) (1.3,0.06) (1.4,0.) (1.5,0.04) (1.6,0.12) (1.7,0.06) (1.8,0.04) (1.9,0.05) (2.,0.07) };
	\draw plot coordinates {(0.1,0.3) (0.2,0.38) (0.3,0.21) (0.4,0.29) (0.5,0.34) (0.6,0.28) (0.7,0.26) (0.8,0.31) (0.9,0.32) (1.,0.23) (1.1,0.22) (1.2,0.33) (1.3,0.26) (1.4,0.2) (1.5,0.36) (1.6,0.33) (1.7,0.26) (1.8,0.28) (1.9,0.37) (2.,0.33) };
	\draw plot coordinates {(0.1,0.5) (0.2,0.54) (0.3,0.59) (0.4,0.46) (0.5,0.44) (0.6,0.48) (0.7,0.46) (0.8,0.51) (0.9,0.58) (1.,0.49) (1.1,0.45) (1.2,0.53) (1.3,0.56) (1.4,0.49) (1.5,0.54) (1.6,0.52) (1.7,0.56) (1.8,0.47) (1.9,0.45) (2.,0.46) };
	\draw plot coordinates {(0.1,0.7) (0.2,0.74) (0.3,0.63) (0.4,0.73) (0.5,0.64) (0.6,0.78) (0.7,0.66) (0.8,0.71) (0.9,0.65) (1.,0.63) (1.1,0.71) (1.2,0.63) (1.3,0.66) (1.4,0.77) (1.5,0.71) (1.6,0.78) (1.7,0.63) (1.8,0.64) (1.9,0.75) (2.,0.72) };
	\draw plot coordinates {(0.1,0.9) (0.2,0.98) (0.3,0.81) (0.4,0.87) (0.5,0.84) (0.6,0.91) (0.7,0.86) (0.8,0.85) (0.9,0.82) (1.,0.93) (1.1,0.82) (1.2,0.96) (1.3,0.96) (1.4,0.91) (1.5,0.88) (1.6,0.91) (1.7,0.87) (1.8,0.94) (1.9,0.95) (2.,0.93) };

	\draw [thick] [red] plot coordinates {(0.4,0.16) (0.5,0.04) (0.6,0.08) (0.7,0.06) (0.8,0.01) (0.9,0.02) (1.,0.14)};
	\draw [dashed] [red] plot coordinates {(1.,0.14) (1.,0.23)};
	\draw [thick] [red] plot coordinates {(1.,0.23) (1.1,0.22) (1.2,0.33) (1.3,0.26) (1.4,0.2) (1.5,0.36)};
	\draw [dashed] [red] plot coordinates {(1.5,0.36) (1.5,0.54)};
	\draw [thick] [red] plot coordinates {(1.5,0.54) (1.6,0.52) (1.7,0.56)};
	\draw [dashed] [red] plot coordinates {(1.7,0.56) (1.7,0.63)};
	\draw [thick] [red] plot coordinates {(1.7,0.63) (1.8,0.64) (1.9,0.75)};
	\draw [dashed] [red] plot coordinates {(1.9,0.75) (1.9,0.95)};
	\draw [thick] [red] plot coordinates {(1.9,0.95) (2.,0.93)};
	
	\draw [thick] [blue] plot coordinates {(0.1,0.1) (0.2,0.) (0.3,0.11)};
	\draw [dashed] [blue] plot coordinates {(0.3,0.11) (0.3,0.21)};
	\draw [thick] [blue] plot coordinates {(0.3,0.21) (0.4,0.29)};
	\draw [dashed] [blue] plot coordinates {(0.4,0.29) (0.4,0.46)};
	\draw [thick] [blue] plot coordinates {(0.4,0.46) (0.5,0.44) (0.6,0.48) (0.7,0.46) (0.8,0.51) (0.9,0.58)};
	\draw [dashed] [blue] plot coordinates {(0.9,0.58) (0.9,0.65)};
	\draw [thick] [blue] plot coordinates {(0.9,0.65) (1.,0.63) (1.1,0.71) (1.2,0.63) (1.3,0.66) (1.4,0.77) (1.5,0.71) (1.6,0.78)};
	\draw [dashed] [blue] plot coordinates {(1.6,0.78) (1.6,0.91)};
	\draw [thick] [blue] plot coordinates {(1.6,0.91) (1.7,0.87) (1.8,0.94)};
	
	\draw [thick] [green] plot coordinates {(0.1,0.3) (0.2,0.38)};
	\draw [dashed] [green] plot coordinates {(0.2,0.38) (0.2,0.54)};
	\draw [thick] [green] plot coordinates {(0.2,0.54) (0.3,0.59)};
	\draw [dashed] [green] plot coordinates {(0.3,0.59) (0.3,0.63)};
	\draw [thick] [green] plot coordinates {(0.3,0.63) (0.4,0.73) (0.5,0.64) (0.6,0.78)};
	\draw [dashed] [green] plot coordinates {(0.6,0.78) (0.6,0.91)};
	\draw [thick] [green] plot coordinates {(0.6,0.91) (0.7,0.86) (0.8,0.85) (0.9,0.82) (1.,0.93) (1.1,0.82) (1.2,0.96) (1.3,0.96) (1.4,0.91)};
	
	\end{tikzpicture}
	\caption{A disjoint optimizer for $k=3$ from $((0,0,0.2),5)$ to $((0.7,0.9,1),1)$.}   \label{fig:dis-opt}
\end{figure}

One way to think about problems of geodesic disjointness and coalescence is in terms of certain \textbf{multi-point last passage values} that generalize \eqref{E:f0k-to-yk}. For collections of points $\bx = (x_1 \le x_2 \le \dots \le x_k)$ and $\by = (y_1 \le \dots \le y_k)$, define
\begin{equation}
\label{E:f0k-to-yk-2}
f[(\bx, n) \to (\by, m)] := \sup_\pi \sum_{i=1}^k \|\pi_i\|_f,
\end{equation}
where the supremum is over all $k$-tuples of disjoint paths $\pi = (\pi_1, \dots, \pi_k)$, where each $\pi_i$ goes from $(x_i, n)$ to $(y_i, m)$. We call a $k$-tuple $\pi$ that achieves this supremum a \textbf{disjoint optimizer}, abbreviated as optimizer. See Figure \ref{fig:dis-opt} for an example of these definitions and Section \ref{S:lppm} for a more precise setup. If there are disjoint geodesics $\pi_i$ from $x_i$ to $y_i$ for $i = 1, \dots, k$, then $f[(\bx, n) \to (\by, m)] = \sum_{i=1}^k f[(x_i, n) \to (y_i, m)]$. On the other hand, if for any collection of $k$ geodesics from $x_i$ to $y_i$, at least $2$ must coalesce on some interval, then $f[(\bx, n) \to (\by, m)] < \sum_{i=1}^k f[(x_i, n) \to (y_i, m)]$.

The following definition gives the analogue of multi-point last passage percolation in the directed landscape. This paper is devoted to studying this analogue, in order to shed light on both the limit of RSK and the structure of geodesic disjointness and coalescence in $\scrL$.

\begin{definition}
	\label{D:ext-land}
	Let $\fX_\uparrow$ be the space of all points $(\bx, s; \by, t)$, where $s < t$ and $\bx, \by$ lie in the same space $\R^k_\le = \{ \bx \in \R^k: x_1 \le \dots \le x_k \}$ for some $k \in \N$. For $(\bx, s; \by, t) \in \fX_\uparrow$, define
	\begin{equation}
	\label{E:Lextend}
	\scrL(\bx, s; \by, t) = \sup_{\pi_1, \dots, \pi_k} \sum_{i=1}^k \|\pi_i\|_\scrL.
	\end{equation}
	Here and throughout we use the convention that $k$ is such that $\bx, \by \in \R^k_\le$. The supremum is over all $k$-tuples of paths $\pi = (\pi_1, \dots, \pi_k)$ where each $\pi_i$ is a path from $(x_i, s)$ to $(y_i, t)$, and the paths satisfy the disjointness condition $\pi_i(r) \ne \pi_j(r)$ for all $i \ne j$ and $r \in (s, t)$. We call such a collection $\pi$ a \textbf{disjoint $k$-tuple} from $(\bx, s)$ to $(\by, t)$. We call the extension of $\scrL$ from $\fX_\uparrow \to \R \cup \{-\infty\}$ the \textbf{extended directed landscape}, abbreviated as extended landscape.
\end{definition}

See Figure \ref{fig:dis-tuple} for an illustration of Definition \ref{D:ext-land}. Note that $\Rd \sset \fX_\uparrow$, and since geodesics in $\scrL$ always exist, definition \eqref{E:Lextend} on $\Rd$ coincides with the usual definition of $\scrL$. In the course of this paper, we will show that:
\begin{enumerate}
	\item Just as the directed landscape is the limit of single-point Brownian LPP, the extended landscape is the scaling limit of multi-point Brownian LPP.
	\item For any $s < t$, the function $(\bx, \by) \mapsto \scrL(\bx, s; \by, t)$ can be expressed in terms of a more tractable object: the parabolic Airy line ensemble. This makes $\scrL(\cdot, s; \cdot, t)$ more amenable to probabilistic analysis.
	\item The supremum in \eqref{E:f0k-to-yk-2} is always attained, and so 
	$
	\scrL(\bx, s; \by, t) = \sum_{i=1}^k \scrL(x_i, s; y_i, t)
	$ 
	if and only if there are geodesics $\pi_i$ from $(x_i, s)$ to $(y_i, t), i = 1, \dots, k$ that are disjoint on $(s, t)$. When combined with point $2$, this gives a formula for understanding geodesic disjointness and coalescence that uses only a single parabolic Airy line ensemble.
	\item One direction of the RSK bijection passes to the limit.
\end{enumerate}

\begin{figure}
	\centering
	\begin{tikzpicture}[line cap=round,line join=round,>=triangle 45,x=10cm,y=5cm]
	\clip(-0.15,-0.15) rectangle (1.15,1.15);

	\draw [line width=.6pt] (-1.,0.) -- (2.,0.);
	\draw [line width=.6pt] (-1.,1.) -- (2.,1.);
	
	\draw (0.2,0) node[anchor=north]{$x_1=x_2$};
	\draw (0.35,0) node[anchor=north]{$x_3$};
	\draw (0.6,0) node[anchor=north]{$x_4=x_5=x_6$};
	\draw (0.8,0) node[anchor=north]{$x_7$};
	\draw (.2,1.) node[anchor=south]{$y_1$};
	\draw (.4,1.) node[anchor=south]{$y_2=y_3$};
	\draw (.6,1.) node[anchor=south]{$y_4=y_5$};
	\draw (.75,1.) node[anchor=south]{$y_6$};
	\draw (.95,1.) node[anchor=south]{$y_7$};
	
	\draw [red] plot coordinates {(0.2,0.) (0.09,0.1) (0.14,0.2) (0.24,0.3) (0.16,0.4) (0.2,0.5) (0.31,0.6) (0.13,0.7) (0.12,0.8) (0.21,0.9) (0.2,1.) };
	\draw [red] plot coordinates {(0.2,0.) (0.28,0.1) (0.3,0.2) (0.28,0.3) (0.26,0.4) (0.31,0.5) (0.39,0.6) (0.37,0.7) (0.41,0.8) (0.33,0.9) (0.4,1.) };
	\draw [red] plot coordinates {(0.35,0.) (0.43,0.1) (0.36,0.2) (0.38,0.3) (0.42,0.4) (0.4,0.5) (0.49,0.6) (0.42,0.7) (0.45,0.8) (0.37,0.9) (0.4,1.) };
	\draw [red] plot coordinates {(0.6,0.) (0.55,0.1) (0.46,0.2) (0.51,0.3) (0.48,0.4) (0.52,0.5) (0.59,0.6) (0.62,0.7) (0.54,0.8) (0.55,0.9) (0.6,1.) };
	\draw [red] plot coordinates {(0.6,0.) (0.67,0.1) (0.66,0.2) (0.61,0.3) (0.68,0.4) (0.58,0.5) (0.66,0.6) (0.72,0.7) (0.69,0.8) (0.67,0.9) (0.6,1.) };
	\draw [red] plot coordinates {(0.6,0.) (0.73,0.1) (0.76,0.2) (0.68,0.3) (0.71,0.4) (0.78,0.5) (0.86,0.6) (0.83,0.7) (0.72,0.8) (0.81,0.9) (0.75,1.) };
	\draw [red] plot coordinates {(0.8,0.) (0.79,0.1) (0.86,0.2) (0.71,0.3) (0.86,0.4) (0.98,0.5) (0.99,0.6) (0.85,0.7) (0.84,0.8) (0.91,0.9) (0.95,1.) };
	\draw [fill=uuuuuu] (0.2,0.) circle (1.0pt);
	\draw [fill=uuuuuu] (0.35,0.) circle (1.0pt);
	\draw [fill=uuuuuu] (0.6,0.) circle (1.0pt);
	\draw [fill=uuuuuu] (0.8,0.) circle (1.0pt);
	\draw [fill=uuuuuu] (0.2,1.) circle (1.0pt);
	\draw [fill=uuuuuu] (0.4,1.) circle (1.0pt);
	\draw [fill=uuuuuu] (0.6,1.) circle (1.0pt);
	\draw [fill=uuuuuu] (0.75,1.) circle (1.0pt);
	\draw [fill=uuuuuu] (0.95,1.) circle (1.0pt);
	\end{tikzpicture}
	\caption{A disjoint $k$-tuple.}   \label{fig:dis-tuple}
\end{figure}
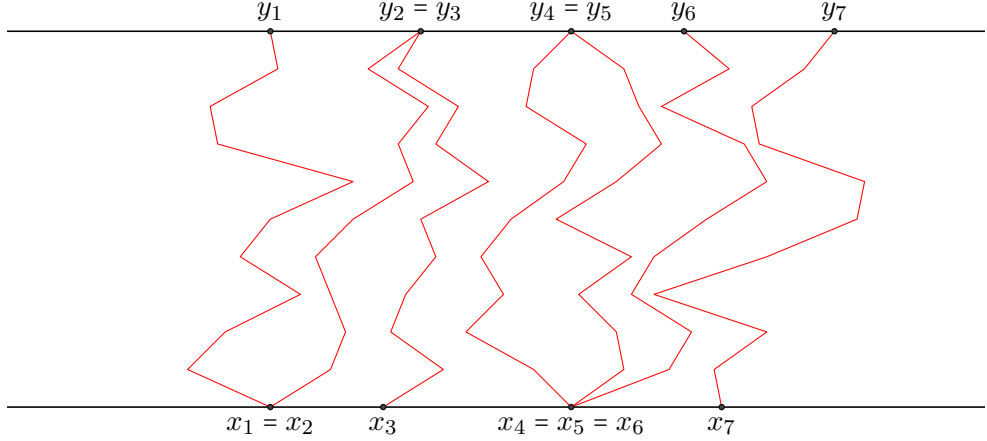

\subsection{Brownian LPP and the extended Airy sheet}
To understand the extended landscape, we need to go back to understand multi-point LPP in the prelimit.
We first focus on understanding the scaling limit of multi-point Brownian LPP from line $n$ to line $1$ as $n \to \infty$. 

\begin{theorem}
	\label{T:extended-sheet}
	Let $B = \{B_i : i \in \Z\}$ be a collection of independent two-sided standard Brownian motions. Let $\fX = \bigcup_{k =1}^\infty \R^k_\le \X \R^k_\le$. For $(\bx, \by) \in \R^k_\le \X \R^k_\le$, define
	$$
	\scrS^n(\bx, \by) = n^{1/6} \lf(B[(2n^{-1/3}\bx, n) \to (1 + 2n^{-1/3}\by, 1)] - 2k\sqrt{n} - n^{1/6} \sum_{i=1}^k 2(y_i - x_i) \rg),
	$$
	Then $\scrS^n \cvgd \scrS$ for some random continuous function $\scrS:\fX \to \R$. The underlying topology here is uniform convergence on compact subsets of $\fX$. The limit $\scrS$ is called the \textbf{extended Airy sheet}.
\end{theorem} 

Certain marginals of the extended Airy sheet are familiar. Indeed, let $\scrA = \{\scrA_i:\R \to \R, i \in \N\}$ be the Airy line ensemble and let $\scrB_i(x) = \scrA_i(x) - x^2$
be the \textbf{parabolic Airy line ensemble}. Then the system $\scrB = \{\scrB_i:\R \to \R, i \in \N\}$ can be coupled with $\scrS$ so that
\begin{equation}
\label{E:line-sheet-relation}
\sum_{i=1}^k \scrB_i(y) = \scrS(0^k, y^k)
\end{equation}
for all $k \in \N, y \in \R$.

The usual Airy sheet, constructed in \cite{DOV}, is given by $\scrS|_{\R^2}$. It is the scaling limit of single-point last passage values from line $n$ to line $1$. The construction of the Airy sheet in \cite{DOV} relies on showing that the \textbf{half-Airy sheet} $\scrS|_{[0, \infty) \X \R}$ is equal to $h(\scrB)$ for an explicit function $h$. The function $h$ is defined in terms of a last passage problem involving the parabolic Airy line ensemble, see Section \ref{S:Airysheet} and Section \ref{SS:outline} for some discussion of how this description arises from an identity in the prelimit. Our Theorem \ref{T:extended-sheet} also relies on characterizing $\scrS$ in terms of last passage percolation across $\scrB$. Doing so requires formalizing a notion of last passage percolation along infinite paths across $\scrB$.

\begin{figure}
    \centering
    \includegraphics{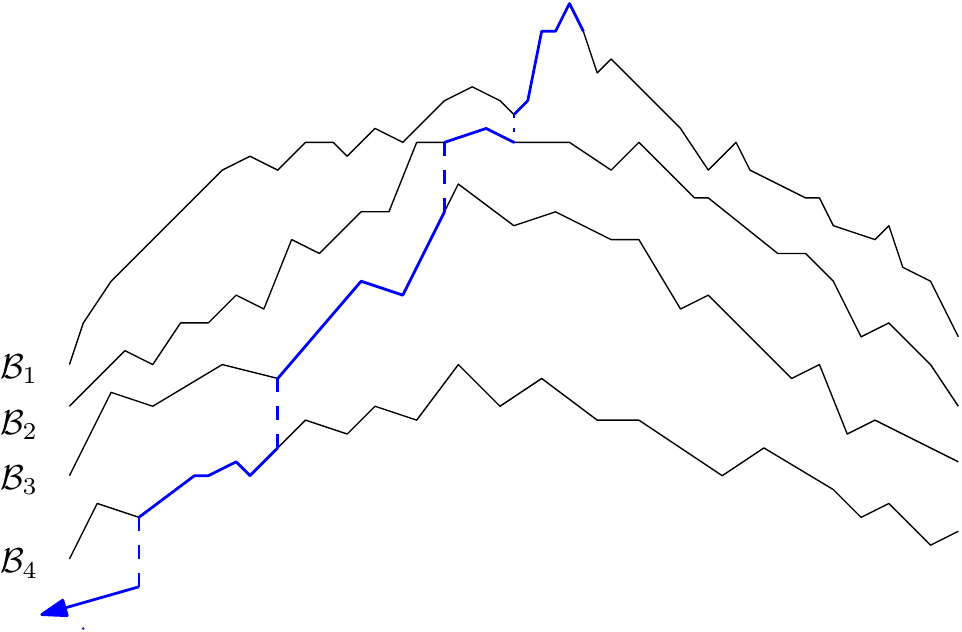}
    \caption{A parabolic path across the parabolic Airy line ensemble.}
    \label{fig:parabolic-path}
\end{figure}
For $x \in [0, \infty), z\in \R$, we say that a nonincreasing cadlag function $\pi:(-\infty, z] \to \N$ is a \textbf{parabolic path} from $x$ to $z$ if
$$
\lim_{y \to -\infty} \frac{\pi(y)}{2y^2} =x.
$$
See Figure \ref{fig:parabolic-path} for an example. For a parabolic Airy line ensemble $\scrB$ with corresponding half-Airy sheet $h(\scrB):[0, \infty) \X \R \to \R$, define the path length
$$
\|\pi\|_\scrB = h(\scrB)(x, z) + \lim_{y \to -\infty} \lf( \|\pi|_{[y,z]}\|_\scrB - \scrB[(y, \pi(y)) \to (z, 1)]\rg). 
$$
See Section \ref{S:lpp-across-ALE} for more context regarding this definition.
For $(\bx, \by) \in \fX$ with $x_1 \ge 0$, we can then define the (multi-point) \textbf{last passage value}
\begin{equation}
\label{E:Bxy}
\scrB[\bx \to \by] = \sup_{\pi_1, \dots \pi_k} \sum_{i=1}^k \|\pi_i\|_\scrB,
\end{equation}
where the supremum is over $k$-tuples of parabolic paths from $x_i$ to $y_i$ that are disjoint away from the right endpoints $y_i$. 

\begin{theorem}
	\label{T:extended-sheet-char}
The extended Airy sheet $\scrS$ satisfies the following properties:
\begin{itemize}
	\item $\scrS$ is shift invariant. More precisely, for $(\bx,\by) \in \fX$ and $c \in \R$, let $T_c(\bx,\by) = (x_1 + c, \dots, x_k + c, y_1+c, \dots, y_k+c)$. Then 
	$
	\scrS \eqd \scrS \circ T_c
	$
	for all $c \in \R$.
	\item $\scrS$ can be coupled with a parabolic Airy line ensemble $\scrB$ so that
	\begin{equation}
	\label{E:Sbxy-int}
	\scrS(\bx, \by) = \scrB[\bx \to \by]
	\end{equation}
	for all $(\bx, \by) \in \fX$ with $x_1 \ge 0$. 
\end{itemize}
Moreover, the law of $\scrS$ is the unique distribution on continuous functions on $\fX$ satisfying these properties.
\end{theorem}
The parabolic Airy line ensemble $\scrB$ in the coupling in Theorem \ref{T:extended-sheet-char} can be recovered from $\scrS$ via \eqref{E:line-sheet-relation}. While the definition of path length and last passage percolation across $\scrB$ are fairly involved, they are still workable. In Sections \ref{S:lpp-across-ALE} and \ref{S:limits-of-optimizers} we prove basic properties of these structures that help make \eqref{E:Sbxy-int} a useful respresentation of the extended Airy sheet. As part of this work, we show that $\scrS(0^k, \by)$ has a particularly accessible structure depending only on the top $k$ lines of $\scrB$ in the compact set $[y_1, y_k]$ (see Proposition \ref{P:high-paths-B}).
We also prove certain symmetries of $\scrS$ (Lemma \ref{L:basic-sym}), a two-point tail bound (Lemma \ref{l:change-spatial}) that shows $\scrS$ is H\"older-$(1/2)^-$, and a metric composition law (Proposition \ref{P:metric-composition-ex-sheet}). 

\subsection{The full scaling limit of multi-point Brownian LPP} In \cite{DOV}, the directed landscape is built out of independent Airy sheets via a metric composition law inherited from Brownian LPP. The authors then show that this describes the full scaling limit of single-point Brownian LPP. A similar procedure allows us to quickly construct the full scaling limit of multi-point Brownian LPP. For this next theorem, we say $\scrS_s$ is an extended Airy sheet of scale $s$ if
$$
\scrS_s(\bx, \by) \eqd s \scrS(s^{-2}\bx, s^{-2} \by)
$$
jointly in all $\bx, \by$.
\begin{theorem}
	\label{T:unique-L*}
There is a unique (in law) random continuous function $\scrL^*:\fX_\uparrow \to \R$ such that 
\begin{itemize}
	\item For any $(\bx, s; \by, t) \in \fX_\uparrow$ and $r \in (s, t)$, almost surely, 
	$$
	\scrL^*(\bx, s; \by, t) = \max_{\bz} \scrL^*(\bx, s; \bz, r) + \scrL^*(\bz, r; \by, t).
	$$
	Here the maximum is over all $\bz \in \R^k_\le$, where $k$ is the cardinality of $\bx$ and $\by$.
	\item For any finite collection of disjoint time intervals $(t_i, t_i + s_i^3)$, the functions $\scrL^*(\cdot, t_i; \cdot; t_i + s_i^3)$ are independent extended Airy sheets of scale $s_i$.
\end{itemize}

\end{theorem}
\begin{theorem}
\label{T:BLPP-convergence}
Let $(\bx, s)_n = (s+2\bx n^{-1/3}, - \floor{sn})$, and define
\begin{equation}
\scrL_n(\bx,t;\by,s) = n^{1/6} \left(
B[(\bx, s)_n\to (\by, t)_n] - 2k(t-s)\sqrt{n} - n^{1/6}\sum_{i=1}^k 2(y_i-x_i)
\right).
\end{equation}
 Then $\scrL_n \cvgd \scrL^*$, with $\scrL^*$ as in Theorem \ref{T:unique-L*}. Here the underlying topology is uniform convergence on compact subsets of $\fX_\uparrow$.
\end{theorem}

We can think of Theorems \ref{T:unique-L*} and \ref{T:BLPP-convergence} as an alternate way of constructing an extended directed landscape by first going back to the prelimit. The advantage of having done this is that the definition of the extended Airy sheet that underlies $\scrL^*$ is much more tractable than Definition \ref{D:ext-land} for $\scrL$. However, it is not clear from their constructions that $\scrL$ and $\scrL^*$ represent the same object. Much of the second half of the paper is devoted to showing this. 

\begin{theorem}\label{T:extended-landscape}
$\scrL^* = \scrL$.
\end{theorem}

The key difficulty in proving Theorem \ref{T:extended-landscape} is in showing that disjoint optimizers in \eqref{E:f0k-to-yk-2} remain disjoint after passing to the limit. As an upshot of the proof of this fact, we show that the supremum \eqref{E:Lextend} is always attained.

\begin{theorem}
	\label{T:disjoint-optimizers-in-L}
	Almost surely, the supremum in \eqref{E:Lextend} is attained for every $\bu = (\bx, s; \by, t) \in \fX_\uparrow$ by some disjoint $k$-tuple $\pi$. We call $\pi$ a \textbf{disjoint optimizer} for $\bu$ in $\scrL$. Moreover, for any fixed $\bu \in \fX_\uparrow$, almost surely there is a unique disjoint optimizer $\pi_\bu$ for $\bu$ in $\scrL$.
\end{theorem} 

Given that $\scrL^* = \scrL$, we can show that optimizers in the prelimit converge to optimizers in the limit. This theorem is the analogue of \cite[Theorem 1.8]{DOV}.
\begin{theorem}
	\label{T:limit-theorem}
	With $\scrL_n$ and $\scrL^* = \scrL$ as in Theorem \ref{T:BLPP-convergence}, consider a coupling where $\scrL_n \to \scrL$ almost surely uniformly on compact subsets of $\fX_\uparrow$. For $\bu = (\bx, s; \by, t) \in \fX_\uparrow$, let $C_\bu$ be the set of probability $1$ where there is an unique disjoint optimizer $\pi = (\pi_1, \dots, \pi_k)$ for $\bu$ in $\scrL$.
	
	In this coupling, there exists a set $\Om$ of probability $1$, such that the following holds. Consider any sequence of points $\bu_n = (\ba_n, m_n; \bb_n, \ell_n)$ which rescale to $\bu$ in the setup of Theorem \ref{T:BLPP-convergence}. That is, 
	$$
	\lf(\frac{n^{-2/3} m_n + n^{1/3} \ba_n}2, -\frac{m_n}n ; \frac{n^{-2/3} \ell_n + n^{1/3} \bb_n}2, -\frac{\ell_n}n \rg) \to \bu.
	$$
	Also consider any sequence of disjoint optimizers $\pi^{(n)} = (\pi^{(n)}_1, \dots, \pi^{(n)}_k)$ for $\bu_n$ across the Brownian motions that give rise to $\scrL_n$. Let $h_{n, i}$ be the order-preserving, linear function mapping $[s, t]$ onto $[a_{n, i}, b_{n, i}]$. Then on $\Om \cap C_\bu$, for all $1 \le i \le k$, we have
	$$
\frac{\pi^{(n)}_i \circ h_{n, i} + n h_{n, i}}{2 n^{2/3}} \to \pi_i 
	$$ 
	uniformly as functions from $[s, t]$ to $\R$.
\end{theorem}

In our exploration of the extended landscape, we also find continuity properties analogously to known properties for the directed landscape. The extended landscape is H\"older-$(1/3)^-$ in time (a consequence of Lemma \ref{l:change-time}), H\"older-$(1/2)^-$ in space (a consequence of Lemma \ref{l:change-spatial}) and its optimizers are H\"older-$(2/3)^-$ (a consequence of Lemma \ref{L:transfluc}).

\subsection{Consequences}
\label{S:consequence}

The structure of the extended landscape established in the previous theorems allows us to use the object to understand the limiting analogue of the RSK bijection, and the structure of geodesic disjointness and coalescence. 

We start with the RSK bijection. If we apply the RSK bijection to a random array or a sequence of continuous functions, then the KPZ scaling limit of the resulting pair of Young tableaux is a single parabolic Airy line ensemble $\scrB$. On the other hand, the KPZ scaling limit of the array itself is the directed landscape, with times restricted to the interval $[0, 1]$. 

As a consequence of our work, we show that the limiting parabolic Airy line ensemble can be reconstructed from the directed landscape restricted to times in $[0, 1]$ via the natural limiting analogue of RSK. This shows that one direction of the RSK bijection survives into the limit.

\begin{corollary}
	\label{C:rsk}
	Let $\scrL$ be the directed landscape restricted to the set $\{(x,s ; y, t) : x, y \in \R, 0 \le s < t \le 1\} \sset \Rd$. Then there is a function $f$ such that
	$
	f(\scrL) = \scrB,
	$
	where $\scrB$ is a parabolic Airy line ensemble. More precisely,
	$$
	\sum_{i=1}^k \scrB_i(y) = \scrL(0^k, 0; y^k, 1),
	$$
	where the right-hand side is an extended landscape value defined from $\scrL$ as in Definition \ref{D:ext-land}.
\end{corollary}

It is natural to ask whether the RSK map in Corollary \ref{C:rsk} is still invertible in the limit. We believe that almost surely, this is the case.

\begin{conj}
	\label{Q:rsk} There is an analogue of the RSK bijection in the KPZ limit.
	More precisely, let $f$ be as in Corollary \ref{C:rsk}, let $\scrB$ be a parabolic Airy line ensemble and let $\scrL$ be a directed landscape restricted to times in the interval $[0, 1]$. Then there exists a function $g$ such that almost surely,
	$f \circ g(\scrB) = \scrB$ and  $g \circ f(\scrL) = \scrL$.
\end{conj}

While we expect that such a function $g$ exists, we do not expect it to resemble the inverse of the usual RSK bijection; this inverse no longer makes sense in the limit. Rather, we believe that such a $g$ should exist because of certain almost sure probabilistic properties of $\scrL$ (e.g. laws of large numbers, $0-1$ laws).

Our work on the extended landscape gives the following criterion for geodesic disjointness and coalescence.

\begin{corollary}
\label{C:disjointness}
Almost surely the following holds.
For every $(\bx, s; \by, t) \in \fX_\uparrow$,
\begin{equation}
\label{E:Lxsys}
\cL(\bx, s; \by, t) = \sum_{i=1}^k \cL(x_i, s; y_i, t)
\end{equation}
if and only if there exist $\cL$-geodesics $\pi_1, \dots, \pi_k$ where $\pi_i$ goes from $(x_i, s)$ to $(y_i, t)$, satisfying $\pi_i(r) < \pi_{i+1}(r)$ for all $i \in \II{1, k-1}$ and $r \in (s, t)$.
\end{corollary}

For a fixed $s, t$, equation \eqref{E:Lxsys} is an equation about a single extended Airy sheet. In particular, by Theorem \ref{T:extended-sheet-char} it can be tackled by understanding a last passage problem across the parabolic Airy line ensemble $\scrB$. Because of the semi-discrete and locally Brownian nature of $\scrB$, understanding this problem is easier than understanding geodesic disjointness and coalescence in $\scrL$ directly.

\subsection{More related work}  \label{ss:relatedwork}
We do not attempt to give a full history of work on the KPZ universality class, and will instead focus on the circle of ideas most closely related to the present work.
For a gentle introduction to the KPZ universality class suitable for a newcomer to the area, see Romik \citep{romik2015surprising}. Review articles and books focusing on more recent developments include Corwin \citep{corwin2016kardar}; Ganguly \citep{ganguly2021random}; Ferrari and Spohn \citep{ferrari2010random}; Quastel \citep{quastel2011introduction}; Weiss, Ferrari, and Spohn \citep{weiss2017reflected}; and Zygouras \citep{zygouras2018some}.

Many of the initial breakthroughs in the area of KPZ relied on understanding integrable models via the RSK bijection. These include the previously discussed papers of \citep{baik1999distribution, johansson2000shape, prahofer2002scale} that establish Tracy-Widom convergence and Airy process convergence. A connection between the RSK correspondence and systems of nonintersecting random walks or Brownian motions was later discovered and understood in a series of papers by O'Connell and coauthors \cite{o2001brownian, o2002random, biane2005littelmann, konig2002non}. A description of the RSK correspondence in \cite{biane2005littelmann} leads to an identity for last passage percolation, see \eqref{E:fWf} below, which is the prelimiting version of the description of the Airy sheet in terms of the Airy line ensemble. Going beyond RSK, newer integrable ideas have yielded a richer set of formulas for limit objects, e.g. see  \citep{matetski2016kpz, johansson2019multi, liu2019multi, quastel2022kp}.

The works discussed above provide a strong integrable framework for understanding the directed landscape. More recently, probabilistic and geometric methods have been used in conjunction with a few key integrable inputs to prove regularity results, convergence statements, and exponent estimates in such models.

Corwin and Hammond \citep{CH} showed that the parabolic Airy line ensemble $\scrB$ satisfies a certain Brownian Gibbs property, making it amenable to probabilistic analysis. Hammond \citep{hammond2016brownian, hammond2019modulus, hammond2019patchwork}; Dauvergne and Vir\'ag \citep{DV}; and Calvert, Hammond, and Hegde \citep{CHH20} used Brownian Gibbs analysis to quantitatively understand the Brownian nature of the parabolic Airy line ensemble. The parabolic Airy line ensemble plays a central role in our paper, and we will require several consequences of this research program. Having a strong understanding of the Brownian nature of $\scrB$ is what makes results like Theorem \ref{T:extended-sheet} and Corollary \ref{C:disjointness} useful in practice.

There are many other papers that use Brownian Gibbs analysis and related ideas to study the structure of geodesics, near geodesics, and disjoint optimizers in the directed landscape and other last passage models. Some prominent recent examples include Hammond \citep{hammond2017exponents}; Ganguly and Hammond \citep{ganguly2020geometry, ganguly2020stability}; Basu, Ganguly, and Zhang \cite{basu2019temporal}; Sarkar, Dauvergne, and Vir\'ag \cite{DSV}; and Bates, Ganguly, and Hammond \cite{bates2019hausdorff}. 

Beyond \cite{DOV}, perhaps the two papers most closely linked with our own are \cite{SV20} and \cite{basu2020interlacing}. In \cite{SV20}, Sarkar and Vir\'ag show Brownian absolute continuity of the KPZ fixed point. One key idea in their work is to construct infinite last passage geodesics across the parabolic Airy line ensemble. Their setup for doing this is different than the setup we require for Theorem \ref{T:extended-sheet}, but still based around the Airy sheet construction in \cite{DOV}. In \cite{basu2020interlacing}, Basu, Ganguly, Hammond, and Hegde study the geometry of disjoint optimizers between $k$ identical start and endpoints for lattice last passage models, or ``geodesic watermelons''. They find scaling exponents in $k$ for the total length and transversal fluctuations of these optimizers.

Results, techniques, and frameworks developed in this paper
have already been used to analyze the Airy sheet, and geodesics across the parabolic Airy line ensemble and in the directed landscape. For example, the work \cite{dauvergne2021last} which came out after the first version of this paper appeared online, uses the framework of this paper to relate marginals of the Airy sheet to marginals in Brownian last passage percolation.
The work \cite{GZfract} (which also came out after the first version of this paper) about fractal geometry in the directed landscape requires an understanding of coalescence and disjointness between various pairs of landscape geodesics, which are equivalent to relations between infinite last passage geodesics across the parabolic Airy line ensemble. 
The analysis in \cite{GZfract} uses the framework of last passage percolation across the parabolic Airy line ensemble in Section \ref{S:lpp-across-ALE}, and the existence of disjoint optimizers across the parabolic Airy line ensemble proven in Section \ref{S:limits-of-optimizers}. Corollary \ref{C:disjointness} and related ideas will be used to analyze disjointness of $\scrL$-geodesics in the forthcoming work \cite{DVgeodesics}.

\subsection{Outline of the paper and a primer about the proofs}
\label{SS:outline}

While the structure of the paper is similar to \cite{DOV}, the proofs are mostly distinct. Indeed, the main difficulties that were resolved in \cite{DOV} yield lemmas that can be applied immediately here without need for generalization. 
As a consequence, the main difficulties in our work are unique to the multi-point setting and require different types of ideas. In this outline, we emphasize the differences between the two papers and some of the additional difficulties in multi-point setting. Generally, Sections \ref{S:construct-extended-landscape} and \ref{S:convergence-optimizers} follow a similar flow to corresponding sections in \cite{DOV}, and Sections \ref{S:lpp-across-ALE}, \ref{S:limits-of-optimizers}, \ref{S:paths-extended-landscape}, and \ref{S:disjoitness-optimizers} contain the most novel ideas. Section \ref{S:lpp} is a blend of background and new deterministic results for multi-point LPP, and Section \ref{S:tightness} applies these multi-point LPP results to prove tightness for key objects.

The first half of the paper (Sections \ref{S:lpp} to \ref{S:limits-of-optimizers}) is devoted to constructing the extended Airy sheet. This is the part of the paper that leans most heavily on technical machinery from \cite{DOV}, and so to appreciate these sections we recommend that the reader have some familiarity with the construction of the Airy sheet from \cite{DOV}. See Section \ref{S:Airysheet} for details on the exact inputs we use.

The starting point for the construction of the extended Airy sheet is a combinatorial identity about the RSK bijection. In essence, this identity shows that given a collection of functions $f = (f_1, \dots, f_n)$, we can construct a collection of ordered functions $Wf = (Wf_1 \ge \dots \ge Wf_n)$ with $Wf(0) = (0, \dots, 0)$ such that
\begin{equation}
\label{E:fWf}
f[(\bx, n) \to (\by, 1)] = Wf[(\bx, n) \to (\by, 1)],
\end{equation}
for all $\bx, \by$ with $x_1 \ge 0$. We refer to $Wf$ as the melon of $f$, as ordered paths in $Wf$ emanating from $0$ resemble stripes on a watermelon. Versions of this identity go back to \cite{noumi2002tropical} and \cite{biane2005littelmann}. When $f$ is given by a collection of independent Brownian motions, then $Wf$ is given by a collection of nonintersecting Brownian motions. In the scaling window we care about, the top lines of $Wf$ converge to the parabolic Airy line ensemble $\scrB$. What Theorems \ref{T:extended-sheet} and \ref{T:extended-sheet-char} say (in particular, equation \eqref{E:Sbxy-int}) is that in this scaling, the identity \eqref{E:fWf} also passes to the limit. The left-hand side becomes the extended Airy sheet and the right-hand side becomes a last passage problem along parabolic paths in $\scrB$. 

At the level of single points $x, y$, this limiting picture was developed in \cite{DOV} to construct the usual Airy sheet. 
However, the construction of the Airy sheet does not require a well-developed notion of last passage percolation along infinite paths across the parabolic Airy line ensemble. We develop this theory in Sections \ref{S:lpp-across-ALE} and \ref{S:limits-of-optimizers}, expanding on the discussion prior to Theorem \ref{T:extended-sheet-char} above. Note that the theory of LPP along infinite paths has subtleties that are not present in the finite case. For example, it is not straightforward to show that the function $\scrB[\bx \to \by]$ is almost surely finite or continuous in $\bx$ and $\by$, see Proposition \ref{P:B-continuity}.

To take advantage of this theory and prove Theorem \ref{T:extended-sheet}, we need to prove tightness of both the extended sheets $\scrS^n$ and optimizers across the Brownian melon. To avoid obtaining new analytic estimates here, we take advantage of a variety of useful \textit{quadrangle inequalities and monotonicity properties} for multi-point LPP that generalize corresponding properties for single-point LPP, see Section \ref{S:basic-properties} and Lemma \ref{L:b-quadrangle}. These inequalities allow us to quickly deduce tightness and a modulus of continuity for the extended Airy sheet from bounds on the prelimiting Airy line ensembles and tightness of melon optimizers from tightness and coalescence properties of melon geodesics, see Section \ref{S:tightness}. These deterministic properties continue to appear as crucial tools throughout the paper. The construction of the extended Airy sheet (Theorems \ref{T:extended-sheet} and \ref{T:extended-sheet-char}) is the culmination of Sections \ref{S:lpp}-\ref{S:limits-of-optimizers}.

The remainder of the paper proves Theorems \ref{T:BLPP-convergence}, \ref{T:extended-landscape}, \ref{T:disjoint-optimizers-in-L}, \ref{T:limit-theorem} and Corollaries \ref{C:rsk} and \ref{C:disjointness}. This part of the paper does not use technical machinery from \cite{DOV}, though as mentioned previously, Sections \ref{S:construct-extended-landscape} and \ref{S:convergence-optimizers} follow a similar flow of ideas to \cite{DOV}. 

The limit $\scrL^*$ of Brownian LPP can be patched together from extended Airy sheets, just as the directed landscape can be built from Airy sheets. The procedure just requires a few technical estimates. We prove these along with Theorems \ref{T:unique-L*} and \ref{T:BLPP-convergence}, in Section \ref{S:construct-extended-landscape}.

Just as path length can be defined in the directed landscape by \eqref{E:length-L}, we can define the length of a continuous multi-path $\pi:[s, t] \to \R^k_\le$ in $\scrL^*$ by setting
$$
\|\pi\|_{\scrL^*} = \inf_{m\in\N} \inf_{s=t_0<t_1<\cdots<t_m=t}
\sum_{i=1}^m \scrL^*(\pi(t_{i-1}), t_{i-1}; \pi(t_i), t_i).
$$
We say that $\pi$ is an optimizer from $(\pi(s), s)$ to $(\pi(t), t)$ if $\|\pi\|_{\scrL^*} = \scrL^*(\pi(s), s; \pi(t), t).$ Preliminary results about paths and length in $\scrL^*$ are developed in Section \ref{S:paths-extended-landscape}. Again, there are some subtleties that arise in the study of these objects that do not exist either in the prelimit or in the setting of single paths. For example, unlike for geodesics it is not straightforward that for any $(\bp; \bq) \in \fX_\uparrow$ there is almost surely a unique $\scrL^*$-optimizer from $\bp$ to $\bq$. This requires a resampling argument in the parabolic Airy line ensemble, see Section \ref{S:optimizers-transversal}.

To show that the limit $\scrL^*$ can alternately be described by Definition \ref{D:ext-land}, the key step is Proposition \ref{P:disjointness}, which shows that almost surely, for every point in $\fX_\uparrow$ there exists an optimizer in $\scrL^*$ consisting of \textit{disjoint} paths. This is a three-step process, carried out in Section \ref{S:disjoitness-optimizers}.  

We first prove Proposition \ref{P:disjointness} for endpoints of the form $((x, x), s), ((y, y), t)$. This is an easier problem since the midpoint of such an optimizer can be characterized using only the top two lines of two independent parabolic Airy line ensembles $\scrB, \scrB'$. The key technical point that makes this observation useful is that for any compact set $K \sset \R$ and any $k \in \N$, on $K$ the top $k$ lines of $\scrB, \scrB'$ are absolutely continuous with respect to $2k$ independent Brownian motions with a well-controlled Radon-Nikodym derivative, see Theorem \ref{T:radon-n-deri}. At the level of any single Airy line, such a Radon-Nikodym derivative estimate was proven in \cite{CHH20}. The extension to multiple lines can be extracted by combining various intermediate lemmas in \cite{CHH20}, see Appendix \ref{app:radon-n-deri}.

Next, we move to endpoints of the more general form $((x_1, x_2), s), ((y_1, y_2), t)$. We do this with a resampling argument which shows that for any $[s', t'] \sset (s, t)$, there is optimizer from $((x_1, x_2), s)$ to $((y_1, y_2), t)$ that coincides on $[s', t']$ with the optimizer from $((0,0), s-1)$ to $((0,0), t + 1).$ Finally, we treat the case of $k \ge 3$ endpoints by induction. The $k=2$ case is both the base case and the key input for the inductive step.

Given Proposition \ref{P:disjointness}, Theorems \ref{T:extended-landscape} and \ref{T:disjoint-optimizers-in-L} and Corollaries \ref{C:rsk} and \ref{C:disjointness} follow easily. In a final short section (Section \ref{S:convergence-optimizers}) we give a deterministic argument to prove Theorem \ref{T:limit-theorem} from Theorems \ref{T:BLPP-convergence} and \ref{T:extended-landscape}. This section is quite similar to Section 13 of \cite{DOV}, though the arguments have been simplified a bit.

\subsection{Acknowledgements}
The first author would like to thank B\'alint Vir\'ag for many useful discussions about this project.
The authors would also like to thank anonymous referees for carefully reading this paper and providing many valuable comments that help improve the expository.

\section{Last passage percolation across lines}
\label{S:lpp}

In this section, we recall and prove combinatorial properties of last passage percolation across lines, and gather necessary limiting results for Brownian LPP. Our presentation aligns with that of \cite{DOV}, where notation and coordinate orientation are set up so that last passage geodesics will rise from the bottom of the page to the top of the page in the Airy line ensemble limit, see Figure \ref{fig:parabolic-path}.

Recall from the introduction that a path from $(x, n)$ to $(y, m)$ is a cadlag, nonincreasing function $\pi:[x, y] \to  \llbracket m, n\rrbracket$ with $\pi(y) = m$. We denote the left limit of $\pi$ at a point $t$ by $\pi(t^-)$. This is defined for all $t \in (x, y]$. We will also extend this to the point $x$ by setting $\pi(x^-) = n$. For any path $\pi$, we can define a sequence of \textbf{jump times} $x = t_{n+1} \le t_n \dots  \le t_{m+1} \in [x, y]$, where
$$
t_i = \inf \{ t \in [x, y] : \pi(t) < i \}.
$$
Typically, this is the jump when $\pi$ jumps from line $i$ to $i-1$. We also set $t_m = y$.
The \textbf{zigzag graph} of $\pi$ is
$$
\Ga(\pi) = \{(t, k) \in [a, b] \X \llbracket m, n\rrbracket : \pi(t^-) \ge k \ge \pi(t)\}.
$$
In other words, the zigzag graph of $\pi$ connects up the graph of $\pi$ by vertical lines at its jumps. We can make the set of paths into a topological space -- path space -- by specifying that $\pi_n \to \pi$ if $\Ga(\pi_n) \to \Ga(\pi)$ in the Hausdorff topology. Equivalently, $\pi_n \to \pi$ if the endpoints and jump times of $\pi_n$ converge to the endpoints and jump times of $\pi$. With this definition, the space of all paths from $p$ to $q$ is compact.

We will also introduce a partial order on paths. Let $(p, q) = (x, n; y, m), (p', q') = (x', n'; y', m')$ be such that $x \le x', y \le y'$. Then for paths $\pi, \pi'$ from $p$ to $q$ and $p'$ to $q'$ respectively, we say that $\pi \le \pi'$ if for every $t \in [x, y] \cap [x', y']$, we have $\pi(t) \le \pi'(t)$.

Now consider a sequence of continuous functions $f = (f_i : i \in I)$, where $I \sset \Z$ and each $f_i:\R\to \R$. We call the space of such functions $\scrC^I$. We will alternately think of $f$ as a function from $\R \X I$ to $\R$, or as a function from $\R \X \Z$ to $\R$, where $f$ is set equal to $0$ outside of its natural domain. When $\llbracket m, n\rrbracket \sset I$, recall from the introduction that the $f$-length of a path $\pi$ from $(x, n)$ to $(y, m)$ with jump times $t_i$ is
$$
\|\pi\|_f = \sum_{i=m}^n f_i(t_i) - f_i(t_{i+1}).
$$
Observe that $f$-length is a continuous function in path space by the continuity of $f$.
Now, for $(p, q) = (x, n; y, m)$ with $x \le y, n \ge m$ we define the \textbf{last passage value}
$$
f[p \to q] = \sup_{\pi} \|\pi\|_f,
$$
where the supremum is over all paths $\pi$ from $p$ to $q$. Continuity of path length and compactness of the set of paths from $p$ to $q$ ensures that this supremum is always attained. 

We call a path that attains the supremum a \textbf{geodesic} from $p=(x,n)$ to $q=(y,m)$. We say that $\pi$ is a \textbf{rightmost geodesic} from $p$ to $q$ if $\pi(t) \ge \tau(t)$ for all $t \in [x, y]$ for any other geodesic $\tau$ from $p$ to $q$. We similarly define the leftmost geodesic $\tau$ from $p$ to $q$ with the opposite inequality. Note our notion of rightmost and leftmost paths is with respect to the picture in Figure \ref{fig:dis-opt}, where the line order is increasing as we go from top to bottom.  Rightmost and leftmost geodesics between two points always exist by a basic compactness and continuity argument in path space, see \cite[Lemma 3.5]{DOV}. Moreover, these paths exhibit a particular tree structure and monotonicity, which can be straightforwardly deduced from their definitions.

\begin{prop}[\protect{\cite[Proposition 3.7]{DOV}}]
	\label{P:tree-structure}
	Take any $x_1 \le x_2$ and $y_1 \le y_2$, and let $\pi^+[x_i, y_i]$ denote the rightmost geodesic from $(x_i, n)$ to $(y_i, 1)$ across a function $f$. Then $\pi^+[x_1, y_1] \le \pi^+[x_2, y_2]$ and $\Ga(\pi^+[x_1, y_1]) \cap \Ga(\pi^+[x_2, y_2])$ is the zigzag graph of some path whenever this set is nonempty.
	
	In particular, if $x_1 = x_2$, then the rightmost geodesics to $y_1$ and $y_2$ are equal on some interval $[x_1, z)$, and $\pi^+[x_1, y_1](z') < \pi^+[x_2, y_2](z')$ whenever $z' \ge z$ is in the domain of both paths. We can think of the two paths as forming two branches in a tree. The same structure holds with rightmost paths replaced by leftmost paths.
\end{prop}

Often, there will be a unique geodesic between two points across the functions that we consider. In this case, the tree structure in Proposition \ref{P:tree-structure} will automatically hold; a unique geodesic is both a rightmost and leftmost geodesic. 

\subsection{Last passage with multiple paths}
\label{S:lppm}

We can extend the definition of last passage percolation to multiple disjoint paths. We say that $\pi$ and $\tau$ with domains $[a, b]$ and $[a', b']$ are \textbf{essentially disjoint} if
\begin{itemize}[nosep]
	\item $\pi(t) \ne \tau(t)$ for all $t \in(a, b) \cap (a', b')$
	\item Either $\pi \le \tau$ or $\tau \le \pi$.
\end{itemize} 
Note that since all paths are cadlag, the first condition above is equivalent to the property that the intersection of the closed graphs $\Ga(\pi) \cap \Ga(\tau)$ is finite. This characterization will often be more useful for proofs. Essential disjointness is a closed condition: if $\pi_n, \tau_n$ are sequences of essentially disjoint paths converging to paths $\pi, \tau$, then $\pi$ and $\tau$ are essentially disjoint.

Now, consider vectors $\mathbf{p} = (p_1, \dots, p_k) = ((x_1, n_1), \dots, (x_k, n_k))$ and $\mathbf{q} = (q_1, \dots, q_k) = ((y_1, m_1), \dots, (y_k, m_k))$ in $(\R \X \Z)^k$.
We say that $(\bp, \bq)$ is an \textbf{endpoint pair} of size $k$, if $n_i \ge m_i$ and $x_i \le y_i, x_i \le x_{i+1}, y_i \le y_{i+1}$ for all $i$, and there is at least one \textbf{disjoint $k$-tuple (of paths)} from $\mathbf{p}$ to $\mathbf{q}$.
Here a \textbf{disjoint $k$-tuple (of paths)} from $\mathbf{p}$ to $\mathbf{q}$ is a vector $\pi = (\pi_1,\ldots ,\pi_k)$, where 
\begin{itemize}[nosep]
	\item $\pi_i$ is a path from $(x_i, n_i)$ to $(y_i, m_i)$,
	\item $\pi_i$ and $\pi_j$ are essentially disjoint for all $i \ne j$,
	\item $\pi_i \le \pi_j$ for $i < j$.
\end{itemize}

We put the product topology on the space of all $k$-tuples of paths: $\pi \to \tau$ if $\pi_i \to \tau_i$ for all $i$. The space of disjoint $k$-tuples is a closed subset of this space, since essential disjointness and all ordering requirements are closed conditions. As in the single path case, the set of all disjoint $k$-tuples from $\bp$ to $\bq$ is compact for any endpoint pair $(\bp, \bq)$.

Now, for a disjoint $k$-tuple $\pi$ and $f \in \scrC^I$, let
$
\|\pi\|_f = \sum_{i=1}^k \|\pi_i\|_f.
$
For any endpoint pair $(\bp, \bq)$ and $f\in\scrC^I$ with $\II{m,n}\subset I$, define the last passage value
$$
f[\bp \to \bq] = \sup_{\pi} \|\pi\|_f,
$$
where the supremum is over disjoint $k$-tuples $\pi$ from $\bp$ to $\bq$. This supremum is always attained since length is a continuous function in path space and the set of all disjoint $k$-tuples from $\bp$ to $\bq$ is compact.
A disjoint $k$-tuple that attains this supremum is a \textbf{disjoint optimizer}, abbreviated to optimizer.

For most parts of the paper, we will only be concerned with endpoint pairs where all the $n_i$ are equal to some $n$, and all the $m_i$ are equal to some $m$. As a slight abuse of notation we write $\bp = (\bx, n)$ and $\bq = (\by, m)$ in this case.

\subsection{Basic properties of disjoint optimizers and last passage values}
\label{S:basic-properties}
Disjoint optimizers share certain features with geodesics. In particular, leftmost and rightmost optimizers still exist, and we have monotonicity and a useful quadrangle inequality.
Throughout this subsection we take $f\in\scrC^I$, for some suitable $I\subset\Z$.

For two disjoint $k$-tuples of paths $\pi, \tau$, we say that $\pi \le \tau$ if $\pi_i \le \tau_i$ for all $i$.
\begin{lemma}
	\label{L:rightmost-multi-path}For any endpoint pair $(\bp, \bq)$, there exists an optimizer $\pi = (\pi_1, \dots, \pi_k)$ from $\bp$ to $\bq$ such that for any other optimizer $\tau$ from $\bp$ to $\bq$, $\tau \le \pi$. We call $\pi$ the \textbf{rightmost optimizer} from $\bp$ to $\bq$. Similarly, there always exists a \textbf{leftmost optimizer} from $\bp$ to $\bq$.
\end{lemma}

\begin{proof}
	We first show that for any optimizers $\tau, \pi$ from $\bp$ to $\bq$, there exists optimizers $\ze, \ze'$ from $\bp$ to $\bq$ such that $\ze \ge \pi \ge \ze'$ and $\ze \ge \tau \ge \ze'$. For each $i, t$, set $\ze_i(t) = \max(\pi_i(t), \tau_i(t))$ and $\ze_i'(t)= \min(\pi_i(t), \tau_i(t))$. We first check that $\ze, \ze'$ are disjoint $k$-tuples from $\bp$ to $\bq$. The arguments are symmetric, so we just check $\ze$. 
	
	It is immediate from the definitions that each $\ze_i$ is a path from $p_i$ to $q_i$ and that $\ze_i \le \ze_j$ whenever $i < j$. Now, for $i \ne j$, if $\ze_i(t) = \ze_j(t)$ for some $t$, then the ordering properties for $\pi, \tau$ ensure that either $\pi_i(t) = \pi_j(t)$ or $\tau_i(t) = \tau_j(t)$. Also, if $\ze_i(t) = \ze_j(t)$ for some $t \in (x_i, y_i) \cap (x_j, y_j)$, then since both $\ze_i, \ze_j$ are cadlag, $\ze_i = \ze_j$ on some interval $[t, t + \ep)$ for some $\ep > 0$. Therefore either $\pi_i(t) = \pi_j(t)$ or $\tau_i(t) = \tau_j(t)$ for infinitely many points in this interval, contradicting the essential disjointness of either $\pi_i$ and $\pi_j$, or $\tau_i$ and $\tau_j$. Therefore $\ze_i, \ze_j$ must also be essentially disjoint, and so $\ze$ is a disjoint $k$-tuple from $\bp$ to $\bq$. Now, by the construction of $\ze_i, \ze'_i$ we have $\|\ze_i\|_f + \|\ze'_i\|_f = \|\tau_i\|_f + \|\pi_i\|_f$ for all $i$. Therefore
	$$
	\|\zeta\|_f +\|\zeta'\|_f = \|\tau\|_f + \|\pi\|_f,
	$$
	and so both $\zeta$ and $\zeta'$ must also be optimizers from $\bp$ to $\bq$.
	
	We can complete the proof by appealing to Zorn's lemma. Indeed, the set of optimizers from $\bp$ to $\bq$ is a partially ordered set. Moreover, this set is compact by the continuity of length in path space, and the fact that the set of all disjoint $k$-tuples from $\bp$ to $\bq$ is compact. Therefore by Zorn's lemma, maximal optimizers exist. Finally, if $\tau, \pi$ are two maximal optimizers, then by the argument above there is an optimizer $\ze$ with $\ze \ge \tau, \ze \ge \pi$. By maximality, this implies $\ze = \tau = \pi$ is the unique maximal optimizer: the rightmost optimizer. By a symmetric argument there exists a leftmost optimizer.
\end{proof}

In order to state the monotonicity lemma for multiple paths, we introduce a partial order on endpoint pairs starting on the same line $n$ and ending on the same line $m$. For two endpoint pairs $(\bp, \bq) = (\bx, n; \by, m)$ and $(\bp', \bq') = (\bx', n; \by', m)$ of size $k = k'$, we say that $(\bp, \bq) \le (\bp', \bq')$ if $x_i \le x_i'$ and $y_i \le y_i'$ for all $i$. If the sizes of the endpoint pairs differ or if we do not have an ordering between all endpoints, then we may still be able to compare the endpoint pairs. For two endpoint pairs $(\bp, \bq)$ and $(\bp', \bq')$ of size $k, k'$ that start and end on the same line, and $s \in \Z$, define
$$
(\bp, \bq) \le_s (\bp', \bq')
$$
if $x_{i + s} \le x_i'$ and $y_{i+s} \le y_i'$ for all $i$ such that either $i + s\in \II{1, k}$ or $i \in \II{1, k'}$. Here the coordinates $x_j, y_j$ are defined to be equal to $\infty$ for $j > k$ and $- \infty$ for $j < 1$, and $x_j', y_j'$ are defined similarly in terms of $k'$. 

This definition can be thought in the following way. 
 First pad the endpoint pairs $(\bp, \bq)$ and $(\bp', \bq')$ with points that are arbitrarily far to the right or left so that the indices $i + s$ in $(\bp, \bq)$ and $i$ in $(\bp', \bq')$ are now lined up and the new endpoint pairs have the same size. The ordering $\le_s$ is then just the usual ordering $\le$ on the padded endpoint pairs.

\begin{lemma}
\label{L:mono-tree-multi-path}
Let $(\bp, \bq)$ and $(\bp', \bq')$ be two endpoint pairs of sizes $k, k'$ starting and ending on the same line. Let $\pi$ be the rightmost optimizer from $\bp$ to $\bq$, and $\pi'$ be the rightmost optimizer from $\bp'$ to $\bq'$.
\begin{enumerate}[label=(\roman*), nosep]
	\item Suppose that $k = k'$, and that $(\bp, \bq) \le (\bp', \bq')$. Then $\pi \le \pi'$.
	\item Suppose that $(\bp, \bq) \le_s (\bp', \bq')$ for some $s \in \Z$. Then $\pi_{i} \le \pi'_{i + s}$ for all $i \in \II{1, k} \cap \II{1-s, k'-s}$.
\end{enumerate}
The same statements hold with leftmost optimizers in place of rightmost ones.
\end{lemma}

\begin{proof}
We will just prove (i), as (ii) can be reduced to (i) by the padding procedure described above. We use a similar construction to Lemma \ref{L:rightmost-multi-path}. For each $i$, define paths $\ze_i, \ze'_i$ as follows. On $[x_i, y_i] \cap [x_i', y_i']$, set $\ze_i(t) = \min(\pi_i(t), \pi_i'(t))$ and set $\ze'_i(t) = \max(\pi_i(t), \pi_i'(t))$. Extend $\ze_i$ to all of $[x_i, y_i]$ by setting it equal to $\pi_i$ on $[x_i, y_i] \smin [x_i', y_i']$ and extend $\ze'_i$ to all of $[x_i', y_i']$ by setting it equal to $\pi'_i$ on $[x_i', y'_i] \smin [x_i, y_i]$.

 With these definitions, because $x_i \le x_i'$ and $y_i \le y_i'$, $\ze_i$ is a path from $p_i$ to $q_i$ and $\ze'_i$ is a path from $p_i'$ to $q_i'$. Moreover, exactly as in the proof of Lemma \ref{L:rightmost-multi-path}, we can check that $\ze = (\ze_1, \dots, \ze_k)$ is a disjoint $k$-tuple from $\bp'$ to $\bq'$, $\ze'$ is a disjoint $k$-tuple from $\bp$ to $\bq$, and
 $$
 \|\ze\|_f + \|\ze'\|_f = \|\pi\|_f + \|\pi'\|_f.
 $$
 Therefore $\ze, \ze'$ must both be optimizers. Since $\pi' \le \ze'$ and $\pi'$ is a rightmost optimizer, we have $\ze' = \pi'$. Also, $\pi \le \ze'$ by construction, yielding (i).
\end{proof}

We will also need two \textbf{quadrangle inequalities} for multi-point last passage values. These are generalizations of a commonly used quadrangle inequality for single-point last passage values, see for example, Proposition 3.8 in \cite{DOV}.

\begin{lemma}
\label{L:quadrangle}
Let $(\bp, \bq) = (\bx, n; \by, m), (\bp', \bq') = (\bx', n; \by', m)$ be endpoint pairs of size $k$. Define $\bx^\ell, \by^\ell, \bx^r, \by^r\in \R^k_\le$ by setting $x^\ell_i=x_i\wedge x_i'$, $y^\ell_i=y_i\wedge y_i'$, and $x^r_i=x_i\vee x_i'$, $y^r_i=y_i\vee y_i'$, for each $1\le i \le k$, and let $\bp^\ell = (\bx^\ell, n), \bp^r = (\bx^r, n), \bq^\ell = (\by^\ell, m), \bq^r = (\by^r, m)$. Then
$$
f[\bp \to \bq] + f[\bp' \to \bq'] \le f[\bp^\ell \to \bq^\ell] + f[\bp^r \to \bq^r].
$$
In particular, if $(\bp, \bq') \le (\bp', \bq)$, then
$$
f[\bp \to \bq] + f[\bp' \to \bq'] \le f[\bp \to \bq'] + f[\bp' \to \bq].
$$
\end{lemma}

\begin{proof}
Let $\pi$ be an optimizer from $\bp$ to $\bq$, and let $\pi'$ be an optimizer from $\bp'$ to $\bq'$. We can define disjoint $k$-tuples $\tau^\ell, \tau^r$ as follows. For each $i$, set $\tau^\ell_i = \min(\pi_i, \pi_i')$ on $[x_i^r, y_i^\ell]$ and set $\tau^r_i = \max(\pi_i, \pi_i')$ on $[x_i^r, y_i^\ell]$. On $[x_i^\ell, x_i^r)$, we set $\tau_i^\ell$ to be either $\pi_i$ or $\pi_i'$, depending on whether $x_i^\ell$ equals $x_i$ or $x_i'$. Similarly, on $(y_i^\ell, y_i^r]$, set $\tau^r_i$ to be either $\pi_i$ or $\pi_i'$. As in the proof of Lemma \ref{L:rightmost-multi-path}, one can check that $\tau^\ell, \tau^r$ are disjoint $k$-tuples from $\bp^\ell$ to $\bq^\ell$ and $\bp^r$ to $\bq^r$, respectively. Therefore
\[
f[\bp \to \bq] + f[\bp' \to \bq']  = \|\pi\|_f + \|\pi'\|_f = \|\tau\|_f + \|\tau'\|_f \le f[\bp^\ell \to \bq^\ell] + f[\bp^r \to \bq^r].
\]
The second part of the theorem follows from the fact that if $(\bp, \bq') \le (\bp', \bq)$, then $\bp = \bp^\ell, \bp' = \bp^r, \bq = \bq^r,$ and $\bq' = \bq^\ell$.
\end{proof}

\begin{lemma} \label{L:quadrangle-2}
	Let $(\bp, \bq), (\bp, \bq')$ be endpoint pairs of size $k \ge 2$ that start and end on the same line with $(\bp, \bq) \le (\bp, \bq')$. Fix $1\le\ell < k$, and let $\bp^L, \bq^L,\bq'^L$ be the first $\ell$ coordinates of $\bp, \bq, \bq'$, and $\bp^R, \bq^R, \bq'^R$ be the last $k-\ell$ coordinates of $\bp, \bq, \bq'$.
	Suppose first that $\bq^R = \bq'^R$. Then
	\[
	f[\bp\to\bq] + f[\bp^L \to \bq'^L]\ge f[\bp\to\bq'] + f[\bp^L \to \bq^L].
	\]
	Similarly, suppose that $\bq^L = \bq'^L$. Then
	\[
	f[\bp\to\bq] + f[\bp^R \to \bq'^R]\le f[\bp\to\bq'] + f[\bp^R \to \bq^R],
	\]

\end{lemma}
\begin{proof}
	We prove the first inequality since the second one follows similarly. Let $\pi$ be an optimizer from $\bp$ to $\bq'$, and let $\tau$ be an optimizer from $\bp^L$ to $\bq^L$. For $i \le \ell$, we can define paths $\sig_i$ by setting $\sig_i = \min(\pi_i, \tau_i)$ on $[x_i, y_i]$. We also set $\sig_i = \pi_i$ for $i \in \{\ell +1, \dots, k\}$. As in the proof of Lemma \ref{L:rightmost-multi-path}, one can check that $\sig$ is a disjoint $k$-tuple from $\bp$ to $(\bq^L, \bq'^R)$. Moreover, $(\bq^L, \bq'^R) = \bq$ since $\bq^R = \bq'^R$.
	Similarly set $\sig_i' = \max(\pi_i, \tau_i)$ on $[x_i, y_i]$ and set $\sig_i = \pi_i$ on $(y_i, y_i']$. Again as in the proof of Lemma \ref{L:rightmost-multi-path}, $\sig'$ is a disjoint $k$-tuple from $\bp^L$ to $\bq'^L$. Therefore
	\[
	f[\bp\to\bq'] + f[\bp^L \to \bq^L] = \|\pi\|_f + \|\tau\|_f = \|\sig\|_f + \|\sig'\|_f \le f[\bp\to\bq] + f[\bp^L \to \bq'^L]. \qedhere
	\]
\end{proof}

We next record three deterministic bounds on multi-point last passage values which will be used to prove tightness. 
The first bound controls the difference between two last passage values.
For this lemma, define the fluctuation of a function $f \in \scrC^I$ on a set $A\sset \R\X I$ by
$$
\om(f, A) = \sup_{(x, i), (y, i) \in A} |f_i(x) - f_i(y)|.
$$
\begin{lemma}
\label{L:naive-bounds-diff}
Let $(\bp, \bq), (\bp, \bq')$ be two endpoint pairs of size $k$ with $\bp=(\bx,n)$ to $\bq=(\by,m)$ and $\bq'=(\by',m)$ differing only on a single coordinate $y_i < y_i'$. Let $\pi, \pi'$ be optimizers from $\bp$ to $\bq$ and $\bp$ to $\bq'$.  Then 
\begin{align*}
f[\bp \to \bq'] - f[\bp \to \bq] &\le  |\pi_i'(y_i) + 1 - m|\om(f, [y_i,y_i'] \X \II{m, \pi_i'(y_i)}),\\
f[\bp \to \bq] - f[\bp \to \bq'] &\le (2(k-i) +1)\om(f, [y_i,y_i'] \X \II{m, m+k-i}).
\end{align*}
\end{lemma}

\begin{proof}
First observe that we can take the disjoint $k$-tuple $\pi'$ and produce a disjoint $k$-tuple $\tau$ from $\bp$ to $\bq$ by restricting the path $\pi'_i$ to the interval $[x_i, y_i]$ (and possibly redefining the value at the right endpoint $y_i$). The change in length from doing this is
$
\|\pi'_i|_{[y_i, y_i']}\|_f$. This is bounded above by the last passage value $f[(y_i, \pi_i'(y_i)) \to (y_i', m)]$, which is bounded above by
$$
|\pi_i'(y_i) + 1 - m|\om(f, [y_i,y_i'] \X \II{m, \pi_i'(y_i)}).
$$
Since $\|\tau\|_f \le f[\bp \to \bq]$, this yields the first bound in the lemma.

For the other bound, we can take the $k$-tuple $\pi$ and extend the component $\pi_i$ to a path $\pi_i^*$ from $(x_i, n)$ to $(y_i', m)$ by letting $\pi_i = m$ on the interval $[y_i, y_i']$. This may break the essential disjointness with the path $\pi_{i+1}$, so we may need to redefine $\pi_{i+1}$ on the interval $[y_i, y_i']$. We can deal with this by defining a new path $\pi^*_{i+1}$ so that $\pi^*_{i+1} = \max\{m + 1, \pi_{i+1}\}$ on the intersection $[y_i, y_i') \cap [x_{i+1}, y_{i+1})$, and setting $\pi^*_{i+1} = \pi_{i+1}$ elsewhere. Continuing in this way, we can redefine all of the paths $\pi_i, \dots, \pi_k$ to get functions $\pi_{i + j}^*$ that are equal to $\max\{m + j, \pi_{i + j}\}$ on each of the intervals $[a_{i+j}, b_{i+j}) = [y_i, y_i') \cap [x_{i+j}, y_{i+j})$, and are equal to $\pi_{i+j}$ elsewhere.

 We check that this process yields a disjoint $k$-tuple. The functions $\pi^*_{i+j}$ are cadlag and nonincreasing on the interval $[a_{i+j}, b_{i+j})$ where the path was redefined. Since this interval is closed on the left and open on the right, this ensures that $\pi^*_{i+j}$ is cadlag everywhere. Now, since $\pi^*_{i+j} \ge \pi_{i+j}$ on the interval $[a_{i+j}, b_{i+j})$, we have that $\pi^*_{i+j}$ is nonincreasing on $[a_i, y_{i+j}]$. To check that $\pi^*_{i+j}$ is nonincreasing everywhere it just remains to check the endpoint $a_{i+j}$, when $a_{i+j} = y_i$. For this, observe that the essential disjointness of $\pi_i, \pi_{i+1}, \dots, \pi_{i+j}$ implies that 
 $$
 \pi_i(y_i^-) < \pi_{i+1}(y_i^-) < \dots < \pi_{i+j}(y_i^-)
 $$
 which forces $\pi_{i+j}(y_i^-) \ge m+j$. Since $\pi^*_{i+j}(y_i) = \max\{m + j, \pi_{i + j}(y_i)\}$, this implies that $\pi^*_{i+j}$ is nonincreasing at $y_i$. Finally, observe that for any $j < j'$, the new definitions imply $\pi^*_{i+j} \le \pi^*_{i+j'}$ and the two paths are essentially disjoint on the interval $[a_{i+j}, b_{i+j}) \cap [a'_{i+j}, b'_{i+j})$. Hence  $\pi^* = (\pi_1, \dots, \pi_{i-1}, \pi_i^*, \dots, \pi^*_k)$ is a disjoint $k$-tuple from $\bp$ to $\bq'$.
 
  Moreover, for each $j \ge 1$ we have
$$
\|\pi_{i+j}\|_f - \|\pi^*_{i + j}\|_f \le 2 \om(f, [y_i,y_i'] \X \II{m, m + j}) \le 2 \om(f, [y_i,y_i'] \X \II{m, m + k-i}).
$$
For $j = 0$ we have the same bound, except with the $2$ removed since $\pi_i$ is not defined on $[y_i, y_i']$.
Summing over $j \in \II{i, k}$ and using that $\|\pi^*\|_f \le f[\bp \to \bq']$ yields the second inequality.
\end{proof}

The second lemma helps controls the weight of an individual path in a disjoint optimizer.

\begin{lemma}
\label{L:one-path-below}
For an endpoint pair $(p, q)$ of single points, let $(p^k, q^k)$ be an endpoint pair of size $k \ge 2$, where $p^k = (p, \dots, p)$ and $q^k = (q, \dots, q)$. Let $\pi = (\pi_1, \dots, \pi_k)$ be a disjoint optimizer for this endpoint pair. Then for all $i \in \II{1, k},$ we have
$$
\|\pi_i\|_f \ge f[p^k \to q^k] - f[p^{k-1} \to q^{k-1}].
$$
\end{lemma}

\begin{proof}
For each $i$, the collection $(\pi_j : j \ne i, j \in \II{1, k})$ is a disjoint $(k-1)$-tuple from $p^{k-1}$ to $q^{k-1}$. Therefore
\begin{equation*}
\label{E:pkqk}
f[p^k \to q^k] = \|\pi\|_f = \|\pi_i\|_f + \sum_{j \ne i, j \in \II{1, k}} \|\pi_j\|_f \le \|\pi_i\|_f + f[p^{k-1} \to q^{k-1}].
\end{equation*}
The lemma follows by rearranging the above inequality.
\end{proof}

The next lemma gives naive bounds on the value of $f[\bp \to \bq]$ in terms of single-point last passage values and last passage values with clustered endpoints.

\begin{lemma}
\label{L:f-naive}
Let $(\bp, \bq) = (\bx, n; \by, m)$ be an endpoint pair of size $k \ge 2$. 
Then
$$
\sum_{i=1}^k \lf(f[p_i^k \to q_i^k] - f[p_i^{k-1} \to q_i^{k-1}]\rg) \le f[\bp \to \bq] \le \sum_{i=1}^k f[p_i \to q_i],
$$
where the notation $p^k$ is as in Lemma \ref{L:one-path-below}.
\end{lemma}

\begin{proof}
	The upper bound follows since any disjoint $k$-tuple from $\bp$ to $\bq$ gives rise to $k$ paths from $p_i$ to $q_i$. For the lower bound, we construct a disjoint $k$-tuple from $\bp$ to $\bq$ using a diagonal argument. For each $i \in \II{1, k}$, let $\tau^i$ be a disjoint optimizer from $p_i^k$ to $q_i^k$. By the monotonicity established in Lemma \ref{L:mono-tree-multi-path}, the components $\tau^1_1, \dots, \tau^k_k$ form $k$ disjoint paths from $\bp$ to $\bq$. Finally,
	$\|\tau^i_i\|_f 
	\ge  f[p_i^k \to q_i^k]- f[p_i^{k-1} \to q_i^{k-1}]$ by Lemma \ref{L:one-path-below}. The conclusion follows.
\end{proof}

We finish this subsection by recording a metric composition law, which can deduced from the definition of last passage values without much difficulty. (This is also recorded as \cite[Lemma 4.4]{DOV}.)

\begin{lemma}
	\label{L:split-path}
	Let $(\bp, \bq) = (\bx, n; \by, m)$ be an endpoint pair of size $k$ and let $\ell \in \{m + 1, \dots, n\}$. Then
	$$
	f[\bp \to \bq] = \max_\bz f[\bp \to (\bz, \ell)] +  f[(\bz, \ell - 1) \to \bq],
	$$
	where the maximum is taken over $\bz \in \R^k_\le$ such that both $(\bp; \bz, \ell)$ and $(\bz, \ell - 1; \bq)$ are endpoint pairs.
\end{lemma}
We note that the $\argmax$ of the right-hand side of the display above is precisely the location in $\R^k_\le$ where an optimizer from $\bp$ to $\bq$ jumps from line $\ell$ to line $\ell-1$.

\subsection{Melons}
\label{S:melons}
Let $f \in \scrC^{\II{1, n}}$. For any point $t \in \R$, the \textbf{melon of $f$ opened up at $t$} is a sequence of functions $W_t f = (W_t f_1, \dots, W_t f_n)$ from $[t, \infty)$ to $\R$ defined as follows. Set $W_t f_1(s) = f[(t, n) \to (s, 1)]$ and for $k \in \II{2, n}$ let
$$
W_t f_k(s) = f[(t, n)^k \to (s, 1)^k] - f[(t, n)^{k-1} \to (s, 1)^{k-1}].
$$
The functions $W_t f_i$ satisfy $W_t f_i(t) = 0$ for all $i$ and are ordered: $W_t f_1 \ge \dots \ge W_t f_n$, see the discussion in \cite[Section 4]{DOV} (before Proposition 4.1). Surprisingly, the melon operation preserves last passage values. This fact was essentially shown by Noumi and Yamada \cite{noumi2002tropical}. A version for single-point last passage values across continuous functions was proven by Biane, Bougerol, and O'Connell \cite{biane2005littelmann}. We quote a multi-point version from \cite{DOV} which applies to our context.
\begin{theorem}[\protect{\cite[Proposition 4.1]{DOV}}]
	\label{T:melon-lpp}
	Let $f \in \scrC^{\II{1, n}}$, and let $(\bp, \bq) = (\bx, n; \by, 1)$ be any endpoint pair. Then for all $t \le x_1$, we have
	$$
	f[\bp \to \bq] = W_tf[\bp \to \bq].
	$$
\end{theorem}

A consequence of Theorem \ref{T:melon-lpp} is that disjointness of optimizers across the melon $W_tf$ is equivalent to disjointness across the original functions $f$. Let $(\bp, \bq)$ and $(\bp', \bq')$ be endpoint pairs such that the concatenation $(\bp \cup \bp', \bq \cup \bq')$ remains an endpoint pair. For disjoint $k$-tuples $\pi, \tau$ from $\bp$ to $\bq$ and $\bp'$ to $\bq'$, we say that $\pi$ and $\tau$ are \textbf{essentially disjoint} if $(\pi, \tau)$ is a disjoint $k$-tuple from $\bp \cup \bp'$ to $\bq \cup \bq'$.

For the next lemma, let $\pi_f^+[\bp, \bq]$ denote the rightmost optimizer from $\bp$ to $\bq$ across a function $f \in \scrC^{\II{1, n}}$, and let $\pi_f^-[\bp, \bq]$ denote the leftmost optimizer.

\begin{lemma}
	\label{L:disjoint-melon}
	Let $f \in \scrC^{\II{1, n}}$, and let $(\bp, \bq) = (\bx, n; \by, 1), (\bp', \bq') = (\bx', n; \by', 1)$ be two endpoint pairs, such that the concatenation $(\bp \cup \bp', \bq \cup \bq')$ remains an endpoint pair. Fix $t \le x_1.$
	
	Then $\pi_f^-[\bp, \bq]$ and $\pi_f^+[\bp', \bq']$ are essentially disjoint if and only if $\pi_{W_tf}^-[\bp', \bq']$ and $\pi_{W_tf}^+[\bu_2, \bv_2]$ are essentially disjoint.
\end{lemma}

Lemma \ref{L:disjoint-melon} is essentially Lemma 4.5 from \cite{DOV}, but for paths with multiple starting and ending points. The proofs are identical up to trivial notational changes.

Optimizers across melons will often be simpler to analyze than optimizers across the original functions. For example, we have the following simple lemma from \cite{DOV}. In this lemma, the function $f$ takes the form of a melon opened up at $0$.

\begin{lemma}[\protect{\cite[Lemma 5.1]{DOV}}]
	\label{L:high-paths}
	Let $f \in \scrC^{\II{1, n}}$ be such that $f_i(0) = 0$ for all $i \in \II{1, n}$ and $f_i \ge f_{i+1}$ for all $i \in \II{1, n-1}$. Fix $j  \le k \le n \in \N$. Let $(\bp, \bq) = (\bx, n; \by, 1)$ be an endpoint pair of size $k$ with $x_i = 0$ for all $i \in \II{1, j}$. Then there exists an optimizer
	$
	\pi
	$
	from $\bp$ to $\bq$
	such that $\pi_i(t) = i$ for all $t \in (0, y_1), i \in \II{1, j}$.
\end{lemma}

In particular, Lemma \ref{L:high-paths} gives that the leftmost optimizer from $(t, n)^k$ to any $(\by, 1)$ (with $\by \in \R^k_\le$) in any melon $W_t f$ will only use the top $k$ lines $W_tf_1, \dots W_t f_k$.

\subsection{Brownian melons and the parabolic Airy line ensemble}
\label{S:brownian-melons}

\FloatBarrier

Melons have a remarkable probabilistic structure when the input function consists of $n$ independent two-sided Brownian motions $B^n = (B^n_1, \dots, B^n_n)$. In this case, the \textbf{Brownian $n$-melon} $W^n := W_0 B^n$ is given by $n$ Brownian motions started at $0$, conditioned to never intersect. This was first shown in \cite[Theorem 7]{o2002representation}. This structure allows one to find the scaling limit of $W^n$ at the edge.
See Figure \ref{fig:Brown-melon} for an illustration.
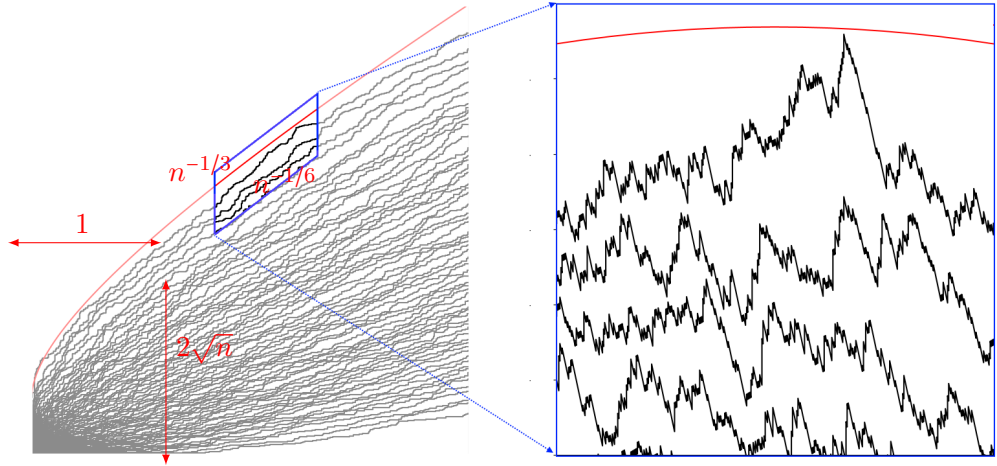
\begin{figure}[t]
    \centering
\begin{tikzpicture}[line cap=round,line join=round,>=triangle 45,x=3cm,y=4cm]

 \node[anchor=south west] at (0,0) {\includegraphics[width=\textwidth]{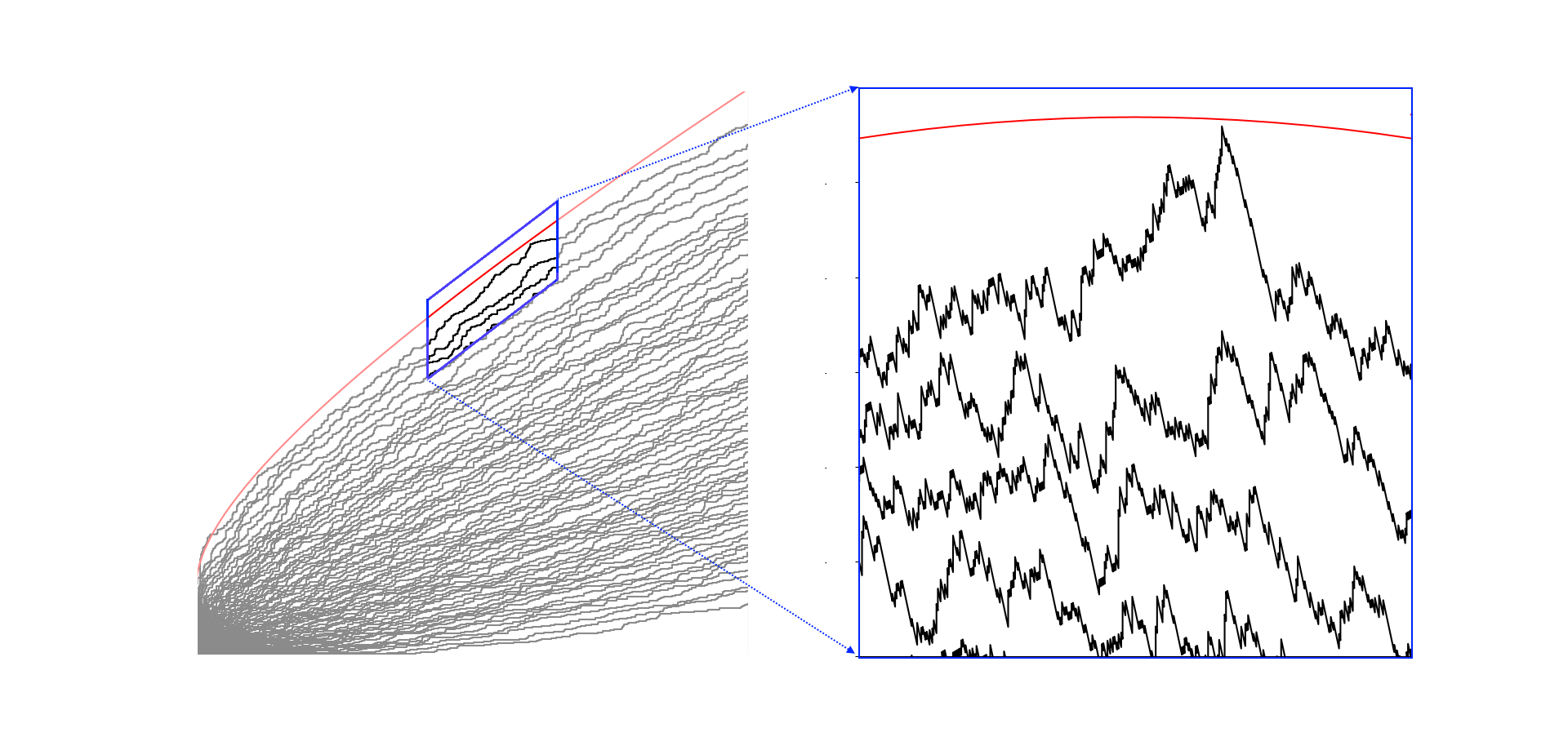}};

\clip(-0.03,-0.03) rectangle (4.95,1.45);

\draw (1.67,1.16) [red] node[anchor=west]{$n^{-1/6}$};
\draw (1.67,1.12) [red] node[anchor=south east]{$n^{-1/3}$};
\draw[{latex}-{latex}] [red] (0.63,0.95) -- (1.31,0.95);
\draw[{latex}-{latex}] [red] (1.33,0.21) -- (1.33,0.83);
\draw (0.96,0.95) [red] node[anchor=south]{$1$};
\draw (1.33,0.6) [red] node[anchor=west]{$2\sqrt{n}$};

\end{tikzpicture}
\caption{A simulation of a Brownian melon (from \cite{DNV}) and a window in which it converges to the parabolic Airy line ensemble.}   \label{fig:Brown-melon}
\end{figure}

First tilt and rescale the melons $W^n = (W^n_1, \dots, W^n_n)$. Define $\scrB^n = (\scrB^n_1, \dots, \scrB^n_n)$ by
\begin{equation}
\label{E:Bnk-def}
\scrB^n_i(y) = n^{1/6} \lf(W_i^n(1 + 2yn^{-1/3}) - 2 \sqrt{n} - 2 y n^{1/6} \rg).
\end{equation}
Then the functions $\scrB^n$ converges in distribution to a continuous limit known as the parabolic Airy line ensemble. 

\begin{theorem}
[\protect{\cite[Theorem 3.1]{CH}}]\label{T:airy-line-ensemble}
The sequence $\scrB^n$ converges in distribution to a continuous limit $\scrB:\R\X\N \to \R$, in the topology of uniform convergence on compact subsets of $\R\X\N$. The limit $\scrB$ is the \textbf{parabolic Airy line ensemble}. 
\end{theorem}

The qualifier parabolic comes from the fact that the process $\scrA(x) = \scrB(x) + x^2$ is stationary, so $\scrB$ has a parabolic shape. The process $\scrA$ is known as the \textbf{(stationary) Airy line ensemble}. Note that Corwin and Hammond technically worked with nonintersecting Brownian bridges (with diffusion parameter $1$) from time $0$ to time $2$, rather than nonintersecting Brownian motions $\scrB^n$ (with diffusion parameter $1$). The two objects are equivalent in the Airy line ensemble scaling limit by virtue of the standard transformation between Brownian bridge and Brownian motion.

Both Brownian melons and the parabolic Airy line ensemble are strictly ordered and satisfy a useful resampling property called the \textbf{Brownian Gibbs property}. This makes these objects useful in practice. The next theorem gathers results from \cite{CH}, from Definition 2.13 and Theorem 1. We choose not to introduce the Brownian Gibbs property as formally as in that paper, since it only plays a tangential role in this paper.

\begin{theorem}
	\label{T:melon-Airy-facts}
Let $W^n$ denote a Brownian $n$-melon, let $\scrB$ denote the parabolic Airy line ensemble and let $\tilde \scrB = 2^{-1/2} \scrB$. Almost surely,
\begin{equation}
\label{E:Wnii}
\begin{split}
W^n_i(t) &> \; W^n_{i+1}(t) \quad \text{ for all } i \in \II{1, n}, t > 0, \quad \text{ and } \\
\quad \tilde \scrB_i(t) &> \; \tilde\scrB_{i+1}(t) \quad \text{ for all } i \in \N, t \in \R.
\end{split}
\end{equation}
Moreover, for any box $S = \II{\ell, k} \X [a, b]$ with $a > 0$ and $k \le n$, the process $W^n|_S$ given $W^n|_S^c$ is just given by $k - \ell + 1$ Brownian bridges (with diffusion parameter $1$) connecting up the points $W^n_i(a)$ and $W^n_i(b)$, conditioned so that the nonintersection conditions in \eqref{E:Wnii} hold. This property is called the Brownian Gibbs property.

Similarly, for any box $S = \II{\ell, k} \X [a, b]$, the process $\tilde \scrB|_S$ given $\tilde \scrB|_S^c$ is just given by $k - \ell + 1$ Brownian bridges (with diffusion parameter $1$) connecting up the points $\tilde\scrB_i(a)$ and $\tilde\scrB_i(b)$, conditioned so that the nonintersection conditions in \eqref{E:Wnii} hold.

\end{theorem}

We end this subsection by recording a few uniqueness results for Brownian last passage percolation. These results are stated for last passage percolation between multiple points on potentially different lines.

\begin{lemma}
	\label{L:brown-unique}
	Let $(\bp, \bq) = (p_1, \dots, p_k, q_1, \dots, q_k) \in ([x, y] \X \II{m,n})^{2k}$ be an endpoint pair. Let $I \sset \Z$ be an integer interval containing $\II{m,n}$, and let $B= \{B_i : i  \in I\}$ be a sequence of random continuous functions with the following property.
	For $[a, b] \sset (x, y)$ and any $i \in I$, let $\scrF_{[a, b] \X \{i\}}$ be the $\sig$-algebra generated by all increments $B_j(t) - B_j(s)$ with 
	$$
	[t, s] \X \{j\} \sset [x, y] \X \II{m, n} \smin \lf( (a, b) \X \{i\} \rg).
	$$
	Suppose that for any $[a, b] \sset (x, y)$ and $i \in I$, the conditional distribution
	\begin{equation}
	\label{E:pBiBi}
	\p(B_i(b) - B_i(a) \in \cdot \;|\; \scrF_{[a, b] \X \{i\}})
	\end{equation}
	is a continuous distribution almost surely. Then there is almost surely a unique optimizer $\pi$ from $\bp$ to $\bq$.
\end{lemma}

This lemma is due to Hammond, see \cite[Lemma B.1]{hammond2019patchwork}. 
However, since we have stated it in greater generality than in that paper, we include a brief proof using Hammond's method.

\begin{proof}[Proof of Lemma \ref{L:brown-unique}]
For any $\gamma = [a, b] \X \{i\}$ with $[a, b] \sset (x, y)$, $i \in I$, and $j \in \II{1, k}$, let 
$$
B_{\ga, j}[\bp \to \bq] = \sup \|\pi\|_B,
$$
where the supremum is taken over all disjoint $k$-tuples from $\bp$ to $\bq$ subject to the constraint that $\ga \sset \Ga(\pi_j)$. Define $B_{\ga^c}[\bp \to \bq]$ similarly, but with the supremum taken over all disjoint $k$-tuples from $\bp$ to $\bq$ subject to the constraint that $\ga \cap \Ga(\pi_i) = \emptyset$ for all $i \in \II{1, k}$.
 We claim that almost surely, 
\begin{equation}
\label{E:Bgaj}
B_{\ga, j}[\bp \to \bq] \ne B_{\ga^c}[\bp \to \bq] 
\end{equation}
for all $j$. Indeed, $B_{\ga^c}[\bp \to \bq]$ is $\scrF_\ga$-measurable, and $B_{\ga, j}[\bp \to \bq] = X + B_i(b) - B_i(a)$, where $X$ is an $\scrF_\ga$-measurable random variable. Since $B_i(b) - B_i(a)$ has a continuous distribution, conditionally on $\scrF_\ga$, this yields \eqref{E:Bgaj}. Now, \eqref{E:Bgaj} holds simultaneously almost surely for all $\ga$ with rational endpoints and $j \in \II{1, k}$. On the other hand, if there were two optimizers $\pi, \pi'$ from $\bp$ to $\bq$, then there would exist a $j \in \II{1, k}$ and a $\ga$ with rational endpoints such that $\ga \cap \Ga(\pi_i) = \emptyset$ for all $i \in \II{1, k}$ but $\ga \sset \Ga(\pi'_j)$. Therefore
$$
B_{\ga, j}[\bp \to \bq] = \|\pi'\|_\scrB = \|\pi\|_\scrB= B_{\ga^c}[\bp \to \bq],
$$
contradicting \eqref{E:Bgaj}.
\end{proof}
 
The conditions of the lemma are set up so that they apply to all the objects that we work with.

 \begin{lemma}
	\label{L:specifics}
	The conditions of Lemma \ref{L:brown-unique} are satisfied when $B$ is a collection of independent Brownian motions for any $I$ and $(\bp, \bq)$, when $B = W^n$ is a Brownian melon with $x \ge 0$ and $I \sset \II{1,n}$, and when $B = \scrB$ is the parabolic Airy line ensemble and $I \sset \N$.
	\end{lemma}

\begin{proof}
If $B$ is a collection of independent Brownian motions, then \eqref{E:pBiBi} is a normal distribution almost surely, and hence is continuous. We treat the remaining two cases together by appealing to the Brownian Gibbs property in Theorem \ref{T:melon-Airy-facts} for either $B$ or $2^{-1/2} B$ (where either $B = W^n$ or $2^{1/2} B = \scrB$). By possibly increasing the size of $I$, we may assume $I = \II{1, m}$ for some $m$. Let $[a, b] \sset (x, y)$. 

By the Brownian Gibbs property, conditionally on the $\sig$-algebra $\scrG$ generated by $B_j(t)$ for all $(j, t) \notin \{i\} \X [a, y + 1],$ the process $B_i(t) - B_i(a), t \in [a, y + 1]$ is a Brownian bridge connecting $0$ and $B_i(y+1) - B_i(a)$, conditioned so that the ensemble $B$ remains nonintersecting. In particular, conditionally on the $\sig$-algebra $\scrG'$ generated by $\scrG$ and $B_i(y) - B_i(r), r \in [b, y]$, almost surely the distribution of
$$
B_i(b) - B_i(a)
$$
is absolutely continuous with respect to Lebesgue measure on $\R$. Finally, $\scrF_{[a, b] \X \{i\}} \sset \scrG'$, giving the result.
\end{proof}

Lemmas \ref{L:brown-unique} and \ref{L:specifics} together allow us to speak of a single optimizer or geodesic when considering last passage problems across these Brownian motions, Brownian melons, and the parabolic Airy line ensemble.

\subsection{Melon geodesics and the Airy sheet}
\label{S:Airysheet}
Recall the definition of the prelimiting extended Airy sheets 
\begin{equation}
\label{E:Sndef}
\scrS^n(\bx, \by)= n^{1/6} \lf(B^n[(2n^{-1/3}\bx, n) \to (1 + 2n^{-1/3}\by, 1)] - 2k\sqrt{n} - n^{1/6} \sum_{i=1}^k 2(y_i - x_i) \rg),
\end{equation}
from Theorem \ref{T:extended-sheet}, where $B^n$ is a collection of $n$ independent two-sided standard Brownian motions. For thinking about the prelimiting sheets $\scrS^n$, it will be helpful to use an alternate formula for $\scrS^n$ in terms of the prelimiting Airy line ensembles $\scrB^n$ (defined in \eqref{E:Bnk-def}).
Recall the Brownian $n$-melon $W^n=W_0B^n$.
When $\bx \in \R^k_\le \cap [0, \infty)^k$, by Theorem \ref{T:melon-lpp} we have
\[
\scrS^n(\bx, \by)= n^{1/6} \lf(W^n[(2n^{-1/3}\bx, n) \to (1 + 2n^{-1/3}\by, 1)] - 2k\sqrt{n} - n^{1/6} \sum_{i=1}^k 2(y_i - x_i) \rg).
\]
Note that $\scrB^n$ is just an affine transformation of $W^n$, given by \eqref{E:Bnk-def}.
Using the fact that last passage values commute with affine shifts, we have
\begin{equation}
\label{E:Sn-compact-form}
\scrS^n(\bx, \by) = \scrB^n[(\bx - n^{1/3}/2, n) \to (\by, 1)] - k n^{2/3}.
\end{equation}
One of the main insights of \cite{DOV} was finding a way to take a limit of the right-hand side of \eqref{E:Sn-compact-form} for single points $x, y$ in order to define the Airy sheet $\scrS:\R \to \R$ in terms of the Airy line ensemble. The basic idea there was to carefully analyze the location and coalescence structure of geodesics across $\scrB^n$. The culmination of this analysis showed that if $\pi^n\{x, y\}, \pi^n\{x, z\}$ are the rightmost geodesics across $\scrB^n$ from $(x-n^{1/3}/2, n)$ to $(y, 1)$ and $(z, 1)$, then with high probability $\pi^n\{x, y\}$ and $\pi^n\{x, z\}$ coincide outside of an $O(1)$ region around the points $(y, 1), (z, 1)$. This suggests that difference
$$
\scrS^n(x, y) - \scrS^n(x, z)
$$
should converge to a difference of last passage problems in the Airy line ensemble.  This, along with an estimate on the location of the paths $\pi^n\{x, z\}$, motivates the following definition of the Airy sheet on $[0, \infty) \X \R.$

\begin{definition}\label{D:halfsheet} For a parabolic Airy line ensemble $\scrB$, we define the \textbf{half Airy sheet} of $\scrB$ to be the function $\scrS_\scrB: [0, \infty) \X \R \to \R$ specified by
the formulas 
\begin{itemize}[nosep]
	\item $\scrS_\scrB(0, y) = \scrB_1(y)$ for $y \in \R$.
	\item For $x > 0$ and $y, z \in \R$, we have 
	\begin{equation}
	\label{E:sheet-form-1}
	\scrS_\scrB(x, y) - \scrS_\scrB(x, z) = \lim_{m \to \infty} \scrB[(-\sqrt{m/(2x)}, m) \to (y, 1)] - \scrB[(-\sqrt{m/(2x)}, m) \to (z, 1)].
	\end{equation}
	\item For any $x \in \Q \cap (0, \infty)$ and $y \in \R$, we have
\begin{equation}
\label{E:S-integrate}
\scrS_\scrB(x, y) = \lim_{a \to \infty} \frac{1}{a}\int_{-a}^0 \lf(\scrS_\scrB(x, y) - \scrS_\scrB(x, z) - (x - z)^2 + \xi \rg)dz,
\end{equation}
where $\xi$ is the expectation of the GUE Tracy-Widom distribution.  Note that we could  have integrated on the right-hand side of \eqref{E:S-integrate} over any interval of length $a$ containing $0$.
\end{itemize}
\end{definition}

It turns out that almost surely, all the limits above exist, and the resulting function $\scrS_\scrB$ is continuous. The existence of such an object follows from \cite[Theorem 8.3]{DOV}. The first bullet is part of \cite[Definition 8.1(ii)]{DOV}, the second bullet is \cite[Remark 8.1]{DOV}, and the third bullet is given by the second display in the proof of \cite[Proposition 8.2]{DOV}.

The half-Airy sheet can be extended to all of $\R^2$ by a stationarity relationship, see \cite[Definition 8.1 and Theorem 8.3]{DOV}.

\begin{definition}
\label{D:airysheet}
The \textbf{Airy sheet} is the unique (in law) random continuous function $\scrS:\R^2 \to \R$ satisfying
\begin{itemize}
	\item $\scrS(\cdot, \cdot) \eqd \scrS(t + \cdot, t + \cdot)$ for all $t \in \R$
	\item $\scrS|_{[0, \infty) \X \R}$ is a half Airy sheet.
\end{itemize}
\end{definition} 

Having defined the Airy sheet, we can now state the main convergence result from \cite{DOV}. When stating this result, we also record convergence information for rightmost geodesics across $\scrB^n$ (parts (ii, iii) below) which is the crucial input in defining the Airy sheet.

\begin{theorem}
	\label{T:sheet-structure} For any subsequence $Y \sset \N$, there exists a further subsequence $Y' \sset Y$ and a coupling of $\scrB$, and $\{\scrB^n : n \in Y\}$ such that the following statements all hold almost surely:
	\begin{enumerate}[label=(\roman*)]
		\item The pair $(\scrB^n, \scrS^n|_{[0, \infty) \X \R})$ converges in the uniform-on-compact topology to $(\scrB, \scrS_\scrB)$. Here $\scrB$ is a parabolic Airy line ensemble, and $\scrS_\scrB$ is the half-Airy sheet of $\scrB$.
		\item Let $Z^n_{m}(x, y)$ denote the jump time from line $m+1$ to $m$ for the rightmost geodesic $\pi^n\{x, y\}$ from $(x-n^{1/3}/2, n)$ to $(y, 1)$. For all $x \in \Q \cap (0, \infty), y \in \Q, m \in \N$, the random variables $Z^n_m(x, y)$ converge almost surely to limits $Z_m(x, y)$.  Moreover, 
		$$
		\lim_{m \to \infty} \frac{Z_m(x, y)}{\sqrt{m}} = \frac{-1}{\sqrt{2x}}.
		$$
		\item For every $x \in \Q \cap (0, \infty)$ and $y < z \in \Q$, there are points $X_1 < x < X_2$ with $X_1, X_2 \in \Q \cap (0, \infty)$ and $T < \min(y, z)$ such that for all large enough $n$, we have
		$$
		\Ga(\pi^n\{X_1, y\}|_{[T, y]}) \cap \Ga(\pi^n\{X_2, z\}|_{[T, z]}) \ne \emptyset.
		$$
		Here recall that $\Ga(\pi)$ denotes the zigzag graph of $\pi$.
		\end{enumerate}
\end{theorem}

This coupling is constructed in \cite[Section 8]{DOV} (after Lemma 8.4). The construction there shows that condition (ii) above is satisfied. Property (i) of the coupling is shown as \cite[Lemma 8.5]{DOV}, and property (iii) of the coupling is shown in the proof of \cite[Lemma 8.5]{DOV}.
Note that the notion of a point `lying along the path $\pi$' used in that proof means that a point is contained in the zigzag graph of $\pi$. 
We remark that while the rationals $\Q$ are used in the Theorem \ref{T:sheet-structure}, they play no special role. The theorem would still hold with any other countable dense set $D$ in place of $\Q$.

To prove convergence of $\scrS^n$ jointly over all $k$ and $\bx, \by\in\R^k_\le$, we will similarly focus on understanding optimizers across $\scrB^n$. We use Theorem \ref{T:sheet-structure} of \cite{DOV} as a starting point for our analysis. In the remainder of this section, we record a few auxiliary results from \cite{DOV} that will also be needed in our analysis, along with some simple consequences of that paper.
We start with two technical lemmas that are stepping stones along the path to Theorem \ref{T:sheet-structure}. 

For the first lemma, for a random array $\{R_{n, m} :n, m \in \N\}$, we write
\begin{equation}
\label{E:onotation}
R_{n,m} = \oo(r_{m})  \qquad \text{ if for all }\epsilon>0 \qquad  \;\; \sum_{m=1}^\infty \limsup_{n\to\infty} \prob(|R_{n,m}/r_{m}|>\epsilon) < \infty.
\end{equation}
\begin{lemma}[\protect{\cite[Lemma 7.1]{DOV}}]
	\label{L:zk-unif-bound}
	Let $K$ be a compact subset of $(0, \infty) \X \R$. Then we have
	$$
	\sup_{(x,y)\in K} \lf|Z^n_m(x, y) + \sqrt{\frac{m}{2x}}\rg| =\oo(\sqrt{m})
	$$
	and $Z^n_m(x, y)$ is tight as a function of $n$ for each fixed $m \in \N, (x, y) \in (0, \infty) \X \R$.
\end{lemma}

We also require a useful lemma about disjointness of geodesics.

\begin{lemma}[\protect{\cite[Lemma 7.2]{DOV}}]
	\label{L:close-not-dj}
	Fix $x > 0$ and $y_1 < y_2$. Then
	$$
	\lim_{\ep \to 0^+} \limsup_{n \to \infty} \p \big( \pi^n\{x - \ep, y_1\} \;\;\; \mathand \;\;\; \pi^n\{x + \ep, y_2\} \text{ are essentially disjoint} \big) = 0.
	$$
\end{lemma}
Note that these two lemmas are quite non-trivial. In particular, Lemma \ref{L:zk-unif-bound} requires the full power of a difficult structural theorem for the Airy line ensemble from \cite{DV}.

We also record a corollary of Lemma \ref{L:close-not-dj} that follows from symmetries of Brownian LPP.

\begin{corollary}
	\label{C:close-center-DJ}
	Fix $y_1 < y_2$. Then
	$$
	\lim_{\ep \to 0^+} \limsup_{n \to \infty} \p \big( \pi^n\{0, y_1\} \;\;\; \mathand \;\;\; \pi^n\{\ep, y_2\} \text{ are essentially disjoint} \big) = 0.
	$$
\end{corollary}

\begin{proof}
We write $\pi^n[x, y]$ for the rightmost geodesic across the original Brownian motions from $(2n^{-1/3}x, n)$ to $(1 + 2n^{-1/3}y, 1)$. 
	By Lemma \ref{L:close-not-dj} and Lemma \ref{L:disjoint-melon}, for any $x > 0$ we have
	$$
	\lim_{\ep \to 0^+} \limsup_{n \to \infty} \p \big( \pi^n[x-\ep, y_1] \;\;\; \mathand \;\;\; \pi^n[x + \ep, y_2] \text{ are essentially disjoint} \big) = 0.
	$$
	for any $x > 0$. Translation invariance of Brownian increments implies that the above statement holds for any $x \in \R$, not just $x > 0$, and monotonicity of geodesics (Proposition \ref{P:tree-structure}) implies that the statement holds with $x -\ep$ replaced by $x$. Setting $x = 0$ and translating back to the melon environment via Lemma \ref{L:disjoint-melon} yields the result.
\end{proof}

Next, we record a few basic facts and symmetries about the Airy sheet.
\begin{lemma}[see \protect{\cite[Lemma 9.1 and Remark 1.1.6]{DOV}}]
\label{L:skew-sym}
The process $(x, y) \mapsto \scrS(x, y) + (x-y)^2$ is translation invariant in both $x$ and $y$. Also, $\scrS(x, y) \eqd \scrS(-y, -x)$. Here the distributional equality is joint in all $x, y \in \R$. 

Moreover, $\scrS(0,0)$ has GUE Tracy-Widom distribution, and hence satisfies the tail bound
$$
\p(|\scrS(0,0)| > m) \le c e ^{-dm^{3/2}}
$$
for universal constants $c, d > 0$ and all $m > 0$.
\end{lemma}
Note that the tail bound in Lemma \ref{L:skew-sym} on the GUE Tracy-Widom distribution goes back to \cite{tracy1994level}.

We end this section by giving a more flexible description of the difference $\scrS_\scrB(x, y) - \scrS_\scrB(x, z)$ defined in \eqref{E:sheet-form-1}.

\begin{lemma}  \label{l:sheet-diff-alter}
Almost surely the following is true.
Take any $x\ge 0$ and $z_1<z_2$.
Let $\pi:(-\infty, z_1]\to\N$ be a nonincreasing cadlag function, such that $\displaystyle \lim_{y \to -\infty} \frac{\pi(y)}{2 y^2} = x$. Then
\begin{equation}
\label{E:S-scrB}
\scrS_\scrB(x,z_1)-\scrS_\scrB(x,z_2) = \lim_{y\to-\infty}\scrB[(y,\pi(y))\to (z_1,1)] - \scrB[(y,\pi(y))\to (z_2,1)]. 
\end{equation}
\end{lemma}
\begin{proof}
By applying Lemma \ref{L:quadrangle} with the endpoint pairs $((y,\pi(y)), (z_2,1))$ and $((-\sqrt{\pi(y)/(2x+\delta)},\pi(y)), (z_1,1))$, and sending $y\to-\infty$, we have that for any $\delta>0$, 
\[
\begin{split}
&\liminf_{y\to-\infty}\scrB[(y,\pi(y))\to (z_1,1)] - \scrB[(y,\pi(y))\to (z_2,1)]\\
\ge\; & \lim_{m\to\infty}\scrB[(-\sqrt{m/(2x+\delta)},m)\to (z_1,1)] - \scrB[(-\sqrt{m/(2x+\delta)},m)\to (z_2,1)]
\\
=\; & \;
\scrS_\scrB(x+\delta/2,z_1)-\scrS_\scrB(x+\delta/2,z_2)
.
\end{split}
\]
Therefore by continuity of $\scrS_\scrB$, the right-hand side of \eqref{E:S-scrB} is bounded below by the left-hand side. For $x > 0$, the opposite inequality holds by symmetric reasoning. For $x = 0$, the opposite inequality holds since 
\[
\scrB[p \to (z_1,1)] - \scrB[p \to (z_2,1)] \le \scrB_1(z_1) - \scrB_1(z_2) = \scrS_\scrB(0, z_1) -\scrS_\scrB(0, z_2). 
\] 
for any point $p \in (-\infty, z_1] \X \N$. Indeed, any path $\pi$ from $p$ to $(z_1, 1)$ can always be extended to a path from $p$ to $(z_2, 1)$ by extending $\pi$ to be equal to $1$ on the interval $[z_1, z_2]$. This picks up the increment $\scrB_1(z_2) - \scrB_1(z_1)$.
\end{proof}

\section{Tightness}
\label{S:tightness}

\subsection{Tightness of prelimiting sheets}

Recall from Theorem \ref{T:extended-sheet} the space $
\fX = \bigcup_{k=1}^\infty \R_{\le}^k \X \R_{\le}^k.
$ Topologically, $\fX$ is a disjoint union of certain subsets of $\R^{2k}$. Let $\scrC(\fX, \R)$ be the space of functions from $\fX$ to $\R$ with the uniform-on-compact topology.
The main goal of this section is to prove the following theorem.

\begin{theorem}
	\label{T:S-tight}
	The functions $\scrS^n$ are tight in $\scrC(\fX, \R)$.
\end{theorem}

Note that $\scrS^n$ (from formula \eqref{E:Sndef}) is not defined on all of $\fX$. To formally define $\scrS^n$ as a random element of $\scrC(\fX, \R)$, we arbitrarily extend $\scrS$ to all of $\fX$ in a continuous way so that $\scrS^n \in \scrC(\fX, \R)$. For any compact set $K \sset \fX$, $\scrS^n|_K$ is well-defined by \eqref{E:Sndef} for all large enough $n$, so the arbitrary choice of extension does not affect any convergence or tightness statements.

Theorem \ref{T:S-tight} will follow from the deterministic bounds and inequalities in Section \ref{S:basic-properties}, and explicit tightness bounds for the prelimiting Airy line ensemble, which we quote from \cite{DV}. For this proposition, $W^n$ is a Brownian melon.

\begin{prop}[\protect{\cite[Proposition 4.1]{DV}}]
	\label{P:dyson-tails}
	Fix $k \in \N$ and $c > 0$. There exist constants $c_k, d_k > 0$ such that for every $n \in \N$, $t > 0, \;s \in (0, ctn^{-1/3}]$, and $a > 0$ we have
	$$
	\p\bigg(\lf| W_k^n(t) - W_k^n(t +s) + \frac{s\sqrt{n}}{\sqrt{t}}\rg| > a \sqrt{s} \bigg) \le c_ke^{-d_ka^{3/2}}.
	$$
\end{prop}

We can translate Proposition \ref{P:dyson-tails} into a modulus of continuity on the prelimiting parabolic Airy line ensembles $\scrB^n$. To do this, we will employ a general lemma for establishing a modulus of continuity, also developed in \cite{DV}. This lemma will also be used later on when establishing a general modulus of continuity for the extended directed landscape.

\begin{lemma}[\protect{\cite[Lemma 3.3]{DV}}]
	Let $T=I_1\X \dots \X I_k$ be a product of bounded real intervals of length $b_1, \dots, b_k$. Let $c, d>0$.
	Let $\scrH$ be a random continuous function from $T$ taking values in a vector space $V$ with norm $|\cdot |$. Assume that for every $i \in \II{1, k}$, that there exist $\al_i \in (0,1), \beta_i, r_i > 0$ such that
	\label{L:levy-est}
	\begin{equation}
	\label{E:tail-bd}
	\p(|\scrH(t+e_i u) - \scrH(t)| \ge a u^{{\alpha_i}}) \le c e^{-d{a^{{\beta_i}}}}
	\end{equation}
	for every coordinate vector $e_i$, every $a>0$, and every $t,t+u e_i\in T$ with $u < r_i$. Set $\beta = \min_i \beta_i, \al = \max_i \al_i$, and $r = \max_i r_i$. Then with probability one we have
	\begin{equation}
	|\scrH(t + s) - \scrH(t)| \le C \lf(\sum_{i=1}^k |s_i|^{\al_i} \log^{1/\beta_i} \lf(\frac{2 r^{\al/\al_i}}{|s_i|} \rg) \rg),
	\end{equation}
	for every $t,t+s\in T$ with $|s_i| \le r_i$ for all $i$ (here $s = (s_1, \dots, s_k)$).
	Here $C$ is random constant satisfying
	$$
	\p(C > a) \le \lf[\prod_{i=1}^k \frac{b_i}{r_i} \rg] c c_0 e^{-c_1 a^{\beta}},
	$$
	where $c_0$ and $c_1$ are constants that depend on $\al_1, \dots, \al_k, \beta_1, \dots, \beta_k, k$ and $d$. Notably, they do not depend on $b_1, \dots, b_k ,c$ or $r_1, \dots, r_k$.
\end{lemma}

\begin{corollary}  \label{C:dyson-tails-uniform}
Fix $k \in \N$ and $c > 0$. There exist positive constants $c_k$, $d_k$ such that for every $n \in \N$, $t > 0$, $s \in [0, ct n^{-1/3}]$, and $a > 0$ we have
	\begin{equation}
	\label{E:biggmax}
	\p\bigg(\max_{t\le x<y\le t+s}\lf| W_k^n(x) - W_k^n(y) + \frac{(y-x)\sqrt{n}}{\sqrt{t}}\rg| > a\sqrt{s} \bigg) \le c_ke^{-d_ka^{3/2}}.
	\end{equation}
\end{corollary}

\begin{proof}
We apply Lemma \ref{L:levy-est} with $k=1$, $\al_1 = 1/2, \beta_1 = 3/2$ and the function $W_k^n(x) - x\sqrt{n/t}$.  The assumption \eqref{E:tail-bd} is implied by Proposition \ref{P:dyson-tails}. Then we have
	$$
	\lf| W_k^n(x) - W_k^n(y) + \frac{(y-x)\sqrt{n}}{\sqrt{t}}\rg| \le C \sqrt{y-x} \log^{2/3}\left(\frac{2s}{y-x}\right)
	$$
	for all $1\le x<y\le 1+s$, where $C$ is a random constant satisfying the tail bound on the right-hand side of \eqref{E:biggmax} for some constants $c_k, d_k$. The right-hand side above is bounded above by $C \sqrt{s}$ for all $t\le x<y\le t+s$, yielding \eqref{E:biggmax}. 
\end{proof}

We are now in a position to prove a two-point tail bound for $\scrS^n$. We first define the stationary version $\scrR^n: \fX \to \R$ by
$$
\scrR^n(\bx, \by) = \scrS^n(\bx, \by) + \sum_{i=1}^{k}(x_i - y_i)^2,
$$
where $\bx, \by \in \R^k_\le$.
\begin{lemma} \label{l:change-spatial-prelim}
	Take any $k,n\in\N$, $\bu = (\bx, \by), \bu' = (\bx',\by') \in \R^k_\le \X \R^k_\le$ with $\|\bu - \bu'\|_2 < 1$, $\|\bx\|_2, \|\by\|_2, \|\bx'\|_2, \|\by'\|_2 < n^{1/6}$ and $a>0$.
	Then
	$$
	\p(|\scrR^n(\bx',\by') - \scrR^n(\bx,\by)| >
	a\sqrt{\|\bu - \bu'\|_2}) < c e^{-da^{3/2}} ,
	$$
	for some constants $c, d > 0$ depending only on $k$.
\end{lemma}

In this proof and throughout the paper, for $\bx \in \R^k_\le$ we write $-\bx$ for the unique element of $\R^k_\le$ given by rearranging the coordinates of $-\bx$. The basic idea of the proof of Lemma \ref{l:change-spatial-prelim} is to reduce bounds on differences in $\scrR$ to bounds on differences of lines in $W^n$ by using deterministic last passage inequalities from Section \ref{S:basic-properties}.

\begin{proof}[Proof of Lemma \ref{l:change-spatial-prelim}]
We first consider fixed $n$.
Recall (from Theorem \ref{T:extended-sheet}) that $\scrS^n$ is defined using multi-point last passage values across a collection of independent two-sided standard Brownian motions $B = \{B_i : i \in \Z\}$.
By using Lemma \ref{L:naive-bounds-diff} repeatedly and changing one coordinate at a time, we have
\[
|\scrS^n(\bx',\by') - \scrS^n(\bx,\by)| < 2k(2k+n) n^{1/6}\sup_{i, a, b} |B_i(a)-B_i(b)| + 2n^{1/3}(\|\bx-\by\|_1+\|\bx'-\by'\|_1),
\]
where the supremum is over all $1\le i \le n$, and $a, b \in \R$, such that both $a, b$ are in either $[2n^{-1/3}(x_i\wedge x_i'), 2n^{-1/3}(x_i\vee x_i')]$ or $[1+2n^{-1/3}(y_i\wedge y_i'), 1+2n^{-1/3}(y_i\vee y_i')]$. 
Therefore by using standard tail bounds on Brownian motion increments, for any fixed $n$ the bound in the lemma holds by taking $c$ large and $d$ small (depending on $n$ and $k$).

For the remainder of the proof it suffices to consider $n$ sufficiently large (depending on $k$).
	By the triangle inequality, and by the symmetry $\scrR^n(\bx, \by) \eqd \scrR^n(-\by, -\bx)$, it suffices to prove the bound when $\bx = \bx'$ and $\by, \by'$ agree at all points except for a single coordinate $y_\ell < y_\ell'$. Moreover, if we let $T_c$ be the map translating all coordinates in a vector by $c$, then $\scrR^n(T_c\bx, T_c\by) \eqd \scrR^n(\bx, \by)$ for all $c$, so we may assume $x_\ell = 0$ if we relax the norm bounds to $\|\bx\|_2, \|\by\|_2, \|\bx'\|_2, \|\by'\|_2 < 2n^{1/6}$.
	With these simplifications, the inequality is equivalent to
	\[
	\p(
	|\scrR^n(\bx, \by') - \scrR^n(\bx, \by)| >  a(y_\ell'-y_\ell)^{1/2}) < c e^{-da^{3/2}} .    
	\]
	Now, by the representation \eqref{E:Sn-compact-form} for $\scrS^n$, we can write 
	\begin{equation}
	\label{E:scrRn}
	\scrR^n(\bx, \by') - \scrR^n(\bx, \by) = \scrA^n[(\bx - n^{1/3}/2, n) \to (\by', 1)] - \scrA^n[(\bx - n^{1/3}/2, n) \to (\by, 1)],
	\end{equation}
	where $\scrA^n_i(x) = \scrB^n_i(x) + x^2$.
	As in Lemma \ref{L:quadrangle-2}, let $\bx^L$ denote the first $\ell$ coordinates of $\bx$ and let $\by^R$ denote the last $k - \ell$ coordinates of $x$. By two applications of that lemma, and \eqref{E:scrRn}, we have
		\[
	\scrR^n(\bx^R, \by'^R) - \scrR^n(\bx^R, \by^R) \le \scrR^n(\bx, \by') - \scrR^n(\bx, \by) \le
	\scrR^n(\bx^L, \by'^L) - \scrR^n(\bx^L, \by^L).
	\]
	Now let $0^i \in \R^i$ denote the vector whose coordinates are all $x_\ell = 0$. We can bound the left- and right-hand sides above using Lemma \ref{L:quadrangle} applied to the points $0^{k-\ell+1} \le \bx^R, \by^R \le \by'^R$ and  $\bx^L \le 0^\ell, \by^L \le \by'^L$ to get that
	\[
	\scrR^n(0^{k-\ell+1}, \by'^R) - \scrR^n(0^{k-\ell+1}, \by^R) \le \scrR^n(\bx, \by') - \scrR^n(\bx, \by) \le
	\scrR^n(0^{\ell}, \by'^L) - \scrR^n(0^{\ell}, \by^L).
	\]
	By these inequalities and \eqref{E:scrRn}, it then suffices to bound
	\begin{align*}
	&\p(\scrA^n[(- n^{1/3}/2, n)^{k-\ell + 1} \to (\by'^R, 1)] - \scrA^n[(- n^{1/3}/2, n)^{k-\ell + 1} \to (\by^R, 1)] < -a(y_\ell'-y_\ell)^{1/2}), \qquad \mathand \\
	&\p(\scrA^n[(- n^{1/3}/2, n)^\ell \to (\by'^L, 1)] - \scrA^n[(- n^{1/3}/2, n)^\ell \to (\by^L, 1)]
	 > a(y_\ell'-y_\ell)^{1/2}).
	\end{align*}
	By Lemma \ref{L:high-paths} applied to $\scrA^n$, for any endpoint pair starting $(- n^{1/3}/2, n)^i$ for some $i \le k$, there is an optimizer that only uses the top $k$ lines. By Lemma \ref{L:naive-bounds-diff}, the above two probabilities are bounded by
	$$
	\p\left(
	2k\max_{1\le i\le k, y_\ell\le x<y\le y_\ell'} |\scrA_i^n(x)-\scrA_i^n(y)| > a(y_\ell'-y_\ell)^{1/2}
	\right).   
	$$
	Rewriting this probability in terms of $W^n$ gives
	\begin{multline}
	\label{E:p-1}
    \p\left(
    2k\max_{1\le i\le k, y_\ell\le x<y\le y_\ell'} |n^{1/6}(W_i^n(1+2n^{-1/3}x)-W_i^n(1+2n^{-1/3}y))+x^2-y^2-2n^{1/3}(x - y)| \right. \\ \left. > a(y_\ell'-y_\ell)^{1/2}
    \right).   
	\end{multline}
	Now, using that $|y_\ell|< 2n^{1/6}$ and $|y_\ell-y_\ell'|<1$, for any $y_\ell\le x<y\le y_\ell'$ and $n$ large enough, by a Taylor expansion we have
	\[
	\lf|\frac{2n^{1/3}(x-y)}{\sqrt{1+2n^{-1/3}y_\ell}} - 2n^{1/3}(x - y) + 2(x-y)y_\ell \rg| \le  3n^{-1/3}(y-x)y_\ell^2 \le 12(y-x).
	\]
Therefore
	\begin{multline}
	\label{E:p-2}
	\left|x^2-y^2-2n^{1/3}(x - y) - n^{1/6}\frac{2n^{-1/3}(y-x)\sqrt{n}}{\sqrt{1+2n^{-1/3}y_\ell}}\right|
	< |x^2-y^2-2(x-y)y_\ell| + 12(y-x) \\
	\le (y-x)(|x+y-2y_\ell| + 12) \le  14(y_\ell'-y_\ell)^{1/2}.
	\end{multline}
	The conclusion then follows by combining \eqref{E:p-1}, \eqref{E:p-2}, Corollary \ref{C:dyson-tails-uniform}, and a union bound.
\end{proof}

\begin{proof}[Proof of Theorem \ref{T:S-tight}]
	It suffices to show that $\scrS^n|_K$ is tight for all compact sets $K \sset \R^k_\le \X \R^k_\le$. We may assume $K$ contains $0$. First, $\scrS^n(0^k, 0^k) = \sum_{i=1}^k \scrB^n_i(0)$, so $\scrS^n(0^k, 0^k)$ is tight by Theorem \ref{T:airy-line-ensemble}. Tightness of $\scrS^n|_K$ then follows from Lemma \ref{l:change-spatial-prelim} and the Kolmorogov-Chentsov criterion, see Corollary 14.9 in \cite{kallenberg2006foundations}.
\end{proof}

\subsection{Tightness of melon optimizers}
\label{S:tightness-melon-optimizers}
Here we prove tightness and asymptotic results about melon optimizers.
For this, we extend the notions of geodesics and jump times from Section \ref{S:Airysheet} to the case where the endpoints are not singletons.
For vectors $\bx, \by \in \R^k_\le$ with $x_1 \ge 0$, we write 
$$
\pi^n\{\bx, \by\} = (\pi^n_1\{\bx, \by\}, \dots, \pi^n_k\{\bx, \by\})
$$
for the rightmost optimizer across $\scrB^n$ from $(\bx-n^{1/3}/2, n)$ to $(\by, 1)$.
We will write $\pi^n[\bx, \by]$ for the rightmost optimizer across the original Brownian motions between the corresponding points $(2n^{-1/3}\bx, n)$ and $(1 + 2n^{-1/3}\by, 1)$. We write
$
Z^n_{i, m}(\bx, \by) 
$
for the jump time from line $m+1$ to $m$ for the path $\pi^n_i\{\bx, \by\}$.

We start with a weak tightness result.

\begin{lemma}
	\label{L:Zk1-tight}
Let $k \in \N$, $x > 0$ and $y \in \R$, and set $x^k = (x, \dots, x), y^k = (y, \dots, y)\in \R^k$. Then for every $i \in \II{1, k}$, the sequence of jump times
$
\{Z^n_{i, 1}(x^k, y^k) : n \in \N\}
$
is tight.
\end{lemma} 

\begin{proof}
We write $\pi^n_i:=\pi^n_i\{x^k, y^k\}$ for the $i^{\mathrm{th}}$ path in the disjoint optimizer $\pi^n\{x^k, y^k\}$. By Lemma \ref{L:one-path-below} and \eqref{E:Sn-compact-form}, we have
$$
\|\pi_i^n\|_{\scrB^n} - n^{2/3} \ge \scrS^n(x^k, y^k) - \scrS^n(x^{k-1}, y^{k-1}).
$$
In particular, by Theorem \ref{T:S-tight}, the random variables
$
Y_n:= \|\pi_i^n\|_{\scrB^n} - n^{2/3}
$
are tight. Now suppose that $Z^n_{i, 1}(x^k, y^k) < r$ for some $r \in \R$. Then $\pi_i^n(z) = 1$ for all $z \in [r, y]$, so 
\begin{align*}
\|\pi_i^n\|_{\scrB^n} &= \|\pi_i^n|_{[x - n^{1/3}/2, r]}\|_{\scrB^n} + \scrB^n_1(y) - \scrB^n_1(r) \\
&\le \scrB^n[(x - n^{1/3}/2, n) \to (r, 1)] + \scrB^n_1(y) - \scrB^n_1(r).
\end{align*}
Therefore by \eqref{E:Sn-compact-form} again, we have
\begin{equation}
\label{E:Yn}
Y_n \wedge 0 \le \mathbf{1}(Z^n_{i, 1}(x^k, y^k) < r)[\scrS^n(x, r) + \scrS^n(0, y) - \scrS^n(0, r)]
\end{equation}
The term multiplying the indicator on right-hand side of \eqref{E:Yn} converges to 
$
X(x, y, r) = \scrS(x, r) + \scrS(0, y) - \scrS(0, r).
$
Since $\scrS(x, y) + (x - y)^2$ has GUE Tracy-Widom distribution for all $x, y \in \R$ (Lemma \ref{L:skew-sym}),  by a union bound, for all $x, y, r, m > 0$ we have
$$
\p (X(x, y, r) > m - x^2 - y^2 + 2 x r) \le ce^{-d m^{3/2}},
$$
for some constants $c,d>0$.
In particular, $X(x, y, r) \cvgd -\infty$ as $r \to -\infty$ for fixed $x, y$. Combining this, \eqref{E:Yn}, and the tightness of $Y_n$ gives that 
$$
\lim_{r \to -\infty} \limsup_{n \to \infty} \p(Z^n_{i, 1}(x^k, y^k) < r) = 0.
$$
Since all the random variables $Z^n_{i, 1}(x^k, y^k)$ are bounded above by $y$, this implies that the sequence $Z^n_{i, 1}(x^k, y^k)$ is tight.
\end{proof}

We can use Lemma \ref{L:Zk1-tight} to prove a disjointness lemma for optimizers.

\begin{lemma}
	\label{L:cucumber-sandwich}
	Consider $\pi^n[x^k, y^k]$, the (almost surely unique) optimizer from $(2xn^{-1/3}, n)^k$ to $(1+ 2yn^{-1/3}, 1)^k$ across $n$ independent standard Brownian motions $B^n$. Then for every $x, y \in \R, k \in \N$ and $\ep > 0$, we have that
	$$
	\lim_{r \to \infty} \liminf_{n \to \infty} \p \lf( \pi^n[x^k, y^k]\text{ is essentially disjoint from } \pi^n[x - \ep, y - r] \mathand \pi^n[x + \ep, y + r] \rg) = 1.
	$$
\end{lemma}

\begin{proof}
Throughout the proof we write $\pi^{n,k}[x,y]:=\pi^n[x^k,y^k]$.
First, by a union bound it suffices to show that for all $x, y \in \R, k \in \N$ and $\ep > 0$,
\begin{align}
\label{E:rinfty1}
&\lim_{r \to \infty} \liminf_{n \to \infty} \p \lf( \pi^{n,k}[x, y]\text{ is essentially disjoint from } \pi^n[x - \ep, y - r]\rg) = 1, \qquad \text{ and} \\
\label{E:rinfty2}
&\lim_{r \to \infty} \liminf_{n \to \infty} \p \lf( \pi^{n,k}[x, y]\text{ is essentially disjoint from } \pi^n[x + \ep, y + r] \rg) = 1.
\end{align}
We first simplify \eqref{E:rinfty1} and \eqref{E:rinfty2}.
Translation invariance of Brownian increments and Brownian scaling gives that
\begin{align*}
B^n(t) &\eqd \al_n^{-1/2} \lf(B^n(\al_n(t - 2n^{-1/3}(x-\ep))) - B^n(2\al_n n^{-1/3}(x-\ep)) \rg), \\
\;\; \text{ where } \;\; \al_n &= \frac{1}{1 + 2(-x + \ep + y)n^{-1/3}},
\end{align*}
and so \eqref{E:rinfty1} is equal to 
$$
\p \lf( \pi^{n,k}[\ep + O(n^{-1/3}), 0]\text{ is essentially disjoint from } \pi^n[0, -r + O(n^{-1/3})]\rg).
$$
Here the $O(n^{-1/3})$ terms are small in the sense that for fixed $r, x, y$, there exists $c > 0$ such that $|O(n^{-1/3})| \le c n^{-1/3}$. 
In particular, for large enough $n$, monotonicity of optimizers (Lemma \ref{L:mono-tree-multi-path}) implies that this is bounded below by
\begin{equation}
\label{E:pinkr1}
\p \lf( \pi^{n,k}[\ep/2, 0]\text{ is essentially disjoint from } \pi^n[0, -r/2]\rg).
\end{equation}
By applying translation invariance and Brownian scaling, we can similarly show that \eqref{E:rinfty2} is equal to 
$$
\p\lf( \pi^{n,k}[0, -r + O(n^{-1/3})]\text{ is essentially disjoint from } \pi^n[\ep + O(n^{-1/3}), 0]\rg),
$$
which is again bounded below by
\begin{equation}
\label{E:pinkr2}
\p\lf( \pi^{n,k}[0, -r/2]\text{ is essentially disjoint from } \pi^n[\ep/2, 0]\rg)
\end{equation}
for large enough $n$. Next, by Lemma \ref{L:disjoint-melon} and Lemma \ref{L:brown-unique}, the probabilities \eqref{E:pinkr1} and \eqref{E:pinkr2} are the same as the corresponding probabilities with melon paths $\pi^n\{, \}$ in place of the original Brownian paths $\pi^n[, ]$. 

Now, for any $0 < b, a < c$ and $k, \ell \in \N$, since the melon path $\pi^n\{0^k, a^k\}$ only uses the top $k$ lines by Lemma \ref{L:high-paths}, $\pi^n\{0^k, a^k\}$ is disjoint from $\pi^n\{b^\ell, c^\ell\}$ whenever the jump time
$$
Z^n_{1, k}(b^\ell, c^\ell) > a.
$$
Therefore to prove \eqref{E:pinkr1} and \eqref{E:pinkr2}, we just need to show that
\begin{align}
\label{E:Zinfty1}
&\lim_{r \to \infty} \liminf_{n \to \infty} \p \lf( Z^n_{1,1}((\ep/2)^k, 0^k) > -r/2 \rg) = 1, \qquad \text{ and} \\
\label{E:Zinfty2}
&\lim_{r \to \infty} \liminf_{n \to \infty} \p \lf( Z^n_k(\ep/2, 0) > - r/2 \rg) = 1.
\end{align}
Equation \eqref{E:Zinfty1} follows from the tightness of $Z^n_{1, 1}(x^k, y^k)$ for fixed $x, y$ in Lemma \ref{L:Zk1-tight}.
Equation \eqref{E:Zinfty2} follows from the tightness of $Z^n_k(\ep/2)$ for fixed $k, \ep$ in Lemma \ref{L:zk-unif-bound}.
\end{proof}

Lemma \ref{L:cucumber-sandwich} can be combined with the asymptotics in Lemma \ref{L:zk-unif-bound} to give tightness and asymptotics for jump times on optimizers across the melon. For this next lemma, we set $(0, \infty)^k_\le = (0, \infty)^k \cap \R^k_\le$.

\begin{lemma}
	\label{L:tight-multipoint} For any $k \in \N$ and any compact set $K \sset (0, \infty)_\le^k \X \R_\le^k$, we have that
	\begin{equation}
	\label{E:xyK}
\sup_{(\bx, \by) \in K, i \in \II{1, k}} \lf|Z^n_{i, m}(\bx, \by) -\sqrt{\frac{m}{2x_i}}\rg| =\oo(\sqrt{m}).
	\end{equation}
	Moreover, for any fixed $\bx, \by, m,$ and $i$, the sequence $Z^n_{i, m}(\bx, \by)$ is tight in $n$.
\end{lemma}

\begin{proof}
We first prove this for a single $(\bx, \by)$. In the notation of Lemma \ref{L:mono-tree-multi-path}, for every $i$ we have $(x_i^k, y_i^k) \le_{1-i} (\bx, \by) \le_{k-i} (x_i^k, y_i^k)$. Therefore by that lemma, we have
$$
Z^n_{1, m}(x_i^k, y_i^k) \le Z^n_{i, m}(\bx, \by) \le Z^n_{k, m}(x_i^k, y_i^k),
$$
so it suffices to prove bounds when $\bx, \by$ consist only of repeated points. For this, observe that on the event $A_{\ep, r}$ where the melon optimizer $\pi^n\{x_i^k, y_i^k\}$ is essentially disjoint from $\pi^n\{x_i - \ep, -r\}$ and $\pi^n\{x_i + \ep, r\}$, that 
\begin{equation}
\label{E:Znjk}
Z^n_{j, m}(x_i^k, y_i^k) \in [Z^n_m(x_i -\ep, -r), Z^n_m(x_i + \ep, r)]
\end{equation}
for all $m \in \N, j \in \II{1, k}$. By Lemma \ref{L:disjoint-melon}, essential disjointness of $\pi^n\{x_i^k, y_i^k\}$ from $\pi^n\{x_i - \ep, -r\}$ and $\pi^n\{x_i + \ep, r\}$ is equivalent to essential disjointness of the original Brownian optimizers $\pi^n[x_i^k, y_i^k]$ from $\pi^n[x_i - \ep, -r]$ and $\pi^n[x_i + \ep, r]$.
Therefore by Lemma \ref{L:cucumber-sandwich},
$$
\lim_{r \to \infty} \liminf_{n \to \infty} \p A_{\ep, r} = 1.
$$
Moreover, the asymptotics of the interval on the right-hand side of \eqref{E:Znjk} are given by Lemma \ref{L:zk-unif-bound}. Putting these together proves \eqref{E:xyK} for a single point. The extension to the entire compact set follows again from monotonicity (Lemma \ref{L:mono-tree-multi-path}).

Finally, the tightness claim for fixed $k$ follows from \eqref{E:xyK}, the definition of the notation $\oo$, and the fact that the $Z^n_{i, m}(\bx, \by)$ are nonincreasing in $m$: $Z^n_{i, 1}(\bx, \by) \ge Z^n_{i, 2}(\bx, \by) \ge \ldots$.
\end{proof}

For this next corollary, we extend the definition of path space to include paths with noncompact domains. Let $\scrP$ be the space of all nonincreasing cadlag functions from any closed interval $I \sset \R$ to $\Z$. For a sequence $\pi_n \in \scrP$, we say that $\pi^n \to \pi$ if 
\begin{equation*}
\Ga(\pi_n) \cap [-n, n] \X \II{-n, n} \to \Ga(\pi) \cap [-n, n] \X \{-n, \dots, n\}
\end{equation*}
in the Hausdorff topology for all $n \in \N$. This is a Polish space, since the Hausdorff topology on paths whose zigzag graphs live in $[-n, n] \X \II{-n, n}$ is Polish for all $n \in \N$.

\begin{corollary}
	\label{C:path-tightness}
For any $(\bx, \by) \in (0, \infty)_\le^k \X \R_\le^k$, the paths $\pi^n\{\bx, \by\}$ are tight in distribution in the product of $k$ path spaces. Subsequential limits are $k$-tuples of nonincreasing paths $\pi_i:(-\infty, y_i] \to \N$. Moreover, any distributional subsequential limit $(\pi, \scrB)$ of $(\pi^n\{\bx, \by\}, \scrB^n)$ satisfies the following property:

For any set of times $\bz = (z_1, \dots, z_k)$ and $m \in \N$ such that $(z_i,m) \in \Ga(\pi_i)$ for all $i$, the restricted paths $\{\pi_i|_{[z_i, y_i]} : i \in \II{1, k}\}$ form a disjoint optimizer in $\scrB$ from $(\bz, k)$ to $(\by, 1)$.

\end{corollary}

\begin{proof}
Tightness is immediate from the tightness of each of the jump time sequences $Z^n_{i, k}(\bx, \by)$ established in Lemma \ref{L:tight-multipoint}, and the definition of the topology on path space. Now, consider a subsequential limit $(\pi, \scrB)$ of $(\pi^n\{\bx, \by\}, \scrB^n)$, a coupling where $(\pi^n\{\bx, \by\}, \scrB^n) \to (\pi, \scrB)$ almost surely, and a set of times $\bz$ as above. Since essential disjointness and path ordering are preversed under taking limits, on the almost sure set where this convergence holds, $\{\pi_i|_{[z_i, y_i]}\}$ is a disjoint $k$-tuple from $(\bz, m)$ to $(\mathbf{y}, 1)$. Moreover, since $\scrB^n \to \scrB$ uniformly on compact sets, we have
\begin{equation}
\label{E:scrBn}
\scrB^n[(\bz, k) \to (\by, 1)] = \sum_{i=1}^m \|\pi_i^n\{\bx, \by\}|_{[z_i, y_i]}\|_{\scrB^n} \to \sum_{i=1}^m  \|\pi_i|_{[z_i, y_i]}\|_{\scrB}
\end{equation}
as $n \to \infty$. Finally, $\scrB^n[(\bz, k) \to (\by, 1)] \to \scrB[(\bz, k) \to (\by, 1)]$ by uniform-on-compact convergence, so \eqref{E:scrBn} implies that $\{\pi_i|_{[z_i, y_i]}\}$ is a disjoint optimizer in $\scrB$ from $(\bz, k)$ to $(\mathbf{y}, 1)$.
\end{proof}

\section{Last passage percolation across the parabolic Airy line ensemble}   \label{S:lpp-across-ALE}

Having established tightness of melon optimizers and prelimiting sheets, our next goal is to construct the limits of these objects. To do this, we introduce a notion of length and last passage percolation for infinite paths in $\scrB$.

\subsection{Parabolic paths, length, and geodesics in $\scrB$}
A \textbf{parabolic path} across $\scrB$ from $x \ge 0$ to $z \in \R$ is a nonincreasing cadlag function $\pi:(-\infty, z] \to \N$ such that 
\begin{equation}
\label{E:asym-dir}
\lim_{y \to -\infty} \frac{\pi(y)}{2 y^2} = x.
\end{equation}
For every $y < z$ define the \textbf{discrepancy of $\pi$ at $y$} by
$$
D_\pi(y) = \|\pi|_{[y, z]}\|_\scrB - \scrB[(y, \pi(y)) \to (z, 1)].
$$
Note that $D_\pi(y) \le 0$ for all $y$.
We then define the length of $\pi$ by
\begin{equation}
\label{E:pi-A}
\|\pi\|_\scrB = \scrS(x, z) + \liminf_{y \to -\infty} D_\pi(y),
\end{equation}
where $\scrS$ is the half-Airy sheet defined from  $\scrB$ as in Definition \ref{D:halfsheet}.
A parabolic path $\pi$ is a \textbf{geodesic} from $x$ to $y$ if the length $\|\pi\|_\scrB$ is finite, and is maximal among all paths in $\scrB$ from $x$ to $z$. A parabolic path $\pi$ is \textbf{locally geodesic} if $\pi|_{[a, b]}$ is a geodesic for every compact interval $[a, b]$.
We first record some basic properties of lengths and geodesics in $\scrB$. The first lemma records useful deterministic facts.

\begin{lemma}
	\label{L:A-paths}
	Let $\scrB$ be a parabolic Airy line ensemble.
	\begin{enumerate}[label=(\roman*)]
		\item For any parabolic path $\pi$, the discrepancy $D_\pi(y)$ is increasing in $y$. In particular, the liminf on the right-hand side of \eqref{E:pi-A} is actually a limit.
		\item A parabolic path $\pi$ from $x$ to $z$ is a geodesic if and only if $\pi$ is locally geodesic, or equivalently $\|\pi\|_\scrB = \scrS(x, z)$.
		\item If $\pi_n$ is a sequence of parabolic paths from $x_n$ to $z_n$ converging to a parabolic path $\pi$ from $x$ to $z$, then $\limsup_{n \to \infty} \|\pi_n\|_\scrB \le \|\pi\|_\scrB$.
	\end{enumerate}
\end{lemma}

\begin{proof}
	For any parabolic path $\pi:(-\infty, z] \to \N$, for $y_1 < y_2\le z$ we have
	$$
	\|\pi|_{[y_1, z]}\|_\scrB = \|\pi|_{[y_1, y_2]}\|_\scrB + \|\pi|_{[y_2, z]}\|_\scrB.
	$$
	Combining this with the triangle inequality for last passage values in $\scrB$ between the points $(\pi(y_1), y_1), (\pi(y_2), y_2),$ and $(\pi(z), z)$, we get that
	$$
	D_\pi(y_2) - D_\pi(y_1) \ge \scrB[(y_1, \pi(y_1)) \to (y_2, \pi(y_2))] - \|\pi|_{[y_1, y_2]}\|_\scrB \ge 0,
	$$
	so $D_\pi$ is increasing, giving (i).
	
	For part (ii), note that $\pi$ is locally geodesic if and only if $D_\pi(y) = 0$ for all $y$, or equivalently, if $\|\pi\|_\scrB = \scrS(x, z)$. Noting that $D_\pi \le 0$ for any path $\pi$, if $D_\pi = 0$, then $\pi$ must be a geodesic. For the opposite direction, suppose that $\pi$ is a path from $x$ to $y$ with 
	$$
\lim_{y \to -\infty} D_\pi(y) = -c < 0.
	$$ 
	Then $D_\pi(y) = -a \in [-c, 0)$ for some $y < z$. We could modify the path $\pi$ by replacing $\pi|_{[y, z]}$ with a geodesic from $(y, \pi(y))$ to $(z, 1)$. The new path $\pi'$ is a also a parabolic path from $x$ to $z$, and
	$$
	D_{\pi'}(y') = D_\pi(y') + a
	$$
	for all $y' < y$. Therefore $\|\pi'\|_\scrB > \|\pi\|_\scrB$, so $\pi$ cannot be a geodesic.  
	
	For part (iii), observe that if $\pi_n \to \pi$, then the domains converge, and by continuity of $\scrB$, the last passage values on any compact interval $[x, y]$ also converge. In particular, $D_{\pi_n}(y) \to D_\pi(y)$ for all $y$. Combining this with the monotonicity from (i) and the continuity of the Airy sheet $\scrS$ (Definition \ref{D:airysheet}) gives (iii). 
\end{proof}

Existence, uniqueness, and other basic structural results about geodesics across $\scrB$ are guaranteed by limiting results for Brownian melons.

\begin{lemma}
	\label{L:A-geod-exist}
	\begin{enumerate}[label=(\roman*)]
		\item (Uniqueness) For any fixed $(x, y) \in [0, \infty) \X \R$, there exists a unique geodesic $\pi\{x, y\}$ in $\scrB$ from $x$ to $y$ almost surely.
		\item (Existence) Almost surely, for every $(x, y) \in [0, \infty) \X \R$, there exists a geodesic $\pi$ in $\scrB$ from $x$ to $y$. Moreover, almost surely for every $x, y \in [0, \infty) \X \R$, there are geodesics $\pi_L\{x,y\}$ and $\pi_R\{x,y\}$ from $x$ to $y$ satisfying $\pi_L\{x,y\}(t) \le\pi (t) \le\pi_R\{x,y\}(t)$ for any geodesic $\pi$ from $x$ to $y$ and all $t \in (-\infty, y]$. We call $\pi_L\{x,y\}$ and $\pi_R\{x,y\}$ the \textbf{leftmost and rightmost geodesics from $x$ to $y$}.
		\item (Overlap in the trunk) For a fixed $x \ge 0$ and $y, y' \in \R$ and any geodesics $\pi$ and $\pi'$ from $x$ to $y$ and $x$ to $y'$, almost surely we have $\pi(z) = \pi'(z)$ for all sufficiently negative $z$.
		\item (Disjointness Structure) Let $x \ge 0$ and $y \in \R$. For any fixed $r > 0$, we have
		\begin{equation}
		\label{E:e0}
		\lim_{\ep \to 0^+} \p (\Ga(\pi\{(x-\ep) \vee 0, y\}) \cap \Ga(\pi\{x + \ep , y + r\}) = \emptyset) = 0.
		\end{equation}
		Also, for any fixed $0 \le x < x'$ and $y \in \R$, we have
		\begin{equation}
		\label{E:R-to-infty}
		\begin{split}
		\lim_{r \to \infty} &\p ( \pi\{x, y\} \mathand \pi\{x', y + r\} \text{ are essentially disjoint}) = 1, \qquad \mathand \\
			\lim_{r \to \infty} &\p ( \pi\{x, y-r\} \mathand \pi\{x', y\} \text{ are essentially disjoint}) = 1.
		\end{split}
		\end{equation}
		\item (Monotonicity and tree structure) Let $\Om$ be the almost sure set where rightmost geodesics $\pi_R\{x, y\}$ from $x$ to $y$ exist for every $(x, y) \in [0, \infty) \X \R$. On $\Om$, for every $0 \le x_1 \le x_2$ and $y_1 \le y_2$, we have 
		$$
		\pi_R\{x_1, y_1\}(t) \le \pi_R\{x_2, y_2\}(t)
		$$
		for all $t \le y_1$, and the overlap of zigzag graphs 
		$$
		\Ga(\pi_R\{x_1, y_1\}) \cap \Ga(\pi_R\{x_2, y_2\})
		$$
		is either empty, or else is the zigzag graph of a cadlag function $\pi$ from a closed interval to $\R$.
	\end{enumerate} 
\end{lemma}

While the proof of Lemma \ref{L:A-geod-exist} is rather lengthy, the basic idea is just to use the limiting structure of Brownian melon geodesics in Theorem \ref{T:sheet-structure}. This theorem guarantees that limits of Brownian melon geodesics are geodesics across the parabolic Airy line ensemble $\scrB$. Moreover, the coalescence claims from this theorem pass to coalescence results for geodesics across $\scrB$.

\begin{proof} We will work with a subsequence $Y \sset \N$ and a coupling of $\scrB^n$ and $\scrB$ so that the following conditions hold almost surely:
	\begin{enumerate}
		\item $\scrB^n \to \scrB$.
		\item For all $(x, z) \in (\Q \cap (0, \infty)) \X \Q$, there exists a geodesic $\pi\{x, z\}$ across $\scrB$ from $x$ to $z$ such that $\pi^n\{x, z\} \to \pi\{x, z\}$.
		\item For all $x \in \Q \cap (0, \infty), z < y \in \Q$, there exist $X_1 < x < X_2$ with $X_1, X_2 \in \Q \cap (0, \infty)$ such that for all large enough $n$, there is a point $(W_n, R_n)$ in the zigzag graph of both $\pi^n\{X_1, z\}$ and $\pi^n\{X_2, y\}$. Moreover, $(W_n, R_n) \to (W, R)$ for some $(W, R) \in \R \X \Z$.
		\item For any $w < y \in \Q$ and $n > m \in \N$, there is almost surely a unique geodesic in $\scrB$ from $(w, n)$ to $(y, m)$.
	\end{enumerate}
The existence of a coupling satisfying conditions 1-3 follows from Theorem \ref{T:sheet-structure}. Condition 1 is immediate from Theorem \ref{T:sheet-structure}(i). Corollary \ref{C:path-tightness} and the asymptotics in Theorem \ref{T:sheet-structure}(ii) guarantees convergence of the finite geodesics $\pi^n\{x, z\}$ to a limiting parabolic path $\pi\{x, z\}$ from $x$ to $z$. The second part of Corollary \ref{C:path-tightness} guarantees that each $\pi\{x, z\}$ is locally geodesic, and hence is a geodesic by Lemma \ref{L:A-paths}(ii). This gives condition 2. 

For condition 3, Theorem \ref{T:sheet-structure}(iii) guarantees that there exist $X_1 < x < X_2$ and $T \in \R$ such that for all large enough $n$, the zigzag graphs of $\pi^n\{X_1, z\}$ and $\pi^n\{X_2, y\}$ overlap on the interval $[T, z]$. Since the paths $\pi^n\{X_1, z\}|_{[T, z]}$ and $\pi^n\{X_2, y\}|_{[T, z]}$ both converge, the region of overlap also converges. Therefore we can find $(W_n, R_n) \in \Ga(\pi^n\{X_1, z\}) \cap \Ga(\pi^n\{X_2, y\})$ that converges to some $(W, R)$. Condition 4 follows from Lemma \ref{L:brown-unique}. For all proofs we work on the almost sure set where the four conditions above hold.

	\textbf{Proof of (i) for $x > 0$:} \qquad Without loss of generality we can assume that $x, y \in \Q$; the general case can be dealt with by working on a version of the above coupling where $(\Q \cap (0, \infty)) \X \Q$ is replaced by $((\Q \cap (0, \infty)) \X \Q) \cup \{(x, y)\}$. The existence of such a coupling still holds in this context, see the discussion after Theorem \ref{T:sheet-structure}. Suppose that $\pi'$ is another geodesic from $x$ to $y$, and let $z < y, z \in \Q$. It is enough to show that $\pi' = \pi\{x, y\}$ on the interval $[z, y]$. 
	
	Let $X_1, X_2$ be as in property $3$ of the coupling for the triple $x, z < y$.  The parabolic shape of the paths $\pi\{X_1, z\}, \pi\{x, y\}, \pi',$ and $\pi\{X_2, z\}$ ensures that for large enough $m \in \N$ we can find times $t_1 < t_2 < W$ and $s, s' \in (t_1, t_2)$ such that
$$
\pi\{X_1, z\}(t_1) = \pi\{X_2, y\}(t_2) = \pi\{x, y\}(s) = \pi'(s') = m
$$
Also, let $r_1, r_2$ be rational times with $s, s' \in (r_1, r_2) \sset (t_1, t_2)$. There are unique finite geodesics $\tau_1, \tau_2$ from $(r_1, m)$, $(r_2, m)$ to $(y, 1)$.
Since the paths $\pi\{X_1, z\}, \pi\{x, y\}, \pi',$ and $\pi\{X_2, z\}$ are locally geodesic by Lemma \ref{L:A-paths}(ii), we can apply the monotonicity in Lemma \ref{L:mono-tree-multi-path}(i) to get that 
$$
\pi\{X_1, z\}|_{[t_1, z]} \le \tau_1 \le \pi\{x, y\}|_{[s, y]} \le \tau_2 \le \pi\{X_1, z\}|_{[t_1, y]}.
$$
The outer two inequalities imply that the point $(W, R)$ is contained in the zigzag graphs of both $\tau_1$ and $\tau_2$. Therefore by the tree structure of geodesics (Proposition \ref{P:tree-structure}) and the uniqueness of $\tau_1, \tau_2$, the paths $\tau_1, \tau_2$ coincide on the interval $[W, y]$. The inner two inequalities above then imply that $\tau_1, \tau_2$ also coincide with $\pi\{x, y\}$ on this interval. The same holds for $\pi'$, and hence $\pi\{x, y\} = \pi'$ on $[W, y]$. Since $W \le z$, this gives the desired claim.

	\textbf{Proof of (ii) for $x > 0$:} \qquad
	By Lemma \ref{L:mono-tree-multi-path}, for any fixed $z$, the functions $(x, y) \mapsto \pi^n\{x, y\}(z)$ are nondecreasing in $x$ and $y$. This property passes to the limits $\pi\{x, y\}$.
	Therefore for any $(x, y) \in (0, \infty) \X \R$, and any monotone decreasing sequences $x_n \downarrow x, y_n \downarrow y$ with $(x_n, y_n)$ rational, the paths $\pi\{x_n, y_n\}$ have a limit in path space. This limit is a function $\pi_R\{x, y\}:(-\infty, y] \to \N$. Since $\pi_R\{x, y\} \le \pi\{x_n, y_n\}$ for all $n$, and each $\pi$ is a parabolic path from $x_n$ to $y_n$, we have
	\begin{equation}
	\label{E:y-infty}
\limsup_{z \to - \infty} \frac{\pi_R\{x, y\}(z)}{2z^2} \le x.
	\end{equation}
	Again by monotonicity, for any rational points $x' < x, y' < y$ and any $n$, we have $\pi\{x_n, y_n\} \ge \pi\{x', y'\}$. Therefore $\pi_R\{x, y\} \ge \pi\{x', y'\}$ as well, and so \eqref{E:y-infty} is an equality with the $\limsup$ replaced by a limit, and hence $\pi_R\{x, y\}$ is a parabolic path from $x$ to $y$. The fact that $\pi_R\{x, y\}$ is a geodesic follows from Lemma \ref{L:A-paths}(iii) and continuity of the Airy sheet $\scrS$ (see Definition \ref{D:airysheet}). This proves existence of geodesics for $x > 0$. 
	
	Next, we show that each $\pi_R\{x, y\}$ must be the rightmost geodesic from $x$ to $y$. Suppose that there were another geodesic $\pi'$ with $\pi'(t) > \pi_R\{x, y\}(t)$ for some $t \in (-\infty, y)$. Since both $\pi, \pi_R\{x, y\}$ are cadlag, there must exist $\ep > 0$ such that $\pi'(s) > \pi_R\{x, y\}(s)$ for all $s \in [t, t+\ep]$. Zigzag graph convergence of $\pi\{x_n, y_n\}$ to $\pi_R\{x, y\}$ implies pointwise convergence at all continuity points of $\pi_R\{x, y\}$. In particular, pointwise convergence holds for some $s \in [t, t + \ep]$. Therefore for all large enough $n$, we have 
	\begin{equation}
	\label{E:ne-to}
	\pi\{x_n, y_n\}(s) = \pi_R\{x, y\}(s) < \pi'(s).
	\end{equation}
	Now define a new function $\pi^*$ on $(-\infty, x_n]$ by $\pi^*(t) = \max \{\pi\{x_n, y_n\}(t), \pi'(t)\}$ for $t \le y$ and $\pi^* = \pi\{x_n, y_n\}$ on $(y, y_n]$. The function $\pi^*$ is a parabolic path from $x_n$ to $y_n$. Also, since geodesics are locally geodesic, $\pi^*$ must also be locally geodesic, and hence is a geodesic from $x_n$ to $y_n$. Since $\pi\{x_n, y_n\} \ne \pi^*$ by \eqref{E:ne-to}, this contradicts the uniqueness of $\pi\{x_n, y_n\}$ shown in (i). The existence of leftmost geodesics is similar.
	
	\textbf{Proof of (iii) for $x > 0$:} \qquad Let $\pi, \pi'$ be two geodesics from $x$ to two points $y < y'$. As in the proof of (i), for every $z < \min (y, y')$, we can find a time $W_z \le z$ and a location $R_z$ such that $(W_z, R_z)$ lies on the zigzag graphs of both $\pi$ and $\pi'$. Let $W = \{W_z : z < \min(y, y')\}$. Also, for any rational points $q < q' < \min(y, y')$, condition 4 of the coupling ensures that $\pi|_{[q, q']}$ is the unique geodesic from $(q, \pi(q))$ to $(q', \pi(q'))$, and $\pi'|_{[q, q']}$ is the unique geodesic from $(q, \pi'(q))$ to $(q', \pi'(q'))$. Therefore $\pi$ and $\pi'$ must agree on the half-open interval 
	$$
	[\inf (W \cap [q, q']), \sup (W \cap [q, q'])).
	$$
	Since $W$ is nonempty and unbounded below, and $q, q'$ were arbitrary, this implies that $\pi(z) = \pi'(z)$ for all sufficiently negative $z$.
	
	\textbf{Proof of (iv) for $x > 0$:} \qquad We start with \eqref{E:e0}. Since $x > 0$, we may replace $(x- \ep) \vee 0$ with $x - \ep$. Also, without loss of generality, we may assume that $x, y + r \in \Q$. By the monotonicity of the paths $\pi\{x, z\}$ in $x$ and $z$, to show the statement \eqref{E:e0}, it is enough to find random $X_1 < x < X_2$ with $X_1, X_2 \in \Q$ such that the zigzag graphs of 
	$$
	\pi\{X_1, y\}, \pi\{X_2, y+r\}
	$$
	overlap. This follows from condition 3 of the coupling. 
	
	Equation \eqref{E:R-to-infty} in the finite-$n$ case follows from Lemma \ref{L:cucumber-sandwich} with $k = 1$, and the translation of essential disjointness of optimizers across the original Brownian motions to essential disjointness of optimizers across the melon, Lemma \ref{L:disjoint-melon}. To pass to the limiting paths from the finite-$n$ statement of Lemma \ref{L:cucumber-sandwich}, we use that essential disjointness is a closed property in path space.
	
\textbf{Proofs of (i), (ii), (iii), and (iv) for $x = 0$:} \qquad
The path $\pi_R\{0, y\} := 1$ is locally geodesic, and hence is always a geodesic from $0$ to $y$ by Lemma \ref{L:A-paths}(ii).
We next prove (iv) when $x=0$ when $\pi\{0, y\}$ is replaced by the path $\pi_R\{0, y\}$. We will later show that $\pi_R\{0, y\}$ is almost surely the unique geodesic from $0$ to $y$, proving (i). 

The path $\pi_R\{0, y\}$ is the almost sure limit of the melon optimizers $\pi^n\{0, y\}$, which simply follow the top path by Lemma \ref{L:high-paths}. In particular, \eqref{E:R-to-infty} then follows from the exact same argument as in the $x \ne 0$ case. 

For the first part of (iv), by Corollary \ref{C:close-center-DJ} and the fact that $\pi^n\{0, y\}(z)= 1$ for all $n, y, z$ we have $\pi\{\ep, y+r\} \to \pi_R\{0, y + r\}$ almost surely in path space as $\ep \to 0^+$. Equation \eqref{E:e0} follows since $\pi^R\{0, y\} = 1$. 

To prove (i), (ii) and (iii) for $x = 0$, we just need to show that almost surely for all $y \in \R$, $\pi_R\{0, y\}$ is the only geodesic from $0$ to $y$. By the first part of \eqref{E:R-to-infty} for the paths $\pi_R\{0, y\}$, we can work on an almost sure set where for every $z \in \Q$ we have $\pi\{\ep, z\} \to \pi_R\{0, z\}$ almost surely in path space as $\ep \to 0^+, \ep \in \Q$. 

Now let $y \in \R$, and suppose that $\pi'$ is any geodesic from $0$ to $y$. For any $\ep \in \Q \cap (0, \infty)$ and $z \in \Q \cap (y, \infty)$, monotonicity of geodesics (Proposition \ref{P:tree-structure}) and the uniqueness of $\pi\{\ep, z\}$ implies that $\pi'(t) \le \pi\{\ep, z\}(t)$ for $t \in (-\infty, y]$. Since $\pi\{\ep, z\} \to \pi_R\{0, z\}$ as $\ep \to 0^+$, this implies $\pi'(t) \le \pi_R\{0, z\}(t) = 1$ for $t \in (-\infty, y]$ and hence $\pi'(t) = 1 = \pi_R\{0, y\}(t) $ for $t \in (-\infty, y]$.

\textbf{Proof of (v):} \qquad This follows from the fact that rightmost (and leftmost) geodesics are locally rightmost (and leftmost) geodesics, and the corresponding result in the finite case, Proposition \ref{P:tree-structure}.
\end{proof}

For a parabolic path across $\scrB$, the definition of its length from \eqref{E:pi-A} is not easy to work with, as it involves a $\liminf$ of the discrepancy.
Thus we record the following lemma, which follows from \eqref{E:pi-A} and Lemma \ref{l:sheet-diff-alter}, and says that for two parabolic paths that agree off of a compact set, their difference in length can be computed locally.
\begin{lemma}  \label{l:diff-weight-finite}
The following statement holds almost surely.
Let $\pi_1,\pi_2$ be any two parabolic paths across $\scrB$ from any point $x$ to any points $z_1, z_2$ respectively, such that for some $z_0<z_1\wedge z_2$, we have $\pi_1(y)=\pi_2(y)$ for any $y\le z_0$.
Then
\[
\|\pi_1\|_\scrB -
\|\pi_1|_{[z_0,z_1]}\|_\scrB
=
\|\pi_2\|_\scrB -
\|\pi_2|_{[z_0,z_2]}\|_\scrB.
\]
\end{lemma}

From this lemma we can deduce the following measurability result. Informally, this result 
says that for any parabolic path $\pi$ without any `jump points' inside a compact set $[a, b] \X \llbracket 1,k\rrbracket$, the length of $\pi$ is determined by the values of $\scrB$ outside of $[a, b] \X \llbracket 1,k\rrbracket$ (see Figure \ref{fig:meas}).

\begin{figure}[hbt!]
    \centering
\begin{tikzpicture}[line cap=round,line join=round,>=triangle 45,x=0.8cm,y=0.4cm]
\clip(-0.5,-1) rectangle (10.5,11);

\fill[line width=0.pt,color=green,fill=green,fill opacity=0.2]
(4,4)-- (7,4)-- (7,11)-- (4,11) -- cycle;

\draw [line width=1pt, color=blue] (4, 10) -- (3.8,10) -- (3.8,8) -- (2.5,8) -- (2.5,6) -- (2.2,6) -- (2.2,4) -- (1.4,4) -- (1.4,2) -- (0.5,2) -- (0.5,0) -- (0,0);

\draw [line width=1pt, color=blue] (9, 10) -- (8,10) -- (8,8) -- (7.6,8) -- (7.6,6)  -- (3.7,6) -- (3.7,4) -- (2.4,4) -- (2.4,2) -- (0.5,2) -- (0.5,0) -- (0,0);

\draw (0,25) -- (10,25);

\foreach \i in {1,...,6}
{
\draw [dotted] [thick] (0, 12-2*\i) -- (10, 12-2*\i);
\begin{scriptsize}
\draw (0,12-2*\i) node[anchor=east]{$\i$};
\end{scriptsize}
}

\draw [fill=uuuuuu] (4,10) circle (1pt);

\begin{scriptsize}
\draw (4,10) node[anchor=south]{$z_I$};
\draw (6,10) node[anchor=south]{$I \X \llbracket 1, k\rrbracket$};
\draw (3.6,9) node[anchor=east]{$\pi'$};
\draw (8.2,9) node[anchor=west]{$\pi$};
\end{scriptsize}

\end{tikzpicture}
\caption{
An illustration of Lemma \ref{l:measurable-B-path-weight} and its proof: for a parabolic path $\pi$ that does not have any jump point in the region $I \X \llbracket 1, k \rrbracket$, its length $\|\pi\|_\scrB$ is determined by $\scrB$ outside $I \X \llbracket 1, k \rrbracket$. 
The proof compares $\|\pi\|_\scrB$ with the length of $\pi'$, a parabolic path ending at $z_I$ (the left endpoint of $I$).
The difference is determined by $\scrB$ outside of $I \X \llbracket 1, k \rrbracket$ by Lemma \ref{l:diff-weight-finite}. Using the definition of path length \eqref{E:pi-A}, $\|\pi'\|_\scrB$ is also determined by $\scrB$ outside of $I \X \llbracket 1, k \rrbracket$. 
}
\label{fig:meas}
\end{figure}
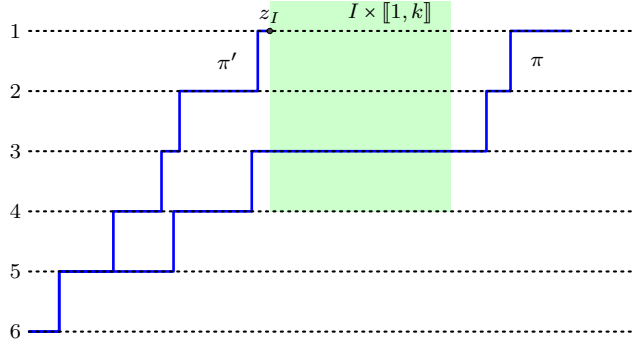

\begin{lemma} \label{l:measurable-B-path-weight}
Take any compact interval $I\subset \R$ and $k\in\N$.
Let $\scrF$ be the $\sigma$-algebra generated by all null sets, all $\scrB_i$ for $i>k$, and $\{\scrB_i(x):x\not\in I\}$ for $1\le i \le k$.
Take any $x\ge 0$, and let $\Sigma_x$ be the set of parabolic paths $\pi$ from $x$ to some $z \in \R$ such that either $\pi(y)>k$ for any $y\in(-\infty,z]\cap I$, or else $I \sset (-\infty, z]$ and $\pi$ is constant on $I$. Let $F:\Sigma_x \to \R$ be the random function recording path length in $\scrB$: $F(\pi) = \|\pi\|_\scrB$. Then $F$ is $\scrF$-measurable.
\end{lemma}

Before proceeding with the proof, let us briefly comment on the technical point that $\scrF$ contains all null sets. This is required to ensure that certain almost sure limits are $\scrF$-measurable. Adding null sets into $\scrF$ does not give us very much extra information since the $\sigma$-algebra generated by null sets contains only the sets $A$ with $\p(A) \in \{0, 1\}$. 

\begin{proof}
Let $z_I$ be the left endpoint of $I$. 
We first show that the length of any parabolic path $\pi'$ from $x$ to $z_I$ is $\scrF$-measurable.
From the definition of path length \eqref{E:pi-A}, it suffices to show that $\scrS(x, z_I)$ is $\scrF$-measurable.

By \eqref{E:sheet-form-1}, for every $y<z_I$ we have that $\scrS(x,z_I)-\scrS(x,y)$ is $\scrF$-measurable.
By \eqref{E:S-integrate} and translation invariance of $\scrS$ (Lemma \ref{L:skew-sym}), outside of a null set we have
\[
\scrS(x, z_I) = \lim_{a \to \infty} \frac{1}{a}\int_{z_I-a}^{z_I} \lf(\scrS(x, z_I) - \scrS(x, y) - (x - y)^2 + \xi \rg)dy,
\]
where $\xi$ is the expectation of a GUE Tracy-Widom random variable.
This implies that $\scrS(x, z_I)$ is $\scrF$-measurable, thus the length of any parabolic path $\pi'$ from $x$ to $z_I$ is $\scrF$-measurable.

For a general parabolic path $\pi \in \Sigma_x$ (i.e., $\pi$ starts and $x$ and does not jump in $I \X \llbracket 1,k\rrbracket$), the idea is to compare the length of $\pi$ with the length of some parabolic path ending at $z_I$ using Lemma \ref{l:diff-weight-finite}, and show that $\|\pi\|_\scrB-\|\pi'\|_\scrB$ is $\scrF$-measurable.
Indeed, one can construct a parabolic path $\pi'$ from $x$ to $z_I$ such that $\pi(y) = \pi'(y)$ for all sufficiently negative $y$; and by Lemma \ref{l:diff-weight-finite} applied to the paths $\pi, \pi'$ we have that $\|\pi\|_\scrB-\|\pi'\|_\scrB$ is $\scrF$-measurable.
The conclusion follows.
\end{proof}
This result will be used in the proof of Lemma \ref{L:two-sheet-sum-unique-max} below, where we use the Brownian Gibbs property (Theorem \ref{T:melon-Airy-facts}) on sets of the form $I \X \llbracket 1,k\rrbracket$ to study distributional properties of parabolic path length.

\subsection{Disjoint optimizers in $\scrB$}

Now that we have a notion of length of parabolic paths in $\scrB$, we can define multi-point last passage values. For $(\bx, \by) \in \fX$ with $x_1 \ge 0$, define
\begin{equation}
\label{E:disjoint-length}
\scrB[\bx \to \by] = \sup_\pi \|\pi\|_\scrB := \sup_{\pi_1, \dots \pi_k} \sum_{i=1}^k\|\pi_i\|_\scrB
\end{equation}
where the supremum is over $k$-tuples of ordered, essentially disjoint parabolic paths from $x_i$ to $y_i$. As in the finite case, we call such a collection $\pi$ a \textbf{disjoint $k$-tuple} from $\bx$ to $\by$, we refer to any disjoint $k$-tuple $\pi = (\pi_1, \dots, \pi_k)$ that attains the above supremum as a \textbf{disjoint optimizer}, as long as $\|\pi\|_\scrB$ is finite. We say that a $k$-tuple $\pi$ is a \textbf{local optimizer} if for all $z \le y_1$, the $k$-tuple consisting of the paths $\pi_i|_{[z, y_i]}$ is a disjoint optimizer. Note that the notation \eqref{E:disjoint-length} is similar to the notation for finite last passage values. The two notations are distinguished by the lack of start and end lines in \eqref{E:disjoint-length}.

We first focus on understanding the structure of disjoint optimizers in $\scrB$ from distinct starting points $\bx = (x_1 < x_2 < \dots < x_k)$ with $x_1 > 0$, as such paths are more easily related to geodesics. 

\begin{prop}
	\label{P:unique-exist}
	Take any $(\bx, \by) \in \fX$ such that $\bx = (x_1 < x_2 < \dots < x_k)$ with $x_1 > 0$.
	\begin{enumerate}[label=(\roman*)]
		\item Suppose that $\pi$ is a disjoint optimizer from $\bx$ to $\by$. Then $\pi$ is locally optimal.
		\item Almost surely there is a unique optimizer $\pi = (\pi_1, \dots, \pi_k)$ from $\bx$ to $\by$ in $\scrB$. Moreover, letting $\pi\{x_i, 0\}$ be the geodesic in $\scrB$ from $x_i$ to $0$, then for every $i$, there exists a (random) $Y \in \R$ such that  $\pi\{x_i, 0\}(t) = \pi_i(t)$ for all $t \le Y$.
		\item Almost surely, the only $k$-tuple $\pi$ from $\bx$ to $\by$ in $\scrB$ which is locally optimal is the unique optimizer from $\bx$ to $\by$.
	\end{enumerate}
\end{prop}

The basic idea of the proof is to first look for pairs $(\bx, \by^-)$ and $(\bx, \by^+)$ with $y_i^- < y_i < y_i^+$ where the optimizers from $\bx$ to $\by^\pm$ simply consists of the geodesics from $x_i$ to $y_i^\pm$. Results about optimizers from $\bx$ to $\by$ can then be inferred using monotonicity and coalescence results.

\begin{proof} 
	Without loss of generality, we may assume that all $x_i, y_i$ are rational.
	For (i), note that if any disjoint optimizer $\pi$ were not locally optimal on some interval of lines $\II{1, m}$, then as in the $k=1$ case of Lemma \ref{L:A-paths}(ii), we can increase its length $\|\pi\|_\scrB$ by replacing $\pi$ on those lines with an optimizer $\pi'$.

	For (ii), we will work on the set where for all $x \in \Q \cap (0, \infty)$ and $y \in \Q$, there is a unique geodesic $\pi\{x, y\}$ from $x$ to $y$ in $\scrB$. We also assume that there is a unique optimizer in $\scrB$ from any rational starting location $\bq = (q_1, \dots, q_k) \in (\Q \X \N)^k$ to $(\by, 1)$. We can do this by Lemma \ref{L:specifics}. (In the more general case where $x_i, y_i$ are not all rational, we just consider the set with $\Q$ replaced by $Q \cup \{x_1, \dots, x_k, y_1, \dots, y_k\}$.
	
	First, by using both parts of \eqref{E:R-to-infty} in Lemma \ref{L:A-geod-exist} (iv), we can find rational $\ep > 0$ and rational points $y_i^\pm$ with
	\begin{equation}
	\label{E:yiorder}
	y^-_1 \ll y^-_2 \ll \dots \ll y^-_k < y_1 \le y_k \ll y^+_1 \ll y^+_2 \dots \ll y^+_k
	\end{equation}
	such that for any $i < j \in \II{1, k}$, the geodesics $\pi\{x_i, y^-_i\}$ and $\pi\{x_j - \ep, y^-_j\}$ are essentially disjoint, as are the  geodesics $\pi\{x_i + \ep, y^+_i\}$ and $\pi\{x_j, y^+_j\}$. 
	
	Now, by monotonicity of geodesics, Lemma \ref{L:A-geod-exist}(v), for any rational $\de \in (0, \ep)$ all of the geodesics $\ga_\de^- = \{\pi\{x_i - \de, y_i^-\} : i \in \II{1, k}\}$ are essentially disjoint from each other. Similarly, the geodesics $\ga_\de^+ = \{\pi\{x_i + \de, y_i^+\} : i \in \II{1,k}\}$ are also essentially disjoint. In particular, the disjoint $k$-tuples $\ga_\de^\pm$ are optimizers. They are also locally optimal by Lemma \ref{L:A-paths}(ii).
	
	We use these geodesics to prove the existence of an optimizer from $\bx$ to $\by$. We start with a disjoint $k$-tuple $\pi^1 = (\pi^1_1, \dots, \pi^1_k)$ from $\bx$ to $\by$. The $k$-tuple $\pi^1$ can be obtained from the geodesics $\pi_1, \dots, \pi_k$ from $x_i$ to $y_i$ in the following way. Since these paths have different asymptotic directions $x_i$, there exists $y^* \in \R$ such that $\pi_i(y) \ne \pi_j(y)$ for all $y < y^*$, and $i \ne j$. Modifying these geodesics in any way for $y > y^*$ to ensure essential disjointness and ordering gives a finite length $k$-tuple from $\bx$ to $\by$. Each of the paths in this modification has finite length by Lemma \ref{l:diff-weight-finite}.
	
	Therefore $\scrB[\bx \to \by] > -\infty$. Let $\pi^m = (\pi^m_1, \dots, \pi^m_k)$ be a sequence of $k$-tuples from $\bx$ to $\by$ whose lengths converge to the supremum in \eqref{E:disjoint-length}. The asymptotic growth rate of parabolic paths guarantees that there exists  a sequence $z_m \to -\infty$ such that
	\begin{equation}
	\label{E:pixiii}
	\pi\{x_i - 1/m, y_i^-\}(z_m) < \pi^m_i(z_m) < \pi\{x_i + 1/m, y_i^+\}(z_m)
	\end{equation}
	for all $i \in \II{1, k}$. Next, for each $m$, modify $\pi^m$ so that $\pi^m|_{(-\infty, z_m]}$ is an optimizer. Doing this can only increase the length, so the new path lengths still converge to the supremal value $\scrB[\bx \to \by]$. Moreover, since the $k$-tuples $\ga^\pm_{1/m}$ are locally optimal, \eqref{E:pixiii}, \eqref{E:yiorder}, and monotonicity of optimizers (Lemma \ref{L:mono-tree-multi-path}) implies that for $t \in [z_m, y_i^-]$, we have
	$$
	\pi\{x_i - 1/m, y_i^-\}(t) \le \pi^m_i(t) \le \pi\{x_i + 1/m, y_i^+\}(t).
	$$
	Now, as $m \to \infty$, each of the path collections $\pi\{x_i - 1/m, y_i^-\}$ converges to $\pi\{x_i, y_i^-\}$ in path space. Similarly each of the paths $\pi\{x_i + 1/m, y_i^+\}$ converges to $\pi\{x_i, y_i^+\}$. This, and monotonicity of geodesics implies that the sequence of $k$-tuples $\pi^m$ is precompact in the product of $k$ path spaces, with subsequential limits $\pi$ that satisfy
	\begin{equation}
	\label{E:xiyi-}
	\pi\{x_i, y_i^-\} \le \pi_i \le  \pi\{x_i, y_i^+\}
	\end{equation}
	for all $i$, and are locally optimal. This implies that $\pi$ is also a $k$-tuple from $\bx$ to $\by$. Since essential disjointness and ordering are preserved under limits, $\pi$ is a disjoint $k$-tuple from $\bx$ to $\by$. Also, Lemma \ref{L:A-paths}(iii) implies that $\|\pi\|_\scrB \ge \scrB[\bx \to \by]$ so $\pi$ is a disjoint optimizer.
	
	Next, we establish uniqueness of $\pi$ by establishing (iii). This also completes the proof of (ii). Let $\ga$ be another $k$-tuple from $\bx$ to $\by$ which is locally optimal, and let $\ga^m = \ga$ for all $m$. By a similar argument as above with $\ga^m$ used in place of $\pi^m$, the bounds \eqref{E:xiyi-} also hold with $\ga_i$ in place of $\pi_i$.
	Next, Lemma \ref{L:A-geod-exist}(iii) and \eqref{E:xiyi-} imply that there exists some $y^*$ such that for all $i$ and all $z < y^*$, $\pi_i(z) = \ga_i(z) = \pi\{x_i, 0\}(z)$. Therefore $\ga, \pi$ are both locally optimal paths which are equal at their endpoints. 
	
	Also, we can find rational points $\{(z_i, m_i) : i \in \II{1, k}\} \in (\Q \cap (-\infty, y^*)) \X \N$ such that $\ga_i(z_i) = \pi_i(z_i) = m_i$ for all $i$. Since we are working on an almost sure set where there are unique optimizers between all rational starting locations and $(\by, 1)$, this implies $\ga = \pi$.
\end{proof}

Understanding the structure of optimizers in $\scrB$ from general starting points is more difficult. We will wait until the construction of the extended Airy sheet to do this.

\section{Limits of melon optimizers and the extended Airy sheet}
\label{S:limits-of-optimizers}

We now have the tools to obtain both the scaling limit of $\scrS^n$, and the joint scaling limit of melon optimizers.
We will focus on first understanding the scaling limit of $\scrS^n$ on the set
$$
\hat \fX = \{(\bx, \by) \in \fX : 0 \le x_1 \le \dots \le x_k\}
$$
Also, let
$$
\hat \Q = \{(\bx, \by) \in \hat \fX : x_i \in \Q \cap (0, \infty), y_i \in \Q \;\; \forall i \in \II{1, k}, \text{ and } 0 < x_1 < x_2 < \dots < x_k \}.
$$
Note that $\hat \fX$ is the closure of $\hat \Q$.
By Theorems \ref{T:airy-line-ensemble}, \ref{T:S-tight}, and Corollary \ref{C:path-tightness}, the functions $\scrB^n, \scrS^n|_{\hat \fX}$ and the paths $\{\pi^n\{\bx, \by\} : (\bx, \by) \in \hat \Q\}$ are jointly tight. Here tightness is with respect to the following topologies:
\begin{itemize}[nosep]
    \item For $\scrB^n$: uniform-on-compact convergence of functions from $\R \X \Z$ to $\R$.
    \item For $\scrS^n|_{\hat \fX}$: uniform-on-compact convergence of functions from $\fX$ to $\R$.
    \item For each of the paths $\pi^n\{\bx, \by\}$: the path space topology defined at the beginning of Section \ref{S:lpp}. 
\end{itemize}
Let 
$$
\scrB, \scrS, \{\pi\{\bx, \by\} : (\bx, \by) \in \hat \Q\}
$$
be any joint distributional subsequential limit along some subsequence $Y$. In this section, we will understand the joint structure of these limiting objects. We start with a lemma and a proposition.

\begin{lemma}
	\label{L:S-sum}
	There exists a subsequence $Y' \sset Y$ such that almost surely,
	$$
	(\scrB^n, \scrS^n, \{\pi^n\{\bx, \by\} : (\bx, \by) \in \hat \Q\}) \to (\scrB, \scrS, \{\pi\{\bx, \by\} : (\bx, \by) \in \hat \Q\}),
	$$
	and for every $\bx \in \Q^k$ with $0 < x_1 < \dots < x_k$, there exist rational points $z_1 < z_2 < \dots < z_k$ such that for all large enough $n$, the paths $\pi^n\{x_i, z_i\}$ are essentially disjoint.
\end{lemma}

\begin{proof}
For $(\bx, \bz) \in \hat \Q$, define the indicator
$$
D^n(\bx, \bz) = \mathbf{1}\lf(\{\pi^n\{x_i, z_i\}\}_{i=1}^k \text{ are essentially disjoint}\rg).
$$
We can find a subsequence $Y' \sset Y$ such that the random variables
$$
\scrB^n, \scrS^n, \{\pi^n\{\bx, \by\} : (\bx, \by) \in \hat \Q\}, \{D^n(\bx, \bz) :  (\bx, \bz) \in \hat \Q \}
$$
converge jointly in distribution. By Skorokhod's representation theorem, we can couple the environments along $Y'$ so that this convergence takes place almost surely. Finally, by Lemma \ref{L:cucumber-sandwich} and Lemma \ref{L:disjoint-melon}, for every $\bx \in \Q^k$ with $0 < x_1 < \dots < x_k$ and any $\ep > 0$, we can find $\bz$ such that
$$
\liminf_{n \to \infty} \E D_n (\bx, \bz) \ge 1 - \ep.
$$	
Therefore on this coupling, the paths $\pi^n\{x_i, z_i\}$ are essentially disjoint for all large enough $n$ with probability at least $1 - \ep$. Since $\ep> 0$ was arbitrary, this holds almost surely for some rational $\bz$. 
\end{proof}

\begin{prop}
	\label{P:BSjoint}
	With notation as above, almost surely the following statements hold. 
	\begin{enumerate}
		\item For all $(\bx, \by) \in \hat \Q$, we have
		$
		\scrS(\bx, \by) = \scrB[\bx \to \by].
		$
		In particular, by continuity $\scrS|_{\hat\fX}$ is a function of $\scrB$.
		\item For all $(\bx, \by) \in \hat \Q$, the $k$-tuple $\pi\{\bx, \by\}$ is the unique optimizer in $\scrB$ from $\bx$ to $\by$. 
	\end{enumerate}
\end{prop}

\begin{proof}
First, Corollary \ref{C:path-tightness} and Lemma \ref{L:tight-multipoint} ensure that each of the $\pi\{\bx, \by\}$ for $(\bx, \by) \in \hat \Q$ is a disjoint $k$-tuple in $\scrB$ from $\bx$ to $\by$ which is locally optimal. Proposition \ref{P:unique-exist}(iii) then implies that $\pi\{\bx, \by\}$ is the unique optimizer in $\scrB$ from $\bx$ to $\by$, yielding statement $2$.

For statement $1$, we first observe that by Theorem \ref{T:sheet-structure}(i), Lemma \ref{L:A-paths}(ii), and the definition \eqref{E:disjoint-length}, we have
$
\scrS(x, y) = \scrB[x \to y]
$
for $x > 0$ and $y \in \R$. Therefore 
to complete the proof it suffices to show that for every $(\bx, \by) \in \hat\Q$ we can find rational points $\bz$ such that 
\begin{equation}
\label{E:difference}
\scrS(\bx, \by) - \sum_{i=1}^k \scrS(x_i, z_i) = \scrB[\bx \to \by] - \sum_{i=1}^k \scrB[x_i \to z_i].
\end{equation}
To prove \eqref{E:difference}, we work with the subsequence $Y'$ and the coupling in Lemma \ref{L:S-sum}.
On this coupling, there exists $\bz$ with $(\bx, \bz) \in \hat \Q$ such that $\pi^n\{x_i, z_i\}$ are essentially disjoint for all large enough $n \in Y'$. Since essential disjointness is a closed condition, the paths $\pi\{x_i, z_i\}$ are also essentially disjoint. Moreover, by Proposition \ref{P:unique-exist}(ii), there exists some $T \in \R$ such that for all $t \le T$, we have
\begin{equation}
\label{E:yiii}
\pi\{x_i, z_i\}(t) =\pi_i\{\bx, \by\}(t).
\end{equation}
In particular, by Lemma \ref{l:diff-weight-finite},
\begin{equation}
\label{E:Bbreakdown}
\scrB[\bx \to \by] - \sum_{i=1}^k \scrB[x_i \to z_i] = \sum_{i=1}^k \|\pi_i\{\bx, \by\}|_{[T, y_i]}\|_\scrB - \|\pi\{x_i, z_i\}|_{[T, z_i]}\|_\scrB,
\end{equation}
Also, \eqref{E:yiii} and the convergence of paths in this coupling implies that there exists $T_n \to T$ such that for all large enough $n \in Y'$, we have
\begin{equation}
\label{E:yiiiy}
\pi^n\{x_i, z_i\}(T_n) =\pi^n_i\{\bx, \by\}(T_n).
\end{equation}
When the $\pi^n\{x_i, z_i\}$ are essentially disjoint, this equality also holds for all $y < T_n$. In particular, this holds for all large enough $n$, and so by \eqref{E:Sn-compact-form},
$$
\scrS^n(\bx, \by) - \sum_{i=1}^k \scrS^n(x_i, z_i) = \sum_{i=1}^k \|\pi^n_i\{\bx, \by\}|_{[Y_n, y_i]}\|_{\scrB^n} - \|\pi^n\{x_i, z_i\}|_{[Y_n, z_i]}\|_{\scrB^n}.
$$
Since the paths $\pi^n\{x_i, z_i\}, \pi^n\{\bx, \by\}$ converge to $\pi\{x_i, z_i\}, \pi\{\bx, \by\}$ and $\scrB^n$ converges uniformly to $\scrB$, the right-hand side above converges to the right-hand side of \eqref{E:Bbreakdown}. The left-hand side above converges to the left-hand side of \eqref{E:difference}, yielding \eqref{E:difference}.
\end{proof}

Proposition \ref{P:BSjoint} uniquely determines $\scrS$ on $\hat \fX$ by continuity. This uniquely determines the distribution of $\scrS$ by translation invariance.

\begin{definition}
	\label{D:extended-sheet}
	Let $\scrC(\fX, \R)$ be the space of continuous functions from $\fX$ to $\R$ with the topology of uniform convergence. A random function $\scrS \in \scrC(\fX, \R)$ is an \textbf{extended Airy sheet} if
	\begin{itemize}
		\item $\scrS$ can be coupled with a parabolic Airy line ensemble $\scrB$ so that 
		$$
		\scrS(\bx, \by) = \scrB[\bx \to \by]
		$$
		for all $(\bx, \by) \in \hat \Q$.
		\item For a vector $\bx \in \R^k$ for some $k$, let $T_c\bx$ denote the shifted vector $(x_1 + c, \dots, x_k + c)$. We can think of $T_c$ as an operator acting on all of $\bigcup_{i=1}^\infty \R^k$. In particular, $T_c$ acts on all of $\fX$ and $T_c(\fX) = \fX$. With this definition, for all $c \in \R$ we have
		$$
		\scrS \eqd \scrS \circ T_c.
		$$
	\end{itemize}
\end{definition}

The above definition clearly yields a unique distribution on $\scrC(\fX, \R)$. Moreover, we have the following theorem. This theorem encompasses Theorem \ref{T:extended-sheet}. 

\begin{theorem}
	\label{T:extended-Airy-sheet}
	The prelimits $\scrS^n$ converge in distribution to an extended Airy sheet $\scrS$.
\end{theorem}

\begin{proof}
Any subsequential limit $\scrS$ of $\scrS^n$ satisfies the first property of Definition \ref{D:extended-sheet} by Proposition \ref{P:BSjoint}. Moreover, for all $c$, translation invariance of Brownian increments guarantees that $\scrS^n \eqd \scrS^n \circ T_c$, and so $\scrS$ also satisfies the second property of Definition \ref{D:extended-sheet}.
\end{proof}

\subsection{Properties of the extended Airy sheet $\scrS$}

In this subsection we record a few basic properties of the extended Airy sheet $\scrS$, and use these properties to better understand the structure of optimizers in $\scrB$. The culmination of this section will be a proof of (the remaining parts of) Theorem \ref{T:extended-sheet-char}. We start with basic symmetries. For this lemma recall that for $\bx = (x_1, \dots, x_k) \in \R^k_\le$ with a slight abuse of notation we write $-\bx = (-x_k, \dots, -x_1)$.

\begin{lemma}
	\label{L:basic-sym}
The extended Airy sheet $\scrS$ satisfies
$
\scrS(\bx, \by) \eqd \scrS(-\by, -\bx),
$
jointly in all $\bx, \by$.
Moreover, the parabolically shifted sheet
$$
\scrR(\bx, \by) := \scrS(\bx, \by) + \sum_{i=1}^{k}(x_i - y_i)^2
$$
is stationary in the sense that for any $c_1, c_2 \in \R$, we have
$$
\scrR(T_{c_1} \bx, T_{c_2} \by) \eqd \scrR(\bx, \by)
$$
jointly in all $(\bx, \by) \in \fX$. Here the shifts $T_{c_i}$ are as in Definition \ref{D:extended-sheet}.
\end{lemma}

\begin{proof}
The first distributional equality follows from the distributional equality 
$
B(\cdot) \eqd B(1) - B(1-\cdot)
$
for Brownian motion. By the second part of Definition \ref{D:extended-sheet}, it is enough to prove the second equality when $c_1 = 0$. Let $\scrR^n(\bx, \by) = \scrS^n(\bx, \by) + \sum_{i=1}^{k}(x_i - y_i)^2$ and $\al_n = 1 + 2 c_2 n^{-1/3}$. By Brownian scaling,
$$
\scrR^n(\bx, \by) \eqd \al_n^{-1/2} \scrR^n( \al_n \bx, T_{c_2} \al_n \by) + e_n(\bx, \by)
$$
jointly in $\bx, \by$, where the error term $e_n(\bx, \by)$ term is deterministic and converges to $0$ uniformly on compact sets. Therefore since $\scrR^n(\bx, \by) \cvgd \scrR$ in the uniform-on-compact topology, $\scrR$ is continuous, and $\al_n \to 1$, we also have $\scrR^n(\bx, T_{c_2} \by) \cvgd \scrR$.
\end{proof}

Using Theorem \ref{T:extended-Airy-sheet} to pass Lemma \ref{l:change-spatial-prelim} to the limit, we get the following result for the parabolically shifted sheet $\scrR$.
\begin{lemma} \label{l:change-spatial}
	Take any $k\in\N$, $\bu = (\bx, \by), \bu' = (\bx',\by') \in \R^k_\le$ with $\|\bu - \bu'\|_2 < 1$, and $a>0$.
	Then
	$$
	\p(|\scrR(\bx',\by') - \scrR(\bx,\by)| >
	a\sqrt{\|\bu - \bu'\|_2}) < c e^{-da^{3/2}} ,
	$$
	for some constants $c , d > 0$ depending only on $k$.
\end{lemma}

We will use this continuity bound to show that $\scrS(\bx, \by) = \scrB[\bx \to \by]$ for all $(\bx, \by) \in \hat \fX$. First, we record an analogue of Lemma \ref{L:quadrangle} from $\scrB$.

\begin{lemma} 
	\label{L:b-quadrangle}
Take any $\bx, \by, \bx', \by'\in \R^k_\le$ such that $x_1, x_1'\ge 0$, and define $\bx^{\ell}, \by^{\ell}, \bx^r, \by^r\in \R^k_\le$ by $x^{\ell}_i=x_i\wedge x_i'$, $y^{\ell}_i=y_i\wedge y_i'$, and $x^r_i=x_i\vee x_i'$, $y^r_i=y_i\vee y_i'$, for each $1\le i \le k$.
Then
\[
\scrB[\bx^{\ell}\to\by^{\ell}] + \scrB[\bx^r\to\by^r] \ge \scrB[\bx\to\by] + \scrB[\bx'\to\by'].
\]
\end{lemma}

\begin{proof}
First, the inequality is trivial if either $\scrB[\bx \to \by]$ or $\scrB[\bx' \to \by']$ is $-\infty$, so we may assume both are finite. Let $\pi_n, \pi_n'$ be sequences of disjoint $k$-tuples from $\bx$ to $\by$ and $\bx'$ to $\by'$ whose weights converge to $ \scrB[\bx\to\by],  \scrB[\bx'\to\by']$ as in \eqref{E:disjoint-length}.
As in the proof of Lemma \ref{L:quadrangle}, we define disjoint $k$-tuples $\tau_n^\ell, \tau_n^r$ from $\bx^{\ell}$ to $\by^{\ell}$, $\bx^r$ to $\by^r$,
by (for each $1\le i \le k$) setting $\tau_{n,i}^\ell = \pi_{n,i}\wedge \pi_{n,i}'$, $\tau_{n,i}^r = \pi_{n,i}\vee \pi_{n,i}'$ on $(-\infty, y_i^\ell]$,
and setting $\tau^r_{n,i}$ to be either $\pi_{n,i}$ or $\pi_{n,i}'$ on $(y_i^\ell, y_i^r]$, depending on whether $y^r_i$ equals $y_i$ or $y_i'$.
Then as in the proof of Lemma \ref{L:rightmost-multi-path}, one can check that $\tau_n^\ell, \tau_n^r$ are disjoint $k$-tuples in $\scrB$ from $\bx^{\ell}$ to $\by^{\ell}$ and from $\bx^r$ to $\by^r$, respectively.
To prove the lemma, we just need to show that
\begin{equation}  \label{eq:quadrangle-eq}
\|\tau_n^\ell\|_\scrB+ \|\tau_n^r\|_\scrB = \|\pi_n\|_\scrB+ \|\pi'_n\|_\scrB.    
\end{equation}
Indeed, for any parabolic path $\pi$ from some $x\ge 0$ to $z\in\R$, and $y\le z$, we denote
\[
P(\pi,y) = \|\pi|_{[y,z]}\|_\scrB - \scrB[(y,\pi(y))\to (z,1)] + \scrS(x, z).
\]
Then from the definition of the path length \eqref{E:pi-A}, we just need to verify that
\[
\lim_{y\to-\infty}\sum_{i=1}^k P(\tau_{n,i}^\ell,y) + P(\tau_{n,i}^r,y) - P(\pi_{n,i},y) - P(\pi_{n,i}',y) = 0.
\]
This follows from Lemma \ref{l:sheet-diff-alter} and the fact that for any $y\le y^\ell_1$, we have
\[
\sum_{i=1}^k \|\tau_{n,i}^\ell|_{[y,y^\ell_i]}\|_\scrB + \|\tau_{n,i}^r|_{[y,y^r_i]}\|_\scrB - \|\pi_{n,i}|_{[y,y_i]}\|_\scrB - \|\pi_{n,i}'|_{[y,y_i']}\|_\scrB = 0. \qedhere
\]
\end{proof}

We next study the extended Airy sheet and optimizers in $\scrB$ for all endpoints in $\hat \fX$ (rather than just $\hat \Q$).
\begin{prop}
\label{P:B-continuity}
The function $(\bx, \by) \mapsto \scrB[\bx \to \by]$ is continuous on $\hat \fX$. In particular, 
$$
\scrS(\bx, \by) =\scrB[\bx \to \by]
$$
for all $(\bx, \by) \in \hat \fX$. Moreover, almost surely, for any $(\bx, \by) \in \hat \fX$ there is an optimizer in $\scrB$ from $\bx$ to $\by$.
\end{prop}

Proposition \ref{P:B-continuity} is the final piece of Theorem \ref{T:extended-sheet-char}. The proof is lengthy and a bit nuanced, so we give a sketch here.

We know that $\scrS(\bx, \by) =\scrB[\bx \to \by]$ for $(\bx, \by) \in \hat \Q$ and $\scrS$ is continuous. Moreover, it is not difficult to show that $\scrB$ is upper semicontinuous using Lemma \ref{L:A-paths}(iii). Therefore the main goal is to show that $\scrB[\bx \to \by] \le \scrS(\bx, \by)$ for general $\bx, \by$. We first prove this for rational $\bx$ with $0 < x_1 < \dots < x_n$. In this case, we can upper bound $\scrB[\bx \to \by]$ by $\scrB[\bx \to \by'] + \ep$ for a nearby rational $\by'$ by making small modifications to a candidate optimizer from $\bx$ to $\by$. 

For general $\bx$, by the quadrangle inequality (Lemma \ref{L:b-quadrangle}) and an approximation of $\bx$ by a rational sequence, we have that $\scrS(\bx, \by) - \scrB[\bx \to \by]$ is non-decreasing in $\by$.
Then it is enough to find a sequence $\bz_n \to (-\infty, \dots, -\infty)$ with 
\begin{equation}
\label{E:lss}
\limsup_{n \to \infty} \scrS(\bx, \bz_n) - \scrB[\bx\to \bz_n] \ge 0.
\end{equation}
The basic idea here is to choose the $\bz_n$ so that with probability tending to $1$ with $n$, we have
$$
\scrS(\bx, \bz_n) = \sum_{i=1}^k \scrS(x_i, z_{n, i}) \ge \scrB[\bx \to \bz_n].
$$
Note that there is a subtlety here because we need \eqref{E:lss} to hold for all $\bx$ simultaneously. We work around this with an approximation argument, making use of the strong modulus of continuity on $\scrS$ from Lemma \ref{l:change-spatial}. 

\begin{proof}[Proof of Proposition \ref{P:B-continuity}]
	Throughout the proof, we fix $k \in \N$ as the size of the set of points $\bx, \by$ we work with. All points $\bx$ have $x_1 \ge 0$.
By Proposition \ref{P:BSjoint}, $\scrB[\bx \to \by]$ almost surely coincides with the continuous function $\scrS(\bx, \by)$ at all points in $\hat \Q$, and there are unique optimizers $\pi\{\bx, \by\}$ in $\scrB$ for all these points by Proposition \ref{P:unique-exist}. Now consider an arbitrary point $(\bx, \by) \in \hat \fX$. We can approximate $(\bx, \by)$ by a sequence of points $(\bx_n, \by_n) \in \hat \Q$ such that $x_{n,i} > x_i$ and $y_{n,i} > y_i$ for all $i$. Now, the collection of optimizers $\{\pi\{\bx, \by\} : (\bx, \by) \in \hat \Q\}$ is monotone in $\bx$ and $\by$; this is inherited from the prelimiting monotonicity, which follows from Lemma \ref{L:mono-tree-multi-path}.
Therefore as in the proof of Lemma \ref{L:A-geod-exist}(ii), monotonicity of optimizers guarantees that the $k$-tuples $\pi\{\bx_n, \by_n\}$ have a limit, which is itself a disjoint $k$-tuple $\pi_R \{\bx, \by\}$ from $\bx$ to $\by$. 
Lemma \ref{L:A-paths}(iii) implies that
$$
\|\pi_R\{\bx, \by\}\|_\scrB \ge \scrS(\bx, \by) = \lim_{n \to \infty} \scrB[\bx_n \to \by_n].
$$
Therefore $\scrB[\bx \to \by] \ge \scrS(\bx, \by)$ for all $(\bx, \by) \in \hat \fX$. If we can show the opposite inequality, then the path $\pi_R\{\bx, \by\}$ is an optimizer, and $\scrS(\bx, \by) = \scrB[\bx \to \by]$. Since $\scrS$ is continuous, this will complete the proof of the proposition. 

For this, we first prove that for a fixed rational $\bx$ with $0 < x_1 < \dots < x_k$, we have $\scrB[\bx \to \by] = \scrS(\bx, \by)$ for all $\by \in \R^k_\le$. Suppose that $\scrB[\bx \to \by] > \scrS(\bx, \by)$ for some $\by$. Then by continuity of $\scrS$, there is an $\ep > 0$ and a disjoint $k$-tuple $\pi$ in $\scrB$ from $\bx$ to $\by$ with 
\begin{equation}
\label{E:piscrB}
\|\pi\|_\scrB > \scrS(\bx, \by') + \ep
\end{equation}
for all $\by'$ with $|\by - \by'| < \ep$. Now, Lemma \ref{l:diff-weight-finite} and the continuity of $\scrB$ and $\scrS$ ensures that there exists a rational $k$-tuple $\by'$ with $\by' < \by$ (coordinatewise) and $|\by - \by'| < \ep$ such that the path $\pi'$ from $\bx$ to $\by'$ defined by $\pi'_i = \pi_i|_{(-\infty, y_i']}$ satisfies 
$$
\|\pi'\|_\scrB > \|\pi\|_\scrB - \ep.
$$
This is greater than $\scrS(\bx, \by')$ by \eqref{E:piscrB}.
On the other hand, $\scrB[\bx \to \by'] \ge \|\pi'\|_\scrB$ and $\scrS(\bx, \by') = \scrB[\bx \to \by']$, giving a contradiction.

Now consider general $\bx$ and let $\by' \le \by \in \R^k_\le$. Consider a sequence of rational $\bx_n$ with $0 < x_{n,1} < \dots < x_{n, k}$ such that $x_{n, i} \cvgdown x_i$. By Lemma \ref{L:b-quadrangle}, we have
\begin{equation*}
\scrB[\bx \to \by] - \scrB[\bx \to \by'] \le \scrB[\bx_n \to \by] - \scrB[\bx_n \to \by'].
\end{equation*}
Since the $\bx_n$ have distinct positive rational entries, the right-hand side above is equal to the same difference with $\scrS(\cdot, \cdot)$ in place of $\scrB[\cdot \to \cdot]$. Therefore by continuity of $\scrS$, we have
 \begin{equation}
 \label{E:scrBB}
 \scrB[\bx \to \by] - \scrB[\bx \to \by'] \le \scrS(\bx, \by) - \scrS(\bx, \by')
 \end{equation}
 for all $\bx$ and $\by' \le \by$. To complete the proof that $\scrS(\bx, \by) \ge \scrB[\bx \to \by]$ it just suffices to show that we can find a sequence $\bz^-_n$ such that $z_{n, i}^- \to -\infty$ as $n \to \infty$ for all $i \in \II{1, k}$, and
 \begin{equation}
 \label{E:Sxzn}
\limsup_{n \to \infty} \scrS(\bx, \bz^-_n) - \scrB[\bx \to \bz^-_n] \ge 0
 \end{equation}
 for every $\bx$.
Indeed, for any $\bx, \by$, \eqref{E:scrBB} gives that
$$
\scrS(\bx, \by) - \scrB[\bx \to \by]\ge \scrS(\bx, \bz_n^-) - \scrB[\bx \to \bz_n^-]
$$ 
for all large enough $n$, so \eqref{E:Sxzn} gives that  $\scrB[\bx \to \by] \le \scrS(\bx, \by)$, as desired.

Let $\bx_n \in \R^k_\le $ be any sequence of points with distinct positive rational entries, such that any $\bx \in \R^k_\le$ satisfies  $|\bx_n - \bx| < n^{-1/k}$ for infinitely many $n$. The fact that we can find such a sequence is a consequence of the fact that the Lebesgue measure of the ball $B(\bx_n, n^{-1/k})$ is $O(1/n)$ as $n\to\infty$, and hence the sum over $n$ is infinite.
By equation \eqref{E:R-to-infty} in Lemma \ref{L:A-geod-exist}, we can find a sequence of deterministic points $\bz_n^-$ such that 
\begin{equation}
\label{E:S-DJ}
\p\lf(\scrS(\bx_n, \bz^-_n) = \scrB[\bx_n \to \bz^-_n] = \sum_{i=1}^k \scrB[x_{n, i} \to z^-_{n, i}] \rg) \to 1
\end{equation}
as $n \to \infty$, and for every $i \in \II{1, k}$ as $n \to \infty$ we have $z_{n, i}^- \to -\infty$. Moreover, the two-point estimate Lemma \ref{l:change-spatial} and Lemma \ref{L:levy-est} gives that
\begin{equation}
\label{E:BYY-prelim}
\sup_{|\by - \bx_n| \le n^{-1/k}} |\scrR(\bx_n, \bz^-_n) - \scrR(\by, \bz^-_n)| \le C_n n^{-1/{2k}}
\end{equation}
for a sequence of constants $C_n$ satisfying $\p (C_n > a) \le c e^{-d a^{3/2}}$ for constants $c, d$ that do not depend on $n$. This strong tail control on $C_n$ ensures that the right-hand side of \eqref{E:BYY-prelim} converges to $0$ almost surely as $n \to \infty$, and hence so does the left-hand side. Similarly,
\begin{equation}
\label{E:BYY}
\lim_{n \to \infty} \sup_{|\by - \bx_n| \le n^{-1/k}} \sum_{i=1}^k |\scrR(x_{n,i}, z^-_{n, i}) - \scrR(y_i, z^-_{n, i})| = 0 \qquad \as.
\end{equation}
Combining \eqref{E:S-DJ}, the convergence of \eqref{E:BYY-prelim}, and \eqref{E:BYY} with the fact that any point $\bx \in \R^k_\le$ satisfies $|\bx_n - \bx| < n^{-1/k}$ infinitely often implies that for all $\bx \in \R^k_\le$,
$$
\limsup_{n \to \infty} \scrR(\bx, \bz^-_n) - \sum_{i=1}^k \scrR(x_i, z^-_{n, i}) = 0,
$$
and so after removing the parabolic correction, 
$$
\limsup_{n \to \infty} \scrS(\bx, \bz^-_n) - \sum_{i=1}^k \scrS(x_i, z^-_{n, i}) \ge 0.
$$
Now, $\sum_{i=1}^k \scrS(x_i, z^-_{n, i}) \ge \scrB[\bx \to \bz^-_n]$ by the definition of parabolic path weight, yielding \eqref{E:Sxzn}.
\end{proof}

The relationship between $\scrS$ and $\scrB$ is particularly tractable when the start point $\bx = 0^k$. This proposition immediately gives the relationship \eqref{E:line-sheet-relation}.
\begin{prop}
	\label{P:high-paths-B}
Almost surely the following holds. For any $k \in \N$ and $\by \in \R^k_\le$ we have
\begin{equation}
\label{E:0ky}
\scrS(0^k, \by) = \sum_{i=1}^k \scrB_i(y_1) + \scrB[(y_1^k, k) \to (\by, 1)].
\end{equation}
Moreover, there is a disjoint optimizer $\pi$ in $\scrB$ from $0^k$ to $\by$ that only uses the top $k$ lines.
\end{prop}

\begin{proof}
Equation \eqref{E:0ky} is true in the prelimit by Lemma \ref{L:high-paths}, and hence holds in the limit as well. The `Moreover' claim follows by an explicit construction. Let $\pi = (\pi_1, \dots, \pi_k)$ be given by $\pi_i|_{(-\infty, y_1)} = i$ and $\pi_i|_{[y_1, y_i]} = \tau_i$, where $\tau$ is a disjoint optimizer from $(y_1^k, k)$ to $(\by, 1)$. We claim that $\|\pi\|_\scrB$ is equal to the right-hand side of \eqref{E:0ky}. The result will then follow from Proposition \ref{P:B-continuity}.

By \eqref{E:pi-A} and the fact that $\scrS(0, y_i) = \scrB_1(y_i)$ for all $i$, it is enough to show that
\begin{equation}
\label{E:Byik}
\scrB[(z, i) \to (y_i, 1)] - [\scrB_1(y_i) - \scrB_i(z)] \to 0
\end{equation}
as $z \to -\infty$. For any interval $[n, n+1] \sset [z, y_i]$, the left-hand side above is always bounded between $I_n := \sup_{x \in [n, n +1]} \scrB_i(x) - \scrB_1(x)$ and $0$. 
This is because the left-hand side of \eqref{E:Byik} can be written as 
\[
\sup_{z\le z_{i-1}\le \cdots \le z_1 \le y_i} \sum_{j=1}^{i-1} \scrB_{j+1}(z_j) - \scrB_j(z_j).
\]
The support of $I_n$ contains $0$ by the Brownian Gibbs property (Theorem \ref{T:melon-Airy-facts}), and $I_n$ is a stationary, ergodic process by the main result of \cite{corwin2014ergodicity}. Hence $\liminf_{n \to -\infty} I_n = 0$, yielding \eqref{E:Byik}.
\end{proof}

\subsection{Metric composition law}
To construct the full scaling limit of multi-point Brownian LPP from the extended Airy sheet, a key property is the following metric composition law. Recall from the introduction that if $\scrS$ is an extended Airy sheet, then $s\scrS(s^{-2}\bx, s^{-2}\by)$ is an \textbf{extended Airy sheet of scale $s$}.
	\begin{prop}
		\label{P:metric-composition-ex-sheet}
		For $s_1,s_2>0$, take independent extended Airy sheets $\scrS_1, \scrS_2$ of scale $s_1, s_2$, respectively. Then almost surely, for any $\bx,\by \in \R^k_\le$ the maximum 
		\begin{equation}
		\label{E:SMC}
		\scrS(\bx, \by) = \max_{\bz \in \R^k_\le} \scrS_1(\bx,\bz)+\scrS_2(\bz,\by)
		\end{equation}
		exists. Moreover, $\scrS$ is an extended Airy sheet of scale $(s_1^3+s_2^3)^{1/3}$.
	\end{prop}

	\begin{proof}
		Without loss of generality, we can assume that $s_1^3 + s_2^3 = 1$. 
		We set up multi-point Brownian LPP converging to an extended Airy sheet $\scrS$ as in Theorem \ref{T:extended-Airy-sheet}. Let $B^n$ be a collection of $n$ independent two-sided standard Brownian motions, and let $\scrS^n$ be the prelimiting extended Airy sheet defined using last passage percolation across $B^n$ (as in Theorem \ref{T:extended-sheet}) so that $\scrS^n \cvgd \scrS$ as $n\to\infty$.
		We now let
	\begin{multline*}
	\scrS^n_1(\bx, \by) = n^{1/6} \\ \times \lf(B[(2n^{-1/3}\bx, n) \to (s_1^3 + 2n^{-1/3}\by, n - \lfloor s_1^3n \rfloor)] - 2ks_1^3\sqrt{n} - n^{1/6} \sum_{i=1}^k 2(y_i - x_i) \rg),
	\end{multline*}
	and
	\begin{multline*}
	\scrS^n_2(\bx, \by) = n^{1/6} \\ \times \lf(B[(s_1^3 + 2n^{-1/3}\bx, n - \lfloor s_1^3n \rfloor - 1) \to (1 + 2n^{-1/3}\by, 1)] - 2ks_2^3\sqrt{n} - n^{1/6} \sum_{i=1}^k 2(y_i - x_i) \rg),
	\end{multline*}
	for any $(\bx, \by) \in \R^k_\le \X \R^k_\le$. Then $\scrS^n_1$ and $\scrS^n_2$ are independent.
	Using Theorem \ref{T:extended-Airy-sheet}, Brownian scaling, and continuity of $\scrS$, we have $\scrS^n_1 \cvgd \scrS_1$ and $\scrS^n_2 \cvgd \scrS_2$ as $n\to\infty$.
	On the other hand, by Lemma \ref{L:split-path} we have
	\[
		\scrS^n(\bx, \by) = \max_{\bz \in \R^k_\le} \scrS^n_1(\bx,\bz)+\scrS^n_2(\bz,\by)
	\]
	for any $\bx,\by \in \R^k_\le$. This passes to the limit as a long as the $\argmax$ for the right-hand side above is tight. 
This tightness for a fixed $\bx, \by$ follows from Lemma \ref{l:maxi-loc} below. Uniform tightness when $(\bx, \by)$ are allowed to range over a compact subset of $\fX$ then follows from monotonicity of optimizers (Lemma \ref{L:mono-tree-multi-path}). 
	\end{proof}
	
\begin{lemma}  \label{l:maxi-loc}
For any $k\in\N$, there exist constants $c,d > 0$ such that the following is true.
Let $n, p, q\in \N$, with $p+q=n$, and denote $t=p/n$.
Take independent standard Brownian motions $B^n=(B^n_1,\ldots,B^n_n)$, and $\bx, \by \in \R^k_\le$ such that $\|\bx\|_2, \|\by\|_2 < n^{1/6}$.
For any $\bz\in \R^k_\le$ define 
\begin{equation}
\label{E:Abzdef}
A(\bz)=B^n[(2n^{-1/3}\bx,n)\to (t+2n^{-1/3}\bz,q+1)]
+
B^n[(t+2n^{-1/3}\bz,q)\to (1+2n^{-1/3}\by,1)].
\end{equation}
We set $A(\bz) = -\infty$ if the right-hand side is not defined.
Then for any $a>0$, with probability at least $1-ce^{-dr^{3/2}}$ the following is true:
for any $\bz^*$ where $A$ achieves its maximum, we must have $\|\bz^*-t\by-(1-t)\bx\|_2<ca^2(t\wedge (1-t))^{1/3}$.
\end{lemma}
This lemma is an analogue of \cite[Lemma 9.3]{DOV}.
Its proof is also similar to the proof of  \cite[Lemma 9.3]{DOV}, involving some technical estimates on the Brownian $n$-melon $W^n$. We leave it to Appendix \ref{app:b-m-est}.

\section{The scaling limit of multipoint Brownian LPP}
\label{S:construct-extended-landscape}
\subsection{Tightness of the prelimiting extended landscape}
\label{S:tightness-landscape} 
Recall from the introduction that
\[
\fX_\uparrow = \{(\bx,s;\by,t) \in \bigcup_{k\in \N}(\R^k_\le \X \R)^2: s<t\}.
\]
Let $B=(B_i)_{i\in\Z}$ be an infinite sequence of independent two-sided standard Brownian motions. As in the introduction, let $(\bx, s)_n = (s+2\bx n^{-1/3}, - \floor{sn})$, and define the prelimiting extended landscape
\begin{equation}
\label{E:Lndef-body}
\scrL_n(\bx,s;\by,t) = n^{1/6} \left(
B[(\bx, s)_n\to (\by, t)_n] - 2k(t-s)\sqrt{n} - n^{1/6}\sum_{i=1}^k 2(y_i-x_i)
\right).
\end{equation}
This is a random function on $\fX_\uparrow$. 
In this section we prove that $\scrL_n$ is tight in an appropriate function space.
Given Lemma \ref{l:change-spatial-prelim}, it remains to prove a two-point tail bound on the deviation of $\scrL_n$ in the time direction. This is the analogue of \cite[Lemma 11.2]{DOV}.
Let $\scrK_n$ be the stationary version of $\scrL_n$, defined as
$$
\scrK_n(\bx,s;\by,t) = \scrL_n(\bx,s;\by,t) + \sum_{i=1}^k \frac{(x_i-y_i)^2}{t-s}.
$$
\begin{lemma} \label{l:change-time-prelim}
Take any $k\in\N$, $\bx, \by\in \R^k_\le$, and $t\in n^{-1}\Z$, such that $\|\bx\|_2, \|\by\|_2 < n^{1/100}$, $1/2\le t<1-n^{-1/100}$. Also take $0<a<n^{1/150}$.
Letting $\by'=t\bx+(1-t)\by$, we have
$$
\p(|\scrK_n(\bx,0;\by',t) - \scrK_n(\bx,0;\by,1)| > a(1-t)^{1/3}|\log(1-t)|) < c e^{-da^{9/8}} ,
$$
for some constants $c, d$ depending only on $k$.
\end{lemma}
We sketch the idea of the proof of this lemma here.
The complete proof is a technical computation and uses similar ideas to the proof of \cite[Lemma 11.2]{DOV}, so we leave it to Appendix \ref{app:b-m-est} (with the proof of Lemma \ref{l:maxi-loc}).
For the upper tail of $\scrK_n(\bx,0;\by',t) - \scrK_n(\bx,0;\by,1)$, via the triangle inequality it suffices to give a lower bound on $\scrK_n(\by',t+n^{-1};\by,1)$. This follows from Lemma \ref{L:f-naive} and tail bounds on points the Brownian melon, see Lemma \ref{l:bd-brownain-passtime}.
For the lower tail, by the metric composition law we need to upper bound 
$$
\sup_{z \in \R^k_\le} (\scrL_n(\bx,0;\bz,t) - \scrL_n(\bx,0;\by',t)) + \scrL_n(\bz,t;\by,1).
$$
The term $\scrL_n(\bz,t;\by,1)$ can be bounded with a curvature estimate on the Brownian melon. When $\|\bz - \by'\|_2$ is large, such a curvature estimate also works to bound $\scrL_n(\bx,0;\bz,t)$, and the aformentioned Lemma \ref{l:bd-brownain-passtime} can be used to bound $\scrL_n(\bx,0;\by',t))$ below.
When $\|\bz-\by'\|_2$ is small we apply the more refined spatial continuity estimate on the prelimiting extended Airy sheets from Lemma \ref{l:change-spatial-prelim} to bound the difference $\scrL_n(\bx,0;\bz,t) - \scrL_n(\bx,0;\by',t)$.
Putting together these bounds gives the desired result.

We now move to tightness. Let $\mathfrak{F}$ be the space of functions from $\fX_\uparrow$ to $\R$ that are either continuous, or of the form \eqref{E:Lndef-body} for some $n$ and some bi-infinite sequence of continuous functions $f$ in place of $B$. This is a Polish space, and so all classical theorems about distributional convergence apply. All of the $\scrL_n$ are random functions on this space.
\begin{prop} \label{p:ln-tight}
The functions $\scrL_n$ are tight in $\mathfrak{F}$, and all subsequential limits are almost surely continuous.
\end{prop}
\begin{proof}
Fix a compact set $K \sset (\R^k_\le\X \R)^2$ for some $k\in\N$. It suffices to show tightness of $\scrL_n|_K$. 
First, we replace $\scrL_n$ by a continuous version $\scrJ_n$ on $K$.
For each $(\bx, \by) \in \fX$ and $s$ with $s \in n^{-1} \Z$, define the function $\scrJ_n(\bx, s; \by, \cdot)$ by setting $\scrJ_n(\bx, s; \by, t) = \scrL_n (\bx, s; \by, t)$ whenever $t \in n^{-1} \Z$
and by linear interpolation at times in between. Then for each $(\bx, \by) \in \fX$ and $t\in \R$, we can define  $\scrJ_n(\bx, s; \by, t)$ by linear interpolation between values when $s \in n^{-1} \Z$. This procedure gives a well-defined continuous function on $K$ for large enough $n$. By Theorem \ref{T:airy-line-ensemble}, $\scrJ_n(0^k, 0; 0^k, 1)$ is tight in $n$.
Moreover, by Lemma \ref{l:change-spatial-prelim}, Lemma \ref{l:change-time-prelim}, and translation and scale invariance properties of $\scrL_n$ we get that for all $\bu, \bu' \in K$  and large enough $n$, 
$$
\p(|\scrJ_n(\bu) - \scrJ_n(\bu')| > a \|\bu - \bu'\|_2^{1/3-\epsilon}) \le ce^{-d a^{9/8}}
$$
for any $a,\epsilon> 0$. Here $c, d > 0$ are $K$-dependent constants. Using the Kolmogorov-Chentsov criterion, see \cite[Corollary 14.9]{kallenberg2006foundations}, we get that the sequence $\scrL_n$ is tight.
\end{proof}

\subsection{The explicit construction of $\scrL^*$}

In this subsection, we construct the scaling limit $\scrL^*$ of multipoint Brownian LPP axiomatically and prove Theorems \ref{T:unique-L*} and \ref{T:BLPP-convergence}. We call this object an extended$^*$ directed landscape, or extended$^*$ landscape for brevity. Later we will show that this object coincides with the extended directed landscape as defined in Definition \ref{D:ext-land}.

\begin{definition}
	\label{D:directed-landscape}
	An \textbf{extended$^*$ directed landscape} is a random continuous function $\mathcal{L}^*$ taking values in the space $\scrC(\fX_\uparrow, \R) \sset \mathfrak{F}$ of continuous functions from $\fX_\uparrow$ to $\R$ with the uniform-on-compact topology. It satisfies the following properties.
	\begin{enumerate}[label=\Roman*.]
		\item (Indepedendent extended Airy sheet marginals) 
		For any disjoint time intervals $\{(s_i, t_i) : i \in \{1, \dots k\}\}$, the random functions
		$$
		(\bx, \by) \mapsto \scrL^*(\bx, s_i ; \by, t_i), \quad i \in \II{1, k}
		$$
		are independent extended Airy sheets of scale $(t_i-s_i)^{1/3}$.
		\item (Metric composition law) For any $r<s<t$, almost surely we have that
		$$
		\scrL^*(\bx,r;\by,t)=\max_{\bz \in \mathbb R_\le^k} \scrL^*(\bx,r;\bz,s)+\scrL^*(\bz,s;\by,t),
		$$
		 for any $\bx, \by \in \R_\le^k$.
	\end{enumerate}
\end{definition}

Note that $\scrL^*|_{\Rd}$ is the usual directed landscape, since extended Airy sheets are simply Airy sheets when restricted to $\R^2$.

While $\scrL^*$ can be constructed directly similarly to how the directed landscape was constructed in \cite[Section 10]{DOV}, we will instead show its existence by proving that it is the scaling limit of $\scrL_n$.
The next result encompasses Theorems \ref{T:unique-L*} and \ref{T:BLPP-convergence}.

\begin{theorem}
	\label{T:Lnconverges}
	The extended$^*$ landscape $\scrL^*$  exists and is unique in law. Moreover, $\scrL_n \cvgd \scrL^*$ as random functions in $\mathfrak{F}$.
\end{theorem}

\begin{proof}
The uniqueness of $\scrL^*$ follows since conditions I and II specify all finite dimensional distributions. Indeed, let $\bu_1, \dots, \bu_k \in \fX_\uparrow$ be any collection of points with time indices $S = \{s_i < t_i  : i \in \II{1, k} \}$. Let $r_1 < \dots< r_\ell$ denote the order statistics of the set $S$, for some $2\le \ell \le 2k$. Then the marginals $\scrL(\cdot, r_i; \cdot, r_{i+1}), i \in \II{1, \ell-1}$ are independent Airy sheets of scale $(r_{i+1} - r_i)^{1/3}$ by I. All the random variables $\scrL(\bu_1), \dots, \scrL(\bu_k)$ are measurable functions of $\scrL(\cdot, r_i; \cdot, r_{i+1}), i \in \II{1, \ell-1}$ by repeated applications of II.

Next, we know $\scrL_n$ is tight in $\scrC(\fX_\uparrow, \R)$ by Proposition \ref{p:ln-tight}. Let $\scrM:\fX \to \Rd$ be any subsequential limit of $\scrL_n$. The function $\scrM$ has independent increments by the independence of the Brownian motions that give rise to $\scrL_n$. These increments must be rescaled extended Airy sheets by Theorem \ref{T:extended-Airy-sheet}, and satisfy metric composition since $\scrL_n$ does, and maximizer locations are tight (Lemma \ref{l:maxi-loc}). Therefore $\scrM$ is an extended$^*$ landscape.
\end{proof}

In the remainder of this section, we gather continuity estimates for $\scrL^*$. 
Let $\scrK$ be the stationary extended landscape, defined as
$$
\scrK(\bx,s;\by,t) = \scrL^*(\bx,s;\by,t) + \sum_{i=1}^k \frac{(x_i-y_i)^2}{t-s}.
$$
By passing Lemma \ref{l:change-time-prelim} to the limit, we have the following two-point bound on $\scrK$ (and hence on $\scrL$) in the time direction. Note that a two-point bound in the spatial direction follows from Lemma \ref{l:change-spatial} and rescaling.

\begin{lemma} \label{l:change-time}
Take any $k\in\N$, $\bx, \by\in \R^k_\le$, and $0<t'<t$ with $2t'\ge t$.
Letting $\by'=(t'/t)\bx+(1-t'/t)\by$, we have
$$
\p(|\scrK(\bx,0;\by',t') - \scrK(\bx,0;\by,t)| > a(t-t')^{1/3}|\log(1-t'/t)|) < c e^{-da^{9/8}} ,
$$
for all $a > 0$. Here $c, d > 0$ are constants depending only on $k$.
\end{lemma}

By Lemmas \ref{l:change-spatial} and \ref{l:change-time}, and using Lemma \ref{L:levy-est}, we have that in any compact subset of $\fX_\uparrow$, the function $\scrL$ is $(1/2-\epsilon)$-H\"older in the spatial coordinate and  $(1/3-\epsilon)$-H\"older in the time coordinate, for any $\epsilon >0$. 

We also need uniform upper and lower bounds on $\scrK$ on $\fX_\uparrow$.
We first give a one-point bound. This is obtained from passing the bound Lemma \ref{l:bd-brownain-passtime} on Brownian last passage values to the limit.
\begin{lemma}  \label{l:bound-K}
For any $k\in\N$, $\bx, \by \in \R^k_\le$, and $t>0$, we have
$$
\p(|\scrK(\bx,0;\by,t)| > at^{1/3}) < ce^{-da^{3/2}}
$$
for all $a > 0$. Here $c, d > 0$ are constants depending only on $k$.
\end{lemma}

Next we use Lemma \ref{l:change-spatial} and Lemma \ref{l:change-time} to upgrade Lemma \ref{l:bound-K} to a uniform bound that will be sufficient for our purposes.

\begin{lemma}  \label{l:uniform-EL-bound}
For any $\eta>0$ and $k\in\N$, there is a random constant $R>1$, such that for any $\bx,\by \in\R^k_\le$ and $s<t$, we have
$$
|\scrK(\bx,s;\by,t)| < RG(\bx,\by,s,t)^{\eta} (t-s)^{1/3}
$$
where
$$
G(\bx,\by,s,t) = \lf(1+\frac{\|\bx\|_1+\|\by\|_1}{(t-s)^{2/3}}\rg)\lf(1+\frac{|s|}{t-s}\rg)(1+|\log(t-s)|).
$$
Also $\p(R>a) < ce^{-da}$ for any $a>0$. Here $c, d > 0$ are constants depending on $k,\eta$.
\end{lemma}
\begin{proof}
Fix $\eta>0, k \in \N$.
Throughout this proof we let $c, d$ be constants depending on $k, \eta$, whose values can vary from line to line.
For each $\ell\in\Z$, let $L_\ell \subset \R^k_\le \X \R$ consist of all $(\bx, s)$, where each coordinate of $\bx$ is in $2^{2\ell}\Z$ and $s \in 2^{3\ell}\Z$.
For any $(\bx, s), (\by, t)\in L_\ell$, denote
\[
F(\bx,\by,s,t,\ell)=(1+2^{-2\ell}(\|\bx\|_1+\|\by\|_1))(1+2^{-3\ell}(|s|+|t|))(1+|\ell|).
\]
Take any $(\bx, s), (\by, t), (\by', t)\in L_\ell$ with $s < t$, such that $\by'$ and $\by$ differ at exactly one coordinate, and by exactly $2^{2\ell}$.
By Lemma \ref{l:change-spatial} we have
\begin{equation}
\label{E:summable1}
\p(|\scrK(\bx,s;\by,t) - \scrK(\bx,s;\by',t)| >
a F(\bx,\by,s,t,\ell)^\eta 2^{\ell}) 
< c e^{-da^{3/2}F(\bx,\by,s,t,\ell)^{3\eta/2}} .
\end{equation}

We then consider $(\bx, s), (\by, t), (\by', t')\in L_\ell$,
such that $s < t'$, $t=t'+2^{3\ell}$, and so that $\bx, \by, \by'$ satisfy the bound $|y_i' - ((t-t')x_i+(t'-s)y_i)/(t-s)| \le 2^{2\ell}$ for $1\le i \le k$.
By Lemmas \ref{l:change-spatial} and \ref{l:change-time} we have
\begin{equation}
\label{E:summable2}
\p(|\scrK(\bx,s;\by,t) - \scrK(\bx,s;\by',t')| >
a F(\bx,\by,s,t,\ell)^\eta 2^\ell
|\log(2^{3\ell}(t-s)^{-1})|
) 
< c e^{-da^{9/8}F(\bx,\by,s,t,\ell)^{9\eta/8}} .
\end{equation}

We next consider any $(\bx, s), (\by, t)\in L_\ell$ with $t-s=2^{3\ell}$.
By Lemma \ref{l:bound-K} we have
\begin{equation}
\label{E:summable-3}
\p(|\scrK(\bx,s;\by,t)|>a F(\bx,\by,s,t,\ell)^\eta 2^{\ell}) < ce^{-da^{3/2}F(\bx,\by,s,t,\ell)^{3\eta/2}}.
\end{equation}
The right-hand sides of \eqref{E:summable1}, \eqref{E:summable2}, and \eqref{E:summable-3} are summable over all allowable $\bx, \by, s, t$ and $\ell$ with sums that decrease at least exponentially in $a$.
In other words, we conclude that there exists a random number $R$, such that $\p(R>a) < ce^{-da}$ and the following is true.
\begin{enumerate}
    \item For any $(\bx, s), (\by, t), (\bx', s), (\by', t)\in L_\ell$ with $s < t$, such that $\bx'$, $\bx$ differ at exactly one coordinate by $2^{2\ell}$ and $\by'$, $\by$ differ at exactly one coordinate by $2^{2\ell}$, we have
    $$
    |\scrK(\bx,s;\by,t) - \scrK(\bx,s;\by',t)| <
R F(\bx,\by,s,t,\ell)^\eta 2^\ell,
    $$
    $$
    |\scrK(\bx',s;\by,t) - \scrK(\bx,s;\by,t)| <
R F(\bx,\by,s,t,\ell)^\eta 2^\ell.
    $$
    The second bound follows by a symmetric analogue of \eqref{E:summable1} where we vary $\bx$ rather than $\by$.
    \item For any $(\bx, s), (\by, t), (\by', t')\in L_\ell$,
such that $s<t'$, $t=t'+2^{3\ell}$, and $\bx, \by, \by'$ satisfy the bound $|y_i' - ((t-t')x_i+(t'-s)y_i)/(t-s)| \le 2^{2\ell}$ for $1\le i \le k$, we have
$$
|\scrK(\bx,s;\by,t) - \scrK(\bx,s;\by',t')| <
R F(\bx,\by,s,t,\ell)^\eta 2^\ell|\log(2^{3\ell}(t-s)^{-1})|.
$$
    \item For any $(\bx, s), (\by, t)\in L_\ell$ with $t-s=2^{3\ell}$, we have
    $$
    |\scrK(\bx,s;\by,t)|<R F(\bx,\by,s,t,\ell)^\eta 2^{\ell}.
    $$
\end{enumerate}
Now consider any $\bx, \by \in \R^k_\le$ and $s<t$ and let $\ell_0 = \lfloor \log_8(t-s) \rfloor -1$. For each $\ell\le \ell_0$, let $s_\ell = 2^{3\ell}\lceil 2^{-3\ell}s \rceil$, $t_\ell = 2^{3\ell}\lfloor 2^{-3\ell}t \rfloor$,
and let $\bx^{(\ell)}, \by^{(\ell)}\in\R^k_\le$ be chosen such that $(\bx^{(\ell)}, s_\ell), (\by^{(\ell)}, t_\ell) \in L_\ell$ and
\begin{equation}  \label{eq:xyil-close}
|x^{(\ell)}_i - ((t-s_\ell)x_i + (s_\ell-s)y_i)/(t-s)|\le 2^{2\ell}, \qquad |y^{(\ell)}_i - ((t-t_\ell)x_i + (t_\ell-s)y_i)/(t-s)|\le 2^{2\ell}    
\end{equation}
for each $1\le i \le k$.
As $\ell\to-\infty$ we have $(\bx^{(\ell)},\by^{(\ell)},s_\ell,t_\ell) \to (\bx, \by, s, t)$.
Therefore by the continuity of $\scrK$ and the triangle inequality we have
\begin{equation}
\label{E:scrK-xxx}
\begin{split}
|\scrK(\bx,s;\by,t)| \le & \; |\scrK(\bx^{(\ell_0)},s_{\ell_0};\by^{(\ell_0)},t_{\ell_0})|\\
&+
\sum_{\ell=\ell_0}^{-\infty}
|\scrK(\bx^{(\ell)},s_{\ell};\by^{(\ell)},t_{\ell}) - \scrK(\bx^{(\ell)},s_{\ell};\by^{(\ell-1)},t_{\ell-1})| \\
&+
|\scrK(\bx^{(\ell)},s_{\ell};\by^{(\ell-1)},t_{\ell-1}) - \scrK(\bx^{(\ell-1)},s_{\ell-1};\by^{(\ell-1)},t_{\ell-1})|.
\end{split}
\end{equation}
Using the above bounds, we have
\begin{equation}
\label{E:scrK-xx}
|\scrK(\bx^{(\ell_0)},s_{\ell_0};\by^{(\ell_0)},t_{\ell_0})| < RF(\bx^{(\ell_0)},\by^{(\ell_0)},s_{\ell_0},t_{\ell_0},\ell_0)^\eta 2^{\ell_0} < cRG(\bx,\by,s,t)^\eta (t-s)^{1/3},
\end{equation}
where the last inequality is by the fact that $2^{-3\ell_0}(t-s)$ and $2^{-3\ell_0}(t_{\ell_0}-s_{\ell_0})$ are upper and lower bounded by constants, and \eqref{eq:xyil-close}.
For each $\ell\le \ell_0$, we can find a sequence
$t_\ell = t_{\ell,1} \ge \cdots \ge t_{\ell,m} = t_{\ell-1}$, and 
$\by^{(\ell)}=\by^{(\ell,1)},\cdots,\by^{(\ell,m)}=\by^{(\ell-1)}\in \R_\le^k$ for some $m\le c$, such that for each $1\le j < m$, one of the following two events happens:
\begin{enumerate}
    \item $t_{\ell,j}=t_{\ell,j+1}+2^{3\ell}$, and $|y_i^{(\ell,j+1)} - ((t_{\ell,j}-t_{\ell,j+1})x_i^{(\ell)}+(t_{\ell,j+1}-s_\ell)y_i^{(\ell,j)})/(t_{\ell,j}-s_\ell)| \le 2^{2\ell}$ for each $1\le i \le k$.
    \item $t_{\ell,j}=t_{\ell,j+1}$ and $\by^{(\ell,j)}$ differ from $\by^{(\ell,j+1)}$ at exactly one coordinate by $2^{2\ell}$.
\end{enumerate}
Thus we have
\[
\begin{split}
|\scrK(\bx^{(\ell)},s_{\ell};\by^{(\ell)},t_{\ell})& - \scrK(\bx^{(\ell)},s_{\ell};\by^{(\ell-1)},t_{\ell-1})| \\
< & \; \sum_{j=1}^{m-1}RF(\bx^{(\ell,j)},\by^{(\ell,j)},s_{\ell},t_{\ell,j},\ell)^\eta 2^{\ell}
|\log(2^{3\ell}(t-s)^{-1})|
\\
< & \; cRG(\bx,\by,s,t)^\eta \lf(\frac{1+|\ell|}{1+|\ell_0|}\rg)^\eta 2^{\ell}
(1+\ell_0-\ell)
.
\end{split}
\]
For $|\scrK(\bx^{(\ell)},s_{\ell};\by^{(\ell-1)},t_{\ell-1}) - \scrK(\bx^{(\ell-1)},s_{\ell-1};\by^{(\ell-1)},t_{\ell-1})|$, arguing similarly we get that the same bound holds. Combining these bounds, summed over all $\ell\le \ell_0$, with the bound \eqref{E:scrK-xx} and the triangle inequality \eqref{E:scrK-xxx}, gives the result.
\end{proof}

As a consequence of Lemma \ref{l:uniform-EL-bound}, we can estimate the location of the maximizer in the metric composition law for the extended landscape.
\begin{lemma} \label{L:transfluc}
For any small enough $\eta>0$ and $k\in\N$, take the random variable $R>1$ and the function $G$ from Lemma \ref{l:uniform-EL-bound}.
For any $\bx,\by,\bz \in \R^k_\le$ and $r<s<t$, if
$\scrL^*(\bx,r;\bz,s)+\scrL^*(\bz,s;\by,t)=\scrL^*(\bx,r;\by,t)$, then
$$
\|\bz-\tbz\|_2
< cRG(\bx,\by,r,t)^{\eta}(t-s)^{1/3}(s-r)^{1/3},
$$
where $\tbz = ((t-s)\bx+(s-r)\by)/(t-r)$ and $c$ is a constant depending only on $\eta, k$.
\end{lemma}

\begin{proof}
In this proof we let $c$ denote a large constant depending on $k,\eta$, whose value may change from line to line.
By Lemma \ref{l:uniform-EL-bound} we have
\begin{multline}
\label{E:big-one}
0 = \scrL^*(\bx,r;\bz,s) + \scrL^*(\bz,s;\by,t) - \scrL^*(\bx,r;\by,t) < \frac{\|\bx-\by\|_2^2}{t-r}
- \frac{\|\bz-\bx\|_2^2}{s-r} - \frac{\|\bz-\by\|_2^2}{t-s}
\\
+ RG(\bx,\bz,r,s)^{\eta} (s-r)^{1/3} + RG(\bz,\by,s,t)^{\eta} (t-s)^{1/3} +
RG(\bx,\by,r,t)^{\eta} (t-r)^{1/3}
\\
=
- \frac{(t-r)\|\bz-\tbz\|_2^2}{(t-s)(s-r)}
+ R(G(\bx,\bz,r,s)^{\eta} (s-r)^{1/3} + G(\bz,\by,s,t)^{\eta} (t-s)^{1/3} +
G(\bx,\by,r,t)^{\eta} (t-r)^{1/3}).
\end{multline}
Now, $\frac{1+|\log(s-r)|}{1+|\log(t-r)|} \le \frac{t-r}{s-r}$, and $\|\bx\|_1+\|\by\|_1\ge \frac{s-r}{t-r}(\|\bx\|_1 + \|\tbz\|_1)$,
so from the definition of $G$ we have
\[
\begin{split}
\frac{G(\bx,\bz,r,s)}{G(\bx,\by,r,t)}
&= \;\frac{1+\frac{\|\bx\|_1+\|\bz\|_1}{(s-r)^{2/3}}}{1+\frac{\|\bx\|_1+\|\by\|_1}{(t-r)^{2/3}}} \times \frac{1+\frac{|r|}{s-r}}{1+\frac{|r|}{t-r}} \times \frac{1+|\log(s-r)|}{1+|\log(t-r)|}
\\
&\le \;
\frac{1+\frac{\|\bx\|_1+\|\bz\|_1}{(t-r)^{2/3}}}{1+\frac{\|\bx\|_1+\|\tbz\|_1}{(t-r)^{2/3}}}\lf(\frac{t-r}{s-r}\rg)^{5/3} \times \lf(\frac{t-r}{s-r}\rg) \times \lf(\frac{t-r}{s-r}\rg)
\\
&\le \;
\lf(1+\lf|\frac{\|\bz\|_1-\|\tbz\|_1}{(t-r)^{2/3}}\rg|\rg) \lf(\frac{t-r}{s-r}\rg)^{11/3}.
\end{split}
\]
Thus we have
\[
\frac{G(\bx,\bz,r,s)}{G(\bx,\by,r,t)}
< c (1+\|\bz-\tbz\|_2(t-r)^{-2/3})\lf(\frac{t-r}{s-r}\rg)^{10},
\]
and similarly
\[
\frac{G(\bz,\by,s,t)}{G(\bx,\by,r,t)}
< c (1+\|\bz-\tbz\|_2(t-r)^{-2/3})\lf(\frac{t-r}{t-s}\rg)^{10}.
\]
Without loss of generality we assume that $t-s\ge s-r$; thus $(t-r)/2 \le t-s \le t-r$.
We plug these two estimates into the inequality \eqref{E:big-one}.
By taking $\eta < 1/30$ we have
\begin{equation}
\label{E:crG}
\|\bz-\tbz\|_2^2 (s-r)^{-1}
< cRG(\bx,\by,r,t)^{\eta} (1+\|\bz-\tbz\|_2(t-r)^{-2/3})^\eta (t-r)^{1/3}.
\end{equation}
Now consider the function
\[
f:Z\mapsto Z^2(s-r)^{-1} - cRG(\bx,\by,r,t)^{\eta} (1+Z(t-r)^{-2/3})^\eta (t-r)^{1/3}.
\]
We have that $f(0)<0$, and on $\R_+$ this function first decreases then increases.
From \eqref{E:crG}, we have that $f(\|\bz-\tbz\|_2)<0$.
Also, $f(cRG(\bx,\by,r,t)^{\eta}(t-s)^{1/3}(s-r)^{1/3})>0$, since
\begin{align*}
& (cRG(\bx,\by,r,t)^{\eta}(t-s)^{1/3}(s-r)^{1/3})^2(s-r)^{-1}
\\
= & (cRG(\bx,\by,r,t)^{\eta})^2(t-s)^{2/3}(s-r)^{-1/3}
\\
\ge & cRG(\bx,\by,r,t)^{\eta}(t-r)^{1/3} + (cRG(\bx,\by,r,t)^{\eta})^2(t-s)^{1/3}(s-r)^{1/3}(t-r)^{-1/3}
\\ \ge &
cRG(\bx,\by,r,t)^{\eta} (1+cRG(\bx,\by,r,t)^{\eta}(t-s)^{1/3}(s-r)^{1/3}(t-r)^{-2/3})^\eta (t-r)^{1/3},
\end{align*}
where the first inequality is by $cRG(\bx,\by,r,t)^{\eta} \ge c$ and taking $c$ large enough, and the second inequality is by taking $\eta<1$.
These imply the conclusion.
\end{proof}

We can now show that metric composition holds everywhere in $\scrL^*$.
\begin{prop}
	\label{P:mc-everywhere}
	Almost surely, for every $r < s < t$ and $(\bx, \by) \in \fX$ we have
	$$
	\scrL^*(\bx, r; \by, t) = \max_{z \in \R} 	\scrL^*(\bx, r; \bz, s) + 	\scrL^*(\bz, s; \by, t).
	$$
	Also, almost surely we have the triangle inequality
	$$
	\scrL^*(\bx, r; \by, t) \ge \scrL^*(\bx, r; \bz, s) + 	\scrL^*(\bz, s; \by, t)
	$$
	for every $r <s < t, \bx, \bz, \by \in \R^k_\le$.
\end{prop}

\begin{proof}
By condition II in Definition \ref{D:directed-landscape}, we can ensure that almost surely, metric composition holds at all rational times $r < s < t$. The triangle inequality then holds at all rational times. This extends to all times by continuity of $\scrL^*$. 

Now, let $r < s < t$ and $(\bx, \by) \in \fX$. Consider rational sequences $r_n \to r, s_n \to s, t_n \to t$. By the metric composition law at rational times, for every $n$ we can find $\bz_n$ such that
\begin{equation}
\label{E:scrL*-bxby}
\scrL^*(\bx, r_n; \by, t_n) = \scrL^*(\bx, r_n; \bz_n, s_n) + \scrL^*(\bz_n, s_n; \by, t_n).
\end{equation} 
Lemma \ref{L:transfluc} ensures that all the points $\bz_n$ are contained in a common compact set, and hence we can find a subsequential limit $\bz$. Continuity of $\scrL^*$ ensures that Equation \eqref{E:scrL*-bxby} then holds with the $n$'s removed. Combining this with the triangle inequality yields the metric composition law at $r < s < t$ and $(\bx, \by)$.
\end{proof}

We finish this section by recording some symmetries of $\scrL^*$.
\begin{lemma}  \label{L:sym-L}
Take $q>0$, $r,c\in\R$, and let $T_c\bx$ denote the shifted vector $(x_1 + c, \dots, x_k + c)$.
We have the following equalities in distribution for $\scrL^*$ as functions in $\mathfrak{F}$.
\begin{enumerate}
    \item Stationarity: \quad
    $
\scrL^*(\bx,s,\by,t) \eqd \scrL^*(T_c\bx,s+r,T_c\by,t+r).
$
\item
Flip symmetry: \quad
$
\scrL^*(\bx,s,\by,t) \eqd \scrL^*(-\by,-t,-\bx,-s).
$
\item
Rescaling:\quad
$
\scrL^*(\bx,s,\by,t) \eqd q\scrL^*(q^{-2}\bx,q^{-3}s,q^{-2}\by,q^{-3}t).$
\item
Skew symmetry:
\[
\scrL^*(\bx,s,\by,t) + (t-s)^{-1}\|\bx-\by\|_2^2 \eqd \scrL^*(\bx,s,T_c\by,t) + (t-s)^{-1}\|\bx-T_c\by\|_2^2.
\]
\end{enumerate}
\end{lemma}
\begin{proof}
The first three symmetries of $\scrL^*$ can be deduced by the convergence from $\scrL_n$ (Theorem \ref{T:Lnconverges}), since finite versions hold for $\scrL_n$. The final symmetry follows from the corresponding symmetry in Lemma \ref{L:basic-sym} and the characterization of $\scrL^*$ in Definition \ref{D:directed-landscape}.
\end{proof}

\section{Paths in the extended landscape}
\label{S:paths-extended-landscape}

Having constructed $\scrL^*$, our next goal is to understand its optimizers. In this section we introduce both paths and optimizers in $\scrL^*$, and prove a selection of basic properties.
\subsection{Path weights}
We call a continuous function $\pi:[s,t]\to \R^k_\le$ for some interval $[s, t]$ a \textbf{multi-path of size $k$}.
For any multi-path $\pi:[s, t] \to \R^k_\le$, define its \textbf{length} in $\scrL^*$ by
$$
\|\pi\|_{\scrL^*} = \inf_{m\in\N} \inf_{s=t_0<t_1<\cdots<t_m=t}
\sum_{i=1}^m \scrL^*(\pi(t_{i-1}), t_{i-1}; \pi(t_i), t_i).
$$
This is the $\scrL^*$-analogue of the formula \eqref{E:length-L}.
For any $\pi:[r,t]\to\R^k_\le$, and a sequence $\pi^{(i)}:[r_i, t_i]\to \R^k_\le$ for $i\in\N$, we say that $\pi^{(i)}\to \pi$ in the \textbf{dyadic pointwise topology}, if $r_i\to r$, $t_i\to t$, and $\pi^{(i)}(s)\to \pi(s)$ for each $s\in \Q_2\cap [r,t]$, where $\Q_2$ is the set of dyadic rational numbers. This is a Polish topology, making it easy to work with probabilistically. Note that the length above can also be defined for discontinuous functions $\pi$. However, almost surely all discontinuous functions will have length $-\infty$ by Lemma \ref{l:uniform-EL-bound}. 

\begin{lemma}  \label{L:upper-semi-conti-path}
For a sequence of multi-paths $\{\pi^{(i)}\}_{i\in\N}$ and a multi-path $\pi$, such that $\pi^{(i)}\to \pi$ in the dyadic pointwise topology, we have $\limsup_{i\to\infty}\|\pi^{(i)}\|_{\scrL^*} \le \|\pi\|_{\scrL^*}$.
\end{lemma}

\begin{proof}
Suppose that $\pi$ is on $[r,t]$ and each $\pi^{(i)}$ is on $[r_i,t_i]$.
Take any $m\in\N$ and any sequence $r=s_0<s_1<\cdots<s_m=t$, such that $s_j\in\Q_2$ for each $0<j<m$.
For each $i \in \N$ and $0<j<m$ we denote $s_{i,j}=s_j$, and $s_{i,0}=r_i$, $s_{i,m}=t_i$. 
By the definition of $\|\pi^{(i)}\|_{\scrL^*}$, for all $i$ large enough so that $r_i < s_1, s_{m-1}  < t_i$, we have that
$$
\|\pi^{(i)}\|_{\scrL^*} \le \sum_{j=1}^m \scrL^*(\pi^{(i)}(s_{i,j-1}), s_{i,j-1}; \pi^{(i)}(s_{i,j}), s_{i,j}).
$$
As $i\to\infty$ the right-hand side converges to $\sum_{j=1}^m \scrL^*(\pi(s_{j-1}), s_{j-1}; \pi(s_j), s_j)$, by the convergence of $\pi^{(i)}$ to $\pi$ in the dyadic pointwise topology and the continuity of $\scrL^*$.
Therefore \[
\limsup_{i\to\infty}\|\pi^{(i)}\|_{\scrL^*} \le \sum_{j=1}^m \scrL^*(\pi(s_{j-1}), s_{j-1}; \pi(s_j), s_j).\]
By the continuity of $\scrL^*$ and of $\pi$, this inequality holds even when the points $s_i$ are not in $\Q_2$.
The conclusion then follows from the definition of $\|\pi\|_{\scrL^*}$.
\end{proof}

\subsection{Optimizers and transversal fluctuation}
\label{S:optimizers-transversal}
From the definition of $\|\cdot\|_{\scrL^*}$, and the triangle inequality for $\scrL^*$ (Proposition \ref{P:mc-everywhere}), for any $\pi:[s,t]\to \R^k_\le$ we have that 
\begin{equation}
\label{E:ext-L-ineq}
\|\pi\|_{\scrL^*}\le\scrL^*(\pi(s), s; \pi(t), t).
\end{equation}
We call a multi-path $\pi$ an \textbf{optimizer in $\scrL^*$} from $(\pi(s), s)$ to $(\pi(t), t)$, if equality holds in \eqref{E:ext-L-ineq}. If $\pi$ is an optimizer, then
$$
\scrL^*(\pi(s), s; \pi(t), t) = 
\sum_{i=1}^m \scrL^*(\pi(t_{i-1}), t_{i-1}; \pi(t_i), t_i)
$$
for any partition $s=t_0<t_1<\cdots<t_m=t$ of $[s, t]$. In the case where $k=1$, this defines a geodesic in the directed landscape, since $\scrL^*|_{\Rd} = \scrL$.
We next address the existence and uniqueness of optimizers. We start with a fixed pair of endpoints.
\begin{lemma}  \label{L:as-unique-geo-fixedend}
Given $(\bx,r;\by,t)\in \fX_\uparrow$, almost surely there is a unique optimizer in $\scrL^*$ from $(\bx,r)$ to $(\by,t)$.
\end{lemma}
We need the following result on the uniqueness of the maximum in the metric composition law.
\begin{lemma}  \label{L:two-sheet-sum-unique-max}
Given $\bx,\by\in\R^k_\le$ and $r<s<t$, almost surely the function $A_s(\bz) = \scrL^*(\bx,r;\bz,s)+\scrL^*(\bz,s;\by,t)$ has a unique maximum.
\end{lemma}
We leave the proof of this lemma to the end of this subsection, and continue our discussion of existence and  uniqueness of optimizers.
\begin{proof}[Proof of Lemma \ref{L:as-unique-geo-fixedend}]
By Lemma \ref{L:two-sheet-sum-unique-max}, almost surely for each rational $s\in(r,t)$, the function $A_s(\bz) = \scrL^*(\bx,r;\bz,s)+\scrL^*(\bz,s;\by,t)$ has a unique maximum. Therefore the value of any optimizer from $(\bx,r)$ to $(\by,t)$ is uniquely determined at all rational times, and hence the optimizer itself is uniquely determined by continuity.

For existence, we can construct an optimizer $\pi$ as follows. Let $\pi(r)=\bx, \pi(t)=\by$, and for any rational $s\in(r,t)$ let $\pi(s)\in\R^k_\le$ be the unique maximum of $A_s$.
Then for any triple $s_1<s_2<s_3\in ((r,t) \cap \Q) \cup \{r, t\}$ we claim that
\begin{equation}  \label{E:as-uni-geo}
\scrL^*(\pi(s_1),s_1;\pi(s_2),s_2)+\scrL^*(\pi(s_2),s_2;\pi(s_3),s_3)=\scrL^*(\pi(s_1),s_1;\pi(s_3),s_3).
\end{equation}
Indeed, by the metric composition law there exists $\bz^{(1)}, \bz^{(2)}, \bz^{(3)}\in\R^k_\le$, such that 
\begin{multline*}
\scrL^*(\bx,r;\by,t)=\scrL^*(\bx,r;\bz^{(1)},s_1)
+\scrL^*(\bz^{(1)},s_1;\bz^{(2)},s_2)+\scrL^*(\bz^{(2)},s_2;\bz^{(3)},s_3)\\ +\scrL^*(\bz^{(3)},s_3;\by,t).    
\end{multline*}
The triangle inequality for $\scrL^*$ (Proposition \ref{P:mc-everywhere}) and the uniqueness of maxima for the functions $A_s$ ensures that $\bz^{(i)}=\pi(s_i)$ for $i=1,2,3$ and enforces equation \eqref{E:as-uni-geo}.  
By Lemma \ref{L:transfluc}, $\pi$ is continuous at rational points and at $r, t$. Therefore we can extend $\pi$ to a continuous function on $[r,t]$. 

Finally we check that for any $r=t_0<t_1<\cdots<t_m=t$, we have
$$
\sum_{i=1}^m \scrL^*(\pi(t_{i-1}), t_{i-1}; \pi(t_i), t_i) = \scrL^*(\bx,r;\by,t).
$$
If for all $0<i<m$, $t_i$ is rational, this follows by \eqref{E:as-uni-geo}. This extends to general times $t_i$ by the continuity of $\scrL^*$. We conclude that $\pi$ is an optimizer.
\end{proof}
We can upgrade the existence of optimizers to hold simultaneously for all pairs of endpoints, although the same cannot be achieved for uniqueness.
\begin{lemma}
	\label{L:existence}
Almost surely, for any $(\bx,s;\by,t)\in \fX_\uparrow$, there is an optimizer in the extended landscape $\scrL^*$ from $(\bx,s)$ to $(\by,t)$.
\end{lemma}
\begin{proof}
By Lemma \ref{L:as-unique-geo-fixedend}, almost surely there is a unique optimizer between any pair of rational endpoints.
For any $\bu = (\bx,s;\by,t)\in \fX_\uparrow$, we can take a sequence of rational points $(\bx^{(i)},s_i;\by^{(i)},t_i)\in \fX_\uparrow$ converging to $\bu$, and let $\pi^{(i)}$ be the unique optimizer from $(\bx^{(i)},s_i)$ to $(\by^{(i)},t_i)$. All these optimizers are H\"older-$(2/3)^-$ with a common H\"older constant $c$ by Lemma \ref{L:transfluc}. Therefore the sequence $\pi^{(i)}$ has a subsequential limit $\pi$ in the dyadic pointwise topology which is itself H\"older-$(2/3)^-$ continuous. By Lemma \ref{L:upper-semi-conti-path} and the continuity of $\scrL^*$, we have that $\pi$ is an optimizer from $(\bx,s)$ to $(\by,t)$.
\end{proof}

We finish this subsection with the proof of Lemma \ref{L:two-sheet-sum-unique-max}.
We first reduce the problem to understanding the sum of two multi-point last passage values across two independent parabolic Airy line ensembles $\scrB, \scrB'$. This is an optimization problem involving parabolic paths across $\scrB, \scrB'$.
The remainder of the proof is essentially a resampling argument for parabolic paths, similar in spirit to the proof of Lemma \ref{L:brown-unique}.

The general idea here is to fix a rational interval $I$, and consider two restricted versions of the optimization problem: one which forces exactly one of the parabolic paths to jump in the interval $I$ either on or off of the line $\scrB_1$, and one which does not allow any jumps on or off of $\scrB_1$ in the interval $I$. Using the Brownian Gibbs property for $\scrB$ (Theorem \ref{T:melon-Airy-facts}), we conclude that almost surely these two restricted optimization problems have different maxima.
Now if the function $A_s$ has two different maxima $\bz, \bz'$, then we show that almost surely we can find a rational interval $I$ where the two optimizations problems discussed above have the same maximum, yielding a contradiction.

\begin{proof}[Proof of Lemma \ref{L:two-sheet-sum-unique-max}]
We now implement the arguments summarized above.

\textbf{Step 1: Reducing to multi-point last passage values across $\scrB$.}
By symmetries of $\scrL^*$ (Lemma \ref{L:sym-L}), we can assume that $r=-1, s=0, t>0$ and $x_1 > 0 > y_k$.
Since $\scrL^*$ has independent extended sheet marginals, using symmetries of extended sheets (Lemma \ref{L:basic-sym}) we have
$$
(\scrL^*(\bx, -1; \bz, 0), \scrL^*(\bz, 0; \by, t)) \eqd (\scrS(\bx, \bz), \scrS_t(-\by, -\bz)),
$$
where $\scrS, \scrS_t$ are independent extended sheets of scale $1$ and $t^{1/3}$.
Therefore by Proposition \ref{P:B-continuity}, we have
\begin{equation}
\label{E:RHSisB}
A_s(\cdot) \eqd \scrB[\bx\to\cdot] + \scrB'[-\by\to -\; \cdot],
\end{equation}
where $\scrB=\{\scrB_i\}_{i\in\N}$ is a parabolic Airy line ensemble, $\scrB'=\{\scrB_i'\}_{i\in\N}$ is independent of $\scrB$, and
$\{x\mapsto t^{-1/3}\scrB_i'(t^{2/3}x)\}_{i\in\N}$ is a parabolic Airy line ensemble. We will show that the right-hand side of \eqref{E:RHSisB} has a unique maximum $\bz$. While the argument is similar in spirit to the one used in Lemma \ref{L:brown-unique}, there are extra complexities coming from the definition of length for parabolic paths. 

\textbf{Step 2: Setup of paths intersecting/avoiding an interval.}
We next define certain collections of parabolic paths that do or not jump in the given interval. We will later show that almost surely, these collections have different maximum weights.

We start by setting up notation for jump times. For a disjoint $k$-tuple of parabolic paths $\pi=(\pi_1,\ldots,\pi_k)$ from $\bx$ to $\bz$, let 
\begin{equation} \label{eq:zjump}
z[\pi]_{i,m} = \sup\{w\le z_i: \pi_i(w)\ge m+1\},
\end{equation}
for any $i\in \llbracket 1, k\rrbracket$ and non-negative integer $m$.
In words, $z[\pi]_{i,m}$ is the jump time from line $m+1$ to line $m$ for $\pi_i$ (when $m\ge 1$), and $z[\pi]_{i,0}=z_i$.

For any interval $I\subset \R$, $1\le i \le k$, and $j\in\{0,1\}$,
let $\scrP_{I,i}^{(j)}$ be the collection of all $k$-tuples of essentially disjoint parabolic paths $\pi$ in $\scrB$ such that $z[\pi]_{i,j} \in I$, and $z[\pi]_{i',m}\not\in I$ for any $(i',m)\neq (i,j)$.
Also let $\scrP_{I}^c$ be the collection of all $k$-tuples of essentially disjoint parabolic paths $\pi$ in $\scrB$ such that $z[\pi]_{i',m}\not\in I$ for any $(i',m)$.

In other words, $\scrP_{I,i}^{(0)}$, $\scrP_{I,i}^{(1)}$ and $\scrP_{I,i}^c$ contain $k$-tuples of essentially disjoint parabolic paths $\pi$, satisfying the following conditions:
\begin{itemize}
    \item $\scrP_{I,i}^{(0)}$: the endpoint of $\pi_i$ is in $I$, and no other endpoint or jump point from line $2$ to line $1$ is in $I$.
    \item $\scrP_{I,i}^{(1)}$: the jump point of $\pi_i$ from line $2$ to line $1$ is in $I$, and no endpoint or other jump point from line $2$ to line $1$ is in $I$.    
    \item $\scrP_{I}^c$: no endpoint or jump point from line $2$ to line $1$ is in $I$.    
\end{itemize}
Now we fix a compact interval $I\subset \R$ and $1\le i \le k$.
Define $M_0, M_1, M_c$ using the same expression
$$
\sup \|\pi\|_\scrB + \|\pi'\|_{\scrB'},
$$
where the supremums are over different sets of pairs of disjoint $k$-tuples $\pi, \pi'$ from $\bx$ to $\bz$ and $-\by$ to $-\bz$ for some $\bz \in \R^k_\le$. 
For $M_0$, we require that $\pi\in \scrP_{I,i}^{(0)}$. For $M_1$, we require $\pi\in \scrP_{I,i}^{(1)}$. For $M_c$, we require $\pi\in \scrP_{I}^c$. We have no additional restriction on $\pi'$.
See Figure \ref{fig:Ms} for an illustration of the paths.

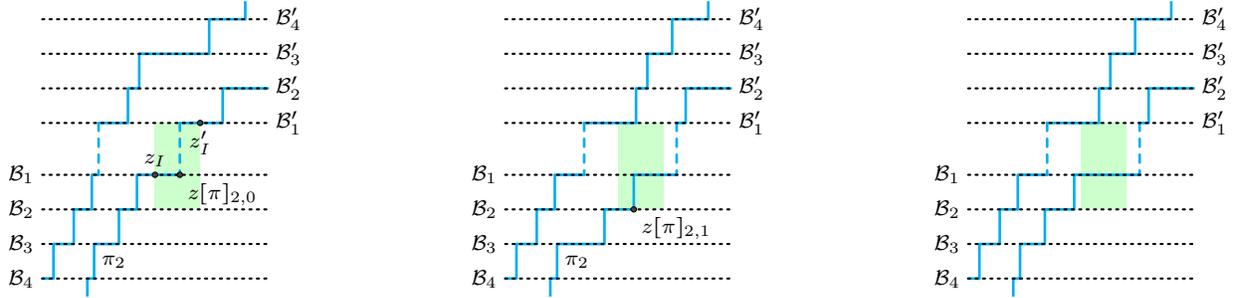
\begin{figure}[hbt!]
    \centering
\begin{tikzpicture}[line cap=round,line join=round,>=triangle 45,x=0.3cm,y=0.23cm]
\clip(-2,-1) rectangle (12,16);

\fill[line width=0.pt,color=green,fill=green,fill opacity=0.2]
(5,4)-- (7,4)-- (7,9)-- (5,9) -- cycle;

\draw [line width=1pt, color=cyan] (9,17) -- (9,15) -- (7.4,15) -- (7.4,13) -- (4.3, 13) -- (4.3, 11) -- (3.8,11) -- (3.8,9) -- (2.5,9);
\draw [line width=1pt, color=cyan] (2.5,6) -- (2.2,6) -- (2.2,4) -- (1.4,4) -- (1.4,2) -- (0.5,2) -- (0.5,0) -- (0,0);
\draw [line width=1pt, color=cyan] [dashed] (2.5,9) -- (2.5,6);

\draw [line width=1pt, color=cyan] (10, 11) -- (8,11) -- (8,9) -- (6.1,9);
\draw [line width=1pt, color=cyan] (6.1,6)  -- (4.2,6) -- (4.2,4) -- (3.4,4) -- (3.4,2) -- (2.3,2) -- (2.3,0) -- (2,0) -- (2,-2);
\draw [line width=1pt, color=cyan] [dashed] (6.1,9) -- (6.1,6);

\draw (0,25) -- (10,25);

\foreach \i in {1,...,4}
{
\draw [dotted] [thick] (0, 8-2*\i) -- (10, 8-2*\i);
\begin{scriptsize}
\draw (0,8-2*\i) node[anchor=east]{$\scrB_{\i}$};
\end{scriptsize}
}

\foreach \i in {1,...,4}
{
\draw [dotted] [thick] (0, 7+2*\i) -- (10, 7+2*\i);
\begin{scriptsize}
\draw (10,7+2*\i) node[anchor=west]{$\scrB_{\i}'$};
\end{scriptsize}
}

\draw [fill=uuuuuu] (5,6) circle (1.0pt);
\draw [fill=uuuuuu] (7,9) circle (1.0pt);
\draw [fill=uuuuuu] (6.1,6) circle (1.0pt);
\begin{scriptsize}
\draw (5,6) node[anchor=south]{$z_I$};
\draw (7,9) node[anchor=north]{$z_I'$};
\draw (6.1,6) node[anchor=north west]{$z[\pi]_{2,0}$};
\draw (2.3,1) node[anchor=west]{$\pi_2$};
\end{scriptsize}

\end{tikzpicture}
\hfill
\begin{tikzpicture}[line cap=round,line join=round,>=triangle 45,x=0.3cm,y=0.23cm]
\clip(-2,-1) rectangle (12,16);

\fill[line width=0.pt,color=green,fill=green,fill opacity=0.2]
(5,4)-- (7,4)-- (7,9)-- (5,9) -- cycle;

\draw [line width=1pt, color=cyan] (9,17) -- (9,15) -- (7.4,15) -- (7.4,13) -- (6.3, 13) -- (6.3, 11) -- (5.8,11) -- (5.8,9) -- (3.5,9);
\draw [line width=1pt, color=cyan] (3.5,6) -- (2.2,6) -- (2.2,4) -- (1.4,4) -- (1.4,2) -- (0.5,2) -- (0.5,0) -- (0,0);
\draw [line width=1pt, color=cyan] [dashed] (3.5,9) -- (3.5,6);

\draw [line width=1pt, color=cyan] (10, 11) -- (8,11) -- (8,9) -- (7.6,9);
\draw [line width=1pt, color=cyan] (7.6,6)  -- (5.7,6) -- (5.7,4) -- (4.4,4) -- (4.4,2) -- (2.3,2) -- (2.3,0) -- (2,0) -- (2,-2);
\draw [line width=1pt, color=cyan] [dashed] (7.6,9) -- (7.6,6);

\draw (0,25) -- (10,25);

\foreach \i in {1,...,4}
{
\draw [dotted] [thick] (0, 8-2*\i) -- (10, 8-2*\i);
\begin{scriptsize}
\draw (0,8-2*\i) node[anchor=east]{$\scrB_{\i}$};
\end{scriptsize}
}

\foreach \i in {1,...,4}
{
\draw [dotted] [thick] (0, 7+2*\i) -- (10, 7+2*\i);
\begin{scriptsize}
\draw (10,7+2*\i) node[anchor=west]{$\scrB_{\i}'$};
\end{scriptsize}
}
\draw [fill=uuuuuu] (5.7,4) circle (1.0pt);
\begin{scriptsize}
\draw (5.7,4) node[anchor=north west]{$z[\pi]_{2,1}$};
\draw (2.3,1) node[anchor=west]{$\pi_2$};
\end{scriptsize}
\end{tikzpicture}
\hfill
\begin{tikzpicture}[line cap=round,line join=round,>=triangle 45,x=0.3cm,y=0.23cm]
\clip(-2,-1) rectangle (12,16);

\fill[line width=0.pt,color=green,fill=green,fill opacity=0.2]
(5,4)-- (7,4)-- (7,9)-- (5,9) -- cycle;

\draw [line width=1pt, color=cyan] (9,17) -- (9,15) -- (7.4,15) -- (7.4,13) -- (6.3, 13) -- (6.3, 11) -- (5.8,11) -- (5.8,9) -- (3.5,9);
\draw [line width=1pt, color=cyan] (3.5,6) -- (2.2,6) -- (2.2,4) -- (1.4,4) -- (1.4,2) -- (0.5,2) -- (0.5,0) -- (0,0);
\draw [line width=1pt, color=cyan] [dashed] (3.5,9) -- (3.5,6);

\draw [line width=1pt, color=cyan] (10, 11) -- (8,11) -- (8,9) -- (7.6,9);
\draw [line width=1pt, color=cyan] (7.6,6)  -- (4.7,6) -- (4.7,4) -- (3.4,4) -- (3.4,2) -- (2.3,2) -- (2.3,0) -- (2,0) -- (2,-2);
\draw [line width=1pt, color=cyan] [dashed] (7.6,9) -- (7.6,6);

\draw (0,25) -- (10,25);

\foreach \i in {1,...,4}
{
\draw [dotted] [thick] (0, 8-2*\i) -- (10, 8-2*\i);
\begin{scriptsize}
\draw (0,8-2*\i) node[anchor=east]{$\scrB_{\i}$};
\end{scriptsize}
}

\foreach \i in {1,...,4}
{
\draw [dotted] [thick] (0, 7+2*\i) -- (10, 7+2*\i);
\begin{scriptsize}
\draw (10,7+2*\i) node[anchor=west]{$\scrB_{\i}'$};
\end{scriptsize}
}

\end{tikzpicture}
\caption{
An illustration of the paths used in defining $M_0$, $M_1$, and $M_c$ (from left to right, respectively). The parabolic Airy line ensemble $\scrB'$ is rotated by $180$ degrees for the picture.
The green regions indicate the interval $I$ and the lines where there are constraints on jumps.
For $M_0$, here we have $\pi\in \scrP_{I,2}^{(0)}$; for $M_1$, we have $\pi \in \scrP_{I,2}^{(1)}$; for $M^c$, we have $\pi \in \scrP_{I}^c$.
}
\label{fig:Ms}
\end{figure}

\textbf{Step 3: Almost surely $M_0 \ne M_c$ and $M_1\ne M_c$.}
The general idea for this step is to show that conditional on $M_c$, the random variable $M_0$ (or $M_1$) has a continuous distribution, using the Brownian Gibbs property for the interval $I$.

We let $\scrF$ be the $\sig$-algebra generated by null sets, $\scrB'$, all $\scrB_m$ for $m\ge 2$, and $\{\scrB_1(x):x\not\in I\}$.
Then $M_c$ is $\scrF$-measurable, since the function recording all lengths of paths $\pi\in \scrP_{I,i}^{c}$ is $\scrF$-measurable by Lemma \ref{l:measurable-B-path-weight}. 

We then deduce that, given $\scrF$, the random variable $M_0$ has continuous distribution.
For this, we investigate how $M_0$ depends on $\scrB_1$ in $I$ (see the left panel of Figure \ref{fig:Ms}).
By Lemma \ref{l:diff-weight-finite} we can write 
\[
M_0 = \sup_{\pi,\pi'} \|\pi\|_\scrB + \|\pi'\|_{\scrB'} + \sup_{x\in I}(\scrB_1(x)+\scrB'_1(x)) - \scrB_1(z_I) - \scrB'_1(-z_I').
\]
Here $z_I, z_I'$ are the left and right endpoints of $I$, i.e., $I=[z_I,z_I']$.
The first supremum above is taken over all $\pi\in \scrP_{I}^c$ and $\pi'\in \scrP_{-I}^c$, where $\pi$ is from $\bx$ to $\bz$ and $\pi'$ is from $-\by$ to $\bz'$, and such that $z_i=z_I, z_{k+1-i}'=-z_I'$, and $z_j=-z_{k+1-j}'$ for any $j\neq i$.
Therefore $M_0-\sup_{x\in I}(\scrB_1(x)+\scrB'_1(x))$ is $\scrF$-measurable by Lemma \ref{l:measurable-B-path-weight}. On the other hand,
the Brownian Gibbs property for $\scrB$ (Theorem \ref{T:melon-Airy-facts}) implies that conditioned on $\scrF$ the law of $\scrB_1$ on $I$ is absolutely continuous to a Brownian bridge (with diffusion parameter $2$). Therefore conditioned on $\scrF$, the random variable $\sup_{x\in I}(\scrB_1(x)+\scrB'_1(x))$ almost surely has a continuous distribution, and hence so does $M_0$. Since $M_c$ is $\scrF$-measurable, $M_0\ne M_c$ almost surely. The argument to show that $M_1\ne M_c$ almost surely is similar.
Moreover, these inequalities hold almost surely simultaneously for all compact rational intervals $I$ and all $i \in \II{1, k}$. Below we assume this probability one event.

\textbf{Step 4: Contradiction with two maximums.}
Now we consider the function $\bz\mapsto \scrB[\bx\to\bz] + \scrB'[-\by\to -\bz]$. Note that by the metric composition law (Proposition \ref{P:metric-composition-ex-sheet}) and symmetry of the extended Airy sheet (Lemma \ref{L:basic-sym}) this function attains its maximum.
Suppose that the maximum is attained at two points $\bz^{(1)}\neq\bz^{(2)}$.
We take any disjoint optimizers $\pi^{(1)}$ and $\pi^{(2)}$ in $\scrB$, from $\bx$ to $\bz^{(1)}$ and $\bz^{(2)}$;
and $\pi'^{(1)}$ and $\pi'^{(2)}$ in $\scrB'$, from $-\by$ to $-\bz^{(1)}$ and $-\bz^{(2)}$. Such optimizers exist by Proposition \ref{P:B-continuity}.
Consider the jump times $z[\pi^{(1)}]_{i,0}=z^{(1)}_i$, $z[\pi^{(1)}]_{i,1}$, and $z[\pi^{(2)}]_{i,0}=z^{(2)}_i$, $z[\pi^{(2)}]_{i,1}$.
By Lemma \ref{L:jump-points-distinct} below (and see Figure \ref{fig:diffz}), almost surely we have
\[
z[\pi^{(1)}]_{1,1}<z[\pi^{(1)}]_{1,0} < z[\pi^{(1)}]_{2,1} < \cdots < z[\pi^{(1)}]_{k,1}<z[\pi^{(1)}]_{k,0},
\]
and
\[
z[\pi^{(2)}]_{1,1}<z[\pi^{(2)}]_{1,0} < z[\pi^{(2)}]_{2,1} < \cdots < z[\pi^{(2)}]_{k,1}<z[\pi^{(2)}]_{k,0}.
\]
Since $\bz^{(1)}\neq\bz^{(2)}$, we can find a number in the first sequence which does not appear in the second sequence. There are two cases:
\begin{enumerate}
    \item We can find $i\in \II{1,k}$, such that $z[\pi^{(1)}]_{i,0}=z^{(1)}_i$ is not in the second sequence.
    Then we can find a rational interval $I$ that contains $z^{(1)}_i$ and does not contain any other number in the two sequences.
    Thus $\pi^{(1)}\in \scrP_{I,i}^{(0)}$ and $\pi^{(2)}\in \scrP_{I,i}^c$.
    Then we have that $M_0 = M_c$ for such $I$ and $i$.
    \item We can find $i\in \II{1,k}$, such that $z[\pi^{(1)}]_{i,1}$ is not in the second sequence.
    Then we can find a rational interval $I$ that contains $z[\pi^{(1)}]_{i,1}$ and does not contain any other number in the two sequences.
    Thus $\pi^{(1)}\in \scrP_{I,i}^{(1)}$ and $\pi^{(2)}\in \scrP_{I,i}^c$.
    Then we have that $M_1 = M_c$ for such $I$ and $i$.
\end{enumerate}
In either case, we get a contradiction with the assumption that $M_0\neq M_c$, $M_1\neq M_c$  for all compact rational intervals $I$ and all $i \in \II{1, k}$.
This means that $A_s$ cannot have two different maximums.
\end{proof}

It remains to prove the following lemma, which says that almost surely, consecutive jumps happen at different locations. 
\begin{lemma}  \label{L:jump-points-distinct}
Almost surely the following statement is true. Fix $t > 0$, and as in the proof of Lemma \ref{L:two-sheet-sum-unique-max}, let $\scrB$ be a parabolic Airy line ensemble, $\scrB'=\{\scrB_i'\}_{i\in\N}$ be independent of $\scrB$ such that
$\{x\mapsto t^{-1/3}\scrB_i'(t^{2/3}x)\}_{i\in\N}$ is a parabolic Airy line ensemble.

Take any $\bx,\by\in\R^k_\le$ such that $x_1>0>y_k$.
Let $\bz^*$ be any maximum of 
$$
\bz\mapsto \scrB[\bx\to\bz] + \scrB'[-\by\to -\bz],
$$
and let $\pi^*, \pi'^*$ be optimizers in $\scrB$ from $\bx$ to $\bz^*$, and in $\scrB'$ from $-\by$ to $-\bz^*$, respectively.
Recall the notation from \eqref{eq:zjump}, and
let $z^*_{i,m}=z[\pi^*]_{i,m}$ and $-z^*_{k+1-i,-m}=z[\pi'^*]_{i,m}$ for any $i\in\II{1,k}$ and non-negative integer $m$ (see Figure \ref{fig:diffz}).
Then for all $m\in\Z$, we have $z^*_{i,m}\neq z^*_{i,m+1}$ for $1\le i\le k$, and $z^*_{i-1,m}\neq z^*_{i,m+1}$ for $2\le i\le k$.
\end{lemma}

\begin{figure}[hbt!]
    \centering
\begin{tikzpicture}[line cap=round,line join=round,>=triangle 45,x=1cm,y=0.23cm]
\clip(-2,-2) rectangle (12,17);

\draw [line width=1pt, color=cyan] (9,17) -- (9,15) -- (7.4,15) -- (7.4,13) -- (5.3, 13) -- (5.3, 11) -- (4.8,11) -- (4.8,9) -- (3.5,9);
\draw [line width=1pt, color=cyan] (3.5,6) -- (2.2,6) -- (2.2,4) -- (1.9,4) -- (1.9,2) -- (1.5,2) -- (1.5,0) -- (0,0);
\draw [line width=1pt, color=cyan] [dashed] (3.5,9) -- (3.5,6);

\draw [line width=1pt, color=cyan] (10,15) -- (9.5,15) -- (9.5,13) -- (8.1,13) -- (8.1, 11) -- (7,11) -- (7,9) -- (6.1,9);
\draw [line width=1pt, color=cyan] (6.1,6)  -- (4.2,6) -- (4.2,4) -- (3.4,4) -- (3.4,2) -- (2.3,2) -- (2.3,0) -- (2,0) -- (2,-2);
\draw [line width=1pt, color=cyan] [dashed] (6.1,9) -- (6.1,6);

\draw [line width=1pt, color=cyan] (10, 11) -- (9,11) -- (9,9) -- (8.1,9);
\draw [line width=1pt, color=cyan] (8.1,6)  -- (7.2,6) -- (7.2,4) -- (5.8,4) -- (5.8,2) -- (5.3,2) -- (5.3,0) -- (4,0) -- (4,-2);
\draw [line width=1pt, color=cyan] [dashed] (8.1,9) -- (8.1,6);

\draw (0,25) -- (10,25);

\foreach \i in {1,...,4}
{
\draw [dotted] [thick] (0, 8-2*\i) -- (10, 8-2*\i);
\begin{scriptsize}
\draw (0,8-2*\i) node[anchor=east]{$\scrB_{\i}$};
\end{scriptsize}
}

\foreach \i in {1,...,4}
{
\draw [dotted] [thick] (0, 7+2*\i) -- (10, 7+2*\i);
\begin{scriptsize}
\draw (10,7+2*\i) node[anchor=west]{$\scrB_{\i}'$};
\end{scriptsize}
}
\draw [fill=uuuuuu] (8.1,6) circle (1.0pt);
\draw [fill=uuuuuu] (7.2,4) circle (1.0pt);
\draw [fill=uuuuuu] (5.8,2) circle (1.0pt);
\draw [fill=uuuuuu] (5.3,0) circle (1.0pt);

\draw [fill=uuuuuu] (6.1,6) circle (1.0pt);
\draw [fill=uuuuuu] (4.2,4) circle (1.0pt);
\draw [fill=uuuuuu] (3.4,2) circle (1.0pt);
\draw [fill=uuuuuu] (2.3,0) circle (1.0pt);

\draw [fill=uuuuuu] (3.5,6) circle (1.0pt);
\draw [fill=uuuuuu] (2.2,4) circle (1.0pt);
\draw [fill=uuuuuu] (1.9,2) circle (1.0pt);
\draw [fill=uuuuuu] (1.5,0) circle (1.0pt);

\draw [fill=uuuuuu] (4.8,11) circle (1.0pt);
\draw [fill=uuuuuu] (5.3,13) circle (1.0pt);
\draw [fill=uuuuuu] (7.4,15) circle (1.0pt);

\draw [fill=uuuuuu] (7,11) circle (1.0pt);
\draw [fill=uuuuuu] (8.1,13) circle (1.0pt);
\draw [fill=uuuuuu] (9.5,15) circle (1.0pt);

\draw [fill=uuuuuu] (9,11) circle (1.0pt);
\begin{scriptsize}
\draw (8.1,6) node[anchor=north]{$z^*_{3,0}$};
\draw (7.2,4) node[anchor=north]{$z^*_{3,1}$};
\draw (5.8,2) node[anchor=north]{$z^*_{3,2}$};
\draw (5.3,0) node[anchor=north]{$z^*_{3,3}$};

\draw (6.1,6) node[anchor=north]{$z^*_{2,0}$};
\draw (4.2,4) node[anchor=north]{$z^*_{2,1}$};
\draw (3.4,2) node[anchor=north]{$z^*_{2,2}$};
\draw (2.3,0) node[anchor=north]{$z^*_{2,3}$};

\draw (3.5,6) node[anchor=north]{$z^*_{1,0}$};
\draw (2.2,4) node[anchor=north]{$z^*_{1,1}$};
\draw (1.9,2) node[anchor=north]{$z^*_{1,2}$};
\draw (1.5,0) node[anchor=north]{$z^*_{1,3}$};

\draw (4.8,11) node[anchor=south]{$z^*_{1,-1}$};
\draw (5.3,13) node[anchor=south]{$z^*_{1,-2}$};
\draw (7.4,15) node[anchor=south]{$z^*_{1,-3}$};

\draw (7,11) node[anchor=south]{$z^*_{2,-1}$};
\draw (8.1,13) node[anchor=south]{$z^*_{2,-2}$};
\draw (9.5,15) node[anchor=south]{$z^*_{2,-3}$};

\draw (9,11) node[anchor=south]{$z^*_{3,-1}$};
\end{scriptsize}

\end{tikzpicture}
\caption{
An illustration of the jump times in the statement of Lemma \ref{L:jump-points-distinct}.
}
\label{fig:diffz}
\end{figure}
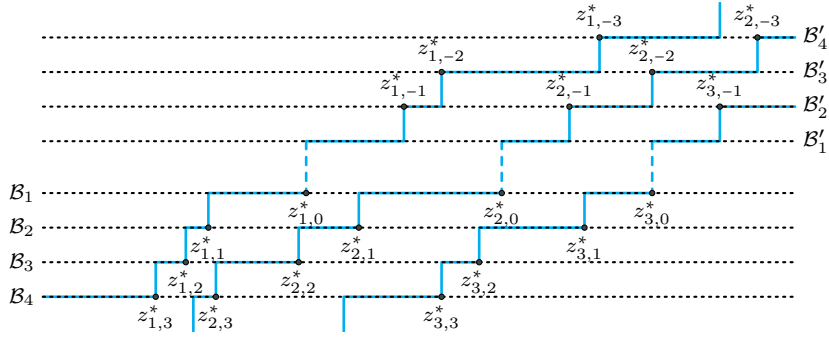

Since $\II{1, k} \times \Z$ is countable, it suffices to show that for fixed $(i, m)$, almost surely $z^*_{i,m} \neq z^*_{i,m+1}$ and $z^*_{i-1,m} \neq z^*_{i,m+1}$ (if $i\ge 2$).
The idea is again to use the Brownian Gibbs property, reducing this to the following statement: on a compact interval, two Brownian motions that are correlated (but not identical or trivially related) almost surely attain their maxima at different points.

\begin{proof}
We note that there are interlacing relations $z^*_{i,m+1}\le z^*_{i,m}$ and $z^*_{i-1,m}\le z^*_{i,m+1}$, from the definition and essential disjointness of the parabolic paths.

Now suppose that there is some $\hat{z}\in\R$, such that $\hat{z} = z^*_{i,m}= z^*_{i,m+1}$ or $\hat{z} = z^*_{i-1,m}= z^*_{i,m+1}$ for some $i,m$.
The plan is to use the Brownian Gibbs property to reduce this problem to studying the $\argmax$ of Brownian motion. 
The interlacing relations impose constraints on the domain of the $\argmax$, so we need to consider all jump times that equal $\hat{z}$.

By looking at all $z^*_{i,m}$ that equal $\hat{z}$, we can find some integers $m_-$ and $m_+$ with $m_+ > m_-$, such that for each $m \in \II{m_-, m_+}$, there is at least one jump time equal to $\hat{z}$ on line $m$, and there is no jump time equal $\hat{z}$ on the lines $m_--1$ and $m_++1$.
We then let $\Phi\subset \llbracket 1,k\rrbracket\X \Z$ be the set of indices for all jump times equal to $\hat{z}$ amongst the lines $\II{m_-, m_+}$.
In other words, we define 
\[
\Phi = \{(i,m): z^*_{i,m}=\hat{z},\;\; m_- \le m \le m_+\},
\]
such that $z^*_{i,m_--1} \neq \hat{z}$, $z^*_{i,m_++1} \neq \hat{z}$ for any $i\in\II{1,k}$.

We note that the number of $\Phi$ of this form is countable. It now suffices to prove the following claim.

\begin{claim}
For any fixed $\Phi \subset \II{1,k}\X\Z$ such that $\Phi \cap (\II{1,k}\X\{m\})$ is nonempty if and only if $m\in \II{m_-,m_+}$ for some $m_-<m_+$, 
almost surely we cannot find a number $\hat{z}$ so that $\Phi = \{(i,m): z^*_{i,m}=\hat{z},\;\; m_--1 \le m \le m_++1\}$.
\end{claim}

To prove this claim, the general strategy is as follows:
\begin{enumerate}
    \item translate this event to properties of the $\argmax$ of some linear combination of lines of $\scrB$;
    \item replace $\scrB$ by independent Brownian motions (with diffusion parameter $2$);
    \item analyze the probability of certain events involving Brownian motions.
\end{enumerate}
\begin{proof}[Claim Proof]
It suffices to show that, for any such $\Phi$ and any rational interval $I$, almost surely the following event does not happen: we can find $\hat{z}\in I'$, where $I'$ is the middle $1/3$ of $I$, such that
\begin{itemize}
    \item for any $(i,m)\in\Phi$, we have $z^*_{i,m}=\hat{z}$,
    \item for any $(i,m)\in\II{m_--1, m_++1} \X \Z \setminus \Phi$, we have $z^*_{i,m}\not\in I$.
\end{itemize}
By comparing path lengths using Lemma \ref{l:diff-weight-finite},
this implies that $(\scrB_{m^-}, \dots, \scrB_{m^+ + 1})\in A$, where $A$ is the following event defined on sequences of continuous functions $f_{m^-}, \dots, f_{m^+ + 1}:I \to \R$. The function
\[
g_f(\{z_{i,m}\}_{(i,m)\in\Phi}) = \sum_{(i,m)\in\Phi}f_{m+1}(z_{i,m})-f_{m}(z_{i,m})
\]
defined on the domain where $z_{i, m} \in I$,  $z_{i,m+1}\le z_{i,m}$ and $z_{i-1,m}\le z_{i,m+1}$ for all $i$ and $m$ attains its maximum when $\{z_{i,m}\}_{(i,m)\in\Phi}=\{\hat{z}\}_{(i,m)\in\Phi}$ for some $\hat{z}\in I'$. Here when defining $(\scrB_{m^-}, \dots, \scrB_{m^+ + 1})$ we let $\scrB_m=\scrB'_{1-m}$ for $m\le 0$.

We next apply the Brownian Gibbs property to $\scrB$ and $\scrB'$, which enables us to replace $(\scrB_{m^-}, \dots, \scrB_{m^+ + 1})$ by independent Brownian motions (with diffusion parameter $2$).
Namely, it now suffices to prove that almost surely $B = (B_{m^-}, \dots, B_{m^+ + 1})\in A$, where the $B_i$ are independent Brownian motions (with diffusion parameter $2$).

The last step is to prove $\p(B \in A)=0$, by studying the Brownian motions and using inequalities given by the maximum condition.
Consider the functions
\[
S_1=\sum_{(i,m)\in\Phi}B_{m+1}-B_{m}, \quad S_2=B_{m_0+1}-B_{m_0},
\]
where $(i_0,m_0)\in\Phi$ is chosen so that $(i_0,m_0-1), (i_0+1,m_0+1)\not\in\Phi$.
Under the event $A$, for any $z \in I$ we have
\[
g_B\big(\{\hat{z}\}_{(i,m)\in\Phi}\big) \ge g_B\big(\{z\}_{(i,m)\in\Phi}\big),
\]
so $S_1(\hat{z})\ge S_1(z)$. 
We can also just deviate the $(i_0, m_0)$ coordinate, but now only in one direction since we must preserve the interlacing conditions. Namely, for any $z\in I$, $z\ge \hat{z}$, under the event $A$ we have 
\[
g_B\big(\{\hat{z}\}_{(i,m)\in\Phi}\big) \ge g_B\big(\{\hat{z}+ \mathbf{1}((i,m)=(i_0,m_0)) (z-\hat{z})\}_{(i,m)\in\Phi}\big),
\]
so $S_2(\hat{z})\ge S_2(z)$.
It remains to show that, for any $\hat{z} \in I'$,
\begin{equation}  \label{eq:bmt0}
\p(S_2(\hat{z})\ge S_2(z), \forall z \ge \hat{z}, z \in I \mid \hat{z} = \argmax_I S_1) = 0.
\end{equation}
Note that $S_1$ and $S_2$ are Brownian motions, and we can write $S_2=\alpha S_1 + \beta S_3$ for some $\alpha, \beta \in \R$, $\beta >0$, where $S_3$ is a Brownian motion independent of $S_1$.
However, for any $\hat{z} \in I'$, we have
\begin{equation}  \label{eq:bmt1}
\p\big( \limsup_{z\searrow\hat{z}} \frac{S_3(z)-S_3(\hat{z})}{S_1(\hat{z})-S_1(z)} > \alpha\beta^{-1} \mid \hat{z} = \argmax_I S_1 \big) = 1.    
\end{equation}
This is because, conditional on $\hat{z} = \argmax_I S_1$ and as $z\searrow\hat{z}$, a rescaling of $S_1(\hat{z})-S_1(z)$ converges in distribution to a Bessel process of order $3$. By Blumenthal's zero–one law, the left-hand side of \eqref{eq:bmt1} is either zero or one. It cannot be zero since the distribution of the ratio of a Brownian motion over an independent Bessel process of order $3$ is unbounded.

From \eqref{eq:bmt1} we get \eqref{eq:bmt0}, which means that $\p(B \in A)=0$ and the conclusion follows.
\end{proof}

Given the above claim, since the number of possible $\Phi$ is countable, the conclusion follows.
\end{proof}

\subsection{Monotonicity of optimizers}
In this subsection we aim to establish monotonicity for optimizers in the extended landscape.
Some arguments are in parallel to those in Section \ref{S:basic-properties} for last passage across lines.
We first establish that leftmost and rightmost optimizers are well-defined. For this lemma we write $\pi \le \pi'$ for two multi-paths $\pi$ and $\pi'$ if the weak inequality holds pointwise and coordinatewise.
\begin{lemma}  \label{L:exist-left-right-most}
The following statement holds almost surely for $\scrL^*$.
For any $\bx, \by \in \R^k_\le$ and $s<t$,
there are optimizers $\pi^\ell, \pi^r$ from $(\bx,s)$ to $(\by,t)$, such that $\pi^\ell \le \pi \le \pi^r$ for any other optimizer $\pi$ from $(\bx,s)$ to $(\by,t)$. 
We call $\pi^\ell$ and $\pi^r$ the leftmost and rightmost optimizers from $(\bx,s)$ to $(\by,t)$, respectively.
\end{lemma}
\begin{proof}
Let $\pi^{(1)}$ and $\pi^{(2)}$ be two optimizers from $(\bx,s)$ to $(\by,t)$.
Let $\pi^{(3)}=\pi^{(1)}\wedge \pi^{(2)}$ and $\pi^{(4)}=\pi^{(1)}\vee \pi^{(2)}$, where $\wedge$ and $\vee$ are defined pointwise and coordinatewise. Since all optimizers are continuous by Lemma \ref{L:transfluc}, $\pi^{(3)}$ and $\pi^{(4)}$ are continuous functions from $[t, s]$ to $\R^k_\le$.
By the definition of $\|\cdot\|_{\scrL^*}$, the fact that $\scrL^*$ has extended Airy sheet marginals, and Lemma \ref{L:b-quadrangle},
we have that 
$$
\|\pi^{(3)}\|_{\scrL^*} + \|\pi^{(4)}\|_{\scrL^*} \ge \|\pi^{(1)}\|_{\scrL^*} + \|\pi^{(2)}\|_{\scrL^*}.
$$
Since $\pi^{(1)}, \pi^{(2)}$ are optimizers, this must an equality.
Thus $\pi^{(3)}, \pi^{(4)}$ are also optimizers.

Now consider any monotone sequence of optimizers $\pi^{(1)} \le \pi^{(2)}, \dots$. By Lemma \ref{L:transfluc}, this sequence has a bounded pointwise limit $\pi'$ on dyadic rationals, and $\pi'$ is continuous. This limit is also an optimizer by Lemma \ref{L:upper-semi-conti-path}. Thus by Zorn's lemma, there is an optimizer $\pi^\ell$, such that for any optimizer $\pi$, the condition $\pi\le \pi^\ell$ implies $\pi = \pi^\ell$. Thus for any optimizer $\pi$, since the multi-path $\pi \wedge \pi^\ell$ is an optimizer satisfying $\pi \wedge \pi^\ell\le \pi^\ell$, we must have $\pi \wedge \pi^\ell = \pi^\ell$, implying that $\pi^\ell\le \pi$. Therefore $\pi^\ell$ is the leftmost optimizer.
The existence of the rightmost optimizer follows similarly.
\end{proof}

\begin{lemma}  \label{L:monotonicity}
The following statements hold almost surely.
For any $\bx\le \bx', \by\le\by' \in \R^k_\le$ and $s<t$,
let $\pi$ be the leftmost (resp. rightmost) optimizer from $(\bx,s)$ to $(\by,t)$ and $\pi'$ be the leftmost (resp. rightmost) optimizer from $(\bx',s)$ to $(\by',t)$.
Then $\pi\leq \pi'$.
\end{lemma}
\begin{proof}
We prove for the case where $\pi, \pi'$ are the leftmost optimizers. The rightmost case follows similarly.
Define $\pi^\ell=\pi\wedge \pi'$ and $\pi^r=\pi\vee \pi'$.
Then $\pi^\ell$ and $\pi^r$ are both continuous multi-paths from $(\bx,s)$ to $(\by,t)$, and from $(\bx',s)$ to $(\by',t)$, respectively.

We claim that $\|\pi^\ell\|_{\scrL^*} + \|\pi^r\|_{\scrL^*} \ge \|\pi\|_{\scrL^*} + \|\pi'\|_{\scrL^*}$.
Indeed, this follows by the definition of $\|\cdot\|_{\scrL^*}$ and Lemma \ref{L:b-quadrangle}.
However, we also have $\|\pi\|_{\scrL^*} \ge \|\pi^\ell\|_{\scrL^*}$ and $\|\pi'\|_{\scrL^*} \ge \|\pi^r\|_{\scrL^*}$, by the definition of optimizers.
Therefore $\|\pi\|_{\scrL^*} = \|\pi^\ell\|_{\scrL^*}$ and $\|\pi'\|_{\scrL^*} = \|\pi^r\|_{\scrL^*}$, and hence $\pi^\ell, \pi^r$ are also optimizers.
As $\pi$ is the leftmost optimizer, we have $\pi\le \pi^\ell$. On the other hand, $\pi^\ell\le \pi, \pi'$ from the definition of $\pi^\ell$.
Thus $\pi=\pi^\ell\le \pi'$.
\end{proof}

\subsection{Sums of disjoint paths}

The goal of this section is to show the following proposition.
\begin{prop}   \label{P:sum-of-disjoint}
Almost surely the following statement is true.
Take any $s < t$ and any multi-paths $\pi:[s,t]\to\R^k_\le$ and $\pi':[s,t]\to\R^{k'}_\le$, such that for any $r\in(s,t)$, we have $\pi_k(r) \le \pi'_1(r)$.
Let $\pi'':[s,t]\to\R_\le^{k+k'}$ be such that $\pi''(r) = (\pi(r), \pi'(r))$ for all $r \in (s, t)$. Then
$$
\|\pi''\|_{\scrL^*} \le \|\pi\|_{\scrL^*} + \|\pi'\|_{\scrL^*}.
$$
Moreover, if $\pi_k(r) < \pi'_1(r)$ for all $r \in (s, t)$, then $\|\pi''\|_{\scrL^*} = \|\pi\|_{\scrL^*} + \|\pi'\|_{\scrL^*}$.
\end{prop}

The inequality in Proposition \ref{P:sum-of-disjoint} is immediate from the definition of $\| \cdot \|_{\scrL^*}$, and the fact that 
$$
\scrL^*(\bx, s, \by, t) + \scrL^*(\bx', s, \by', t) \ge \scrL^*((\bx, \bx'), s, (\by, \by'), t)
$$
for any $\bx, \by, \bx', \by', s, t$ for which both sides above make sense. This inequality is inherited from the prelimit $\scrL_n$, where it is clear. To prove the claimed equality in Proposition \ref{P:sum-of-disjoint}, we require a few lemmas.
\begin{lemma}  \label{l:bound-L-disjoint}
Let $k,\ell\in\N$ and $h,t>0$.
Then
$$
\p(\scrL^*(0^k,0;0^k,t) + \scrL^*( h^{\ell},0; h^{\ell},t) > \scrL^*((0^k, h^{\ell}),0;(0^k, h^{\ell}),t)) < ce^{-dh^3 t^{-2}},
$$
for $c,d$ depending only on $k, \ell$.
\end{lemma}
\begin{proof}
	
By rescaling we can assume that $t=1$. Since $\scrL^*$ has extended Airy sheet marginals, by Proposition \ref{P:B-continuity} the probability in question is the same as the probability of the event
$$
A = \{\scrB[0^k \to 0^k] + \scrB[h^{\ell} \to h^{\ell}] > \scrB[(0^k, h^{\ell}) \to (0^k, h^{\ell})]\}.
$$
By Proposition \ref{P:high-paths-B}, the event $A$ implies that every optimizer $\pi$ from $h^{\ell}$ to $h^{\ell}$ in $\scrB$ intersects the first $k$ lines of $\scrB$ in the interval $(-\infty, 0]$. That is, $\pi_1(0) \le k$, and so 
\begin{equation}
\label{E:BHLL}
\scrB[ h^{\ell}\to  h^{\ell}] \le \scrB[h^{\ell-1} \to h^{\ell-1}] + \scrB[h \to 0] + \scrB_1(h) - \scrB_k(0).
\end{equation}
Now by translation invariance of the extended Airy sheet $\scrS$,  $\scrB[ h^{\ell}\to  h^{\ell}] - \scrB[h^{\ell-1} \to h^{\ell-1}] \eqd \scrB_\ell(0)$ and $\scrB[h \to 0] \eqd \scrB_1(-h).$ Finally, since $\scrB(x)+x^2$ is stationary, \eqref{E:BHLL} is equivalent to an inequality of the form
\begin{equation}
\label{E:Xieqref}
X_1 + X_2 + 2h^2 \le X_3 + X_4, 
\end{equation}
where each of the random variables $X_i$ are equal in distribution to $\scrB_i(0)$ for some $i \le k \wedge \ell$. The points $\scrB_i(0)$ are points in the Airy point process, which are known to have well-controlled tails. For example, we can pass Theorem \ref{T:top-bd} to the limit to get that $\p(|X_i| > a) \le c e^{-da^{3/2}}$ for all $1 \le i \le 4$ and constants $c, d$ that depend only on $k, \ell$. Therefore by a union bound, the probability of \eqref{E:Xieqref} is bounded above by $ce^{-dh^3 t^{-2}}$, completing the proof.
\end{proof}

\begin{lemma}  \label{l:property-of-L-disjoint}
For any $s<t$, almost surely the following statement holds.
For any $\bx^{(1)} \le \bx^{(2)} \le \bx^{(3)} \le \bx^{(4)} \in \R^k_\le$, and $\by^{(1)} \le \by^{(2)} \le \by^{(3)} \le \by^{(4)} \in \R^k_\le$, if 
\[
\scrL^*((\bx^{(2)}, \bx^{(3)}), s; (\by^{(2)}, \by^{(3)}), t) = \scrL^*(\bx^{(2)}, s; \by^{(2)}, t) + \scrL^*(\bx^{(3)}, s; \by^{(3)}, t),    
\]
we then have
\[
\scrL^*((\bx^{(1)}, \bx^{(4)}), s; (\by^{(1)}, \by^{(4)}), t) = \scrL^*(\bx^{(1)}, s; \by^{(1)}, t) + \scrL^*(\bx^{(4)}, s; \by^{(4)}, t).
\]
\end{lemma}
\begin{proof}
Without loss of generality we assume that $s=0$, $t=1$. With these choices, the lemma is a fact about an extended Airy sheet of scale $1$.
We claim that
\begin{equation}  \label{eq:L-disjoint-pf}
\begin{split}
&\scrS(\bx^{(2)}, \by^{(2)}) + \scrS(\bx^{(3)}, \by^{(3)}) - \scrS((\bx^{(2)}, \bx^{(3)}), (\by^{(2)}, \by^{(3)}))
\\
\ge\; &
\scrS(\bx^{(1)}, \by^{(1)}) + \scrS(\bx^{(4)}, \by^{(4)}) - \scrS((\bx^{(1)}, \bx^{(4)}), (\by^{(1)}, \by^{(4)}))
\\
\ge\; & 0.
\end{split}    
\end{equation}
For this we study the prelimiting sheet $\scrS^n$.
By Lemma \ref{L:quadrangle-2}, for $n$ large enough we have
\[
\scrS^n(\bx^{(2)}, \by^{(2)})+
\scrS^n((\bx^{(2)}, \bx^{(3)}), (\by^{(1)}, \by^{(3)}))
\ge 
\scrS^n(\bx^{(2)}, \by^{(1)})+
\scrS^n((\bx^{(2)}, \bx^{(3)}), (\by^{(2)}, \by^{(3)})),
\]
and
\[
\scrS^n(\bx^{(3)}, \by^{(3)})+
\scrS^n((\bx^{(2)}, \bx^{(3)}), (\by^{(1)}, \by^{(4)}))
\ge 
\scrS^n(\bx^{(3)}, \by^{(4)})+
\scrS^n((\bx^{(2)}, \bx^{(3)}), (\by^{(1)}, \by^{(3)})).
\]
Adding up these two inequalities, and passing to the limit via Theorem \ref{T:extended-Airy-sheet}, we get
\[
\begin{split}
&\scrS(\bx^{(2)}, \by^{(2)}) + \scrS(\bx^{(3)}, \by^{(3)}) - \scrS((\bx^{(2)}, \bx^{(3)}), (\by^{(2)}, \by^{(3)}))
\\
\ge\; &
\scrS(\bx^{(2)}, \by^{(1)}) + \scrS(\bx^{(3)}, \by^{(4)}) - \scrS((\bx^{(2)}, \bx^{(3)}), (\by^{(1)}, \by^{(4)})).
\end{split}
\]
Similarly, we also have
\[
\begin{split}
&\scrS(\bx^{(2)}, \by^{(1)}) + \scrS(\bx^{(3)}, \by^{(4)}) - \scrS((\bx^{(2)}, \bx^{(3)}), (\by^{(1)}, \by^{(4)}))
\\
\ge\; &
\scrS(\bx^{(1)}, \by^{(1)}) + \scrS(\bx^{(4)}, \by^{(4)}) - \scrS((\bx^{(1)}, \bx^{(4)}), (\by^{(1)}, \by^{(4)})).
\end{split}
\]
Thus adding up the above two inequalities we get the first inequality in \eqref{eq:L-disjoint-pf}.
The second inequality in \eqref{eq:L-disjoint-pf} is obvious for $\scrS^n$, so by passing to the limit via Theorem \ref{T:extended-Airy-sheet} it also holds for $\scrS$.
Finally, when the first line in \eqref{eq:L-disjoint-pf} equals zero, so does the second line. The conclusion follows.
\end{proof}

For this next lemma and vectors $\bx, \by \in \R^k$, we write
$$
\min (\bx, \by) = \min\{x_i \wedge y_i: i \in \II{1, k}\}.
$$
We similarly define $\max (\bx, \by)$.
\begin{lemma}  \label{L:algebraic-disjoint}
For each $M, h > 0$ and $k, k' \in \N$, there is a random number $P > 0$ such that the following is true.
For any $\bx, \by \in \R^k_\le, \bx', \by' \in \R^{k'}_\le$ and $s, t \in\R$ with $\|\bx\|_2, \|\by\|_2, |s|, |t| < M$, 
$s<t$, $t-s < P$, and 
$
\min(\bx', \by') - \max(\bx, \by) > h$,
we have 
$$
\scrL^*((\bx, \bx'), s; (\by, \by'), t) = \scrL^*(\bx, s; \by, t) + \scrL^*(\bx', s; \by', t).
$$
\end{lemma}
\begin{proof}
For each $\ell \in \Z$, let 
$$
J_\ell = \{(x, x', s, t) : |x|, |s| < M, x \in 2^\ell \Z, x' = x + 2^\ell, s \in 2^{2\ell} \Z, t=s+2^{2\ell-1} \}. 
$$
By Lemma \ref{l:bound-L-disjoint}, 
for any $\ell<0$,
with probability at least 
$1-cM^2 2^{-3\ell}e^{-d2^{-\ell}}$ (for some constants $c, d$ depending on $k, k'$), 
for any 
$(x, x', s, t) \in J_\ell$, we have
$$
\scrL^*(x^k,s;x^k,t) + \scrL^*(x'^{k'},s; x'^{k'},t) = \scrL^*((x^k,x'^{k'}),s;(x^k, x'^{k'}),t).
$$
Then almost surely, there is a random $L_0\in \Z_-$, such that this event happens for all $\ell\le L_0$. Now,  we can choose $P$ small enough such that for any $\bx, \by \in \R^k_\le, \bx', \by' \in \R^{k'}_\le$ and $s,t \in\R$ satisfying the conditions of the lemma, we can find $\ell \le L_0$ and $(\tilde x, \tilde x', \tilde s, \tilde t) \in J_\ell$ such that the following holds:
\begin{itemize}
	\item $\tilde s < s < t < \tilde t$ and $\min (\bx', \by') > \tilde x+h/3$, $\max (\bx, \by) < \tilde x' -h/3$
	\item There exist optimizers $\pi, \pi'$ from $(\tilde x^k, \tilde s)$ to $(\tilde x^k, \tilde t)$ and from $(\tilde x'^{k'}, \tilde s)$ to $(\tilde x'^{k'}, \tilde t)$, such that $\bx\le \pi(s) \le \pi'(s)\le \bx'$ and $\by\le \pi(t) \le \pi'(t)\le \by'$.
\end{itemize}
To ensure the second condition, we have used the transversal fluctuation bound on optimizers from Lemma \ref{L:transfluc}. The fact that $\ell \le L_0$ ensures that 
$$
\scrL^*(\tilde x^k, \tilde s;\tilde x^k, \tilde t) + \scrL^*(\tilde x'^{k'}, \tilde s;\tilde x'^{k'}, \tilde t) = \scrL^*((\tilde x^k, \tilde x'^{k'}),s;(\tilde x^k, \tilde x'^{k'}),t),
$$
which implies that
$$
\scrL^*(\pi(s),s;\pi(t),t) + \scrL^*(\pi'(s),s;\pi'(t),t) = \scrL^*((\pi(s), \pi'(s)),s;(\pi(t), \pi'(t)),t).
$$
Finally, assuming that for any $s<t\in\Q$ the event in Lemma \ref{l:property-of-L-disjoint} holds, we have that $\scrL^*((\bx, \bx'), s; (\by, \by'), t) = \scrL^*(\bx, s; \by, t) + \scrL^*(\bx', s; \by', t)$, if $s,t \in \Q$. By continuity of $\scrL^*$ the conclusion follows.
\end{proof}

\begin{proof}[Proof of the equality in Proposition \ref{P:sum-of-disjoint}]
First assume that $\|\pi''\|_{\scrL^*} > -\infty$. Let $\delta>0$.
Take $s<t_0<t_1<\cdots<t_m<t$, such that
$\sum_{i=1}^m \scrL^*(\pi''(t_{i-1}), t_{i-1}; \pi''(t_i), t_i) < \|\pi''\|_{\scrL^*} + \delta$, and $t_0 - t, s-t_m < \delta$.
We next choose parameters to apply Lemma \ref{L:algebraic-disjoint}. Let 
$$
h = \min_{r\in [t_0,t_m]} \{\min \pi'(r) - \max \pi(t)\}, \qquad M = \max \{ |s|, |t|, \max_{r\in[t,s]} \|\pi''(r)\|_2  \}.
$$
Observe that $h > 0$ by the assumptions of the proposition. Let $P$ be as in Lemma \ref{L:algebraic-disjoint} for this $M, h$.
Then we choose $\overline{m}\ge m$, and $t_0=\overline{t}_0<\overline{t}_1<\cdots<\overline{t}_{\overline{m}}=t_m$
such that $\{t_0, t_1, \cdots, t_m\} \subset \{\overline{t}_0, \overline{t}_1, \cdots, \overline{t}_{\overline{m}}\}$,
and $\overline{t}_i-\overline{t}_{i-1}<P$ for each $1\le i \le \overline{m}$.
Then we have 
\[
\begin{split}
&
\|\pi\|_{\scrL^*}- 
\scrL^*(\pi(s),s; \pi(t_0), t_0)
-\scrL^*(\pi(t_{{m}}), t_{{m}}; \pi(t),t)
\\
&
 + \|\pi'\|_{\scrL^*} 
- 
\scrL^*(\pi'(s),s; \pi'(t_0), t_0)
-\scrL^*(\pi'(t_{{m}}), t_{{m}}; \pi'(t),t)
\\
\le
&\; \sum_{i=1}^{\overline{m}} \scrL^*(\pi(\overline{t}_{i-1}), \overline{t}_{i-1}; \pi(\overline{t}_i), \overline{t}_i) + \scrL^*(\pi'(\overline{t}_{i-1}), \overline{t}_{i-1}; \pi'(\overline{t}_i), \overline{t}_i)\\
=& \; 
\sum_{i=1}^{\overline{m}} \scrL^*(\pi''(\overline{t}_{i-1}), \overline{t}_{i-1}; \pi''(\overline{t}_i), \overline{t}_i)
\\
<& \; \|\pi''\|_{\scrL^*} + \delta.    
\end{split}
\]
Now we send $\delta \to 0$.
By Lemma \ref{l:uniform-EL-bound} we have
$$
\limsup_{\delta\to 0}  \scrL^*(\pi(s),s; \pi(t_0), t_0)
=
\limsup_{\delta\to 0} 
\scrL^*(\pi(t_{{m}}), t_{{m}}; \pi(t),t) \le 0,
$$
$$
\limsup_{\delta\to 0}  \scrL^*(\pi'(s),s; \pi'(t_0), t_0)
=
\limsup_{\delta\to 0} 
\scrL^*(\pi'(t_{{m}}), t_{{m}}; \pi'(t),t) \le 0.
$$
Therefore $\|\pi\|_{\scrL^*} + \|\pi'\|_{\scrL^*} \le \|\pi''\|_{\scrL^*}$, and our conclusion follows. In the case when $\|\pi''\|_{\scrL^*} = -\infty$, we can apply the same argument with an arbitrary $b \in \R$ in place of $\|\pi''\|_{\scrL^*} + \de$ to get the result.
\end{proof}

\section{Disjointness of optimizers}
\label{S:disjoitness-optimizers}
The main goal of this section is to prove the following disjointness result.
\begin{prop}
\label{P:disjointness}
Almost surely, for any $(\bx, r; \by, t)\in \fX_\uparrow$,
there exists an optimizer $\pi$ in $\scrL^*$ from $(\bx,r)$ to $(\by,t)$ such that $\pi_i(s) < \pi_j(s)$ for all $i < j$ and $s \in (r, t)$.
\end{prop}
Proposition \ref{P:disjointness} essentially says that disjoint optimizers in Brownian LPP remain disjoint even as we pass to the limit. In other words, it implies that constituent paths in disjoint optimizers in Brownian LPP are separated from each other by an amount that remains visible in the limiting scaling. This is not an obvious fact, and we will prove it in stages. Note also that Proposition \ref{P:disjointness} still leaves open the possibility that at exceptional points $(\bx, r; \by, t)\in \fX_\uparrow$ where there are multiple $\scrL^*$-optimizers, some of these optimizers may have overlapping constituent paths; we do not believe believe that such optimizers exist but we do not attempt to resolve this issue here.

By using Proposition \ref{P:disjointness} and checking the definition of the extended directed landscape from Definition \ref{D:ext-land}, we will show that $\scrL^* = \scrL$. This is the content of Theorem \ref{T:extended-landscape}.
We will also use Proposition \ref{P:disjointness} to prove Theorem \ref{T:disjoint-optimizers-in-L} and Corollaries \ref{C:rsk} and \ref{C:disjointness}.

\subsection{Convergence in the overlap topology}

We start with the following weaker result, which says that all optimizers are disjoint at fixed time.
\begin{lemma}  \label{l:disjoint-at-fixed-time}
For any fixed $s$ the following holds almost surely for $\scrL^*$.
For any $(\bx, \by) \in \fX$, $r<s<t$, and any optimizer $\pi$ from $(\bx, r)$ to $(\by, t)$, we have that $\pi_1(s) < \pi_2 (s) < \dots < \pi_k(s)$.
\end{lemma}
\begin{proof}
As any optimizer restricted to a smaller interval of time is also an optimizer, it suffices to prove the result for fixed $r=s-\delta$ and $t=s+\delta$ with a fixed small $\delta$, and for $\bx, \by \in \R^k_\le$ with each coordinate in a compact interval.
Since $\scrL^*$ has extended Airy sheet marginals, the conclusion follows from Lemma \ref{L:jump-points-distinct} for compact sets $K \sset \fX$ such that every point $(\bx, \by) \in K$ satisfies $x_1 > 0 > y_k$.
For more general compact sets, the conclusion follows by skew symmetry of $\scrL^*$ (Lemma \ref{L:sym-L}).
\end{proof}

For any continuous paths $\pi_n:[r_n,t_n]\to \R^k_\le$, $n\in\N$, and $\pi:[r,t]\to \R^k_\le$, we say that $\pi_n\to \pi$ \textbf{in the overlap topology}, if for all large enough $n$, $O_n=\{s\in [r,t]\cap [r_n,t_n]: \pi_n(s)=\pi(s)\}$ is an interval, and the endpoints of $O_n$ converge to $r$ and $t$. Overlap was first introduced in \cite[Section 3]{DSV}, and is particularly useful for studying geodesics or optimizers in the directed or extended landscape.

We aim to prove an overlap convergence result for optimizers. 
We will require two closely related results for $\scrL$-geodesics from \cite{DSV}.
To state them, for any path $\pi:[r,t]\to\R$, define the \textbf{graph} of $\pi$ by
\[
\fg \pi:= \{(\pi(s),s):s\in[r,t]\}.
\]
This is the usual graph of a function with coordinates reversed.

\begin{lemma}[\protect{\cite[Lemma 3.1]{DSV}}]\label{L:overlap-0}
Almost surely the following is true.
Let $(p_n; q_n) \to (p; q) \in \R^4_\uparrow$, and let $\pi_n$ be any sequence of geodesics from $p_n$ to $q_n$.
Then the sequence $\fg \pi_n$ is precompact in the Hausdorff metric, and any subsequential limit is the graph of a geodesic from $p$ to $q$.
\end{lemma}
\begin{lemma}[\protect{\cite[Lemma 3.3]{DSV}}]
\label{L:overlap-1} 
Almost surely the following is true. Let $(p_n; q_n) \to (p; q) \in \R^4_\uparrow$, and let $\pi_n$ be any sequence of geodesics from $p_n$ to $q_n$.
Suppose that $(p_n; q_n) \in \Q^4$, and $\fg\pi_n\to \fg\pi$ in the Hausdorff metric, for some geodesic $\pi$ from $p$ to $q$.
Then $\pi_n \to \pi$ in the overlap topology.
\end{lemma}

From these we can deduce the following result.
\begin{lemma}
\label{L:overlap-2}
Almost surely the following is true.
Let $(p_n; q_n) = (x_n, s; y_n, t) \to (p; q) = (x, s; y, t) \in \R^4_\uparrow$, and suppose that $x_n \ge x, y_n \ge y$ for all $n$. Let $\pi_n$ be the sequence of rightmost geodesics from $p_n$ to $q_n$, and let $\pi$ be the rightmost geodesic from $p$ to $q$.
Then $\pi_n \to \pi$ in the overlap topology.
\end{lemma}

The existence of rightmost and leftmost geodesics follows from \cite[Lemma 13.2]{DOV}; alternately, it follows from Lemma \ref{L:exist-left-right-most}. 

\begin{proof}[Proof of Lemma \ref{L:overlap-2}]
First, by Lemma \ref{L:overlap-0} the sequence $\mathfrak{g} \pi_n$ is precompact in the Hausdorff metric and any subsequential limit is the graph of a geodesic from $p$ to $q$. Consider such a subsequential limit $\mathfrak{g} \pi'$. Since the $\pi_n$ are rightmost geodesics, by Lemma \ref{L:monotonicity} we have $\pi_n \ge \pi$ for all $n$, and hence $\pi' \ge \pi$. Since $\pi$ is a rightmost geodesic, this implies $\pi' = \pi$, and therefore $\mathfrak{g} \pi_n \to \mathfrak{g} \pi$. Lemma \ref{L:overlap-1} then completes the proof.
\end{proof}
We can now upgrade the above lemma to optimizers in $\scrL^*$.
\begin{lemma} \label{L:pointwise-to-overlap}
Almost surely the following statement is true.
Take any $r<t$, $\bx, \by \in \R^k_\le$, and two sequences $\bx^{(i)}, \by^{(i)} \in \R^k_\le$, for $i\in\N$.
Suppose that $x^{(i)}_j > x_j$, $y^{(i)}_j > y_i$ for $j \in \II{1, k}$, and that $\bx^{(i)}\to \bx$, $\by^{(i)}\to \by$ as $i\to\infty$.
Let $\pi^{(i)}$ be the rightmost optimizer from $(\bx^{(i)}, r)$ to $(\by^{(i)}, t)$, and $\pi$ be the rightmost optimizer from $(\bx, r)$ to $(\by, t)$.
Then $\pi^{(i)} \to \pi$ in the overlap topology.
\end{lemma}

\begin{proof}
First, the graphs of all the optimizer paths $\pi_j^{(i)}$ are H\"older-$(2/3)^-$ with a common H\"older constant by Lemma \ref{L:transfluc}. Therefore along any subsequence we can take a further subsequence so that $\pi^{(i)}$ converges to a continuous limit in the dyadic pointwise topology.
This limit must be an optimizer from $(\bx, r)$ to $(\by, t)$, by Lemma \ref{L:upper-semi-conti-path} and the continuity of $\scrL^*$. Thus the limit must be $\pi$ since $\pi\le \pi^{(i)}$ for each $i$, by Lemma \ref{L:monotonicity}.
We conclude that $\pi^{(i)} \to \pi$ in the  dyadic pointwise topology.

We take any $s\in\Q_2$ with $r<s<t$ (recall that $\Q_2$ is the set of dyadic rational numbers).
By Lemma \ref{l:disjoint-at-fixed-time} we can assume that $\pi_1(s), \cdots, \pi_k(s)$ are pairwise distinct.
Then in a small neighborhood of $s$, the paths $\pi_1,\cdots, \pi_k$ are mutually disjoint.
By Lemma \ref{L:transfluc}, we can find a (random) $\delta>0$, such that $\delta \in \Q_2$, and for any $0<\delta_1^-,\cdots, \delta_k^-, \delta_1^+,\cdots, \delta_k^+ < \delta$, and any geodesics from $(\pi_j(s-\delta)+\delta_j^-, s-\delta)$ to $(\pi_j(s+\delta)+\delta_j^+, s+\delta)$, $1\le j \le k$, these geodesics are disjoint. By the dyadic convergence established above and Lemma \ref{L:monotonicity}, for all large enough $i$ we have
\begin{equation}
\label{E:pijj}
\pi_j(s \pm \de) \le \pi^{(i)}_j(s \pm \de) < \pi_j(s \pm \de) + \de
\end{equation}
for all $j \in \II{1, k}$. From now on, we work with $i$ such that \eqref{E:pijj} holds.
For each $1\le j \le k$, let $\tau^{(i)}_j$ be the rightmost geodesic from $(\pi^{(i)}_j(s-\delta), s-\delta)$ to $(\pi^{(i)}_j(s+\delta), s+\delta)$. We claim that $\pi^{(i)}_j = \tau^{(i)}_j$ on the interval $I_\de = [s-\delta, s+\delta]$ as long as $i$ is sufficiently large. Indeed, letting $\tau^{(i)} = (\tau_1^{(i)}, \dots, \tau_k^{(i)})$, we have
\begin{equation*}
\|\tau^{(i)}\|_{\scrL^*} = \sum_{j=1}^k \|\tau^{(i)}_j\|_{\scrL^*} \ge \sum_{j=1}^k \|\pi_j^{(i)}|_{I_\de}\|_{\scrL^*} \ge \|\pi^{(i)}|_{I_\de} \|_{\scrL^*}.
\end{equation*}
Here the equality follows from Proposition \ref{P:sum-of-disjoint} and \eqref{E:pijj}, the first inequality uses that each $\tau^{(i)}_j$ is a geodesic, and the second inequality uses
Proposition \ref{P:sum-of-disjoint} again. Since $\pi^{(i)}$ is an optimizer, all inequalities above must be equalities, so $\tau^{(i)}$ must also be an optimizer, and all the paths $\pi_j^{(i)}|_{I_\de}$ must be geodesics. Since $\pi^{(i)}$ is a rightmost optimizer, we have $\pi^{(i)}|_{I_\de} \ge \tau^{(i)}$. Since each of the $\pi_j^{(i)}|_{I_\de}$ are geodesics and each of the $\tau_j^{(i)}$ are rightmost geodesics, this implies that $\tau_j^{(i)} = \pi_j^{(i)}|_{I_\de}$ for all $i$ large enough.

The same argument shows that each $\pi_j|_{I_\de}$ is also a rightmost geodesic. Therefore by Lemma \ref{L:overlap-2}, and the fact that for any $\de > 0$, \eqref{E:pijj} holds for all large enough $i$ shows that for $i$ large enough we have $\pi_j^{(i)}(s) = \pi_j(s)$ for all $j \in \II{1, k}$.

Next we take $s_1, s_2\in\Q_2$ with $r<s_1<s_2<t$.
For $i$ large enough we have $\pi^{(i)}(s_1) = \pi(s_1)$ and $\pi^{(i)}(s_2) = \pi(s_2)$. Thus $\pi^{(i)}(s) = \pi(s)$ for any $s_1<s<s_2$ since on $[s_1, s_2]$ both $\pi^{(i)}$ and $\pi$ are the rightmost optimizer from $(\pi(s_1), s_1)$ to $(\pi(s_2), s_2)$.
By sending $s_1\to r$ and $s_2\to t$ we get the conclusion.
\end{proof}
\begin{remark}
A similar statement to Lemma \ref{L:pointwise-to-overlap} holds for convergence to leftmost optimizers.
\end{remark}

\subsection{Two paths}

To prove Proposition \ref{P:disjointness}, we start with the two-path case with fixed endpoints.
\begin{lemma}
\label{L:disjoint-2}
Fix $\bx, \by \in \R^2_\le$ and $s < t$. Then almost surely, the unique optimizer in $\scrL^*$ from $(\bx, s)$ to $(\by, t)$ consists of two paths that are disjoint, except possibly at the endpoints.
\end{lemma}

The proof of Lemma \ref{L:disjoint-2} consists of two steps. First, we will handle the common endpoint case when $x_1 = x_2, y_1 = y_2$. This can be dealt with by proving a quantitative estimate on the location of optimizers from $\bx$ to $\by$ at a fixed time, Lemma \ref{L:disjoint-fixedtime} below, and then appealing to H\"older continuity of optimizers. We will then compare the case of general $\bx = (x_1 < x_2), \by = (y_1 < y_2)$ to the common endpoint case by applying a resampling argument to show that optimizers from $(\bx, s)$ to $(\by, t)$ can be approximated in overlap by optimizers from $((0,0), s-1)$ to $((0,0), t + 1)$ in a sequence of extended landscapes defined with different, independent noise on $[s-1, s]$ and $[t, t + 1]$.

\begin{lemma}  \label{L:disjoint-fixedtime}
Let $\pi=(\pi_1,\pi_2):[0,1]\to \R^2_\le$ be the almost surely unique optimizer from $((0,0),0)$ to $((0,0),1)$ in $\scrL^*$. 
Take any $\delta, \eta, d>0$.
There exists $c > 0$ depending on $\delta, d, \eta$, such that for any $t\in[\delta, 1-\delta]$, and $\epsilon>0$, we have
$$
\p(|\pi_1(t)-\pi_2(t)| < \epsilon, |\pi_1(t)|, |\pi_2(t)|<d ) < c\epsilon^{2-\eta}.
$$
\end{lemma}

The basic strategy for the proof of Lemma \ref{L:disjoint-fixedtime} is to use Proposition \ref{P:B-continuity} to relate the joint distribution of $\pi_1(t), \pi_2(t)$ to a certain optimization problem across the first two lines of two parabolic Airy line ensembles. This optimization problem is amenable to analysis since the parabolic Airy line ensemble withstands a strong comparison to independent Brownian motions, and is further reduced to the same optimization problem for several independent Brownian motions (see Lemma \ref{L:LPP-2BM} below).

It is worth mentioning that exponents for disjointness of \textit{geodesics} in the directed landscape have previously been analyzed by Hammond \cite{hammond2017exponents} by similarly appealing to a particular statistic in the parabolic Airy line ensemble. Hammond showed that the probability that $k$ geodesics from time $0$ to time $1$ that all start and end within $\ep$ of each other are mutually disjoint is bounded above by $\ep^{(k^2-1)/2 + o(1)}$. This exponent is expected to be sharp. One thing that is new in the analysis in Lemma \ref{L:disjoint-fixedtime} is the use of the precise relationship between the Airy sheet and the parabolic Airy line ensemble, which was not yet known when \cite{hammond2017exponents} was written. 

The optimization problem (Lemma \ref{L:LPP-2BM}) that arises in Lemma \ref{L:disjoint-fixedtime} is not straightforward to analyze, even heuristically. Indeed, we do not expect that the exponent $2$ in Lemma \ref{L:disjoint-fixedtime} is sharp and do not have a prediction for the true exponent. To set up the proof, we need the following result giving a strong comparison between the parabolic Airy line ensemble $\scrB$ on a compact set and a sequence of independent Brownian motions.

\begin{theorem} \label{T:radon-n-deri}
	For $d > 0$, let $\scrC_d$ be the space of continuous functions on $[-d, d]$ which vanish at $-d$. Let $\mu_d$ denote the law of a standard Brownian motion on $[-d, d]$, and for $k \in \N$ let $\mu_d^{\bigotimes k}$ denote the law of $k$-tuples of functions in $\scrC^k_d$ given by the product of $k$ copies of $\mu_d$. For any measurable set $A\subset \scrC^k_d, k \in \N$ and $d \ge 1$ we have
	$$
	\p(\hat \scrB^k \in A) \le \mu_d^{\bigotimes k}(A) \exp \lf(b k d^6 + d e^{bk} \lf(\log [\mu_d^{\bigotimes k}(A)]^{-1} \rg)^{5/6} \rg),
	$$
	where $b>0$ is a universal constant, $\hat \scrB^k = (\hat\scrB^k_1, \dots, \hat \scrB^k_k)$, and each $\hat \scrB^k_i$ is given by
	$$
	\hat \scrB^k_i(x) = 2^{-1/2} \lf(\scrB^k_i(x) - \scrB^k_i(-d)\rg).
	$$
\end{theorem}

The main result in \cite{CHH20} (Theorem 3.11 therein) shows that each of the marginals $\hat \scrB_i^k$ satisfy the above Radon-Nikodym derivative bound with $\mu_d$ in place of $\mu_d^{\bigotimes k}$. While Theorem \ref{T:radon-n-deri} is stronger than \cite[Theorem 3.11]{CHH20}, it can nonetheless be proven by combining the same key technical ingredients developed in Sections 4 and 5 of \cite{CHH20}. We do this in Appendix \ref{app:radon-n-deri}.

\begin{proof}[Proof of Lemma \ref{L:disjoint-fixedtime}]
We define
$$\tbz=\argmax_{-2d\le z_1\le z_2\leq 2d} \scrL^*((0,0), 0; \bz, t) + \scrL^*(\bz, t; (0,0), 1).$$
Since $\scrL^*$ has extended Airy sheet marginals, by Proposition \ref{P:B-continuity} and the symmetry $\scrS(\bx, \by) = \scrS(-\by, -\bx)$ (Lemma \ref{L:basic-sym}), we could alternatively define $\tbz$ as
\begin{equation}
\label{E:tbz-argmax}
\tbz=\argmax_{-2d\le z_1\le z_2\leq 2d} \scrB[(0,0) \to (z_1,z_2)] + \scrB'[(0,0)\to (-z_2,-z_1)],
\end{equation}
where $t^{1/3} \scrB(t^{-2/3} \; \cdot)$ and $(1-t)^{1/3} \scrB'((1-t)^{-2/3} \; \cdot)$ are independent parabolic Airy line ensembles. Uniqueness of the $\argmax$ follows the same arguments as in the proof of Lemma \ref{L:two-sheet-sum-unique-max}.
It suffices to prove
$$
\p(|\tz_1-\tz_2| < \epsilon, |\tz_1|, |\tz_2|< d) < c\epsilon^{2-\eta},
$$
since if $|\pi_1(t)|, |\pi_2(t)|<d$, we must have that $\pi_1(t)=\tz_1$ and $\pi_2(t)=\tz_2$. By Proposition \ref{P:high-paths-B},
\begin{align*}
\scrB[(0,0) \to (z_1,z_2)]
&=
\max_{z_1\le w\le z_2} \scrB_1(z_1) + \scrB_1(z_2) - \scrB_1(w) + \scrB_2(w), \\
\scrB'[(0,0) \to (-z_2,-z_1)]
&=
\max_{-z_2\le w\le -z_1} \scrB'_1(-z_2) + \scrB'_1(-z_1) - \scrB'_1(w) + \scrB'_2(w).
\end{align*}
Therefore by Theorem \ref{T:radon-n-deri} applied to the interval $[-2\delta^{-2/3}d, 2\delta^{-2/3}d]$ and Brownian scaling and time-reversal symmetry of Brownian motion, it suffices to study the same problem when $\scrB_1(\cdot), \scrB_2(\cdot)$, $\scrB_1'(- \; \cdot), \scrB_2'(- \; \cdot)$ are replaced by independent Brownian motions. This is done in Lemma \ref{L:LPP-2BM}, implying the desired result.
\end{proof}

\begin{lemma} \label{L:LPP-2BM}
Take four independent two-sided Brownian motions $B_1, B_2, B_1', B_2':\R\to \R$, with diffusion parameter $2$.
Let $(\tz_1, \tz_2, \tw, \tw')$ be
$$
\argmax_{-2\le z_1\le w, w' \le z_2\le 2}
(B_1(z_1)+B_1(z_2)-B_1(w)+B_2(w))+(B_1'(z_1)+B_1'(z_2)-B_1'(w')+B_2'(w')).
$$
Note that a priori we do not assume the $\argmax$ is unique, and just take an arbitrary one.
Then given any small $\eta>0$, for any small enough $\epsilon>0$ we have $\p(|\tz_1|,|\tz_2|<1, |\tz_1-\tz_2|<\epsilon) < \epsilon^{2-\eta}$.
\end{lemma}

While now we only work with Brownian motions, analyzing the $\argmax$ formula is still involved. Our general strategy is to simplify the problem by restricting the choice of parameters (for example, we will usually take $w=w'$). We expect that these simplifications do not capture the full picture and so we do not expect the exponent of $2-\eta$ to be sharp.

\begin{proof}[Proof of Lemma \ref{L:LPP-2BM}]
Throughout the proof we assume that $\eta > 0$ is small and that $\ep > 0$ is sufficiently small given $\eta$.
	
\textbf{Step 1: Conditioning on the location of the $\argmax$.} We first split $[-1, 1]$ into $\lceil\epsilon^{-1}\rceil$ intervals, each of length at most $2\epsilon$.
We just need to show that, for each interval $I$, we have $\p(\tz_1,\tz_2\in I) < \epsilon^{3-\eta }$.

Let $F(z_1,z_2,w,w')$ denote the function inside the $\argmax$.
Denote the center of $I$ by $z_I$, and $F_I=F(z_I,z_I,z_I,z_I)$.
If $\tz_1,\tz_2\in I$, then one of $\scrE$ and $\scrE_1\cap \scrE_2$ happens, where
\begin{align*}
\scrE:&\quad \max_{z_1\le w, w' \le z_2, z_1,z_2\in I} F(z_1,z_2,w,w') \ge F_I + \epsilon^{1/2-\eta/10},\\
\scrE_1:&\quad \max_{z_I\le w \le z_2 \le 2} F(z_I,z_2,w,w) < F_I + \epsilon^{1/2-\eta/10},\\
\scrE_2:&\quad \max_{-2\le z_1\le w \le z_I} F(z_1,z_I,w,w) < F_I + \epsilon^{1/2-\eta/10}.
\end{align*}
We have taken $w'=w$ in defining $\scrE_1$ and $\scrE_2$ to simplify arguments below.
From the tail of Brownian motions in an interval of length $2\epsilon$ we have $\p(\scrE)<e^{-\epsilon^{-\eta/6}}$.
Also note that $\scrE_1$ and $\scrE_2$ are independent, since they depend only on $B_1+B_1'-B_1(z_I)-B_1'(z_I)$ and $B_2+B_2'-B_2(z_I)-B_2'(z_I)$, to the right and left of $z_I$, respectively.
Thus $\p(\scrE_1\cap\scrE_2) = \p(\scrE_1)\p(\scrE_2)$.
It remains to show that $\p(\scrE_1)<\epsilon^{3/2-\eta/3}$, since similarly we will also have $\p(\scrE_2)<\epsilon^{3/2-\eta/3}$.

\textbf{Step 2: Reducing to two Brownian motions.} To bound $\p(\scrE_1)$, we rewrite $F(z_I,z_2,w,w)$ using two independent Brownian motions.
Consider two processes on $[0,3]$, defined as $\tilde B_1(z)= 2^{-1}(B_1(z_I+z)+B_1'(z_I+z)-B_1(z_I)-B_1'(z_I))$ and $\tilde B_2(z) = 2^{-1}(B_2(z_I+z)+B_2'(z_I+z)-B_2(z_I)-B_2'(z_I))$, respectively.
These are two independent standard Brownian motions.
Letting $h=2^{-1}\epsilon^{1/2-\eta/10}$, we have
\begin{equation}
\label{E:scrE1}
\p(\scrE_1) 
\le \p\left(\max_{0\le w \le z \le 1} \tB_1(z) - \tB_1(w) + \tB_2(w) < h \right).
\end{equation}
Now let $\hat B(w) = \max \{ \tB_1(z)-\tB_1(1-w): 1-w\le z\le 1 \}$. By \cite[Theorem 2.34]{morters2010brownian} and the independence of $\tilde B_1$ and $\tilde B_2$, we have $(\hat B, \tilde B_2) \eqd (|B|, \tilde B_2)$ on $[0, 1]$, where $B$ is another Brownian motion, also independent of $\tilde B_2$. Therefore the right-hand side of \eqref{E:scrE1} equals
\begin{equation}
\label{E:scrE11}
\p \lf(\max_{w\in[0,1]}|B(1-w)|+\tB_2(w)<h \rg).
\end{equation}
Our goal is to show that \eqref{E:scrE11} is $O(h^3)$.

\textbf{Step 3: Computation using the reflection principle.} 
Up to a constant factor, \eqref{E:scrE11} further equals
\begin{equation}
\label{E:double-1}
\iint \p \lf(\max_{w\in[0,1]}|B(1-w)|+\tB_2(w)<h\;\Big \mid \; B(1)=a, \tB_2(1)=b\rg) e^{-(a^2+b^2)/2} dadb.
\end{equation}
Conditioned on $B(1)=a, \tB_2(1)=b$, the processes $B(1-w)-a(1-w)$ and $\tB_2(w)-bw$ are independent Brownian bridges.
Thus we can write the probability in \eqref{E:double-1} as
$$
\p\lf(\max_{w\in[0,1]} \max \{G_1(w)+G_2(w) + a(1-w)+bw, -G_1(w)+G_2(w) - a(1-w)+bw\} < h\rg), 
$$
where $G_1, G_2:[0,1]\to \R$ are two independent Brownian bridges.
Using that $H_1 := 2^{-1/2}(G_1+G_2)$ and $H_2 := 2^{-1/2}(G_1-G_2)$ are independent Brownian bridges, this probability can be further written as
\begin{equation*}
\p\lf(\max_{w\in[0,1]} \sqrt{2} H_1(w) + a(1-w)+bw < h \rg) \p \lf(\max_{w\in[0,1]} \sqrt{2} H_2(w) - a(1-w)+bw < h \rg).
\end{equation*}
These two probabilities can be computed using the reflection principle (see e.g. \cite[Theorem 2.19]{morters2010brownian}).
The first one equals (for $B$ being a standard Brownian motion on $[0,1]$)
\[
\begin{split}
\p\Big(\max_{w\in[0,1]} B(w) < 2^{-1/2}(h-a)
&\mid B(1)=2^{-1/2}(b-a)
\Big)\\
=\;&
\frac{\p\lf(B(1)=2^{-1/2}(b-a)\rg)-\p\lf(B(1)=2^{-1/2}(2h-a-b)\rg)}{\p\lf(B(1)=2^{-1/2}(b-a)\rg)}\\
=\;&
e^{(a-b)^2/4}(e^{-(a-b)^2/4} - e^{-(2h-a-b)^2/4}),    
\end{split}
\]
and similarly for the second one.
Therefore we can write \eqref{E:double-1} as
\[
\begin{split}
&\iint_{|a|,b\le h} e^{-(a^2+b^2)/2} e^{(a-b)^2/4}(e^{-(a-b)^2/4} - e^{-(2h-a-b)^2/4})
e^{(-a-b)^2/4}(e^{-(-a-b)^2/4} - e^{-(2h+a-b)^2/4})
dadb\\
=\;
&\iint_{|a|,b\le h} e^{-(a^2+b^2)/2} (1 - e^{-(h-a)(h-b)})
(1 - e^{-(h+a)(h-b)})
dadb\\
<\;
&\iint_{|a|,b\le h} e^{-(a^2+b^2)/2} (h-a)(h+a)(h-b)^2
dadb\\
<\; &2h^3 \int_{b\le h}e^{-b^2/2}(h-b)^2 db.
\end{split}
\]
We note that the integral in the last line is uniformly bounded for $h<1$.
Thus we conclude that $\p(\scrE_1)<\epsilon^{3/2-\eta/3}$, and our conclusion follows.
\end{proof}

Before moving to the proof of Lemma \ref{L:disjoint-2}, we need one more result.

\begin{lemma}
	\label{L:positive-probability}
	Let $s < t$, and let $\scrF$ denote the $\sig$-algebra generated by $\scrL^*$ restricted to time increments $[r, r'] \sset [s, t]$. Let $\pi$ denote the almost surely unique optimizer from $((0,0),s-1)$ to $((0,0),t+1)$. Then the conditional law of $(\pi(s), \pi(t))$ given $\scrF$ almost surely has full support $\R^2_\le \X \R^2_\le$.
\end{lemma}

\begin{proof}
Let $(\bx_*, \by_*) = \argmax_{(\tbx,\tby)\in \R^2_\le} F(\tilde \bx, \tilde \by),$ where
$$
F(\tilde \bx, \tilde \by) =  \scrL^*((0,0),s-1;\tbx,s)+\scrL^*(\tbx,s;\tby,t)+\scrL^*(\tby,t;(0,0),t+1).
$$
Then $(\pi(s), \pi(t)) = (\bx^*, \by^*)$ and by Lemma \ref{L:as-unique-geo-fixedend}, the argmax is almost surely unique. Now, the outer two functions are independent of $\scrF$.
Moreover, by Proposition \ref{P:high-paths-B}, we have
\begin{equation}
\label{E:L**}
\scrL^*((0,0),s-1;\tbx,s) = \max_{\tilde x_1 \le w \le \tilde x_2} \scrB_1(\tilde x_1) + \scrB_1(\tilde x_2) - \scrB_1(w) + \scrB_2(w),
\end{equation}
where $\scrB$ is a parabolic Airy line ensemble.
A similar decomposition exists for $\scrL^*(\tby,t;(0,0),t+1)$ in terms of an independent parabolic Airy line ensemble $\scrB'$. Now let $(\bx, \by) \in \R^2_\le \X \R^2_\le$. Conditionally on $\scrF$, we can apply the Brownian Gibbs property to resample the first two lines of $\scrB, \scrB'$ on an interval $[-m, m]$ containing $x_1, x_2, y_1, y_2$. Let $F'$ denote the analogue of the original function $F$ after resampling. By \eqref{E:L**}, for any $M, \de > 0$, with positive probability we have
\begin{align*}
F'(\bx, \by) - F(\bx, \by) > M, \qquad |F(\bu) - F'(\bu)| \le \de \text{ for all } \bu \text{ such that } \|\bu - (\bx, \by)\|_2 > \de. 
\end{align*}
Since $F$ achieves its argmax, this implies that $F'$ can achieve its argmax arbitrarily close to $(\bx, \by)$. Since $F \eqd F'$, this gives the result.
\end{proof}

\begin{proof}[Proof of Lemma \ref{L:disjoint-2}]
Let $\pi=(\pi_1,\pi_2)$ be the optimizer from $(\bx,s)$ to $(\by,t)$. By Lemma \ref{L:as-unique-geo-fixedend} we assume that it is the unique one.

\textbf{Step 1.}
We first prove the case where $x_1=x_2, y_1=y_2$.
By the symmetries of $\scrL^*$ (Lemma \ref{L:sym-L}) we may assume $x_1=x_2=y_1=y_2=0$, and $s=0$, $t=1$.

Fix some small $\delta$ with $0<\delta<1$.
Take a large $N\in\N$, and let $t_i=\delta+(1-2\delta)i/N$ for $i=0,\ldots,N$.
By Lemma \ref{L:disjoint-fixedtime},
for any fixed $d$ and $\eta$ there is some constant $c>0$ such that
$$
\p(\exists i, |\pi_1(t_i)-\pi_2(t_i)| < N^{\eta-2/3}, |\pi_1(t_i)|, |\pi_2(t_i)|<d ) < cN^{3\eta-1/3}.
$$
By Lemma \ref{L:transfluc}, each $\pi_i$ is H\"older $2/3^-$. Therefore taking $N \to \infty$ we have
$$
\p(\exists t'\in[\delta, 1-\delta], \pi_1(t')=\pi_2(t'), |\pi_1(t')|< d ) =0.
$$
Since $d$ and $\de$ are arbitrary, we have $\pi_1(t')\neq \pi_2(t')$, $\forall t'\in(0, 1)$.

\textbf{Step 2.} Now we prove the general case by a resampling argument.

Let $\pi'$ be the optimizer from $((0,0),s-1)$ to $((0,0),t+1)$, which is assumed to be unique by Lemma \ref{L:as-unique-geo-fixedend}.
Setting
\begin{equation}
\label{E:bxby}
(\bx_*, \by_*) = \argmax_{(\tbx,\tby)\in \R^2_\le} \scrL^*((0,0),s-1;\tbx,s)+\scrL^*(\tbx,s;\tby,t)+\scrL^*(\tby,t;(0,0),t+1),
\end{equation}
then $\pi'$ is the concatenation of the optimizers from $((0,0),t-1)$ to $(\bx_*,t)$, from $(\bx_*,t)$ to $(\by_*,s)$, and from $(\by_*,s)$ to $((0,0),s+1)$.
Each of these three optimizers must be unique, otherwise $\pi'$ is not unique.

Now we take a series of independent samples of $\scrL^*$, 
denoted as $\scrL^{*,i}$ for $i\in\N$. Using these samples, we can define landscapes $\hat \scrL^{*, i}$ by setting $\hat \scrL^{*, i}(\cdot, r; \cdot, r')$ equal to $\scrL^{*, i}$ when $[r, r'] \sset (s, t)^c$ and equal to $\scrL^*$ when $[r, r'] \sset [s, t]$. Defining $\hat \scrL^{*, i}$ at all other time increments via metric composition yields an extended landscape.

We denote by $\pi^{(i)}$ the optimizer from $((0,0),s-1)$ to $((0,0),t+1)$ in $\hat \scrL^{*, i}$, and define $(\bx_*^{(i)}, \by_*^{(i)})$ as in \eqref{E:bxby} with $\hat \scrL^{*, i}$ in place of $\scrL^*$, so that arguing as before, $\pi^{(i)}$ is a concatenation of the unique optimizer from $((0,0),s-1)$ to $(\bx_*^{(i)},s)$ in $\scrL^{*,i}$, the unique optimizer from $(\bx_*^{(i)},s)$ to $(\by_*^{(i)},t)$ in $\scrL^*$, and the unique optimizer from $(\by_*^{(i)},t)$ to $((0,0),t+1)$ in $\scrL^{*,i}$. In addition, from the first step, we have that each $\pi^{(i)}$ consists of disjoint paths, except for the endpoints.

Conditioned on $\scrL^*$, for any fixed $\bx,\by \in \R^2_\le$ and any $\epsilon > 0$, by Lemma \ref{L:positive-probability} there is a positive probability that $x_{*, j}^{(i)}>x_j$ and $y_{*, j}^{(i)}>y_j$ for all $j \in \II{1, k}$, and that $\|\bx_*^{(i)}-\bx\|_2<\epsilon$ and $\|\by_*^{(i)}-\by\|_2<\epsilon$. 
Thus almost surely, we can find a sequence $i_1<i_2<\cdots$, such that $\bx_*^{(i_\ell)} \to \bx$ and $\by_*^{(i_\ell)} \to \by$ as $\ell \to \infty$, and $x_{*, j}^{(i_\ell)}>x_j$ and $y_{*, j}^{(i_\ell)}>y_j$ for all $\ell$.
Then by Lemma \ref{L:pointwise-to-overlap}, 
the optimizer from $(\bx_*^{(i_k)},s)$ to 
$(\by_*^{(i_k)},t)$ converges to the rightmost optimizer from $(\bx,s)$ to $(\by,t)$, in the overlap topology.
Since for each $i$, the optimizer from $(\bx_*^{(i)},s)$ to 
$(\by_*^{(i)},t)$ consists of disjoint paths, the optimizer from $(\bx,s)$ to $(\by,t)$ must also consists of disjoint paths, except possibly at the endpoints.
\end{proof}
We upgrade Lemma \ref{L:disjoint-2} to all endpoints simultaneously.

\begin{lemma} \label{L:disjoint-2-allpair}
Almost surely the following statement is true.
For any $\bx, \by \in \R^2_\le$ and $s < t$,
there exists an optimizer in the extended landscape from $(\bx,s)$ to $(\by,t)$, that consists of two paths that are disjoint, except possibly at the endpoints.
\end{lemma}
\begin{proof}
By Lemma \ref{L:as-unique-geo-fixedend} and Lemma \ref{L:disjoint-2}, almost surely for all $\bx, \by \in \R^2_\le\cap \Q^2$ and $s < t \in \Q$ the above statement is true, and there is a unique optimizer from $(\bx,s)$ to $(\by,t)$.
For any general $\bx, \by\in \R^2_\le$ and $s < t$, we take a sequence $\bx^{(i)}, \by^{(i)} \in \R^2_\le$ and $s^{(i)} < t^{(i)}$ consisting of rational numbers, and satisfying the following conditions:
\begin{itemize}
    \item $s^{(i)} < s < t < t^{(i)}$, and $s^{(i)} \to s$, $t^{(i)} \to t$ as $i\to \infty$;
    \item for each $1\le j\le k$ and $i\in \N$, 
    $x^{(i)}_j - x_j > (s-s^{(i)})^{1/5}$, $y^{(i)}_j - y_j > (t^{(i)}-t)^{1/5}$;
    \item $\bx^{(i)} \to \bx$, $\by^{(i)} \to \by$ as $i\to \infty$.
\end{itemize}
Let $\pi^{(i)}$ be the unique optimizer from $(\bx^{(i)}, s^{(i)})$ to $(\by^{(i)}, t^{(i)})$.
By Lemma \ref{L:transfluc}, for $i$ large enough we have $\pi^{(i)}(s) > \bx$ and $\pi^{(i)}(t) > \by$, while $\pi^{(i)}(s) \to \bx$ and $\pi^{(i)}(t) \to \by$ as $i\to\infty$.
Then by Lemma \ref{L:pointwise-to-overlap}, as $i\to\infty$
the (unique) optimizer from $(\pi^{(i)}(s), s)$ to $(\pi^{(i)}(t), t)$ converges to the rightmost optimizer from $(\bx, s)$ to $(\by, t)$ in the overlap topology.
This means that the rightmost optimizer from $(\bx,s)$ to $(\by,t)$ consists of two paths that are disjoint, except possibly at the endpoints.
\end{proof}

\subsection{Multiple paths}

We now extend from the two-path case to the general $k$-path case. The basic idea is that if two adjacent paths in $k$-path optimizer overlap at a point $(z, r)$, then a similar situation must occur with two paths in a two-path optimizer whose endpoints are chosen close to $(z, r)$.
\begin{proof}[Proof of Proposition \ref{P:disjointness}]
We show that for each $k\ge 2$, almost surely, for any $\bx, \by \in \R^k_\le$ and $r < t$,
there exists an optimizer from $(\bx,r)$ to $(\by,t)$ consisting of paths that are disjoint, except possibly at the endpoints.
We prove this by induction on $k$.
The $k=2$ case is Lemma \ref{L:disjoint-2-allpair}.
Now suppose that $k>2$ and that the statement is true for $k-1$.

We first prove the fixed endpoint version, i.e., for any fixed $\bx, \by \in \R^k_\le$ and $r < t$, almost surely the unique optimizer $\pi$ from $(\bx, r)$ to $(\by, t)$ is disjoint, except possibly at the endpoints.

Take $\bx', \by' \in \R^{k-1}_\le$ consisting of the first $k-1$ coordinates of $\bx, \by$, respectively.
We assume that the optimizer from $(\bx', r)$ to $(\by', t)$ is unique, and denote it as $\pi'$.
We then have that almost surely, these optimizers interlace; i.e. 
for each $1\le i \le k-1$ and $r\le s\le t$ we have
$$\pi_i(s)\le \pi'_i(s) \le \pi_{i+1}(s).$$
This follows from Lemma \ref{L:monotonicity}.
Indeed, by Lemma \ref{L:transfluc} one can find large enough $x'', y'' \in \R$, such that the geodesic from $(x'',r)$ to $(y'',t)$ (denoted as $\pi''$) is disjoint from $\pi'$.
Then $(\pi',\pi'')$ is a optimizer from $((\bx',x''),r)$ to $((\by',y''),t)$, 
by Proposition \ref{P:sum-of-disjoint};
and by Lemma \ref{L:monotonicity} applied to $\pi$ and $(\pi',\pi'')$ the first inequality is obtained. The second inequality follows similarly by taking $x'',y''$ sufficiently negative.

By the inductive hypothesis, $\pi'$ consists of paths that are disjoint, except possibly at the endpoints.
This implies that $\pi_i$ and $\pi_{i+2}$ are disjoint except possibly at the endpoints, for each $1\le i \le k-2$.

Now suppose that $\pi_i(s)=\pi_{i+1}(s)$ for some $r<s<t$ and $1\le i \le k-1$.
Then $\pi_{i-1}(s)$ (if $i>1$) and $\pi_{i+2}(s)$ (if $i+1<k$) are different from $\pi_i(s)=\pi_{i+1}(s)$.
Then there exists some $\epsilon>0$, such that 
$$
\max_{s'\in[s-\epsilon, s+\epsilon]} \pi_{i-1}(s') + \epsilon < \min_{s'\in[s-\epsilon, s+\epsilon]} \pi_{i}(s'),$$
$$\min_{s'\in[s-\epsilon, s+\epsilon]} \pi_{i+2}(s') - \epsilon > \max_{s'\in[s-\epsilon, s+\epsilon]} \pi_{i+1}(s').
$$
Now for small enough $\de>0$, let $\bx^\de = (\pi_i(s - \de), \pi_{i+1}(s - \de))$ and let $\by^\de = (\pi_i(s + \de), \pi_{i+1}(s + \de))$. By Lemma \ref{L:disjoint-2-allpair}, we can find an optimizer $\pi^\de$ from $(\bx^\de, s - \de)$ to $(\by^\de, s + \de)$ with $\pi^\de_1(s) < \pi^\de_2(s)$; in particular, $\pi^\de \ne (\pi_i, \pi_{i+1})$.
By Lemma \ref{L:transfluc}, for small enough $\de > 0$, the optimizer $\pi^\de$ is disjoint from $\pi_{i-1}$ (if $i>1$) and $\pi_{i+2}$ (if $i+1<k$). Therefore letting $\tau^\de$ denote $\pi$ with $\pi^\de$ in place of $(\pi_i, \pi_{i+1})$ on the interval $[s- \de, s+ \de]$, Proposition \ref{P:sum-of-disjoint} ensures that $\|\tau^\de\|_{\scrL^*}  \ge \|\pi\|_{\scrL^*}$. 
Thus this new path is also an optimizer from $(\bx,r)$ to $(\by,t)$, contradicting the uniqueness assumption.

To upgrade this to hold for all endpoints simultaneously, we use the arguments in the proof of Lemma \ref{L:disjoint-2-allpair}, essentially verbatim.
\end{proof}

\begin{proof}[Proof of Theorems \ref{T:extended-landscape} and \ref{T:disjoint-optimizers-in-L} and Corollaries \ref{C:rsk} and \ref{C:disjointness}]
First, we can couple $\scrL^*, \scrL$ so that $\scrL^*|_{\Rd} = \scrL|_{\Rd}$. 
By Proposition \ref{P:sum-of-disjoint}, for any multi-path $\pi:[s, t] \to \R^k_\le$ with $\pi_i(r) < \pi_{i+1}(r)$ for all $r \in (s, t)$, we have
\begin{equation}
\label{E:pi-same}
\|\pi\|_{\scrL^*} = \sum_{i=1}^k \|\pi_i\|_{\scrL^*} = \sum_{i=1}^k \|\pi_i\|_\scrL.
\end{equation}
Moreover, almost surely for all $(\bx, s; \by, t) \in \fX_\uparrow$, Proposition \ref{P:disjointness} guarantees that
$
\scrL^*(\bx, s; \by, t) = \sup_\pi \|\pi\|_{\scrL^*},
$
where the supremum is over all multi-paths $\pi$ from $(\bx, s)$ to $(\by, t)$ that are disjoint away from the endpoints. Comparing this with Definition \ref{D:ext-land} gives that $\scrL^* = \scrL$, proving Theorem \ref{T:extended-landscape}. Theorem \ref{T:disjoint-optimizers-in-L} then follows from \eqref{E:pi-same}, Proposition \ref{P:disjointness}, and Lemma \ref{L:as-unique-geo-fixedend}. Corollary \ref{C:rsk} follows from \eqref{E:line-sheet-relation} and Theorem \ref{T:extended-landscape}. For Corollary \ref{C:disjointness}, \eqref{E:Lxsys} follows from the existence of disjoint geodesics by definition. The opposite direction uses Theorem \ref{T:disjoint-optimizers-in-L}. 
\end{proof}

\section{Convergence of optimizers}
\label{S:convergence-optimizers}

In this section we prove Theorem \ref{T:limit-theorem}, which shows that disjoint optimizers in Brownian LPP converge to disjoint optimizers in $\scrL$. The convergence for geodesics was shown in \cite{DOV}. The argument in \cite{DOV} is purely deterministic, relying only the metric composition law for $\scrL$ and a few basic regularity properties. We will adopt a similar strategy here. 
In this section, we will work in a coupling where the following conditions hold on some set $\Om$ of probability $1$.
\begin{enumerate}[label=(\roman*), nosep]
	\item $\cL_n \to \cL$ uniformly on compact subsets of $\fX_\uparrow$.
	\item For every bounded set $K = [-b, b]^4 \cap \R^4_\uparrow$, there exists some finite $C_b$ such that for all $\ep \in (0, 1)$ we have
	$$
	\limsup_{n \to \infty} \sup_{(x, s; y; t) \in K} \scrL_n(x, s; y; t) + \frac{(x -y)^2}{t-s + \ep} \le C_b.
	$$
	\item For any $\eta > 0$, there is a constant $R > 0$ such that
$$
\lf|\scrL(x, s; y, t) + \frac{(x - y)^2}{t-s} \rg|\le R (t-s)^{1/3} G(x, s; y, t)^\eta.
$$
Here the function $G$ is as in Lemma \ref{l:uniform-EL-bound}.
\end{enumerate}
The fact that such a coupling exists follows from Theorem \ref{T:BLPP-convergence} for the first statement, \cite[Lemma 13.3]{DOV} for the second statement, and Lemma \ref{l:uniform-EL-bound} for the third statement (or alternately, \cite[Corollary 10.7]{DOV}). We let $B^n = (B^n_i : i \in \Z)$ denote the collection of standard Brownian motions that give rise to $\scrL_n$ in this coupling. We work on $\Om$ for all statements and proofs in this section.

Most of this section is focused on proving Hausdorff convergence of rescaled zigzag graphs; we translate to the language of Theorem \ref{T:limit-theorem} at the end.
 For a path $\pi:[a, b] \to \Z$, recall from Section \ref{S:lpp} that its zigzag graph is
$$
\Ga(\pi) = \{(c, y) \in \R \X \Z : c \in [a, b], \pi(r) \le y \le \pi(r^-) \}.
$$
Note that we write $\pi(r^-)$ for the left-hand limit at $r$, and that $\pi(r^-)$ is always defined, see Section \ref{S:lpp}.
Also let $A_n$ be the linear transformation of $\R^2$ given by the matrix
\begin{equation}
A_n = \lf[ \begin{array}{cc}
n^{1/3}/2 & n^{-2/3}/2 \\
0 & -n^{-1}
\end{array}\rg].
\end{equation}
For any path $\pi$, its transformed zigzag graph $A_n \Ga(\pi)$ is contained in $\R \X n^{-1} \Z$. Moreover, the restriction $A_n|_{\R \X \Z}: \R \X \Z \to \R \X n^{-1} \Z$ is the inverse of the map $(x, s) \mapsto (x, s)_n$ used in the construction of $\scrL_n$ in Theorem \ref{T:BLPP-convergence}. Therefore for any path $\pi$ from $(a, m)$ to $(b, \ell)$ and any $n \in \N$, after tracing through the definitions we get that
\begin{equation}
\label{E:pi-length}
\begin{split}
 \|\pi\|_{\scrL_n} &= \inf \sum_{i=1}^k \scrL_n(p_{i-1}; p_i), \qquad \qquad \text{ where } \\  \|\pi\|_{\scrL_n} :&= \|\pi\|_{B^n} + 2\sqrt{n} (b-a) + n^{1/6}(A_n(b, \ell)_1 - A_n(a, m)_1).
\end{split}
\end{equation}
Here the infimum is over all finite sequences $p_0, \dots, p_k \sset A_n \Ga(\pi)$ such that
$$
a = (A_n^{-1} p_0)_1 < (A_n^{-1} p_1)_1 < \dots < (A_n^{-1} p_k)_1  = b.
$$
Here and in \eqref{E:pi-length}, $(A^{-1}_n p)_1$ denotes the first coordinate of $A_n^{-1} p$. We begin with a tightness statement for zigzag graphs.

\begin{lemma}
	\label{L:tight}
Let $\pi_n$ be a sequence of paths from $(a_n, m_n)$ to $(b_n, \ell_n)$ such that 
\begin{equation}
\label{E:An-1}
A_n(a_n, m_n) \to (x, r) \qquad \mathand \qquad A_n(b_n, \ell_n) \to (y, t)
\end{equation} 
as $n \to \infty$. Suppose also that
\begin{equation}
\label{E:liminfpin}
\liminf_{n \to \infty} \|\pi_n\|_{\scrL_n} > -\infty
\end{equation}
almost surely. Then on $\Om$, the sequence $A_n\Ga(\pi_n)$ is precompact in the Hausdorff metric. Moreover, any subsequential limit of $A_n\Ga(\pi_n)$ is equal to $\mathfrak{g} \pi = \{(\pi(s), s) : s \in [r, t]\}$ for some continuous function $\pi:[r, t] \to \R$ with $\pi(r) = x$ and $\pi(t) = y$.  
\end{lemma}

\begin{proof}
First, let 
$$
\Ga'_n(\pi_n) = \{x \sset \R^2 :  d(x, A_n\Ga(\pi)) \le n^{-2/3}\}.
$$
Here $d(x, A)$ denotes the Euclidean distance between a point and a set. The definitions of $A_n$ and $\Ga_n(\pi)$ ensure that the sets $\Ga'_n(\pi_n)$ are all connected. Moreover, the Hausdorff distance $d_H(\Ga'_n(\pi_n), \Ga_n(\pi_n))$ is at most $n^{-2/3}$, so it suffices to prove all statements in the lemma for $\Ga'_n(\pi_n)$. Next, fix an interval $[-b, b] \sset \R$. The definition of the scaling matrix $A_n$ and the limiting statements \eqref{E:An-1} guarantee that $\Ga'_n(\pi_n) \cap ([-b, b] \X \R)$ is precompact in the Hausdorff topology, with subsequential limits contained in $[-b, b] \X [r, t]$. 

Take $b$ large enough so that $x, y \in (-b, b)$. Connectedness of the sets $\Ga'_n(\pi_n)$ implies either there is a subsequential limit of $\Ga'_n(\pi_n) \cap ([-b, b] \X \R)$ that intersects the boundary $\{-b, b\} \X\R$, or else the sequence $\Ga_n'(\pi_n)$ is precompact, and all subsequential limits are contained in $(-b, b) \X \R$.

Suppose that some subsequential limit of $\Ga'_n(\pi_n) \cap [-b, b] \X \R$ intersects the boundary $\{-b, b\} \X\R$ at a point $p \in \R^2$. Then there exists a sequence of points $p_n \in A_n\Ga(\pi_n)$ that converge to $p$. By the triangle inequality for $\scrL_n$ and \eqref{E:pi-length} we have
$$
\|\pi\|_{\scrL_n} \le \scrL_n(A_n(a_n, m_n), p_n) + \scrL_n(p_n, A_n(b_n, \ell_n) ).
$$
If $p_n \to (z, s)$ for some $(z, s) \in \{-b, b\} \X \{r, t\}$, then the right-hand side above converges to $-\infty$ by condition (ii) above, contradicting \eqref{E:liminfpin}. 
If $p_n \to (z, s)$ for some $s \in (r, t)$ and $z = \pm b$, then uniform-on-compact convergence of $\scrL_n$ to $\scrL$ guarantees that the right-hand side above converges to 
$$
\scrL(x, r; z, s) + \scrL(z, s; y, t).
$$
For $b$ large enough, condition (iii) above guarantees that this quantity can become arbitrarily large and negative, contradicting \eqref{E:liminfpin}. Therefore the sequence $\Ga_n'(\pi_n)$ is precompact, and all subsequential limits are contained in $(-B, B) \X [r, t]$ for some random $B > 0$. Since all subsequential limits of $\Ga_n'(\pi_n)$ are connected and contain the points $(x, r)$ and $(y, t)$, to show that any subsequential limit $\Ga$ is of the form $\{(\pi(s), s) : s \in [r, t]\}$ for some continuous function $\pi:[r, t] \to \R$ with $\pi(r) = x$ and $\pi(t) = y$, we just need to show that $\Ga$ intersects each horizontal line at most once.

Suppose that this is not the case, and that $p=(z, s),p'=(z', s) \in \Ga$ for some $z \ne z'$. Then there are sequences $p_n \in A_n \Ga(\pi_n)$ and $p_n' \in A_n\Ga(\pi_n)$ converging to $p, p'$, respectively. 
Without loss of generality, we may assume that $(A_n^{-1} p_n)_1 < (A_n^{-1} p_n')_1$ infinitely often, so that by \eqref{E:pi-length} we have
$$
\|\pi_n\|_{\scrL_n} \le \scrL_n(A_n(a_n, m_n); p_n) + \scrL_n(p_n; p_n')+ \scrL_n(p_n'; A_n(b_n, \ell_n))
$$
for infinitely many $n$. Condition (ii) guarantees that almost surely, the middle term on the right-hand side above converges to $-\infty$, whereas the first and third terms are bounded above. Again, this contradicts \eqref{E:liminfpin}.
\end{proof}

\begin{theorem}
\label{T:path-cvge}
Fix $\bu = (\bx, s; \by, t) \in \fX_\uparrow$, and let $C_\bu$ be the almost sure set where there is a unique disjoint optimizer $\pi$ in $\scrL$ from $(\bx, s)$ to $(\by, t)$. Let $\pi^{(n)}$ be any sequence of $\scrL_n$-optimizers from $(\ba_n, m_n)$ to $(\bb_n, \ell_n)$ where 
$A_n (a_{n, i}, m_n) \to (x_i, s)$ and $A_n (b_{n, i}, \ell_n) \to (y_i, t)$.

Then on $\Om \cap  C_\bu$, $A_n \Ga(\pi^{(n)}_i) \to \mathfrak{g} \pi_i = \{(\pi_i(r), r) : r \in [s, t]\}$ in the Hausdorff metric for all $i$. Moreover, letting $h_{n, i}:[s, t] \to [a_{n, i}, b_{n, i}]$ be the linear function satisfying $h_{n,i}(s) = a_{n, i},  h_{n,i}(t) = b_{n, i}$, on $\Om \cap C_\bu$ we have the uniform convergence
$$
\tilde \pi^{(n)}_i := \frac{\pi^{(n)}_i \circ h_{n, i}+ n h_{n, i}}{2 n^{2/3}} \to \pi_i,
$$
as functions from $[s, t]$ to $\R$.
\end{theorem}

The `Moreover' in Theorem \ref{T:path-cvge} is Theorem \ref{T:limit-theorem}.
\begin{proof}[Proof of Theorem \ref{T:path-cvge}]
	In the proof, we work on the set $C_\bu \cap \Om$.
	Let $k$ be such that $\bx, \by \in \R^k_\le$. Since the $\pi^{(n)}$ are optimizers, and $\scrL_n \to \scrL$ uniformly on compact sets, we have
\begin{equation}
\label{E:ppinn}
\sum_{i=1}^k \|\pi_i^{(n)}\|_{\scrL_n} = \scrL_n(A_n (\ba_n, m_n), A_n (\bb_n, \ell_n)) \to  \scrL(\bx, s; \by, t).
\end{equation}
Also, for each $i$, we have
\begin{equation}
\label{E:pini}
\|\pi^{(n)}_i\|_{\scrL_n} \le \scrL_n(A_n (a_{n, i}, m_n), A_n (a_{n, i}, \ell_n)).
\end{equation}
The right-hand side above converges to $\scrL(x_i, s; y_i, t)$, so for all $i$, by \eqref{E:ppinn} and \eqref{E:pini}, we have
$$
\liminf_{n \to \infty} \|\pi^{(n)}_i\|_{\scrL_n} \ge  \scrL(\bx, s; \by, t) - \sum_{1 \le j \le k, j \ne i} \scrL(x_j, s; y_j, t) > -\infty.
$$
Hence by Lemma \ref{L:tight}, each of the sequences $\{A_n \Ga(\pi^{(n)}_i) : n \in \N\}$ is precompact, with subsequential limits that are of the form $(\mathfrak{g} \ga_1, \dots, \mathfrak{g} \ga_k)$ for some continuous multi-path $\ga$ from $(\bx, s)$ to $(\by, t)$. Now, let $P_n \sset \R^k_\le \X \Z$ be the set of all points $(\bz, j)$ such that 
\begin{equation}
\label{E:Mcprior}
B^n[(\ba_n, m_n) \to (\bb_n, \ell_n)] = B^n[(\bb_n, m_n) \to (\bz, j)] + B^n[(\bz, j-1) \to (\bb_n, \ell_n)],
\end{equation}
and such that $(z_i, j), (z_i, j-1) \in \Ga(\pi^{(n)}_i)$ for all $i$.
Metric composition (Lemma \ref{L:split-path}) and the fact that $\pi^{(n)}$ is an optimizer guarantees that for every $j \in \{\ell_n + 1, \dots, m_n\}$, there exists $(\bz, j) \in P_n$. In particular, this implies that along a subsequence where $A_n \Ga(\pi^{(n)}_i) \to \mathfrak{g} \ga_i$ for all $i$, we have
$$
A_n P_n \to \mathfrak{g} \ga = \{(\ga(r), r) : r \in [s, t]\} 
$$
and so \eqref{E:Mcprior} passes to the limit to give that
$$
\scrL(\bx, s; \by, t) = \scrL(\bx, s; \ga(r), r) +  \scrL( \ga(r), r; \by, t)
$$
for all $r \in (s, t)$.
This can only occur if $\ga$ is the unique optimizer in $\scrL$ from $(\bx, s)$ to $(\by, t)$, yielding the first part of the theorem.

For the `Moreover', it is enough to show that $\mathfrak{g} \tilde \pi^{(n)}_i \to \mathfrak{g} \pi_i$ for all $i$ in the Hausdorff metric, since Hausdorff convergence of graphs implies uniform convergence of functions when the limit is continuous. For this, by the first part of the theorem we just need to show that the Hausdorff distance $d_H(\mathfrak{g} \tilde \pi_i^{(n)}, A_n \Ga(\pi_i^{(n)}))$ converges to $0$ with $n$.

Since $A_n (a_{n, i}, m_n) \to (x_i, s)$ and $A_n (b_{n, i}, \ell_n) \to (y_i, t)$ we have $a_{n,i}\to s$ and $b_{n,i}\to t$.
Then the function $h_{n, i}$ converges to the identity, so $d_H(\mathfrak{g} \tilde \pi_i^{(n)}, \mathfrak{g} \hat \pi^{(n)}_i) \to 0$, where
$$
\hat \pi^{(n)}_i(x) = \frac{\pi^n_i(x) + nx}{2n^{2/3}}.
$$
Moreover, letting $\Lambda \pi^{(n)}_i = \{(c, \pi^{(n)}_i(c)) : c \in [a, b]\}$ denote the graph of $\pi^{(n)}_i$, the first part of the theorem guarantees that $d_H(A_n \Lambda \pi_i^{(n)}, A_n\Ga \pi_i^{(n)}) \to 0$. Therefore it suffices to show that $d_H(A_n \Lambda \pi_i^{(n)}, \mathfrak{g} \hat \pi_i^{(n)}) \to 0$ with $n$. This boils down to a matrix computation. We have $\mathfrak{g} \hat \pi_i^{(n)} = D_n \mathfrak{g} \pi_i^{(n)}$ and $ \mathfrak{g} \pi_i^{(n)} = R \Lambda \pi^{(n)}_i$, where
\begin{equation}
D_n = \lf[ \begin{array}{cc}
n^{-2/3}/2 & n^{1/3}/2 \\
0 & 1
\end{array}\rg], \qquad R = \lf[ \begin{array}{cc}
0 & 1 \\
1 & 0
\end{array}\rg].
\end{equation}
Therefore $D_n R A^{-1}_n (A_n \Lambda \pi^{(n)}_i) = \mathfrak{g} \hat \pi_i^{(n)}$. A quick computation shows that $D_n R A_n^{-1} \to I$, yielding the result.
\end{proof}

\bibliographystyle{alpha}
\bibliography{bibliography}

\appendix
\section{Brownian melon estimates}
\label{app:b-m-est}
In this appendix we prove Lemma \ref{l:maxi-loc} and \ref{l:change-time-prelim}, using some Brownian melon estimates from the literature.
We start by quoting these results.
\begin{theorem}[\protect{\cite[Theorem 3.1]{DV}}]
\label{T:top-bd}
There exist positive constants $c_k, d_k, k \in \N$ such that the following holds. For all $m \in (0, 5n^{2/3})$ and $n \ge 1$ we have
\[
\prob(W^n_k(1) - 2 \sqrt{n} \ge m n^{-1/6} ) \le c_1 e^{-d_1 m^{3/2}}, \]
\[
\prob(W^n_k(1) - 2 \sqrt{n} \le - m n^{-1/6} ) \le c_k e^{-d_k m^3}.
\]
Also, for all $m \ge 5n^{2/3}$ and $n \ge 1$ we have
\[
\p(|W^n_k(1) - 2 \sqrt{n}| \ge m n^{-1/6}) \le c_1 e^{-d_1 n^{-1/3} m^2}.
\]
\end{theorem}

For any $n\in \N$, $x,a,b,w>0$, denote 
\[
\cN_{b,w}(n,x,a)=
2 \sqrt{nx} + \sqrt{x} n^{-1/6}(a + b\log^{2/3}(n^{1/3} |\log(x/w)| + 1)).
\]
Note that for any $\alpha>0$, we have
\begin{equation} \label{eq:rescale-cN}
\cN_{b,\alpha w}(n,\alpha x,a) = \sqrt{\alpha}\cN_{b,w}(n,x,a).
\end{equation}
\begin{prop}[\protect{\cite[Proposition 4.3]{DV}}]
\label{P:cross-prob}
There exist positive constants $b, c$ and $d$ such that for all $w,a > 0$ and $n \ge 1$, the probability that
\[
W^n_1(x) \le \cN_{b,w}(n,x,a), \;\; \forall x \in (0, \infty)
\]
is greater than or equal to $1- c e^{-d a^{3/2}}$.
\end{prop}
The following estimate is also necessary. This estimate is simply a deterministic inequality and does not involve any probabilistic objects.
\begin{lemma}[\protect{\cite[Lemma 9.4]{DOV}}]
	\label{L:calculation}
	Let $b > 0$ be a fixed constant. Then there exists a constant $c$ such that for all $n \in \N, t \in \{1/n, 2/n, \dots, (n-1)/n\}, a>1$ and
	$$
	z \in [0, t - c(t\wedge(1- t))^{1/3}a^2n^{-1/3}] \cup [t + c(t\wedge(1- t))^{1/3}a^2n^{-1/3}, 1],
	$$
	we have that
	\[
	\cN_{b,t}(nt,z,a) + \cN_{b,1-t}(n(1-t),1-z,a)\le 2\sqrt{n} - an^{-1/6}.
	\]
\end{lemma}

We first use Theorem \ref{T:top-bd} to deduce an estimate on last passage values across Brownian motions. To clean up the notation in this lemma, its proof, and in the subsequent proof of Lemma \ref{l:maxi-loc}, for a vector $\bx$, we let
$$
\hat \bx = (2 n^{-1/3} \bx, n) \qquad \mathand \qquad \tilde \bx = (1 + 2 n^{-1/3} \bx, 1).
$$
The dependence on $n$ in the notation is implicit.
\begin{lemma}  \label{l:bd-brownain-passtime}
Take independent standard Brownian motions $B^n=(B^n_1,\ldots,B^n_n)$, and $\bx, \by \in \R^k_\le$ such that $\|\bx\|_2, \|\by\|_2 < n^{1/6}$.
For any $a>0$ we have
\[
\p\left(\lf| B^n[\hat \bx \to \tilde \by]-\sum_{i=1}^k 
2\sqrt{n(1+2n^{-1/3}(y_i-x_i))}
\rg| > an^{-1/6}\right) < ce^{-da^{3/2}},  
\]
where $c,d$ are constants depending only on $k$.
\end{lemma}

\begin{proof}
By Lemma \ref{L:f-naive} we have
\[
\sum_{i=1}^k(B^n[\hat x_i^k \to \tilde y_i^k] - B^n[\hat x_i^{k-1}\to \tilde y_i^{k-1}])\le B^n[\hat \bx \to \tilde \by] \le \sum_{i=1}^k B^n[\hat x_i \to \tilde y_i].
\]
If $W^n$ is the $n$ dimensional Brownian melon, then by Theorem \ref{T:melon-lpp},  
\begin{align*}
B^n[\hat x_i \to \tilde y_i] &\eqd W_1^n(1+2n^{-1/3}(y_i-x_i)) \quad \mathand \quad \\
B^n[\hat x_i^k \to \tilde y_i^k] - B^n[\hat x_i^{k-1}\to \tilde y_i^{k-1}] &\eqd W_k^n(1+2n^{-1/3}(y_i-x_i)).
\end{align*}
Then the conclusion follows from Theorem \ref{T:top-bd}, using that $\|\bx\|_2, \|\by\|_2 < n^{1/6}$, and scale invariance of the Brownian melon: $\sqrt{\al} W^n(\cdot) \eqd W^n(\al \; \cdot)$ for any $\al > 0$.
\end{proof}

\begin{proof}[Proof of Lemma \ref{l:maxi-loc}]
Throughout this proof we let $c, d$ denote constants depending on $k$, whose values may change from line to line. We also assume that $n$ is large enough, since otherwise the conclusion follows by taking $c$ large and $d$ small.

By Lemma \ref{l:bd-brownain-passtime}, we have
\begin{equation} \label{eq:maxi-loc-pf}
\p\left(\max_{\bz\in\R^k_\le} A(\bz)<\sum_{i=1}^k 
2\sqrt{n(1+2n^{-1/3}(y_i-x_i))}
- an^{-1/6}\right) < ce^{-da^{3/2}}.
\end{equation}
For each $\bz\in\R^k_\le$ for which $A(\bz)$ is not equal to $-\infty$ (i.e. when it is defined by \eqref{E:Abzdef}), by Lemma \ref{L:f-naive} we also have
\begin{equation}
\label{E:Abzz}
A(\bz) \le 
\sum_{i=1}^k\big(B^n[\hat x_i \to (t+2n^{-1/3}z_i, q+1)] + B^n[(t+2n^{-1/3}z_i, q)\to \tilde y_i]\big).
\end{equation}
By Proposition \ref{P:cross-prob}, with probability at least $1-ce^{-da^{3/2}}$, the $i$th summand on the right-hand side of \eqref{E:Abzz} is bounded above by 
\[
\cN_{b,t(1+2n^{-1/3}(y_i-x_i))}(p,t+2n^{-1/3}(z_i-x_i),a)
+
\cN_{b,(1-t)(1+2n^{-1/3}(y_i-x_i))}(q,1-t+2n^{-1/3}(y_i-z_i),a),
\]
where $b$ is a universal constant.
By \eqref{eq:rescale-cN} this equals
\begin{multline*}
\sqrt{1+2n^{-1/3}(y_i-x_i)}
\\
\X\lf(\cN_{b,t}\lf(nt,\frac{t+2n^{-1/3}(z_i-x_i)}{1+2n^{-1/3}(y_i-x_i)},a\rg)
+
\cN_{b,(1-t)}\lf(n(1-t),\frac{1-t+2n^{-1/3}(y_i-z_i)}{1+2n^{-1/3}(y_i-x_i)},a\rg)\rg).
\end{multline*}
Recall that we require $\|\bx\|_2,\|\by\|_2<n^{1/6}$.
By Lemma \ref{L:calculation}, for any $a>1$ the above can be bounded by $\sqrt{1+2n^{-1/3}(y_i-x_i)}(2\sqrt{n}-an^{-1/6})$,
when $|z_i-ty_i-(1-t)x_i|>ca^2(t\wedge(1-t))^{1/3}$.
Thus we conclude that, for any $a>0$, with probability at least $1-ce^{-da^{3/2}}$ we have
\[
A(\bz) < \sum_{i=1}^k
2\sqrt{n(1+n^{-1/3}(y_i-x_i))} - an^{-1/6}
\]
for any $\bz$ with $\|\bz-t\by-(1-t)\bx\|_2>ca^2(t\wedge(1-t))^{1/3}$.
This with \eqref{eq:maxi-loc-pf} finishes the proof.
\end{proof}

Now we complete proving Lemma \ref{l:change-time-prelim}, following the outline in Section \ref{S:tightness-landscape}. 
\begin{proof}[Proof of Lemma \ref{l:change-time-prelim}]
In this proof we let $c, d$ denote large and small constants depending on $k$, whose values may change from line to line. 

We first upper bound $\scrK_n(\bx,0;\by',t) - \scrK_n(\bx,0;\by,1)$.
By the triangle inequality we have
$$
\scrL_n(\bx,0;\by',t) - \scrL_n(\bx,0;\by,1)
\le -\scrL_n(\by',t+n^{-1};\by,1).
$$
Thus we have
\[
\begin{split}
&\p(\scrK_n(\bx,0;\by',t) - \scrK_n(\bx,0;\by,1)> a(1-t)^{1/3})\\
\le &\;
\p(\scrK_n(\by',t+n^{-1};\by,1) < -a(1-t)^{1/3}) \le ce^{-da^{3/2}}.
\end{split}
\]
Here the last inequality follows by applying Lemma \ref{l:bd-brownain-passtime} to $(1-t)n-1$ Brownian motions, and elementary calculations.
We next lower bound $\scrK_n(\bx,0;\by',t) - \scrK_n(\bx,0;\by,1)$. For any $\bz\in \R^k_\le$ we denote $A(\bz)=(\scrL_n(\bx,0;\bz,t) - \scrL_n(\bx,0;\by',t)) + \scrL_n(\bz,t;\by,1)$.
It remains to bound the probability of this event
$$
\sup_{\bz\in\R^k_\le}A(\bz) >  -\|\by-\bx\|_2^2(1-t) + a(1-t)^{1/3}|\log(1-t)|.
$$

To bound $A(\bz)$, we collect some estimates on $\scrL_n(\bz,t;\by,1)$ and $\scrL_n(\bx,0;\bz,t)-\scrL_n(\bx,0;\by',t)$.
For this, take any $1<\hat{a}<n^{1/100}$.

\noindent\textbf{Estimate 1.}
By Lemma \ref{L:f-naive} we have
$
\scrL_n(\bz,t;\by,1) \le \sum_{i=1}^k
\scrL_n(z_i,t;y_i,1)
$.
By Proposition \ref{P:cross-prob} and using the notation there, for some constant $b>0$, with probability $>1-ce^{-d\hat{a}^{3/2}}$ we have
\begin{multline}  \label{eq:timed-pre-pf1}
\scrL_n(\bz,t;\by,1) < \sum_{i=1}^k \Big[
n^{1/6}\cN_{b,1-t+2(y_i-y_i')n^{-1/3}}((1-t)n,1-t+2(y_i-z_i)n^{-1/3},\hat{a})\\-2(1-t)n^{2/3}-2n^{1/3}(y_i-z_i)\Big],
\end{multline}
for any $\bz\in\R^k_\le$ such that $1-t+2(y_i-z_i)n^{-1/3} > 0$ for each $i$.
We now give a more explicit bound for the $i^{\mathrm{th}}$ summand in the right-hand side of \eqref{eq:timed-pre-pf1}, when $|z_i-y_i'|<cn^{1/20}$.
Note that $\|\bx\|_2,\|\by\|_2<n^{1/100}$, $1-t>n^{-1/100}$, so in this case we would have $(y_i-z_i)n^{-1/3}, (y_i-y_i')n^{-1/3}<d(1-t)$.
Further, recall that
\[
\begin{split}
&\cN_{b,1-t+2(y_i-y_i')n^{-1/3}}((1-t)n,1-t+2(y_i-z_i)n^{-1/3},\hat{a})
\\
=& \; 2\sqrt{(1-t)n(1-t+2(y_i-z_i)n^{-1/3})}
+ \sqrt{1-t+2(y_i-z_i)n^{-1/3}}
\\
&\X
((1-t)n)^{-1/6}   \left(\hat{a}+b\log^{2/3}\left(((1-t)n)^{1/3}\left|\log\left(\frac{1-t+2(y_i-z_i)n^{-1/3}}{1-t+2(y_i-y_i')n^{-1/3}}\right)\right| + 1\right)\right).
\end{split}    
\]
By Taylor expansion of $y_i-z_i$, we can bound the first term in the right-hand side by
\[
2(1-t)\sqrt{n}+2n^{1/6}(y_i-z_i)
-n^{-1/6}\frac{(y_i-z_i)^2}{1-t} 
+ cn^{-1/2}\frac{(y_i-z_i)^3}{(1-t)^2}.
\]
For the second term, we use that $\sqrt{1-t+2(y_i-z_i)n^{-1/3}} < c\sqrt{1-t}$, and that $\left|\log\left(\frac{1-t+2(y_i-z_i)n^{-1/3}}{1-t+2(y_i-y_i')n^{-1/3}}\right)\right|<cn^{-1/3}\frac{|y_i'-z_i|}{1-t}$, to bound it by
\[
cn^{-1/6}\hat{a}(1-t)^{1/3}+cn^{-1/6}\frac{|y_i'-z_i|}{(1-t)^{1/3}}.
\]
Thus when $|z_i-y_i'|<cn^{1/20}$ we can bound the $i^{\mathrm{th}}$ summand in the right-hand side of \eqref{eq:timed-pre-pf1} by
\begin{equation}  \label{eq:timed-pre-pf11}
-\frac{(y_i-z_i)^2}{1-t}
+
c\hat{a}(1-t)^{1/3}+c\frac{|y_i'-z_i|}{(1-t)^{1/3}}.
\end{equation}

\noindent\textbf{Estimate 2.} 
For $\scrL_n(\bx,0;\bz,t)-\scrL_n(\bx,0;\by',t)$ we give two different bounds.
The first of these bounds $\scrL_n(\bx,0;\bz,t)$ and $\scrL_n(\bx,0;\by',t)$ separately.

By Lemma \ref{L:f-naive} we have
$\scrL_n(\bx,0;\bz,t) \le \sum_{i=1}^k
\scrL_n(x_i,0;z_i,t)$.
By Proposition \ref{P:cross-prob} (for $\scrL_n(\bx,0;\bz,t)$) and Lemma \ref{l:bd-brownain-passtime} (for $\scrL_n(\bx,0;\by',t)$), with probability $>1-ce^{-d\hat{a}^{3/2}}$ we have
\begin{multline}   \label{eq:timed-pre-pf2}
\scrL_n(\bx,0;\bz,t)-\scrL_n(\bx,0;\by',t)
< \sum_{i=1}^k
\Big[ n^{1/6}\cN_{b,t+2(y_i'-x_i)n^{-1/3}}(tn,t+2(z_i-x_i)n^{-1/3},\hat{a})\\
 - 2n^{2/3}\sqrt{t(t+2(y_i'-x_i)n^{-1/3})}
 +\hat{a}\sqrt{t+2(y_i'-x_i)n^{-1/3}}t^{-1/6}
-2n^{1/3}(z_i-y_i') \Big],
\end{multline}
for any $\bz\in\R^k_\le$ such that $t+2(z_i-x_i)n^{-1/3} > 0$ for each $i$.
Similar to Estimate 1, when $|z_i-y_i'|<cn^{1/20}$, we can bound the $i^{\mathrm{th}}$ summand in the right-hand side of \eqref{eq:timed-pre-pf2} by
\begin{equation}  \label{eq:timed-pre-pf21}
-\frac{(z_i-x_i)^2}{t} +
c\hat{a}{t}^{1/3}+c\frac{|y_i'-z_i|}{{t}^{1/3}} +\frac{({y_i}'-x_i)^2}{t}.
\end{equation}

\noindent\textbf{Estimate 3.} 
The second bound for $\scrL_n(\bx,0;\bz,t)-\scrL_n(\bx,0;\by',t)$ is from the continuity of the prelimiting extended Airy sheet (Lemma \ref{l:change-spatial-prelim}). It is more refined when $\|\bz-\by'\|_2$ is small.

By Lemma \ref{l:change-spatial-prelim} and using Lemma \ref{L:levy-est}, we also have that with probability $>1-ce^{-d\hat{a}^{3/2}}$,
\begin{multline}   \label{eq:timed-pre-pf3}
\scrL_n(\bx,0;\bz,t) - \scrL_n(\bx,0;\by',t) < 
\sum_{i=1}^k \Big[ \hat{a}\log^{2/3}(2|z_i-y_i'|^{-1})\sqrt{|z_i-y_i'|} \Big] \\
-  \frac{\|\bz-\bx\|_2^2-\|\by'-\bx\|_2^2}{t}, 
\end{multline}
for any $\bz\in\R^k_\le$ with $\|\by'-\bz\|_2<1$.

Below we shall bound $A(\bz)$ assuming that the above three estimates \eqref{eq:timed-pre-pf1}, \eqref{eq:timed-pre-pf2}, \eqref{eq:timed-pre-pf3} hold.

\noindent\textbf{Upper bound $A(\bz)$ for $\|\by'-\bz\|_2<1$.}
In this case, by \eqref{eq:timed-pre-pf11} and \eqref{eq:timed-pre-pf3}
we have
\[
\begin{split}
A(\bz) <& \;\sum_{i=1}^k \Big[ -\frac{(y_i-z_i)^2}{1-t} +c\hat{a}(1-t)^{1/3}+c\frac{|y_i'-z_i|}{(1-t)^{1/3}}
\\
&+
\hat{a}\log^{2/3}(2|z_i-y_i'|^{-1})\sqrt{|z_i-y_i'|}
- \frac{(z_i-x_i)^2-(y_i'-x_i)^2}{t} \Big] \\
=& \;
-(1-t)\|\by-\bx\|_2^2 + \sum_{i=1}^k \Big[ -\frac{(y_i'-z_i)^2}{t(1-t)}
+\hat{a}\log^{2/3}(2|z_i-y_i'|^{-1})\sqrt{|z_i-y_i'|} \\ &+
c\hat{a}(1-t)^{1/3}+c\frac{|y_i'-z_i|}{(1-t)^{1/3}} \Big]
\\
< & \; -(1-t)\|\by-\bx\|_2^2 + c\hat{a}^{4/3}(1-t)^{1/3}|\log(1-t)|,
\end{split}
\]
where the last inequality uses that
\[
\frac{(y_i'-z_i)^2}{2t(1-t)} + c^2 (1-t)^{1/3} > c\frac{|y_i'-z_i|}{(1-t)^{1/3}},\]
\[
\frac{(y_i'-z_i)^2}{2t(1-t)} + c\hat{a}^{4/3}(1-t)^{1/3}|\log(1-t)| > \hat{a}\log^{2/3}(2|z_i-y_i'|^{-1})\sqrt{|z_i-y_i'|}.
\]

\noindent\textbf{Upper bound $A(\bz)$ for $\|\by'-\bz\|_2\ge 1$.}
In this case we use \eqref{eq:timed-pre-pf1} and \eqref{eq:timed-pre-pf2}.
Letting $A_i$ be the sum of the $i^{\mathrm{th}}$ term in the right-hand side of \eqref{eq:timed-pre-pf1} and \eqref{eq:timed-pre-pf2}, we have $A(\bz)\le \sum_{i=1}^k A_i$.
By \eqref{eq:rescale-cN} and Lemma \ref{L:calculation}, for each $1\le i \le k$ such that $|z_i-y_i'|>c\lf(\frac{\hat{a}}{1-t}\rg)^2(1-t)^{1/3}$ we have
\[
\begin{split}
&\cN_{b,1-t+2(y_i-y_i')n^{-1/3}}((1-t)n,1-t+2(y_i-z_i)n^{-1/3},\hat{a})\\
&+
\cN_{b,t+2(y_i'-x_i)n^{-1/3}}(tn,t+2(z_i-x_i)n^{-1/3},\hat{a}) \\
\le &\;  \sqrt{1+2(y_i-x_i)n^{-1/3}}
\Bigg( \cN_{b,1-t}\lf((1-t)n,\frac{1-t+2(y_i-z_i)n^{-1/3}}{1+2(y_i-x_i)n^{-1/3}},\frac{\hat{a}}{1-t}\rg)
\\ &
+
\cN_{b,t}\lf(tn,\frac{t+2(z_i-x_i)n^{-1/3}}{1+2(y_i-x_i)n^{-1/3}},\frac{\hat{a}}{1-t}\rg)\Bigg)
\\
\le &\; \sqrt{1+2(y_i-x_i)n^{-1/3}} \lf(2\sqrt{n}-\lf(\frac{\hat{a}}{1-t}\rg)n^{-1/6}\rg).
\end{split}
\]
So using that $\|\bx\|,\|\by\|_2<n^{1/100}$, $1-t>n^{-1/100}$, we have
\[
\begin{split}
A_i\le & \; 2n^{2/3}\sqrt{1+2(y_i-x_i)n^{-1/3}} - \frac{\hat{a}}{1-t} \sqrt{1+2(y_i-x_i)n^{-1/3}}
-2n^{1/3}(y_i-y_i')
\\
& -2(1-t)n^{2/3}
 - 2n^{2/3}\sqrt{t(t+2(y_i'-x_i)n^{-1/3})}
+\hat{a}\sqrt{t+2(y_i'-x_i)n^{-1/3}} t^{-1/6}
\\
\le & \;
 2(1-t)n^{2/3}\sqrt{1+2(y_i-x_i)n^{-1/3}} - 2(1-t)n^{1/3}(y_i-x_i) -2(1-t)n^{2/3} - \frac{d\hat{a}}{1-t}
\\
<& \; -(1-t)(y_i-x_i)^2 + c(1-t)^{1/3} - \frac{d\hat{a}}{1-t}
\end{split}
\]
When $|z_i-y_i'|\le c\lf(\frac{\hat{a}}{1-t}\rg)^2(1-t)^{1/3} < cn^{1/20}$, using \eqref{eq:timed-pre-pf11} and \eqref{eq:timed-pre-pf21} we have
\[
\begin{split}
A_i\le & \;
-\frac{(y_i-z_i)^2}{1-t}
+
c\hat{a}(1-t)^{1/3}+c\frac{|y_i'-z_i|}{(1-t)^{1/3}}
-\frac{(z_i-x_i)^2}{t} \\
&+
c\hat{a}{t}^{1/3}+c\frac{|y_i'-z_i|}{{t}^{1/3}} +\frac{({y_i}'-x_i)^2}{t}
\\
<& \;
-\frac{(y_i'-z_i)^2}{t(1-t)} - (1-t)(y_i-x_i)^2 + c\hat{a} + \frac{c |y_i'-z_i|}{(1-t)^{1/3}}
\\
<&  \; -\frac{(y_i'-z_i)^2}{2(1-t)} - (1-t)(y_i-x_i)^2 + c\hat{a} + c(1-t)^{1/3}.
\end{split}
\]
Thus by the above two inequalities, and using that $\|\by'-\bz\|_2\ge 1$, we get
\[
A(\bz) < c\hat{a} + c(1-t)^{1/3} - \frac{d}{1-t} -(1-t)\|\by-\bx\|_2^2 < c\hat{a}^{4/3}(1-t)^{1/3}-(1-t)\|\by-\bx\|_2^2.
\]
Finally, from these bounds on $A(\bz)$ in each case, we conclude that
\[
\begin{split}
& \p(\scrK_n(\bx,0;\by,1) - \scrK_n(\bx,0;\by',t)> a(1-t)^{1/3}|\log(1-t)|)
\\
<& \; \p(\sup_{\bz\in\R^k_\le}A(\bz) >  -\|\by-\bx\|_2^2(1-t) + a(1-t)^{1/3}|\log(1-t)|)
< ce^{-da^{9/8}}.
\end{split}
\]
The conclusion follows.
\end{proof}

\section{Proof of Theorem \ref{T:radon-n-deri}}  \label{app:radon-n-deri}

In this appendix, we extend the main result of Calvert, Hegde, and Hammond \cite{CHH20} to prove Theorem \ref{T:radon-n-deri}.
For brevity, we don't give full context for the paper \cite{CHH20} here and refer the interested reader to that paper. The paper \cite{hammond2016brownian} may also be a useful reference, as the work \cite{CHH20} builds on results from that paper. We strive to use the same notation as \cite{CHH20} so the interested reader can refer back easily. The main exception to this is that we use the notation $\tilde \scrB_i = 2^{-1/2} \scrB_i$ for lines in the (rescaled) parabolic Airy line ensemble.  In \cite{CHH20}, the authors use the notation $\scrL(i, \cdot)$ for these lines, which conflicts with our notation for the directed landscape. The factor of $2^{-1/2}$ is introduced in \cite{CHH20} so that comparison statements can be made with Brownian motions with diffusion parameter $1$, rather than $2$.

Throughout this section, we let $b > 0$ be a large constant and $b' > 0$ be a small constant, whose values may change from line to line but do not depend on any parameters. Other constants will retain the definitions used in \cite{CHH20}.

First, fix an interval $[-d, d]$ with $d \ge 1$ and a collection of line indices $\II{1, k}$.  For universal positive constants $c, C$, as in \cite{CHH20} we define
\begin{align*}
c_k &= ((3-2^{3/2})^{3/2} 2^{-1} 5^{-3/2})^{k-1} (2^{-5/2} c \wedge 1/8), \\
C_k &= \max \lf \{ 10 \cdot 20^{k-1}5^{k/2} \lf(\frac{10}{3 - 2^{3/2}}\rg)^{k(k-1)/2} C, e^{c/2} \rg\},  \\
D_k &= \max \lf\{ k^{1/3} c_k^{-1/3} (2^{-9/2} - 2^{-5})^{-1/3}, 36 (k^2 - 1), 2 \rg\}.
\end{align*}
The precise values of these constants are not important for our purposes here, but we will record the bounds
\begin{align}
\label{E:Dk-bound}
b' \le D_k \le e^{b k},\quad b' \le C_k \le e^{b k^2}.
\end{align}
Next, let $\ep > 0$ satisfy the $(k, d)$-dependent upper bound
$$
\ep < e^{-1} \wedge (17)^{-1/k} C_k^{-1/k} D_k^{-1} \wedge \exp(-(24)^6 d^6/ D_k^3).
$$
This simplifies to 
\begin{equation}
\label{E:ebbb}
\ep < e^{- b d^6 - b k}.
\end{equation}
Finally, we set $T = D_k (\log \ep^{-1})^{1/3}$. 

Now, for a function $f:[a, b] \to \R$, define its \textbf{bridge version} $f^{[a, b]} = f - L$, where $L$ is the linear function satisfying $L(a) = f(a)$ and $L(b) = f(b)$.
Next, with all parameters $d, k, \ep, T$ fixed as above, let $\scrF_{k}$ be the $\sig$-algebra generated by
\begin{itemize}[nosep]
	\item all the lower curves $\tilde \scrB_i:\R \to \R, i \ge k + 1$,
	\item the top $k$ curves $\tilde \scrB_i$ restricted to the set $\{x \in \R: |x| \ge 2T\}$,
	\item certain $\sig(\tilde \scrB_{k+1})$-measurable random variables $\mathfrak{l} \le \mathfrak{r} \in [-T, T]$, and
	\item the $2k$ bridges $\tilde \scrB_i^{[-2T, \mathfrak{l}]}, i = 1, \dots, k$ and $\tilde \scrB_i^{[\mathfrak{r}, 2T]}, i = 1, \dots, k$.
\end{itemize}
Here we use the notation $\sig(X)$ for the $\sig$-algebra generated by $X$.
We let $\p_{\scrF_k}(\cdot) = \p(\cdot \;|\; {\scrF_k})$ be the conditional law given ${\scrF_k}$. The precise nature of the random variables $\mathfrak{l}$ and $\mathfrak{r}$ is not important for us here, only their potential ranges and that they are functions of the $(k+1)^{\mathrm{st}}$ curve $\tilde \scrB_{k+1}$. For precise definitions, see the beginning of Section 4.1.5 in \cite{CHH20}.

With parameters $d, k, \ep, T$ fixed as above, in \cite{CHH20} and previously in \cite{hammond2016brownian}, the authors define a collection of random functions $J = \{J_i:[-2T, 2T] \to \R, i \in \II{1, k}\}$ known as the \textbf{jump ensemble}. First, for any sequence of functions $X=\{X_i:[-2T, 2T] \to \R, i \in \II{1, k}\}$, we can define a \textbf{resampled ensemble} $\tilde \scrB^{\operatorname{re}, X} = \{\tilde \scrB^{\operatorname{re}, X}_i: \R \to \R, i \in \N\}$. For this definition we let $L(x, a; b, y)$ denote the affine function with $L(x) = a, L(y) = b$. 
\begin{equation}
\label{E:resampled}
\tilde \scrB_i^{\operatorname{re}, X}(x) = \begin{cases}
\tilde \scrB_i(x), \quad & (i, x) \notin \II{1, k} \X [-2T, 2T]\\
\tilde \scrB_i^{[-2T, \mathfrak{l}]}(x) + L(-2T, \tilde \scrB_i(-2T); \mathfrak{l}, X_i(\mathfrak{l}))(x) , \qquad &x \in [-2T, \mathfrak{l}], i \le k\\
\tilde \scrB_i^{[\mathfrak{r}, 2T]}(x) + L(\mathfrak{r}, X_i(\mathfrak{r}); 2T, \tilde \scrB_i(2T))(x), \qquad & x \in [\mathfrak{r}, 2T], i \le k\\
X_i(x), & x\in [\mathfrak{l}, \mathfrak{r}], i \le k.
\end{cases}
\end{equation}
Note that in \cite{CHH20}, the same object is only defined for the top $k$ lines.
Next, in \cite{CHH20}, the authors define an ${\scrF_k}$-measurable finite set $P \sset [\mathfrak{l}, \mathfrak{r}]$ called a \textbf{pole set}, see the discussion in \cite[Section 4.1.5]{CHH20}. The precise nature of this set is not important for us. Let $B = \{B_i:[-2T, 2T] \to \R, i \in \II{1, k}\}$ be a collection of  Brownian bridges with $B_i(\pm 2T) = 0$ that are independent of $\tilde \scrB$ and each other. Finally, we define the ensemble $J$ in the following way.
\begin{itemize}
	\item First let $J' = \{J'_i:[-2T, 2T] \to \R, i \in \II{1, k}\}$ be given by connecting up the points $\tilde \scrB_i(\pm 2T)$ with the Brownian bridges $B_i$. That is, for all $i \in \II{1, k}$, 
	$$
	J'_i = B_i + L(-2T, \tilde \scrB_i(-2T); 2T, \tilde \scrB_i(-2T)).
	$$
	\item Next, let $J$ be given by the ensemble $J'$, conditionally on the events
	\begin{align*}
	\tilde \scrB_1^{\operatorname{re}, J'}(x) > \tilde \scrB_2^{\operatorname{re}, J'}(x) > \dots > \tilde \scrB_{k+1}^{\operatorname{re}, J'}(x) \quad &\mathfor x = [-2T,\mathfrak{l}]\cup[\mathfrak{r},2T], \quad \mathand \\
	\tilde \scrB_i^{\operatorname{re}, J'}(x) \ge \tilde \scrB_{k+1}^{\operatorname{re}, J'}(x), \qquad &\mathforall i \in \II{1, k}, x \in P.
	\end{align*}
\end{itemize}
This is the same as the definition given at the beginning of Section 4.1.6 in \cite{CHH20}. 
Next, let $\Pass(J)$  be the indicator of event where
$$
\tilde \scrB_i^{\operatorname{re}, J}(x) > \tilde \scrB_{i+1}^{\operatorname{re}, J}(x) \qquad \text{ for all } x \in [-2T, 2T], i \in \II{1, k}.
$$ 
The relevance of the jump ensemble $J$ lies in the following lemma.. 

\begin{lemma}
	\label{L:-2T2T}
	We have
	$$
	\p_{\scrF_k}(\tilde \scrB_i^{\operatorname{re}, J}|_{\II{1, k} \X [-2T, 2T]} \in \cdot \;| \; \Pass(J) = 1) = \p_{\scrF_k}(\tilde \scrB|_{\II{1, k} \X [-2T, 2T]} \in \cdot).
	$$ 
	Here the restriction to $\II{1, k} \X [-2T, 2T]$ is a restriction to the top $k$ lines and the interval $[-2T, 2T]$.
\end{lemma} 

This is a special case of Lemma 4.5 in \cite{CHH20}. To compare with that lemma, we take $X' = J$, and replace the deterministic values $\ell$ and $r$ and the set $A$ with random ${\scrF_k}$-measurable values $\mathfrak{l}$ and $\mathfrak{r}$ and the ${\scrF_k}$-measurable set $P$. As noted in \cite{CHH20} in the discussion immediately following Equation (17) (at the end of Section 4.1.6), this replacement with ${\scrF_k}$-measurable random variables follows does not affect the lemma since the claim is about ${\scrF_k}$-conditional distributions. 

The usefulness of Lemma \ref{L:-2T2T} in practice comes from the following four facts.
\begin{enumerate}[label=(\Roman*)]
	\item There exists an ${\scrF_k}$-measurable event $\Fav_{k, \ep}$ such that
	$$
	\p_{\scrF_k}(\Pass(J) = 1) \ge \exp \lf(-3973 k^{7/2} (d_{ip})^2 D_k^2 (\log \ep^{-1})^{2/3}\rg) \mathbf{1}(\Fav_{k, \ep}).
	$$
	This is Proposition 4.2 in \cite{hammond2016brownian}, quoted as Proposition 4.9 in \cite{CHH20}. For use in \cite{hammond2016brownian}, the quantity $d_{ip}$ above is a parameter related to the pole set $P$, but in \cite{CHH20} and for our purposes, we take $d_{ip} = 5d$ (see Equation (18) in \cite[Section 4.1.8]{CHH20} and surrounding discussions). Moving forward, we work with the simplified version of the above bound given by
	$$
	\p_{\scrF_k}(\Pass(J) = 1) \ge \exp \lf(-d^2 e^{b k} (\log \ep^{-1})^{2/3}\rg)  \mathbf{1}(\Fav_{k, \ep}).
	$$
	\item The event $\Fav_{k, \ep}$ satisfies
	$$
	\p(\Fav_{k, \ep}^c) \le \ep.
	$$
	This bound uses Lemma 4.1 in \cite{hammond2016brownian}, cited as Lemma 4.10 in \cite{CHH20}. 
	\item On $\Fav_{k, \ep}$, we have $[-d, d] \sset [\mathfrak{l}, \mathfrak{r}]$. This follows from the paragraph before Lemma 4.10 in \cite{CHH20}, which states that $[-T/2, T/2] \sset [\mathfrak{l}, \mathfrak{r}]$, and the discussions in \cite{CHH20} after Lemma 4.10, which shows that $[-d, d] \sset [-T/2, T/2]$. 
	\item Let $\scrC_d$ denote the space of continuous functions from $[-d, d]$ to $\R$ that vanish at $-d$, equipped with the Borel $\sig$-algebra in the topology of uniform convergence. Let $\mu_{0, *}^{[-d, d]}$ denote the law of a standard Brownian motion on $[-d, d]$ started from the initial condition $B(-d) = 0$. This is a measure on $\scrC_d$.
	Then there exists an absolute positive constant $G$ such that if $\mu_{0, *}^{[-d, d]}(A) = \ep$, then
	$$
	\p_{\scrF_k} \lf(J_k(\cdot) - J_k(-d) \in A \rg) \mathbf{1}(\Fav_{k, \ep}) \le \ep G d^{1/2} D_k^4 (\log \ep^{-1})^{4/3} \exp \lf(792 d D_k^{5/2} (\log \ep^{-1})^{5/6}\rg).
	$$
	This is Theorem 4.11 in \cite{CHH20}. The most important term to keep in mind here is the $(\log \ep^{-1})^{5/6}$ in the exponent. Moving forward, we will work with the simplified version of the bound given by
	\begin{equation}
	\label{E:FJk}
	\p_{\scrF_k} \lf(J_k(\cdot) - J_k(-d) \in A \rg) \mathbf{1}(\Fav_{k, \ep}) \le \ep \exp \lf(d e^{bk} (\log \ep^{-1} )^{5/6} \rg).
	\end{equation}
	Observe also that the inequality \eqref{E:FJk} for all $A$ with $\mu_{0, *}^{[-d, d]}(A) = \ep$ implies that
	\begin{equation}
	\label{E:FJk'}
	\p_{\scrF_k} \lf(J_k(\cdot) - J_k(-d) \in B \rg) \mathbf{1}(\Fav_{k, \ep}) \le \mu_{0, *}^{[-d, d]}(B) \exp \lf(d e^{bk} (\log \ep^{-1} )^{5/6} \rg)
	\end{equation}
	for all $B$ with $\mu_{0, *}^{[-d, d]}(B) \ge \ep$.
\end{enumerate}

In \cite{CHH20}, the authors use these three bounds with Lemma \ref{L:-2T2T} to find explicit Radon-Nikodym derivative estimates for individual parabolic Airy lines versus Brownian motion. With only slightly more work, we can upgrade these estimates to give bounds for multiple parabolic Airy lines versus several independent Brownian motions. We start with a lemma that translates the bounds on the jump ensemble to a conditional bound on parabolic Airy lines. For this lemma, we will also need to define $\Fav_{k, \ep}$ when $\ep$ does not satisfy the bound in \eqref{E:ebbb}. In this case, we set $\Fav_{k, \ep}$ to be the whole space.

\begin{lemma}
	\label{L:Bkbound}
	With $\mu_{0, *}^{[-d, d]}$ as above, for every $d \ge 1, k \in \N,$ and $\ep \in (0, 1]$, for every Borel measurable set $A$ in $\scrC_d$ with $\mu_{0, *}^{[-d, d]}(A) \ge \ep$, we have
	\begin{equation}
	\label{E:conditionalB}
	\p_{\scrF_{k}} \lf(\tilde \scrB_k(\cdot) - \tilde \scrB_k(-d) \in A \rg) \mathbf{1}(\Fav_{k,\ep}) < \mu_{0, *}^{[-d, d]} (A) \exp \lf(b d^6 + d e^{bk} (\log \ep^{-1})^{5/6} \rg).
	\end{equation}
	Moreover, if we let $f_{\scrF_k}$ denote the (random) Radon-Nikodym derivative of the random measure $\p_{\scrF_{k}} \lf(\tilde \scrB_k(\cdot) - \tilde \scrB_k(-d) \in \cdot \rg)$ with respect to $\mu_{0, *}^{[-d, d]}$, then almost surely,
	\begin{equation}
	\label{E:scrFF}
	\E_{\scrF_k} f_{\scrF_k} \mathbf{1}\lf(f_{\scrF_k} \ge \exp \lf(b d^6 + d e^{bk} (\log \ep^{-1})^{5/6} \rg)\rg) \mathbf{1}(\Fav_{k,\ep}) \le \ep \exp \lf(b d^6 + d e^{bk} (\log \ep^{-1})^{5/6} \rg).
	\end{equation}
\end{lemma}

\begin{proof}
	First, the bound \eqref{E:conditionalB} holds trivially whenever $\ep$ does not satisfy \eqref{E:ebbb} as long as $b$ is taken large enough. Therefore we may assume that \eqref{E:ebbb} holds.
	
	In this case, we can let $J$ be the jump ensemble defined with parameters $d, k$, and $\ep$.
	Then by Lemma \ref{L:-2T2T}, we can write 
	\begin{align*}
	\p_{\scrF_k} \lf(\tilde \scrB_k(\cdot) - \tilde \scrB_k(-d) \in A \rg) = \p_{\scrF_k} \lf( \tilde \scrB^{\operatorname{re}, J}_k(\cdot) - \tilde \scrB^{\operatorname{re}, J}_k(-d) \in A \; | \; \Pass(J) = 1 \rg).
	\end{align*}
	Now, by assertion (III) above, $[-d, d]\sset [\mathfrak{l}, \mathfrak{r}]$ on $\Fav_{k,\ep}$. Therefore by the definition \eqref{E:resampled} of the resampled ensemble $\tilde \scrB^{\operatorname{re}, J}_k$, on  $\Fav_{k,\ep}$ we have
	\begin{align*}
	\p_{\scrF_k} \lf( \tilde \scrB^{\operatorname{re}, J}_k(\cdot) - \tilde \scrB^{\operatorname{re}, J}_k(-d) \in A \; | \; \Pass(J) = 1 \rg) &= \p_{\scrF_k} \lf( J_k(\cdot) - J_k(-d) \in A \; | \; \Pass(J) = 1 \rg) \\
	&\le \frac{\p_{\scrF_k} \lf( J_k(\cdot) - J_k(-d) \in A \rg)}{\p_{\scrF_k} \lf( \Pass(J) = 1 \rg)}.
	\end{align*}
	By assertion (I) and \eqref{E:FJk'} above, on $\Fav_{k, \ep}$ the right-hand side above is bounded by
	\begin{equation*}
	\mu_{0, *}^{[-d, d]}(A) \exp \lf(d e^{bk} (\log \ep^{-1})^{5/6} + d^2 e^{b k} (\ep^{-1})^{2/3} \rg).
	\end{equation*}
	The bound \eqref{E:ebbb} on $\ep$ implies that this is bounded above by the right-hand side of \eqref{E:conditionalB}.
	It remains to show \eqref{E:scrFF}. 
	
	We first claim that $\p_{\scrF_{k}} \lf(\tilde \scrB_k(\cdot) - \tilde \scrB_k(-d) \in \cdot \rg)$ is absolutely continuous with respect to $\mu_{0, *}^{[-d, d]}$ and so the Radon-Nikodym derivative $f_{\scrF_k}$ is well-defined. To see this, observe that $\scrF_k$ is contained in the $\sig$-algebra $\scrG_k$ generated by $\tilde \scrB|_{S^c}$, where $S = \II{1, k} \X [-T, T]$, so it suffices to prove the same absolute continuity for $\p_{\scrG_{k}} \lf(\tilde \scrB_k(\cdot) - \tilde \scrB_k(-d) \in \cdot \rg)$. Conditional on $\scrG_k$, the Brownian Gibbs property for $\tilde \scrB$ (Theorem \ref{T:melon-Airy-facts}) ensures that the process $\tilde \scrB_k|_{[-T, T]}$ is absolutely continuous with respect to a Brownian bridge between the $\scrG_k$-measurable endpoints $\tilde \scrB_k(\pm T)$ at times $\pm T$. Since Brownian bridge increments are absolutely continuous with respect to Brownian motion increments away from the endpoints of the bridge interval, $\p_{\scrG_{k}} \lf(\tilde \scrB_k(\cdot) - \tilde \scrB_k(-d) \in \cdot \rg)$ is absolutely continuous with respect to $\mu_{0, *}^{[-d, d]}$.
	
	Next, let
	$$
	A = \lf\{f_{\scrF_k} \ge \exp \lf(b d^6 + d e^{bk} (\log \ep^{-1})^{5/6} \rg) \rg\},
	$$
	so that the left-hand side of \eqref{E:scrFF} is equal to 
	$$
	\p_{\scrF_k}(\tilde \scrB_k(\cdot) - \tilde \scrB_k(-d) \in A) \mathbf{1}(\Fav_{k, \ep}).
	$$
	By the definition of $A$, this is bounded below by 
	$$
	\mu_{0, *}^{[-d, d]} (A) \exp \lf(b d^6 + d e^{bk} (\log \ep^{-1})^{5/6} \rg) \mathbf{1}(\Fav_{k, \ep}).
	$$
	By \eqref{E:conditionalB}, this implies that $\mu_{0, *}^{[-d, d]} (A) \le \ep$. Therefore we can find a set $S$ such that $\mu_{0, *}^{[-d, d]} (A \cup S) = \ep$. Then by \eqref{E:conditionalB}, we have
	\begin{align*}
	\p_{\scrF_k}(\tilde \scrB_k(\cdot) - \tilde \scrB_k(-d) \in A) \mathbf{1}(\Fav_{k, \ep}) &\le \p_{\scrF_k}(\tilde \scrB_k(\cdot) - \tilde \scrB_k(-d) \in A \cup S) \mathbf{1}(\Fav_{k, \ep}) \\
	&\le \ep \exp \lf(b d^6 + d e^{bk} (\log \ep^{-1})^{5/6} \rg),
	\end{align*}
	giving \eqref{E:scrFF}.
\end{proof}
The next theorem is a restatement of Theorem \ref{T:radon-n-deri}.
\begin{theorem}
	\label{T:RN-joint}
	Let $\mu_d^{\bigotimes k}$ denote the law of $k$-tuples of functions in $\scrC^k_d$ given by the product of $k$ copies of $\mu_{0, *}^{[-d, d]}$. Define $\hat \scrB^k = (\hat \scrB_1, \dots \hat \scrB_k)$ by letting $\hat \scrB_i = \tilde \scrB_i(\cdot) - \tilde \scrB_i(-d)$, restricted to the interval $[-d, d]$. Then for any set $A, k \in \N$ and $d \ge 1$, we have
	$$
	\p(\hat \scrB^k \in A) \le \mu_d^{\bigotimes k}(A) \exp \lf(b k d^6 + d e^{bk} (\log [\mu_d^{\bigotimes k}(A)]^{-1})^{5/6} \rg),
	$$
	where $b>0$ is a universal constant.
\end{theorem}

\begin{proof}
	Fix $d \ge 1$. We will first show that for every $k \in \N$, and $\ep \in (0, 1]$, that there exists an $\scrF_{k}$-measurable set $\Favl_{k, \ep}$ with $\p(\Favl^c_{k, \ep}) \le k\ep$ such that for every $\scrC^k_d$-measurable set $A$ with $\mu_d^{\bigotimes k}(A) \ge \ep$, we have
	\begin{equation}
	\label{E:extra-conditional}
	\p_{\scrF_{k}} \lf(\hat \scrB^k \in A \rg)\mathbf{1}(\Favl_{k, \ep}) \le \mu_d^{\bigotimes k}(A) \lf(4 \exp \lf(b d^6 + d e^{b k} (\log \ep^{-1})^{5/6} \rg) \rg)^k.
	\end{equation}
	The set $\Favl_{k, \ep}$ is not from \cite{CHH20}, and is contained in $\Fav_{k, \ep}$ but does not have an explicit construction. Think of the `!' as indicating that it is an extra favourable version of the set $\Fav_{k, \ep}$.
	To prove \eqref{E:extra-conditional} we use induction on $k$. For the proof of \eqref{E:extra-conditional}, we fix the constant $b$, since increasing $b$ during the inductive step would be problematic.
	The $k = 1$ case for all $\ep \in (0, 1]$ is given in Lemma \ref{L:Bkbound} with the set $\Favl_{k, \ep} = \Fav_{k, \ep}$. Now suppose that the claim holds for $k-1$ and all $\ep \in (0, 1]$.
	Let $A \sset \scrC_d^k$ be a Borel measurable set with $\mu_d^{\bigotimes k}(A) = \ep$. For every $x \in \scrC_d$, define the fibre
	$$
	A_{x} = \{y \in \scrC_d^{k-1} : (y, x) \in A \}.
	$$
	Then we can write
	\begin{equation}
	\label{E:conditioning}
	\p_{\scrF_{k}}(\hat \scrB^k \in A) = \p_{\scrF_{k}}(\hat \scrB^{k-1} \in A_{\hat \scrB_k})  = \E_{\scrF_{k}} \lf(\p_{\scrF_{k-1}}(\hat \scrB^{k-1} \in A_{\hat \scrB_k})\rg).
	\end{equation}
	where the last equality uses that $\scrF_{k} \sset \scrF_{k-1}$.
	We use the inductive hypothesis to estimate $\p_{\scrF_{k-1}}(\hat \scrB^{k-1} \in A_{\hat \scrB_k})$. First, let $S$ be any set with $\mu_d^{\bigotimes k - 1}(S) = \ep^2$. 
	Then we can write
	\begin{align}
	\nonumber
	\p_{\scrF_{k-1}}(\hat \scrB^{k-1} \in A_{\hat \scrB_k}) \le\; &\p_{\scrF_{k-1}}(\hat \scrB^{k-1} \in A_{\hat \scrB_k} \cup S)\\
	\nonumber
	\le\; &\p_{\scrF_{k-1}}(\hat \scrB^{k-1} \in A_{\hat \scrB_k} \cup S)\mathbf{1}(\Favl_{k-1,\ep^2}) + \mathbf{1}(\Favl^c_{k-1,\ep^2}) \\
	\label{E:PFk-1}
	\le\; & \mu_d^{\bigotimes k - 1}(A_{\hat \scrB_k} \cup S)\lf(4 \exp \lf(b d^6 + d e^{b (k-1)} (\log \ep^{-1})^{5/6} \rg) \rg)^{k-1}  + \mathbf{1}(\Favl^c_{k-1,\ep^2}).
	\end{align}
	Here the final inequality uses the inductive hypothesis, and the fact that $\mu_d^{\bigotimes k - 1}(A_{\hat \scrB_k} \cup S)$ is always greater than $\ep^2$.
	Next, we want to apply $\E_{\scrF_{k}}$ to the right-hand side of \eqref{E:PFk-1}. We start with the term $\mu_d^{\bigotimes k - 1}(A_{\hat \scrB_k} \cup S)$. As in Lemma \ref{L:Bkbound}, let $f_{\scrF_k}$ denote the Radon-Nikodym derivative of $\p_{\scrF_{k}} \lf(\hat \scrB_k \in \cdot \rg)$ with respect to $\mu_{0, *}^{[-d, d]}$. Letting $W$ be an independent Brownian motion drawn from the distribution $\mu_{0, *}^{[-d, d]}$, we can write
	\begin{align}
	\nonumber
	\E_{\scrF_{k}} \mu_d^{\bigotimes k - 1}(A_{\hat \scrB_k} \cup S) = & \; \E_{\scrF_{k}} f_{\scrF_k}(W) \mu_d^{\bigotimes k - 1}(A_{W} \cup S) \\
	\label{E:bigtobd}
	\le &  \;\E_{\scrF_k} f_{\scrF_k}(W) \mathbf{1}\lf(f_{\scrF_k}(W) \ge \exp \lf(b d^6 + d e^{bk} (\log \ep^{-1})^{5/6} \rg)\rg)  \\
	\nonumber
	&+\exp \lf(b d^6 + d e^{bk} (\log \ep^{-1})^{5/6} \rg) \E_{\scrF_{k}} \mu_d^{\bigotimes k - 1}(A_{W} \cup S).
	\end{align}
	Now, by the definition of the sets $A_x$ and a union bound, we have
	$$
	\E_{\scrF_{k}} \mu_d^{\bigotimes k - 1}(A_{W} \cup S) \le \E_{\scrF_{k}}[ \mu_d^{\bigotimes k - 1}(A_{W}) + \mu_d^{\bigotimes k - 1}(S)] = \mu_d^{\bigotimes k}(A) + \ep^2 \le 2\ep.
	$$ 
	Also, on the event $\Fav_{k, \ep}$ in Lemma \ref{L:Bkbound}, we can bound the first term on the right-hand side of \eqref{E:bigtobd} above using \eqref{E:scrFF}.  Therefore
	\begin{equation}
	\label{E:3ep-bd}
	\E_{\scrF_{k}} \mu_d^{\bigotimes k - 1}(A_{\hat \scrB_k} \cup S) \mathbf{1}(\Fav_{k, \ep}) 
	\le
	3 \ep \exp\lf(b d^6 + d e^{bk} (\log \ep^{-1})^{5/6} \rg).
	\end{equation}
	We now bound the second term on the right-hand side of \eqref{E:PFk-1}. We have
	$$
	\E \E_{\scrF_{k}} (\mathbf{1}(\Favl^c_{k-1,\ep^2})) = \p (\Favl^c_{k-1,\ep^2}) \le (k-1)\ep^2,
	$$
	by the inductive hypothesis,
	so by Markov's inequality, we have
	\begin{equation}
	\label{E:ep^2-up}
	\E_{\scrF_{k}} (\mathbf{1}(\Favl^c_{k-1,\ep^2})) \le \ep
	\end{equation}
	on a set $B$ of probability $1-(k-1)\ep$. We now set $\Favl_{k, \ep} = B \cap \Fav_{k,\ep}$. Assertion (II) and a union bound shows that $\p(\Favl^c_{k,\ep}) \le k \ep$. Finally, gathering the inequalities \eqref{E:PFk-1}, \eqref{E:3ep-bd}, and \eqref{E:ep^2-up}, we have
	\begin{equation*}
	\begin{split}
	&\E_{\scrF_{k}} \p_{\scrF_{k-1}}(\hat \scrB^{k-1} \in A_{\hat \scrB_k}) \mathbf{1}(\Favl_{k, \ep}) \\
	\le &  \;\lf(4 \exp \lf(b d^6 + d e^{b (k-1)} (\log \ep^{-1})^{5/6} \rg) \rg)^{k-1}3 \ep \exp\lf(b d^6 + d e^{bk} (\log \ep^{-1})^{5/6} \rg)  + \ep.
	\end{split}
	\end{equation*} 
	This is bounded above by \eqref{E:extra-conditional} when $\mu_d^{\bigotimes k}(A) = \ep$. As in \eqref{E:FJk'}, the extension of \eqref{E:extra-conditional} to all $A$ with $\mu_d^{\bigotimes k}(A) \ge \ep$ is immediate.
	
	The theorem then follows by averaging over $\scrF_k$. More precisely, let $A$ be any set, and define $\ep = \mu_d^{\bigotimes k}(A)$. Then
	\begin{align*}
	\p(\hat \scrB \in A) &\le \E \p_{\scrF_k}(\hat \scrB \in A) \mathbf{1}(\Favl_{k, \ep}) + \p(\Favl_{k, \ep}^c) \\
	&\le \ep \lf(4 \exp \lf(b d^6 + d e^{b k} (\log \ep^{-1})^{5/6} \rg) \rg)^k + k \ep,
	\end{align*}
	This gives the desired bound after increasing $b$.
\end{proof}

\end{document}